\documentclass{amsart}
\usepackage{amsmath, amssymb, graphicx, epsfig, verbatim, psfrag, color,amscd,stmaryrd} 
\usepackage{tikz}
\usepackage{amsthm}
\usepackage{enumitem}
\usetikzlibrary{arrows}

\newcommand\MGradingSet{S}

\newcommand\OurRing{\mathcal R}
\newcommand\PartInv{\mathcal Q}
\newcommand\orK{\vec{K}}
\newcommand\gr{\mathbf{gr}}

\newcommand\Partition{\Matching}
\newcommand\Matching{M}
\newcommand\DuAlg{{\mathcal A}'}

\usetikzlibrary{matrix,arrows,positioning}
\tikzset{cdlabel/.style={above,sloped,
    execute at begin node=$\scriptstyle,execute at end node=$}}
\tikzset{algarrow/.style={->, thick}}
\tikzset{blgarrow/.style={->, thick}}
\tikzset{clgarrow/.style={->, thick}}
\tikzset{tensoralgarrow/.style={double, double equal sign distance, -implies}}
\tikzset{tensorblgarrow/.style={double, double equal sign distance, -implies}}
\tikzset{tensorclgarrow/.style={double, double equal sign distance, -implies}}
\tikzset{modarrow/.style={->, dashed}}
\tikzset{othmodarrow/.style={->, thick}}
\tikzset{Amodar/.style={->, dashed}}
\tikzset{Dmodar/.style={->, dashed}}
\newcommand\op{\mathrm{op}}
\newcommand\opp{\op}

\newcommand\TerMin{t\GenMin}

\newcommand\Min{\GenMin}
\newcommand\GenMin{\mho}
\newcommand\KHm{{\mathcal H}^-}

\newcommand\KCm{{\mathcal C}^-}

\newcommand\Diag{\mathcal D}

\newcommand\Cwz{\mathcal{C}}
\newcommand\Hwz{\mathcal{J}}

\def\endproof{\relax\ifmmode\expandafter\endproofmath\else
  \unskip\nobreak\hfil\penalty50\hskip.75em\hbox{}\nobreak\hfil\bull
  {\parfillskip=0pt \finalhyphendemerits=0 \bigbreak}\fi}
\def\endproofmath$${\eqno\bull$$\bigbreak}
\def\bull{\vbox{\hrule\hbox{\vrule\kern3pt\vbox{\kern6pt}\kern3pt\vrule}\hrule}}

\def\mathcenter#1{%
  \vcenter{\hbox{$#1$}}%
}
\newcommand\CanonDD{\mathcal K}
\newcommand\North{\mathbf N}
\newcommand\South{\mathbf S}
\newcommand\East{\mathbf E}
\newcommand\West{\mathbf W}

\newcommand\Pos{\mathcal P}
\newcommand\Neg{\mathcal N}
\newcommand\Upwards{\mathcal{S}}

\newtheorem{thm}{Theorem}[section]

\newtheorem{lemma}[thm]{Lemma}
\newtheorem{prop}[thm]{Proposition}
\newtheorem{defn}[thm]{Definition}

\newtheorem{example}[thm]{Example}
\newtheorem{rem}[thm]{Remark}
\newtheorem{remark}[thm]{Remark}

\numberwithin{equation}{section}

\newcommand\OneHalf{\frac{1}{2}}
\newcommand{\smargin}[1]{}

\newcommand\IdempRing{\mathbf{I}}
\newcommand\Idemp[1]{\mathbf{I}_{#1}}
\newcommand\DT{\boxtimes}
\newcommand\x{\mathbf x}
\newcommand\w{\mathbf w}
\newcommand\y{\mathbf y}

\newcommand\lsup[2]{^{#1}{#2}}
\newcommand\lsub[2]{{}_{#1}{#2}}

\newcommand{\AlgB}{{\mathcal B}}

\newcommand{\Mor}{\mathrm{Mor}}

\newcommand\z{\mathbf z}

 \newcommand{\Z}{\mathbb Z}  \newcommand{\Q}{\mathbb Q} \newcommand{\R}{\mathbb R}

\newcommand{\HFa}{{\widehat {\rm {HF}}}}

\newcommand\XX{\mathbf X}
\newcommand\YY{\mathbf Y}
\newcommand\ZZ{\mathbf Z}
\newcommand\Max{\Omega}
\newcommand\Crit{\mathcal E}

\newcommand\Field{\mathbb F}
\DeclareMathOperator{\Hom}{Hom}

\DeclareMathOperator{\Id}{Id}

\newcommand\VRot{\mathcal R}
\newcommand\Alg{\mathcal A}
\newcommand\Blg{\mathcal B}
\newcommand\Clg{\mathcal C}

\newcommand\Tensor{\mathcal T}
\newcommand\Ainf{{\mathcal A}_{\infty}}
\newcommand\Ainfty\Ainf
\newcommand\Zmod[1]{{\mathbb Z}/{#1}{\mathbb Z}}

\newcommand\AlgBZ{\AlgB_0}
\newcommand\BlgZ\AlgBZ
\newcommand\Opposite{o}
\newcommand\KC{C^-}
\newcommand\MinGen{\mathbf T}
\newcommand\DDmin{\mho}
\newcommand\Ynew{\mathcal Y}
\newcommand\BigMin{\Theta}
\newcommand\gen{K}

\newcommand\AlgLoc{\Alg^{\mathrm{loc}}}

\makeatletter
\renewenvironment{proof}[1][\proofname]{\par
\pushQED{\qed}%
\normalfont \topsep6\p@\@plus6\p@\relax
\trivlist
\item\relax
{\bf#1\@addpunct{.}}\hspace\labelsep\ignorespaces
}{%
\popQED\endtrivlist\@endpefalse
}
\makeatother

\begin{document}
\title{Bordered knot algebras with matchings}

\author[Peter S. Ozsv\'ath]{Peter Ozsv\'ath}
\thanks {PSO was supported by NSF grant number DMS-1405114 and DMS-1708284}
\address {Department of Mathematics, Princeton University\\ Princeton, New Jersey 08544} 
\email {petero@math.princeton.edu}

\author[Zolt{\'a}n Szab{\'o}]{Zolt{\'a}n Szab{\'o}}
 \thanks{ZSz was supported by NSF grant numbers DMS-1309152 and DMS-1606571}
\address{Department of Mathematics, Princeton University\\ Princeton, New Jersey 08544}
\email {szabo@math.princeton.edu}

\begin{abstract}
  This paper generalizes the bordered-algebraic knot invariant introduced in an
  earlier paper, giving an invariant now with more algebraic
  structure. It also introduces signs to define these invariants with
  integral coefficients. We describe effective computations of the
  resulting invariant.
\end{abstract}

\newcommand\Hw{{\mathcal J}^{U}}
\maketitle
\section{Introduction}
\label{sec:Intro}

The aim of the present paper is to generalize the bordered-algebraic
knot invariant from~\cite{BorderedKnots}.  This generalization gives a
knot invariant with more algebraic structure than the earlier
invariant.  Another aim is to develop techniques that make both the
new invariant and the one from~\cite{BorderedKnots} more readily
computable.  In fact, C++ code for computing these invariants
can be found here~\cite{Program}.
The constructions presented here are very similar to and build on the material
from~\cite{BorderedKnots}. For other similar algebraic constructions, 
see~\cite{HFa,KhovanovFunctor,KhovanovSeidel,Manion,PetkovaVertesi,Zarev}.
The identification between these constructions and constructions
from the {\em knot Floer homology} of~\cite{OSKnots,RasmussenThesis}
will be given in~\cite{HolKnot}.
A further generalization of this algebra will be given in~\cite{Pong}.

The basic set-up is the following. We start with a collection of $2n$
points on the real line, which we will think of as
$\{1,\dots,2n\}$ equipped with a {\em
  matching}, which is a partition $\Partition$ 
of $\{1,\dots,2n\}$ into pairs of points.
To the collection of $2n$ points on the real line, equipped with a
matching and an integer $k$ with $0\leq k\leq 2n+1$, we associate an
algebra.  Like the algebras from~\cite{BorderedKnots}, idempotents 
$\Idemp{\x}$ in
the algebra correspond to {\em idempotent states}
$\x\subset\{0,\dots,2n\}$ with $|\x|=k$. Additional generators
for the algebra are elements $\{L_i\}_{i=1,\dots,2n}$ and
$\{R_i\}_{i=1,\dots,2n}$, where
\begin{align*}
  L_i &= \sum_{\{\x\big|i\in \x, i-1\not\in \x\}}
  \Idemp{\x}\cdot L_i \cdot\Idemp{(\x\setminus\{i\})\cup\{i-1\}} \\
  R_i &= \sum_{\{\x\big|i-1\in \x, i\not\in \x\}}
  \Idemp{\x}\cdot R_i \cdot\Idemp{(\x\setminus\{i-1\})\cup\{i\}}
\end{align*}
and $\{U_i\}_{i=1,\dots,2n}$, which in certain idempotents can be
expressed as $L_i\cdot R_i$ or $R_i\cdot L_i$. 
The remaining generators for the algebra are supplied by the matching:
for each $\{a,b\}\in{\Partition}$, there is a
central algebra element $C_{\{a,b\}}$ with $d C_{\{a,b\}}=U_a \cdot
U_b$ and $C_{\{a,b\}}^2=0$. Denote this algebra $\Alg(n,k,{\Partition})$.
(For details on these constructions, see Section~\ref{sec:Algebras}.)

We wish to construct a chain complex associated to a knot diagram
$\Diag$.  As a preliminary point, observe that for a generic value of
$t$, the slice in the $(x,y)$ plane with $y=t$ meets $\Diag$ in a
collection of $2n$-points; moreover, the arcs in the portion of the
diagram with $y\geq t$ induce a matching on the $y=t$-points in the
diagram. Thus, for generic $t$, there is an associated algebra for the
$y=t$ slice of the diagram. Now, slice up the diagram into strips;
i.e.  restrict the diagram to the strip in the $(x,y)$-plane with
$t_i\leq y\leq t_{i+1}$, for a suitable increasing sequence
$t_1,\dots,t_m$.  These pieces are called {\em partial knot diagram}s,
and by using sufficiently thin slices, and assuming that the knot
projection is in general position, we can assume that each partial
knot diagram contains either one maximum, one minimum, or one
crossing. To these basic pieces, we associate a $DA$ bimodule, whose
incoming algebra is associated to the $y=t_{i+1}$-slice of the knot
diagram, and whose outgoing algebra is associated to the
$y=t_{i}$-slice.

Tensoring together $DA$ bimodules associated to partial knot diagrams,
we obtain a chain complex
$\Cwz(\Diag)$ associated to an oriented knot diagram $\Diag$ with a
distinguished edge, which we think of as containing the unique global
minimum.  Generators of this chain complex correspond to Kauffman
states for the Alexander polynomial as in~\cite{Kauffman}.

The resulting complex has the following algebraic structure.
Consider the bigraded ring $\OurRing=\Field[U,V]/U\cdot V = 0$,
equipped with a $\Delta$-grading with
\[ \Delta(U)=\Delta(V)=-1\]
and an {\em Alexander}-grading $A$, determined by 
\[ A(U)=-1,\qquad A(V)=+1.\]
An orientation on the knot gives 
the complex $\Cwz(\Diag)$ 
the structure of a  bigraded module over that ring; i.e.
the complex is also equipped with two integer-valued gradings, called $\Delta$ and $A$,
i.e. 
$\Cwz(\Diag)=\bigoplus_{\delta,s} \Cwz_{\delta}(\Diag,s)$,
and \[
  U\colon \Cwz_{\delta}(\Diag,s)\to \Cwz_{\delta-1}(\Diag,s-1),\qquad
  V\colon \Cwz_{\delta}(\Diag,s)\to \Cwz_{\delta-1}(\Diag,s+1).
\]
\[  \partial\colon \Cwz_{\delta}(\Diag,s)\to \Cwz_{\delta-1}(\Diag,s) \]

Recall in~\cite{BorderedKnots}, we used a Maslov grading $M$; that is
related to $\Delta$ and $A$ by the formula
\[ \Delta=M-A.\] While the local formulas take values in $\OneHalf\Z$,
for a knot, summing over the local contributions, the gradings of
generators take values in $\Z$. (For example, the Alexander grading of a generator
corresponding to a Kauffman state is the exponent of $t$ in the
corresponding contribution to the Alexander polynomial in the state
sum formula; see~\cite{Kauffman}.)

 \begin{figure}[h]
 \centering
 \input{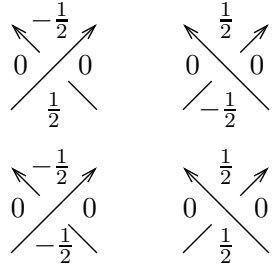}
 \caption{{\bf Local Alexander and
     $\Delta$-contributions.}  The first and second rows illustrates 
   the Alexander and $\Delta$-gradings  of each
   quadrant, respectively.}
 \label{fig:LocalCrossing}
 \end{figure}

Taking homology of this complex, we obtain a bigraded
module $\Hwz(\orK)=H(\Cwz({\mathcal D}))$  over $\OurRing$; i.e.
\[ \Hwz(\orK)=\bigoplus_{\delta,s} \Hwz_{\delta}(\orK,s),\]
with
\[ U\colon \Hwz_{\delta}(\orK,s)\to \Hwz_{\delta-1}(\orK,s-1)
\qquad 
V\colon \Hwz_{\delta}(\orK,s)\to \Hwz_{\delta-1}(\orK,s+1)\]

As the notation suggests, these modules are invariants of the knot:

\begin{thm}
  \label{thm:Invariance}
  The bigraded module $\Hwz(\orK)$ is an invariant of the oriented knot $\orK$.
  The module itself is the homology of a chain complex whose generators correspond to Kauffman states.
\end{thm}

Setting $U=V=0$, and taking the homology of the resulting complex, we
obtain a knot invariant whose bigraded Euler characteristic (using the
Alexander and Maslov gradings) is the symmetrized Alexander
polynomial. This is obvious from the local description of the
bigradings from Figure~\ref{fig:LocalCrossing}; for more on this,
see~\cite{BorderedKnots}.

The complex $\Cwz(\Diag)$ is constructed in
Section~\ref{sec:Construction}, building on the work from earlier
sections; Theorem~\ref{thm:Invariance} is also proved in that section,
where the base algebra is taken with coefficients in
$\Field=\Zmod{2}$, a simplification which we make for the rest of this
introduction. The generalization of the construction with coefficients
in $\Z$ is given in Section~\ref{sec:Signs}.

Given any bigraded $\OurRing$-module $M$, there is a new bigraded module $S(M)$ whose underlying Abelian group is the same as that for $M$, equipped with a map
\[ \sigma\colon M \to S(M) \]
that restricts to $\Field$-isomorphisms $\sigma\colon M_{\delta,s}\to S(M)_{\delta,-s}$, and satisfies the rule
\[ \sigma(U\cdot m)=V\cdot \sigma(m)\qquad{\text{and}}\qquad \sigma(V\cdot m)=U\cdot \sigma(m).\]
Our bordered knot invariants are symmetric in the following sense:

\begin{prop}
        \label{prop:Symmetry}
        The bigraded module $\Hwz(\orK)$ is symmetric, in the sense that $\Hwz(\orK)\cong S(\Hwz(-\orK))$.
\end{prop}

Let $\KCm(\orK)$ be the bordered knot complex from~\cite{BorderedKnots};
which is a bigraded module over $\Field[U]$.
The relationship between the $\KCm(\orK)$
from~\cite{BorderedKnots} and $\Cwz(\orK)$ is summarized as follows:

\begin{prop}
  \label{intro:Compare}
  There is a homotopy equivalence of bigraded chain complexes 
  \[\frac{\KCm(-\orK)}{U=0}\simeq \frac{\Cwz(\orK)}{U=V=0}.\]
  Moreover, $\Cwz(\orK)/(V=0)\cong \KCm(-\orK)$.
\end{prop}

See Section~\ref{sec:Compare} for a proof of a more detailed version.

\subsection{Numerical knot invariants}
\label{intro:NumKnotInv}

The above bigraded knot invariants naturally give rise to some
numerical knot invariants. Their construction hinges on the behaviour of 
$\Hwz(\orK)$ under crossing changes (see Proposition~\ref{prop:CrossingChangeHwz}),
which leads to the  following structure theorem.

For the statement of the structure theorem, will consider the module
$\Hwz(\orK)\otimes_{\Field[U]}\Field[U,U^{-1}]$, which can be thought
of as obtained from $\Hwz(\orK)$ by inverting $U$.

\begin{prop}
        \label{prop:UNonTorsion}
        For any oriented knot $\orK$, we have an identification of
        $\OurRing$-modules $\Hwz(\orK)\otimes_{\Field[U]}
        \Field[U,U^{-1}]\cong \Field[U,U^{-1}]$, where on the right
        hand side, $\Field[U,U^{-1}]$ is thought of as the $\OurRing$
        module for which $V$ acts as multiplication by $0$.
\end{prop}  

It follows quickly from the 
above proposition (see Lemma~\ref{lem:NuMakesSense})
that we can make the following:

\begin{defn} 
   Given any oriented knot $\orK$, let ${\underline\nu}(\orK)\in\Z$ be the knot invariant defined by
   \[ {\underline\nu}(\orK)=-\max\{s\big| U^d\cdot \Hwz(\orK,s)\neq 0~~ \forall d\geq 0\}.\]
\end{defn}

The above defined integer ${\underline\nu}(\orK)$ is clearly a knot invariant, since $\Hwz(\orK)$ is.
Consider $\Hw(\orK)=H(\Cwz(\Diag)/V=0)$. (Recall that $\Hw(\orK)=\KHm(-\orK)$,
by  Proposition~\ref{intro:Compare}.)
There is a canonical map
$\Hwz(\orK)\to \Hw(\orK)$. Clearly, any $U$-nontorsion element in $\Hwz(\orK)$ has $U$-non-torsion image in $\Hw(K)$.
Moreover, for all sufficiently large $s$, $\Hw(\orK,s)$ is zero. Thus, we have the following:

\begin{defn} 
   Given any oriented knot $\orK$, let ${\underline\tau}(\orK)\in\Z$ be the knot
      invariant defined by \[ {\underline\tau}(\orK)=-\max\{s\big|
      U^d\cdot \Hw(\orK,s)\neq 0~~ \forall d\geq 0\}.\]
\end{defn}

Clearly, ${\underline\tau}(\orK)\leq {\underline\nu}(\orK)$. 

Our crossing change result stated in Section~\ref{sec:CrossingChange}
implies the following inequalities:
\begin{prop}
  \label{prop:BoundUnknotting}
        Let $\orK_+$ and $\orK_-$ be two oriented knots that differ in a single crossing, which is positive for $\orK_+$ and negative for $\orK_-$;
        then
        \begin{align*}  
            0 & \leq {\underline \tau}(\orK_+)-{\underline \tau}(\orK_-)\leq 1 \\
            0 & \leq {\underline \nu}(\orK_+)-{\underline \nu}(\orK_-)\leq 1.
        \end{align*}
        In particular, for any knot $\orK$, the unknotting number $u(\orK)$ is bounded by:
        \[ |{\underline\tau}(\orK)|\leq u(\orK)\qquad{\text{and}}\qquad |{\underline\nu}(\orK)|\leq u(\orK).\]
\end{prop}

The invariants ${\underline \tau}(\orK)$ and ${\underline \nu}(\orK)$
are analogues of the knot invariants $\tau(K)$ and $\nu(K)$, defined
in~\cite{FourBall} and~\cite{RatSurg} respectively, using the
holomorphic construction; see also~\cite{RasmussenThesis,HomCables}.

\subsection{Organization}
This paper is organized as follows. In Section~\ref{sec:Algebras}, we
define the relevant algebras, and verify some of their basic formal
properties.  Indeed, in that section we define two algebras $\Alg$ and
$\DuAlg$ which are Koszul dual to one another; we prefer to work with
$\Alg$ whenever possible.  In Section~\ref{sec:Cross} we construct the
bimodules associated to a crossing (both of type $DD$ and of type
$DA$, over $\Alg$), and verify that they satisfy the braid relations
in Section~\ref{sec:BraidRelations}.  In Section~\ref{sec:Max}, we
define the bimodule associated to a local maximum.  In
Section~\ref{sec:DuAlg}, we give DA bimodules associated to crossings
over $\DuAlg$ (instead of $\Alg$, as was constructed earlier).  In
Section~\ref{sec:Min} we construct the bimodule associated to a local
minimum, using the material from Section~\ref{sec:DuAlg} to verify
that this bimodule satisfies the ``trident relation''. In
Section~\ref{sec:Construction}, the pieces are assembled to define the
knot invariant $\Hwz(K)$. In that section, we also prove topological
invariance of the construction. In Section~\ref{sec:Compare}, we
identify specializations of these constructions with the invariants
from~\cite{BorderedKnots}. In that section, we are also able to verify
that the specialized knot invariants are multiplicative under
connected sums.  In Section~\ref{sec:Symmetries}, we verify some
symmetries of our constructions.  In Section~\ref{sec:CrossingChange},
we investigate how the bordered invariants change under crossing
changes, to verify Proposition~\ref{prop:BoundUnknotting}.  In
Section~\ref{sec:Signs}, we explain how to lift the present invariant
to coefficients in $\Z$.
Finally, in Section~\ref{sec:Fast}, we describe some methods for 
optimizing the computer computations of these invariants.

\subsection{Further remarks}

Note that there is an analogous construction of knot Floer homology
with coefficients in $\OurRing$: the differential in that construction
counts pseudo-holomorphic disks in a homotopy class $\phi$, weighted
by $U^{n_{w}(\phi)} V^{n_{z}(\phi)}$. Since $UV=0$, the differential
is counting disks that can cross either one of the two
basepoints. This construction is sufficient for computing $\HFa$ for
surgeries on $K$~\cite{IntSurg,RatSurg}, the concordance invariant
$\tau(K)$~\cite{FourBall}, and indeed Hom's countable collection of
concordance homomorphisms~\cite{HomConc} (defined using $\tau$ and
$\nu$ from~\cite{RatSurg}).  In forthcoming
work~\cite{HolKnot}, we identify the presently defined
invariants with their knot-Floer homological analogues.

In view of that identification, the concordance invariants $\tau(K)$
and $\nu(K)$ agree ${\underline\tau}(K)$ and ${\underline\nu}(K)$
described above; and so we will be able to conclude
Proposition~\ref{prop:BoundUnknotting} from corresponding properties
of the holomorphically defined invariants.  Nonetheless, it is
instructive to have a self-contained proof of
Proposition~\ref{prop:BoundUnknotting} within the present algebraic
framework;
compare~\cite{RasmussenSlice}, \cite{Sarkar}, and \cite[Chapters~6~and
8]{GridBook}.

\newcommand\Image{\mathrm{Image}}
\newcommand\Alex{\mathrm{Alex}}
\newcommand\Filt{\mathcal{F}}
\newcommand\Maslov{\mathrm{Maslov}}
\newcommand\SubCx{{\mathcal E}}

\section{Bordered algebras}
\label{sec:Algebras}

Throughout most of this paper, we suppress signs,
working with coefficients over $\Zmod{2}$.
See Section~\ref{sec:Signs} for the generalization to coefficients in $\Z$.

\subsection{Previous bordered algebras for knot diagrams}

We recall the construction of the algebras $\BlgZ(m,k)$ and $\Blg(m,k)$ from~\cite{BorderedKnots}.
Fix integers $k$ and $m$ with $0\leq k\leq m+1$. 
The algebra $\BlgZ(m,k)$ is an algebra over $\Field[U_1,\dots,U_m]$, whose basic idempotents
correspond to {\em idempotent states}, or {\em $I$-states} which are $k$-element subsets of $\{0,\dots,m\}$. 
Given an I-state $\x$, define its {\em weight}
$v^{\x}\in \Z^{m}$ by
\begin{equation}
  \label{eq:DefOfV}
  v^{\x}_{i} = \# \{x\in \x\big| x\geq i\}.
\end{equation}
Given two I-states $\x$ and $\y$, define their {\em minimal relative weight vector} $w^{\x,\y} \in (\frac{1}{2} \Z)^m$ to be vector with components
\[ w^{\x,\y}_{i} = \frac{1}{2}\big|v^{\x}_i-v^{\y}_i\big|.\]
$\AlgB_0(m,k)$ is defined so that  there is an identification of $\Field[U_1,\dots,U_m]$-modules
$\Idemp{\x} \cdot \AlgB_0(m,k) \cdot \Idemp{\y} \cong \Field[U_1,\dots,U_m]$; denote the identification by 
\[ \phi^{\x,\y}\colon \Field[U_1,\dots,U_m] \rightarrow 
\Idemp{\x} \cdot \AlgB_0(m,k) \cdot \Idemp{\y}. \]

An element of $\BlgZ(m,k)$ is called {\em pure} if it is of the form
$\phi^{\x,\y}(U_1^{t_1}\cdots U_m^{t_m})$ for some non-negative sequence of integers
$t_1,\dots,t_m$.

A grading by $(\OneHalf \Z)^{m}$ on $\Idemp{\x} \cdot \AlgB_0(m,k) \cdot \Idemp{\y}$  is specified by its values on pure algebra elements:
\begin{equation}
  \label{eq:WeightFunction}
  w(\phi^{\x,\y}(U_1^{t_1}\dots U_m^{t_m}))= w^{\x,\y}+(t_1,\dots,t_m).
\end{equation}
An element in $\BlgZ(m,k)$
is called {\em homogeneous of degree $(w_1,\dots,w_m)$} if it can be written
as a sum of pure algebra elements, all of which have 
weight $(w_1,\dots,w_m)$.

Multiplication 
\[ 
\Big(\Idemp{\x} \cdot \AlgB_0(m,k) \cdot \Idemp{\y} \Big)
*
\Big(\Idemp{\y} \cdot \AlgB_0(m,k) \cdot \Idemp{\z} \Big)
\to \Big(\Idemp{\x} \cdot \AlgB_0(m,k) \cdot \Idemp{\z} \Big)
\] 
is the unique non-trivial, grading-preserving
$\Field[U_1,\dots,U_m]$-equivariant map.

If $\x$ is an I-state with $j-1\in \x$ but $j\not\in\x$, we
can form a new I-state $\y=\x\cup \{j\}\setminus\{j-1\}$. Let 
$R_j^\x = \phi^{\x,\y}(1)$,
$L_j^\y = \phi^{\y,\x}(1)$, 
\[ R_{j} = \sum_{\{\x\big|j-1\in \x, j\not\in\x\}} R_j^\x
\qquad{\text{and}}\qquad L_j = \sum_{\{\y\big| j\in \y, j-1\not\in\y\}} L_j^\y.\]

The algebra
$\AlgB(m,k)$ (also denoted $\AlgB(m,k,\emptyset)$ in~\cite{BorderedKnots})
is the quotient of $\AlgBZ(m,k)$ by the relations
\[
 L_{i+1}\cdot L_i = 0, \qquad
 R_i\cdot R_{i+1}=0; \]
and also, if 
$\{x_1,...,x_k\}\cap \{j-1,j\}=\emptyset$, then 
\begin{equation}
\label{eq:KillUs} 
\Idemp{\x}\cdot U_j=0.
\end{equation}

We will use the following terminology from~\cite{BorderedKnots}:

\begin{defn}
  \label{def:CloseEnough}
Let $\x$ and $\y$ be two $I$-states for $\Blg(m,k)$. Write
$\x=x_1<\dots<x_k$ and $\y=y_1<\dots<y_k$. If for all $i=1,\dots,k$,
$|x_i-y_i|\leq 1$, we say that $\x$ and $\y$ are {\em close enough}.
\end{defn}

According to~\cite[Proposition~\ref{BK:prop:IdentifyJ}]{BorderedKnots},
the $I$-states $\x$ and $\y$ are close enough if and only if
$\Idemp{\x}\cdot\Blg\cdot \Idemp{\y}\neq 0$.

The grading on $\BlgZ(m,k)$ by $(\OneHalf \Z)^m$ descends to a grading
on $\Blg(m,k)$. A non-zero element of $\Blg(m,k)$ is called {\em pure} if
it is the image of a pure algebra element in $\BlgZ(m,k)$.

For more on the construction of $\Blg(m,k)$,
see~\cite[Section~\ref{BK:sec:Algebras}]{BorderedKnots}.

\subsection{Algebras associated to matchings}

In~\cite{BorderedKnots}, we constructed various algebras containing
$\Blg(m,k)$, associated to orientations on the points $\{1,\dots,m\}$
or, equivalently, a subset of the points $\{1,\dots,m\}$. In the
present paper, we associate instead an algebra containing
$\Blg(2n,k)$, associated to matchings, as follows.  Let $\Partition$
be a {\em matching}, a partition of $\{1,\dots,2n\}$ into $n$
two-element subsets. Fix any integer $0\leq k \leq 2n+1$.  We will
describe presently two algebras associated to matchings,
$\Alg(n,k,\Partition)$ and $\Alg'(n,k,\Partition)$, both containing
the algebra $\Blg(2n,k)=\Blg(2n,k,\emptyset)$ from above.

The algebra $\Alg(n,k,\Partition)$ is obtained from $\Blg(2n,k)$ by
including $n$ further central elements $C_{\{i,j\}}$
for each $\{i,j\}\in\Partition$, satisfying the properties that
\[ C_{\{i,j\}}^2=0 \qquad dC_{\{i,j\}}=U_i U_j. \]

For $\DuAlg(n,k,\Partition)$, we start from $\Blg(2n,k)$, and  include $2n$ further  algebra
elements $E_i$, one for each $i\in\{1,\dots,2n\}$, so that 
$E_i^2=0$, $dE_i= U_i$, $E_i\cdot b=b\cdot E_i$ for all $b\in\Blg(2n,k)$;
and
\[ E_i \cdot E_j + E_j \cdot E_i=0 \qquad{\text{if $\{i,j\}\not\in\Partition$}}.\]
Observe that for each $\{i,j\}\in\Partition$, there is an associated non-zero algebra element
\[ \llbracket E_i,E_j\rrbracket=E_i E_j+E_jE_i,\] which is closed and central.

There are two gradings on $\Alg(n,k,\Matching)$; one is an {\em
  Alexander multi-grading} with values in $\Q^{2n}$.  To define this,
we say that an algebra element in $\Alg(n,k,\Matching)$ is {\em pure}
if it is the product of a pure algebra element $b\in\Blg(2n,k)$ with
elements of the form $C_{\{i,j\}}$, with $\{i,j\}\in\Matching$.  The
Alexander grading of such element is the grading of $b$ (an element of
$\Q^{2n}$) plus, for each factor of $C_{\{i,j\}}$, the sum of the
basis vectors $e_i+e_j$.  A non-zero algebra element in $\Alg(n,k,\Matching)$
is called {\em homogeneous of degree $v\in\Q^{2n}$} if it can be
written as a sum of pure algebra elements with multi-grading equal to $v$.

Equivalently, the Alexander grading is characterized by the following properties:
\begin{itemize}
\item Each idempotent is homogenous, with weight specified
by the $0$ vector in $\Q^{2n}$.
\item The algebra elements $L_i$ and $R_i$
are both homogeneous, and their weight is
half of the $i^{th}$ standard basis vector $e_i$ in $\Q^{2n}$.
\item The weight of $U_i$ is $e_i$.
\item The weight of $C_{i,j}$ is $e_i+e_j$.
\item Weight is additive under multiplication; i.e.
  if $a$ and $b$ are homogeneous algebra elements
  whose product $a\cdot b$ is non-zero,
  then $a\cdot b$ is homogeneous, too; and 
  \[ w_i(a\cdot b)=w_i(a)+w_i(b)\]
  for all $i=1,\dots,2n$.
\end{itemize}

There is a second grading
$\Delta$, specified by its values on 
pure algebra elements, by
\[ \Delta(a)=\#(\text{$C_{i,j}$~that divide $a$})-\sum_{i} w_i(a).\]

There are analogous gradings on $\Alg'(n,k,\Matching)$, as above.  A
pure algebra element is a product of a pure algebra element in
$\Blg(2n,k)$ and some word in the various $E_i$.   The weight of
$E_i$ is $e_i$, and
\[ \Delta(a)=\#(\text{$E_j$~that divide $a$})-\sum_{i} w_i(a),\]
where now the number of $E_j$ is counted with multiplicity. For example,
if $\{1,2\}\in\Matching$, then $E_1\cdot E_2\cdot E_1\neq 0$, and
\[\Delta(E_1\cdot E_2 \cdot E_1)=0
\qquad{\text{and}}\qquad 
\Delta(E_1 \cdot E_2\cdot U_1)= -1.\]

\subsection {Canonical $DD$-bimodules}
\label{sec:DefCanonDD}

Let 
\begin{equation}
  \label{eq:SpecifyCanonicalDDAlgebras}
    \Alg=\Alg(n,k_1,\Partition), \ \ \ \DuAlg=\DuAlg(n,k_2,\Partition)
\end{equation}
where $k_1+k_2=2n+1$.

Note that there is a natural one-to-one correspondence between the
$I$-states for $\Alg$ and those for 
$\DuAlg$: if $\x\subset \{0,\dots,2n\}$
is a $k_1$-element subset, then its complement $\x'$ is a
$k_2$-element subset of $\{0,\dots,2n\}$. In this case, we say that
$\x$ and $\x'$ are {\em complementary $I$-states}.

A $DD$ bimodule over $\Alg-\DuAlg$ is specified as follows.
Let $\CanonDD$ be the $\Field$-vector space whose generators $\gen_\x$ correspond
to $I$-states for $\Alg(n,k_1,\Partition)$. 
We give $\CanonDD$ the structure of 
a left module over $\IdempRing(\Alg)\otimes \IdempRing(\DuAlg)$,
so that the action of $\IdempRing(\Alg)\otimes \IdempRing(\DuAlg)$ is specified
by
\[(\Idemp{\y} \otimes \Idemp{\w})  \cdot \gen_\x=
\left\{\begin{array}{ll}
\gen_\x &{\text{if $\x=\y$ and $\w$ is complementary to $\x$}} \\ 
0 & {\text{otherwise.}}
\end{array}\right.\]

The algebra element
\[
A = \sum_{i=1}^{2n} \left(L_i\otimes R_i + R_i\otimes L_i\right) + \sum_{i=1}^{2n}
U_i\otimes E_i + 
\sum_{\{i,j\}\in\Partition} C_{\{i,j\}}\otimes \llbracket E_i,E_j\rrbracket\in
  \Alg\otimes\DuAlg\]
specifies a map
\[ \delta^1 \colon \CanonDD \to \Alg\otimes\DuAlg \otimes \CanonDD.\]
by
$\delta^1(v)=A\otimes v$
(where the tensor product is taken over $\IdempRing(\Alg)\otimes \IdempRing(\DuAlg)$).

\begin{lemma}
  \label{lem:CanonicalIsDD}
  The map $\delta^1$ satisfies the type $DD$ structure relation.
\end{lemma}

\begin{proof}
  This is equivalent to the statement that
  \[ dA + A\cdot A =0, \]
  thought of as an element of $\Alg\otimes \DuAlg$.
  This is a straightforward verification; compare~\cite[Lemma~\ref{BK:lem:CanonicalIsDD}]{BorderedKnots}.
\end{proof}

The above defined type $DD$ bimodule $\CanonDD$ is called the 
{\em canonical type $DD$ bimodule} over  $\Alg(n,k_1,\Partition)$ and $\DuAlg(n,k_2,\Partition)$.

\subsection{Koszul duality}
\label{subsec:Koszul}

We consider the candidate inverse module
\[ \Ynew_{\Alg',\Alg}=\Mor^{\Alg}(\lsub{\Alg'}\Alg'_{\Alg'}\DT~ \lsup{\Alg,\DuAlg}\CanonDD,~\lsup{\Alg}\Id_{\Alg}).\]

In words: take $\lsup{\Alg,\DuAlg}\CanonDD$, and turn it into a type
$D$ structure over $\Alg$ by forming the tensor product over $\DuAlg$
with the algebra $\DuAlg$, viewed as a bimodule itself; denote this
type $D$ structure
$\lsub{\Alg'}\Alg'_{\Alg'}\DT~\lsup{\Alg,\DuAlg}\CanonDD$.  The space
of type $D$ morphisms from this type $D$ structure to the identity
$DA$ bimodule $\lsup{\Alg}\Id_{\Alg}$ inherits the structure of a
$\DuAlg$-$\Alg$ module, where the (right) action by $\DuAlg$ is induced from
the left action of $\DuAlg$ on
$\lsub{\Alg'}\Alg'_{\Alg'}\DT~\lsup{\Alg,\DuAlg}\CanonDD$; and the
action by $\Alg$ is induced on the right action on
$\lsup{\Alg}\Id_{\Alg}$.

Note that the pure algebra elements give a basis for $\DuAlg$. If $a\in \DuAlg$ is a pure algebra element, let ${\overline a}$ denote the linear map
$\DuAlg$ to $\Zmod{2}$ which is non-trivial on $a$ and vanishes on all other pure algebra elements. 

As a vector space, $\Ynew$ is spanned by elements of the form
$({\overline a}|b)$, 
where $a\in\Alg'$ and $b\in\Alg$
are pure algebra elements,
subject to the constraint that the left idempotent of ${\overline a}$
(i.e. the right idempotent of $a$) is complementary to the left idempotent of
$b$. Thus, the vector space $\Ynew$ is spanned by pairs of 
pure non-zero algebra elements $a\in\Alg'$ and $b\in\Alg$, with
$a=a\cdot \Idemp{\x}$, $b=\Idemp{\y}\cdot b$, so that $\x$ and $\y$
are complementary idempotent states; the corresponding generator is
denoted $({\overline a}|b)$.

The differential on $\Ynew$ has the form
\begin{align*}
  \partial({\overline a}|b)&=
({\overline a}|d b) + 
({\overline d}({\overline a})|b) \\
&+ \sum_i  (L_i\cdot {\overline a} | R_i \cdot b) +
 (R_i\cdot  {\overline a} | L_i \cdot b) +
 (E_i\cdot {\overline a} | U_i \cdot b) \\ 
& + \sum_{\{i,j\}\in\Partition}
 (\llbracket E_i,E_j\rrbracket  \cdot {\overline a}| C_{i,j}\cdot b).
\end{align*}
The right $\DuAlg$-$\Alg$ bimodule structure on $\Ynew$ is specified by
\[ ({\overline a}|b)\cdot
 (a_2\otimes a_1) =(\xi\mapsto {\overline a}(a_2\cdot \xi)\big| b \cdot
a_1) \]

\begin{thm}
  \label{thm:InvertibleDD}
  The canonical type $DD$ bimodule $\lsup{\Alg,\DuAlg}\CanonDD$ is invertible.
\end{thm}

The above theorem is proved by reducing to a case considered
in~\cite{BorderedKnots}. We review the necessary background, first.

Consider the algebras $\Blg_1=\Blg(2n,k,\emptyset)$ and
$\Blg_2=\Blg(2n,2n+1-k,\{1,\dots,2n\})$.  There is a type $DD$ bimodule
over $\Blg_1$ and $\Blg_2$, $\lsup{\Blg_1,\Blg_2}\CanonDD$, whose
generators correspond to pairs of complementary idempotent states, and
whose differential is specified by the algebra element
\[
\sum_{i=1}^{2n} \left(L_i\otimes R_i + R_i\otimes L_i\right) + \sum_{i=1}^{2n}
U_i\otimes C_i \in \Blg_1\otimes\Blg_2.\]

The candidate inverse  to $\lsup{\Blg_1,\Blg_2}\CanonDD$ 
is given by
\[
Y_{\AlgB_1,\AlgB_2}=\Mor^{\AlgB_1}((\lsub{\AlgB_2}{(\AlgB_2)}_{\AlgB_2})\DT_{\AlgB_2}~
(\lsup{\AlgB_1,\AlgB_2}\CanonDD),\lsup{\AlgB_1}\Id_{\AlgB_1}).\] 
(This is naturally a  $\AlgB_2-\AlgB_1$-bimodule, with both
actions on the right.)
Explicitly, as a vector space, $Y$ is spanned by
elements of the form $({\overline a}\big| b)$, where ${\overline
  a}\in{\overline{\AlgB_2}}$ and $b\in\AlgB_1$, where here
${\overline{\AlgB}}$ is opposite bimodule to $\AlgB$, subject the
restriction that the left idempotent of ${\overline a}$ is
complementary to the left idempotent of $b$.  Recall that ${\overline
  {\AlgB_2}}$ is the $\AlgB_2$-bimodule consisting of maps from
$\AlgB_2$ to $\Field$.

The differential on $Y$ is given by
\[ \partial ({\overline a}|b)
= \sum_{i=1}^{2n}
 (L_i\cdot {\overline a} | R_i \cdot b) + 
 (R_i\cdot  {\overline a} | L_i \cdot b) +
 (C_i\cdot {\overline a} | U_i \cdot b) +
 ({\overline d}({\overline a})|b) +
 ({\overline a}| db).
\]
The key step in showing that $\lsup{\Blg_1,\Blg_2}\CanonDD$ is
invertible consists of verifying that the homology of $Y$ is generated
by elements of the form $({\overline{\Idemp{\x'}}} |\Idemp{\x} )$,
where $\x$ and $\x'$ are complementary idempotents. (This is proved
in~\cite[Proposition~\ref{BK:prop:RankOneHomology}]{BorderedKnots}.)
This argument can be used to verify that $\lsup{\Alg,\DuAlg}\CanonDD$
is invertible, as well:

To relate $Y$ and $\Ynew$, we use the following maps.
First, observe that there is a canonical map 
of differential graded algebras
\[ \alpha\colon \Alg'\to\Blg_2\]
with 
\[ \alpha(\Idemp{\x})=\Idemp{\x},\qquad \alpha(L_i)=L_i,\qquad \alpha(R_i)=R_i,
\qquad 
\alpha(U_i)=U_i,\qquad \alpha(E_i)=C_i.\]
This map is surjective, and so it has an injective dualization
\[ {\overline \alpha}\colon {\overline \Blg}_2\to {\overline \Alg}'.\]
Also, there is a canonical inclusion of $\Blg_1$ into $\Alg$. These
tensor together to give a map
\[ \phi\colon {\overline \Blg}_2\otimes \Blg_1 \to {\overline
  \Alg}'\otimes \Alg,\] where the tensors are taken over the ring of
idempotents of the rings $\Blg_1$, $\Blg_2$, $\Alg$, and $\DuAlg$
(all of which are canonically identified).

\begin{lemma}
  \label{lem:IdentifySubcomplex}
  For $\{i,j\}\in\Matching$, let $\SubCx^{\{i,j\}}\subset {\overline
    \Alg}'\otimes\Alg$ be the subset generated by pairs $({\overline
    a}|b)$ where $b$ is a pure algebra element of $\Alg$ which is not
  divisible by $C_{\{i,j\}}$,
  and $\llbracket E_i, E_j\rrbracket\cdot
  {\overline a}=0$.
  \begin{enumerate}[label=(C-\arabic*),ref=(C-\arabic*)]
    \item 
      \label{eq:IsASubCx}
      $\SubCx^{\{i,j\}}$ is a subcomplex of $\Ynew$.
    \item 
      \label{eq:IsAChainMap}
      The above defined map $\phi$ gives an injective chain map
      from $Y$ to $\Ynew$, and hence realizing
      $Y$ as a subcomplex of $\Ynew$.
    \item 
      \label{eq:IdentifyImagePhi}
      The  intersection $\bigcap_{p\in\Matching}\SubCx^p$ 
      is the image of $\phi$.
  \end{enumerate}
\end{lemma}

\begin{proof}
  First, we note that  $d \llbracket E_i,E_j\rrbracket=0$.
  It follows immediately that 
  \[ \llbracket E_i,E_j\rrbracket \cdot ({\overline d}({\overline a}))
  = {\overline d}(\llbracket E_i,E_j\rrbracket\cdot  {\overline a}); \]
  from which Condition~\ref{eq:IsASubCx} is an easy consequence.

  Next, we verify Condition~\ref{eq:IsAChainMap}.  It is easy to see
  that for $({\overline a}|b)\in{\overline \Blg_2}\otimes \Blg_1$,
  \[ \phi(\partial({\overline a}|b))=\partial(\phi({\overline a}|b)) + 
  \sum_{\{i,j\}\in\Matching} (\llbracket E_i,E_j\rrbracket \cdot
  {\overline\alpha}({\overline a})|C_{i,j} \cdot b).\]
  It remains to check that 
  \begin{equation}
    \label{eq:RemainsToCheck}
    \llbracket E_i,E_j\rrbracket \cdot {\overline\alpha}({\overline a})=0
  \end{equation}
  for all ${\overline a}\in{\overline\Blg_2}$.
  To this end, note that
  $\alpha(E_i
  E_j)=\alpha(E_j E_i)=C_i C_j$; and dually
  \[{\overline \alpha}({\overline{C_i\cdot  C_j}})=
  {\overline{E_i \cdot E_j}} + {\overline {E_j \cdot E_i}}.\]
  
  To verify Condition~\ref{eq:IdentifyImagePhi}, observe that
  Equation~\eqref{eq:RemainsToCheck} gives the containment
  \[ \Image(\phi)\subseteq \bigcap_{p\in\Matching}\SubCx^p.\]

  For the other containment, a straightforward computation shows that
  $\SubCx^{\{i,j\}}$ is contained in the span of $({\overline a}|b)$,
  where $b$ is a pure algebra element not divisible by $C_{\{i,j\}}$;
  and ${\overline a}$ is in the span of ${\overline c}$,
  ${\overline{E_i \cdot c}}$, ${\overline{E_j \cdot c}}$, and
  ${\overline{E_i E_j c}}+{\overline{E_j E_i c}}$, where $c$ is a pure
  algebra element not divisible by $E_i$ or $E_j$.  A straightforward
  induction on the number of $E_k$ so that $E_k\cdot {\overline a}=0$
  now shows that ${\overline c}$ is in the image of ${\overline
    \alpha}$.  
\end{proof}

\begin{proof}[Proof of Theorem~\ref{thm:InvertibleDD}] 
  We wish to show that $\phi$ induces an isomorphism on homology.

  Fix any $\{i,j\}\in\Matching$ with $i<j$, and consider the associated
  linear map $h_0^{\{i,j\}}\colon {\overline\Alg}'\to {\overline\Alg'}$ 
  characterized by its values on ${\overline a}\in{\overline \Alg'}$ 
  dual to a pure algebra element $a$, specified as
  \[ h_0^{\{i,j\}}({\overline a})=\left\{\begin{array}{ll}
      {\overline {\llbracket E_i,E_j\rrbracket \cdot a}} & {\text{if $E_i \cdot a = 0$ or $E_j\cdot a=0$,}} \\ \\
      {\overline {E_i \cdot E_j\cdot a}} &{\text{otherwise.}}
      \end{array}\right.\]
    The following identities are easily verified:
    \begin{align}
      \label{eq:EasyToSee}
      \llbracket E_i,E_j\rrbracket \cdot h_0^{\{i,j\}}({\overline a}) &= {\overline a} \\
      {\overline d}(h_0^{\{i,j\}}({\overline a}))&=h_0^{\{i,j\}}({\overline d}({\overline a})) \label{eq:Easy2}
      \\
       c\cdot h_0^{\{i,j\}}({\overline a}) &= h_0^{\{i,j\}}(c\cdot {\overline a}), \label{eq:Easy3}
      \end{align}
      for any pure algebra element $c$ not divisible by $E_i$ or $E_j$.

  Now, fix some $\{i,j\}\in\Matching$, and
  consider the operator on $h^{\{i,j\}}\colon \Ynew\to \Ynew$ given by
  $h^{\{i,j\}}({\overline a}|C_{\{i,j\}}\cdot b)=(h_0^{\{i,j\}}({\overline a})|b)$.
  Let 
  \[ \Pi^{\{i,j\}}=\Id + \partial \circ h^{\{i,j\}} + h^{\{i,j\}}\circ\partial.\]

  We argue that $\Pi^{\{i,j\}}$ maps into $\SubCx^{\{i,j\}}$.  First
  consider an element of the form $({\overline a}|C_{\{i,j\}}\cdot
  b)$.  The component of $\Pi^{\{i,j\}}({\overline a}|C_{\{i,j\}}b)$ in
  $({\overline a}|C_{\{i,j\}} b)$ vanishes by
  Equation~\eqref{eq:EasyToSee}; and the other components vanish by
  Equations~\eqref{eq:Easy2} and~\eqref{eq:Easy3}.  Next, consider an
  element of the form $({\overline a}|b)$ where $b$ is pure and
  $C_{\{i,j\}}$ does not divide $b$. Then,
  \[ \Pi^{\{i,j\}}({\overline a}|b)=({\overline a}+h_0^{\{i,j\}}(\llbracket E_i,E_j\rrbracket \cdot {\overline a})|b).\]
  From Equation~\eqref{eq:EasyToSee}, it follows that 
  \[\llbracket E_i,E_j\rrbracket \left({\overline a}+h_0^{\{i,j\}}(\llbracket E_i,E_j\rrbracket \cdot {\overline a}\right)=0.\]

  We claim also that for all $\{i',j'\}\in\Matching$, the 
  map
  $\Pi^{\{i,j\}}$ maps $\SubCx^{\{i',j'\}}$ into
  itself. This follows from the fact that the elements $\llbracket E_{i'},E_{j'}\rrbracket$
  are both central and closed, and they commute with $h_0^{\{i,j\}}$ (Equation~\eqref{eq:Easy3}). 
  Thus, the composition $\Pi$ of all the $\Pi^{p}$ for $p\in\Matching$
  induces a quasi-isomorphism of
  ${\overline\Alg'}\otimes \Alg$ with its subcomplex
  $\cap_{p}\SubCx^{p}$; i.e. in view of
  Lemma~\ref{lem:IdentifySubcomplex}, the inclusion map $\phi$ is a
  quasi-isomorphism.
\end{proof}

\subsection{Symmetries}
\label{sec:AlgSymm}
The algebras $\Alg$ and $\Alg'$ inherit symmetries from $\Blg$;
cf.~\cite[Section~\ref{BK:subsec:Symmetry}]{BorderedKnots}.

Consider the map $\rho_n\colon \{0,\dots,2n\}\to \{0,\dots,2n\}$ with
$\rho_n(i)=2n-i$; and let $\rho'_n\colon
\{1,\dots,2n\}\to\{1,\dots,2n\}$ be the map 
\[ \rho'_n(i)=2n+1-i.\] We
will drop the subscript $n$ on $\rho$ and $\rho'$ when it is clear
from the context. There is a map
\begin{equation}
  \label{eq:DefVRot}
  \VRot\colon \Alg(n,k,\Matching)\to \Alg(n,k,\rho'(\Matching))
\end{equation}
characterized as follows.
First,
\[ \VRot(\Idemp{\x})=\Idemp{\rho(\x)};\]
and if $a\in\Blg(2n,k)$ is non-zero and homogeneous with specified weights 
\[(w_1(a),\dots,w_{2n}(a)),\] 
then $\VRot(a)=b$ is the non-zero element that is homogeneous with specified weights 
$w_{i}(b)=w_{\rho'(i)}(a)$. 
This defines an isomorphism $\VRot \colon \Blg(2n,k)\to\Blg(2n,k)$.

We extend this 
to the desired map on $\Alg$ by
requiring 
\[\VRot(C_p\cdot a)=C_{\rho'(p)}\cdot \VRot(a).\]
Note that for all $i=1,\dots,2n$ 
\[\begin{array}{llll}
 \VRot(L_i)=R_{\rho'(i)} &\VRot(R_i)=L_{\rho'(i)} &
 \VRot(U_i)=U_{\rho'(i)} & \VRot(C_{\{p,q\}})=C_{\{\rho'(p),\rho'(q)\}}.
\end{array}\]
Clearly, this induced map $\VRot$ induces an isomorphism of differential graded algebras.

We can also extend $\VRot$ to an isomorphism
\[ \VRot\colon \DuAlg(n,k,\Matching)\to \DuAlg(n,k,\rho'(\Matching)).\]
For this extension, we require that
\[\VRot(E_i\cdot a)=E_{\rho'(i)}\cdot \VRot(a).\]

Another symmetry is constructed as follows.
Recall that 
the map
$\Opposite(\Idemp{\x})=\Idemp{\x}$
extends to an isomorphism of differential graded algebras
\begin{equation}
  \label{eq:OppositeIsomorphism}
  \Opposite\colon \Blg(2n,k)\to \Blg(2n,k)^{\op},
\end{equation}
with 
\[ \begin{array}{lll}
 \Opposite(L_i)=R_i &\Opposite(R_i)=L_{i} &
 \Opposite(U_i)=U_{i} 
\end{array}\]
The map can be extended to either $\Alg(n,k,\Partition)$ or $\DuAlg(n,k,\Partition)$;
we denote these extensions also by $\Opposite$.

\subsection{Relationship with knot diagrams}

Recall that the algebras used here are associated to slices of a knot
diagram: generic horizontal slices of the knot diagram give rise to an
even number of points. The matchings correspond to pairs of points
that are connected by arcs in the portion above the horizontal slice.

Moreover, if the knot diagram is oriented, the orientation can be used
to choose a preferred section $\Upwards$ of the matching: $\Upwards$
consists of those points on the boundary where the knot is oriented to
point into the upper part of the knot diagram.

In~\cite{BorderedKnots}, we defined an Alexander
grading, depending on the choice of $\Upwards$,
given by
\[ 
  \Alex(a)= -\sum_{s\in \Upwards} w_s(a)+ \sum_{t\not\in \Upwards} w_t (a).
\]

With the choice of orientation, we could now define a Maslov grading
on the algebras $\Alg$ and $\DuAlg$, defined by
\[ \Delta=\Maslov-\Alex.\]
This grading generalizes the Maslov grading from~\cite{BorderedKnots}.
For more on the relationship between this algebra and the one from~\cite{BorderedKnots}, see Section~\ref{sec:Compare}.

\subsection{Gradings}
As in~\cite{BorderedKnots}, the bimodules we will construct here will
have a special grading set.  Let $W_1$ be a disjoint union of finitely
many intervals, equipped with a partition of its boundary $\partial
W_1=Y_1 \cup Y_2$ into two sets of points, where both $|Y_1|$ and
$|Y_2|$ are even.  In this case, we write $W_1\colon Y_1\to Y_2$.

Fix a matching $M_1$ on $Y_1$. There is an equivalence relation on
points in $Y=Y_1\cup Y_2$, generated by the relation $y\sim y'$ if
they are connected by an interval in $W_1$, or $y\sim y'$ if both are in
$Y_1$, and $\{y,y'\}\in M_1$. It is easy to see that each equivalence
class of points contains either zero or two points in $Y_2$; let $M_2$
denote the induced matching on $Y_2$. 

\begin{defn}
  \label{def:Compatible}
  The matching $M_1$ on $Y_1$ and the one-manifold $W_1$ with
  $\partial W_1=Y_1\cup Y_2$ are called {\em compatible} if every
  point point $y\in Y_1$ is equivalent to some point $y'\in Y_2$;
  equivalently, if we think of $M_1$ as a union $W_0$ of abstract arcs
  joining the points in $Y_1$, then $M_1$ and $W_1$ are compatible if
  the one-manifold $W_0\cup W_1$ has no closed components.
\end{defn}

Let
\begin{equation}
\label{eq:DefAlg1Alg2}
\Alg_1=\Alg(|Y_1|/2,s+s_1,\Matching_1)\qquad\text{and}\qquad
\Alg_2=\Alg(|Y_2|/2,s+s_2,\Matching_2),
\end{equation}
where $s_1=s_1(W_1)$ is the number of components of $W_1$ that connect $Y_1$ to itself;
$s_2=s_2(W_2)$ is the number of components of $W_1$ that connect $Y_2$ to itself,
and 
$0\leq s \leq s_0+1$, where $s_0=s_0(W_1)$ is the number of intervals in $W_1$ 
that match some point in $Y_1$ to $Y_2$.
We can think of the Alexander
multi-grading on $\Alg_i$ as taking values in $H^0(Y_i;\Q)$: the weights
of algebra elements are functions on the points in $Y_i$.
The sum of grading groups $H^0(Y_1;\Q)\oplus H^0(Y_2;\Q)$ act on 
$H^1(W,\partial W;\Q)$, via the map
\[ H^0(Y_1)\oplus H^0(Y_2)\to H^1(W,\partial W)\]
given by $(y_1,y_2)\mapsto -d^0(y_1)+d^0(y_2)$.

\begin{defn}
  \label{def:AdaptedDA}
  Fix $W_1$, $\Alg_1$, and $\Alg_2$ as above.  A  type $DA$ bimodule
  $X=\lsup{\Alg_2}X_{\Alg_1}$ is called {\em adapted to $W_1$} if the following conditions hold:
  \begin{enumerate}[label=(Ad-\arabic*),ref=(Ad-\arabic*)]
    \item $X$
      \label{Adapted:grading} 
      is multi-graded by $H^1(W_1,\partial W_1)$, as described above;
      and it has an additional $\Q$-grading, denoted $\Delta_X$, which
      is compatible with the $\Z$-grading on the algebra, in the sense
      that if $b$ and $a_1,\dots,a_\ell$ are $\Delta$-homogeneous
      algebra elements, and $\x,\y$ are $\Delta_X$-homogeneous module
      elements in $X$, so that $b\otimes \y$ appears with non-zero coefficient in
      $\delta^1_\ell(\x,a_1,\dots,a_\ell)$, then
      \[ \Delta(b)+\Delta_X(\y) =
      \Delta_X(\x)+\Delta(a_1)+\dots+\Delta(a_{\ell-1})-\ell+2.\]
  \item \label{Adapted:FD}
    $X$ is a finite dimensional $\Field$ vector space
    \end{enumerate}
  The above definition specializes readily to the case of $D$ modules
    $\lsup{\Alg_2}X$:
    in this case, $\partial W_1=Y_2$ (i.e. $Y_1=\emptyset$).
    Similarly, if we let $\DuAlg_1=\Alg(|Y_1|/2,|Y_1|-(s+s_1),\Matching_1)$, a type $DD$ bimodule
    $X=\lsup{\Alg_2,\DuAlg_1}X$ is called {\em adapted to $W_1$} if 
    the above two properties hold.
\end{defn}

\begin{example}
  The type $DD$ bimodule $\lsup{\Alg,\DuAlg}\CanonDD$ is adapted to
  the one-manifold $W_1=\{1,\dots,2n\}\times [0,1]$.
\end{example}

Let $\Alg_1$ and $\Alg_2$ be as in Equation~\eqref{eq:DefAlg1Alg2}.
In addition, choose a one-manifold $W_2$ connecting $Y_2$ to $Y_3$, and let
\begin{equation}
  \label{eq:DefAlg3}
  \Alg_3=\Alg(|Y_3|/2,s+s_3,\Matching_3),
\end{equation}
where $M_3$ is the matching induced by $M_2$ and $W_2$, and $s$ is
chosen so that $0\leq s\leq s_0(W_1\cup W_2)+1$, where $s_0$ is the
number of components of $W_1\cup W_2$ connecting $Y_1$ to $Y_3$.

The following is an adaptation
of~\cite[Proposition~\ref{BK:prop:AdaptedTensorProducts}]{BorderedKnots}:

\begin{prop}
  \label{prop:AdaptedTensorProd}
  Choose $W_1\colon Y_1\to Y_2$, $W_2\colon Y_2\to Y_3$, $\Alg_1$,
  $\Alg_2$, and $\Alg_3$ as above.  Suppose moreover that $W_1\cup
  W_2$ has no closed components, i.e. it is a disjoint union of
  finitely many intervals joining $Y_1$ to $Y_3$.  Given any two
  bimodules $\lsup{\Alg_2}X^1_{\Alg_1}$ and
  $\lsup{\Alg_3}X^2_{\Alg_2}$ adapted to $W_1$ and $W_2$ respectively,
  we can form their tensor product
  $\lsup{\Alg_3}X^2_{\Alg_2}\DT~\lsup{\Alg_2}X^1_{\Alg_1}$ (i.e. only
  finitely many terms in the infinite sums in its definition
  are non-zero); and moreover, it is a bimodule that is adapted to
  $W=W_1\cup W_2$.
\end{prop}

\begin{proof}
  Consider equivalence classes of points in $Y_1\cup Y_2\cup Y_3$,
  where two points are considered equivalent if
  they can be joined by arcs in $W$. Fix
  $\x_1, \x_1' \in X_1$, $\x_2, \x_2'\in X_2$, and a sequence $a_1,\dots,a_\ell$
  in $\Alg_1$.
  Suppose that
  $(b_1\otimes \dots\otimes b_j)\otimes \x_1'$ appears with non-zero 
  multiplicity in $\lsup{X_1}\delta^{j}_{\ell+1}(\x_1,a_1,\dots,a_{\ell})$, and
  $c\otimes \x_2'$ appears with non-zero multiplicity in
  $\lsup{X_2}\delta^1_{j+1}(\x_2,b_1,\dots,b_j)$. Then,
  \[ \Delta(c)+\Delta(\x_1')+\Delta(\x_2') =
  \ell-1+\Delta(\x_1)+\Delta(\x_2)+\sum \Delta(a_i).\] In $\Alg_3$, the
  number of $C_{p}$ that divide any algebra element is universally
  bounded, so we obtain a universal bound on $|w_i(c)|$ for all
  $i\in 1,\dots,|Y_3|$ in terms of the input.

  At points $i\in Y_2$ that are equivalent to points $Y_1$ in $W$, the
  grading assumptions bound
  $|w_i(\delta^j_{\ell+1}(\x_1,a_1,\dots,a_{\ell}))|$ in terms of 
  the inputs $\x_1$ and $a_1,\dots, a_{\ell}$.  In more detail,
  first
  suppose that some point in $Y_2$ is matched with a point $z\in Y_1$ by
  $W_1$. Let $\xi_i$ and $\xi_i'$ be the coefficient of $\gr(\x_1)$
  and $\gr(\x_1')$ in the component of $W_1$ that contains $i\in Y_2$.
  By the grading hypotheses,
  \begin{equation}
    \label{eq:MatchY1Y2}
    w_i(b_1\otimes\dots\otimes b_j)+\xi_i'=\xi_i+w_{z}(a_1\otimes\dots\otimes a_{\ell}).
  \end{equation}
  The claimed bound on $|w_i(b_1\otimes\dots\otimes b_j)|$
  in terms of the inputs (appearing on the right-hand-side) is immediate in this case. 
  Next, suppose that $i,i'\in Y_2$ are matched by some component in $W_2$,
  let $\eta_i$ and $\eta_i'$ be the corresponding components of $\gr(\x_2)$
  and $\gr(\x_2')$,
  then
  \begin{equation}
    \label{eq:MatchY2inW1}
    w_i(b_1\otimes\dots\otimes b_j)+\eta_i
    = w_{i'}(b_1\otimes\dots\otimes b_j)+\eta_i'.
  \end{equation}
  Finally, if $i,i'\in Y_2$ are matched by some component in $W_1$, 
  then 
  \begin{equation}
    \label{eq:MatchY2inW2}
    w_i(b_1\otimes\dots\otimes b_j)+\xi_i
    = w_{i'}(b_1\otimes\dots\otimes b_j)+\xi_{i'}.
  \end{equation}
  Putting together
  Equations~\eqref{eq:MatchY1Y2},~\eqref{eq:MatchY2inW1}
  and~\eqref{eq:MatchY2inW2}, and using the fact that the $\xi_i$,
  $\xi_{i'}$, and $\eta_i$ are universally bounded (since $X_1$ and
  $X_2$ are finite dimensional), we obtain a universal bound on
  $|w_i(b_1\otimes\dots\otimes b_j)|$ in terms of
  $w_z(a_1\otimes\dots\otimes a_\ell)$ provided that $i\in Y_2$ is
  equivalent to $z\in Y_1$; and hence the claimed bound of
  $|w_i(\delta^j_{\ell+1}(\x_1,a_1,\dots,a_\ell))|$ in terms of the
  inputs.

  For points $i\in Y_2$ that are equivalent to points $i'\in Y_3$, we
  obtain a similar bound of $|w_i(\delta^j_{\ell+1}(\x_1,a_1,\dots,a_{\ell}))-
  w_{i'}(c)|$.   

  Since $W_2\cup W_1$ has no closed components, 
  for given input sequence
  $a_1,\dots,a_{\ell}$,
  we have obtained an upper bound on
  $w_i(b_1\otimes\dots\otimes b_j)$ for all $i\in Y_2$, depending only on the weights of the inputs $a_1,\dots,a_\ell$.
  If $j$ could be arbitrarily large, we would be able to find
  arbitrarily large $k$ for which
  $\delta^{k}(x')$ contains some term of the form $y'\in \IdempRing(\Alg_2)\otimes X_1 \subset \Alg_2\otimes X_1$,
  so that $\Delta(y')=\Delta(x')-k$, violating the hypothesis that $X_1$ is finitely generated.
  
  Having defined $\lsup{\Alg_3}X^2_{\Alg_2}\DT~\lsup{\Alg_2}X^1_{\Alg_1}$,
  the verification that it is adapted to $W_1\cup W_2$ is straightforward.
\end{proof}

\begin{rem}
  \label{rem:OtherBoundedness}
  It is a key point in the above proof that in $\Alg(n,k)$, the number
  of $C_p$ that divide any pure algebra element is universally
  bounded. The corresponding statement does not hold over
  $\DuAlg(n,k)$: if $\{i,j\}\in\Matching$, the elements $\llbracket E_i,E_j\rrbracket^m$
  are non-zero for all $m\geq 0$; and correspondingly other boundedness
  criteria must be satisfied when forming $\DT$ over $\DuAlg(n,k)$.
  (See for example Lemma~\ref{lem:DuAlgDualCross}.)
\end{rem}

Proposition~\ref{prop:AdaptedTensorProd} has the following statement
for type $DD$ bimodules.  Choose $W_1\colon Y_1\to Y_2$, $W_2\colon Y_2\to Y_3$,
$\Alg_2$ and $\Alg_3$ as above; and let
 \begin{align*}
   \DuAlg_1&=\DuAlg(|Y_1|/2,|Y_1|-(s+s_1),\Matching_1).
 \end{align*}

\begin{prop}
  \label{prop:AdaptedTensorDD}
  Fix $\Alg_1$, $\Alg_2$, and $\Alg_3$ as in
  Equation~\eqref{eq:DefAlg1Alg2} and~\eqref{eq:DefAlg3}. Let $W_0$ be
  the one-manifold with boundary consisting of intervals that connect
  pairs of points in $Y_1$ matched by $M_1$. Suppose that $W_0\cup
  W_1\cup W_2$ has no closed components.  Given any two bimodules
  $\lsup{\Alg_2,\DuAlg_1}X^1$ and $\lsup{\Alg_3}X^2_{\Alg_2}$ adapted
  to $W_1$ and $W_2$ respectively, we can form their tensor product
  $\lsup{\Alg_3}X^2_{\Alg_2}\DT~\lsup{\Alg_2,\DuAlg_1}X^1$; and
  moreover, it is a bimodule that is adapted to $W=W_1\cup W_2$.
\end{prop}

\begin{proof}
  Let $W$ be the grading manifold for
  $X=\lsup{\Alg_3}X^2_{\Alg_2}\DT~\lsup{\Alg_2,\DuAlg_1}X^1$, with
  $\partial W =Y_1\cup Y_3$.  Fix $\x_1,\x_1'\in X_1$ and
  $\x_2,\x_2'\in X_2$.  Suppose that $(b_1\otimes b_1')\otimes\dots
  (b_j\otimes b_j')\otimes \x_1'$ appears with non-zero multiplicity
  in $\lsup{X_1}\delta^j(\x_1)$ and $c\otimes \x_2'$ appears with
  non-zero multiplicity in
  $\lsup{X_2}\delta^1_{j+1}(\x_2,b_1,\dots,b_j)$.  To show that the
  sums defining $\DT$ are finite amounts to establishing an upper
  bound on $j$.

  Note first that
  \begin{equation} 
    \label{eq:DeltaBound}
    \Delta(c)+\Delta(c')+\Delta(\x_1')+\Delta(\x_2')
    =\Delta(\x_1)+\Delta(\x_2)-1,
  \end{equation}
  for $c'=b_1'\cdots b_j'$.
  Since $\Delta(c')\leq 0$ 
  (in fact, this is true for any $c'\in \Alg_1'$), 
  Equation~\eqref{eq:DeltaBound} gives an upper bound on 
  $w_i(c)$ for all $i\in Y_3$. 

  For all $i\in Y_2$ that are equivalent via $W_2$ to some $j\in Y_3$,
  the fact that $X$ finitely generated  gives
  an upper bound on $w_i(b_1\otimes\dots \otimes b_j)$.  Moreover
  since $X^1$ is finitely generated and graded by $W_1$, for every $i\in
  Y_2$ that is equivalent via $W_1$ to some $k\in Y_1$, we obtain a
  lower bound on $|w_i(b_1\otimes\dots \otimes b_j)-w_k(c')|$.
  Finally, in the algebra $\DuAlg_1$, there is a bound on
  $|w_i(c')-w_k(c')|$ for any $\{i,k\}\in \Matching_1$.  Since
  $W_0\cup W_1\cup W_2$ has no closed components, it follows that any
  any $i\in Y_1\cup Y_2$ is equivalent to some $k\in Y_3$, and hence
  the above reasoning gives bounds on $w_i((b_1\otimes
  b_1')\otimes\dots \otimes (b_j\otimes b_j'))$ for all $i\in Y_1\cup
  Y_2$.  The desired bound on $j$ now follows from the
  $\Delta$-grading on $X^1$, and the fact that it is finite
  dimensional, as before.

  The fact that the tensor product is adapted to $W_1\cup W_2$ follows
  easily.
\end{proof}

\subsection{Standard type $D$ structures}
\label{sec:Standard}

A type $D$ structure $X$, specified by
\[ \delta^1\colon X \to \Alg(n,k,\Partition)\otimes X, \]
is called {\em standard} 
if it is adapted to a one-manifold with boundary $W$, and it has the following form:
\[ \delta^1(\x)= (\sum_{p\in\Partition} C_p)\otimes \x + \epsilon^1(\x),\]
where $\epsilon^1(\x)\in \Blg(2n,k)\otimes X \subset \Alg(n,k,\Partition)\otimes X$,
for all $\x\in X$.

A {\em standard sequence} $(a_1,\dots,a_\ell)$ in $\Alg(n,k,\Partition)$
is a sequence of algebra elements for which each $a_i$ is either 
in $\Blg(2n,k)$, or it is equal to some $C_{p}$ for $p\in\Partition$.

A type $A$ module is characterized up to homotopy by its values on
standard sequences.  This follows from Lemma~\ref{lem:DeterminedByStandard} below,
stated in terms of the following generalization of type $D$ structures to bimodules.

For $i=1,2$, fix integers $n_i$ and $k_i$ with $0\leq k_i\leq 2n_i+1$;
fix also matchings $\Partition_i$ on $\{1,\dots,2n_i\}$.
Let $\Blg_i=\Blg(2n_i,k_i)\subset \Alg_i(n_i,k_i,\Partition_i)=\Alg_i$
for $i=1,2$.
Let $C^i=\sum_{p\in\Partition_i} C_p\in\Alg_i$.

\begin{defn}
  \label{def:Standard}
  A type $DA$ bimodule $\lsup{\Alg_2}X_{\Alg_1}$ is called {\em standard} if
  the following conditions hold:
  \begin{enumerate}[label=(DA-\arabic*),ref=(DA-\arabic*)]
    \item 
      The bimodule $X$ is adapted to some one-manifold $W$, 
      in the sense of Definition~\ref{def:AdaptedDA}.
    \item 
      The one-manifold $W$ is compatible with the 
      matching $\Matching_1$ used in the algebra $\Alg_1$.
  \item\label{prop:LandsInB} For any standard sequence of elements $a_1,\dots,a_{\ell-1}$ 
    with at least some $a_i\in\Blg_1$, 
    $\delta^1_{\ell}(\x,a_1,\dots,a_{\ell-1})\in\Blg_2\otimes X$.
  \item 
    \label{prop:CStandard}
    For any $\x\in X$, 
    \[
      C^2\otimes \x +    \sum_{\ell=0}^{\infty}\delta^1_{1+\ell}(\x,\overbrace{C^1,\dots,C^1}^{\ell})
      \in \Blg_2\otimes X.
      \]
\end{enumerate}
\end{defn}

\begin{lemma}
\label{lem:StandardTimesStandard}
Let $\Alg_i=\Alg(n_i,k_i,\Partition_i)$
for $i=1,2,3$. 
Suppose that $\lsup{\Alg_3}X_{\Alg_2}$ and $\lsup{\Alg_2}Y_{\Alg_1}$ 
are standard type $DA$ bimodules adapted to one-manifolds $W_1$ and $W_2$, so that $W_1\cup W_2$ contains no closed components.
Then their
product  $X\DT Y$ is standard, too.
\end{lemma}

\begin{proof}
  The fact that $W_1\cup W_2$ has no closed components ensures that
  the tensor product makes sense, according to
  Proposition~\ref{prop:AdaptedTensorProd}. The fact that $X\DT Y$ is
  standard is clear now from the definition of $\DT$.
\end{proof}

\begin{lemma}
  \label{lem:DeterminedByStandard}
  A standard type $DA$ bimodule is determined uniquely (up to homotopy) by its
  bimodule structure over the idempotent algebras, together with 
  values on standard sequences.
\end{lemma}

\begin{proof}
  Suppose that $\lsup{\Alg_2}X_{\Alg_1}$ and $\lsup{\Alg_2}Y_{\Alg_1}$
  are two type $DA$ bimodule structures on the same underlying
  $\Idemp{\Alg_2}\otimes\Idemp{\Alg_1}$-bimodules, that have the same
  values on standard sequences. By
  Proposition~\ref{prop:AdaptedTensorDD}, we can form the tensor
  products $\lsup{\Alg_2}X_{\Alg_1}\DT~\lsup{\Alg_1,\DuAlg_1}\CanonDD$
  and $\lsup{\Alg_2}Y_{\Alg_1}\DT~\lsup{\Alg_1,\DuAlg_1}\CanonDD$.

  We claim that in this case,
  \[ \lsup{\Alg_2}X_{\Alg_1}\DT~\lsup{\Alg_1,\DuAlg_1}\CanonDD
  \cong 
  \lsup{\Alg_2}Y_{\Alg_1}\DT~\lsup{\Alg_1,\DuAlg_1}\CanonDD,\]
  because $\CanonDD$ has the property that its outputs in the $\Alg_1$ tensor factor
  are either
  in $\Blg_1$ or they are algebra elements $C_p\in\Alg_1$;
  and so the tensor product of $X$ (or $Y$) with $\CanonDD$
  is determined by the values of the $DA$ bimodule on standard sequences.
  By Theorem~\ref{thm:InvertibleDD}, we can conclude that
  $X\simeq Y$, as needed.
\end{proof}

\newcommand\LocPos{P}
\newcommand\Initial{I}
\newcommand\ed{\widetilde\delta}
\newcommand\Xelt{X}
\newcommand\Yelt{Y}
\newcommand\Zelt{Z}
\section{Bimodules associated to crossings}
\label{sec:Cross}

Bimodules associated to crossings are defined very similarly to those from~\cite{BorderedKnots}.
We start with the type $DD$ bimodules, and then describe the corresponding type $DA$ bimodules.

\subsection{$DD$ bimodules}
\label{subsec:DDcross}

Fix $i$ with $1\leq i\leq 2n-1$,
fix a matching $\Partition$ on $\{1,\dots,2n\}$,
and fix an auxiliary integer $k$ with $0\leq k \leq 2n+1$.

Let $\tau=\tau_i\colon \{1,\dots,2n\}\to  \{1,\dots,2n\}$ 
be the transposition that 
switches $i$ and $i+1$, and let
$\tau(\Partition)$ be the induced matching that matches
$\tau(\alpha)$ and $\tau(\beta)$ iff $\alpha$ and $\beta$ 
are matched in $\Partition$.
Let 
\begin{equation}
  \label{eq:DefA1A2}
  \Alg_1=\Alg(n,k,\Partition)~\qquad{\text{and}}\qquad\DuAlg_2=
  \Alg'(n,2n+1-k,\tau_i(\Partition));
\end{equation} 
correspondingly
\[ \IdempRing_1=\IdempRing(2n,k) \qquad{\text{and}}\qquad\IdempRing_2'=\IdempRing(2n,2n+1-k).\]
We think of the algebra $\Alg_1$ as being below the crossing and $\DuAlg_2$ as
above the crossing.

\begin{figure}[ht]
\input{PosCrossDD.pstex_t}
\caption{\label{fig:PosCrossDD} {\bf{Positive crossing $DD$ bimodule generators.}}  
The four generator types are pictured to the right.}
\end{figure}

As an $\IdempRing_1-\IdempRing_2'$-bimodule, $\Pos_i$ is the submodule
of $\IdempRing_1\otimes_{\Field}\IdempRing_2'$ generated by elements
$\Idemp{\x}\otimes \Idemp{\y}$ where $\x\cap\y=\emptyset$ or 
\begin{equation}
  \label{eq:PosGens}
  \x\cap\y=\{i\}\qquad{\text{and}}\qquad \{0,\dots,2n\}\setminus (\x\cup \y) =\{i-1\}~\text{or}~\{i+1\}.
\end{equation}

In a little more detail,
generators correspond to certain pairs of idempotent
states $\x$ and $\y$, where $|\x|=k$ and $|\y|=2n+1-k$. They are further classified into four types,
$\North$, $\South$, $\West$, and $\East$. For generators of type $\North$ 
the subsets $\x$ and $\y$ are complementary subsets of $\{0,\dots,2n\}$ 
and $i\in \x$.
For generators of type $\South$, 
$\x$ and $\y$ are complementary subsets of $\{0,\dots,2n\}$ with $i\in \y$.
For generators of type $\West$, $i-1\not\in \x$ and $i-1\not\in \y$,
and $\x\cap\y=\{i\}$.
For generators of type $\East$, $i+1\not\in \x$ and $i+1\not\in \y$,
and $\x\cap\y=\{i\}$.

The differential has the following types of terms:
\begin{enumerate}[label=(P-\arabic*),ref=(P-\arabic*)]
\item 
  \label{type:OutsideLRP}
  $R_j\otimes L_j$ and $L_j\otimes R_j$ for all $j\in
  \{1,\dots,2n\}\setminus \{i,i+1\}$; these connect generators of the
  same type.
\item
  \label{type:UCP}
  $U_{j}\otimes E_{\tau(j)}$
  for all $j=1,\dots,2n$, connecting generators of the same type.
\item
  \label{type:UCCP}
  $C_{\{\alpha,\beta\}}\otimes [E_{\tau(\alpha)},E_{\tau(\beta)}]$,
  for all  $\{\alpha,\beta\}\in\Partition$; these connect generators of the same type
\item 
  \label{type:InsideP}
  Terms in the diagram below that connect  generators
  of different types:
  \begin{equation}
    \label{eq:PositiveCrossing}
    \begin{tikzpicture}[scale=1.8]
    \node at (0,3) (N) {$\North$} ;
    \node at (-2,2) (W) {$\West$} ;
    \node at (2,2) (E) {$\East$} ;
    \node at (0,1) (S) {$\South$} ;
    \draw[->] (S) [bend left=7] to node[below,sloped] {\tiny{$R_i\otimes U_{i+1}+L_{i+1}\otimes R_{i+1}R_i$}}  (W)  ;
    \draw[->] (W) [bend left=7] to node[above,sloped] {\tiny{$L_{i}\otimes 1$}}  (S)  ;
    \draw[->] (E)[bend right=7] to node[above,sloped] {\tiny{$R_{i+1}\otimes 1$}}  (S)  ;
    \draw[->] (S)[bend right=7] to node[below,sloped] {\tiny{$L_{i+1}\otimes U_i + R_i \otimes L_{i} L_{i+1}$}} (E) ;
    \draw[->] (W)[bend right=7] to node[below,sloped] {\tiny{$1\otimes L_i$}} (N) ;
    \draw[->] (N)[bend right=7] to node[above,sloped] {\tiny{$U_{i+1}\otimes R_i + R_{i+1} R_i \otimes L_{i+1}$}} (W) ;
    \draw[->] (E)[bend left=7] to node[below,sloped]{\tiny{$1\otimes R_{i+1}$}} (N) ;
    \draw[->] (N)[bend left=7] to node[above,sloped]{\tiny{$U_{i}\otimes L_{i+1} + L_{i} L_{i+1}\otimes R_i$}} (E) ;
  \end{tikzpicture}
\end{equation}
\end{enumerate}

Note that for a generator of type $\East$, the terms of
Type~\ref{type:OutsideLRP} with $j=i+2$ vanish; while for one of type
$\West$, the terms of Type~\ref{type:OutsideLRP} with $j=i-1$ vanish.

The bimodules $\Pos_i$ are graded by the set $\MGradingSet=\Q^{2n}$ 
as follows.
Let $e_1,\dots,e_{2n}$ be the standard basis for $\Q^{2n}$, and set
\begin{equation}
  \label{eq:GradeCrossing}
  \begin{array}{llll}
  \gr(\North)=\frac{e_i + e_{i+1}}{4} & 
  \gr(\West)=\frac{e_{i}-e_{i+1}}{4} & 
  \gr(\East)=\frac{-e_{i}+e_{i+1}}{4} &
  \gr(\South)=\frac{-e_{i}-e_{i+1}}{4}.
  \end{array}
\end{equation}

\begin{prop}
  \label{prop:PosIsDD}
  The bimodule $\lsup{\Alg_1,\DuAlg_2}\Pos_i$ is a type $DD$ bimodule.
\end{prop}

\begin{proof}
  The proof is straightforward; compare~\cite[Proposition~\ref{BK:prop:PosIsDD}]{BorderedKnots}.
\end{proof}

Taking opposite modules, we can form
\[{\overline{\lsup{\Alg_1,\DuAlg_2}
    \Pos_i}}=\overline{\Pos}_i^{\Alg_1,\DuAlg_2}
=\lsup{\Alg_1^{\opp},(\DuAlg_2)^{\opp}}{\overline{\Pos}}_i.\] Combining
this with the identification $\Opposite$ of $\Alg_1$ and $\DuAlg_2$
with their opposites, we arrive at a type $DD$ bimodule, denoted
$\lsup{\Alg_1,\DuAlg_2}\Neg_i$. In a little more detail, 
$\Neg_i$ has generators of type $\North$, $\South$, $\West$, $\East$ as in the definition of $\Pos_i$.
Differentials are also as enumerated earlier (Types~\ref{type:OutsideLRP}-\ref{type:UCCP}),
with those of Type~\ref{type:InsideP} replaced by those specified in the following diagram:
\begin{equation}
  \label{eq:NegativeCrossing}
\begin{tikzpicture}[scale=1.8]
    \node at (0,3) (N) {$\North$} ;
    \node at (-2,2) (W) {$\West$} ;
    \node at (2,2) (E) {$\East$} ;
    \node at (0,1) (S) {$\South$} ;
    \draw[->] (W) [bend left=7] to node[above,sloped] {\tiny{${U_{i+1}}\otimes{L_i}+{L_{i} L_{i+1}}\otimes{R_{i+1}}$}}  (N)  ;
    \draw[->] (N) [bend left=7] to node[below,sloped] {\tiny{${1}\otimes{R_{i}}$}}  (W)  ;
    \draw[->] (N)[bend right=7] to node[below,sloped] {\tiny{${1}\otimes{L_{i+1}}$}}  (E)  ;
    \draw[->] (E)[bend right=7] to node[above,sloped] {\tiny{${U_i}\otimes{R_{i+1}} + {R_{i+1} R_{i}}\otimes{L_i}$}} (N) ;
    \draw[->] (S)[bend right=7] to node[above,sloped] {\tiny{${R_i}\otimes{1}$}} (W) ;
    \draw[->] (W)[bend right=7] to node[below,sloped] {\tiny{${L_i}\otimes{U_{i+1}} + {R_{i+1}}\otimes{L_{i} L_{i+1}}$}} (S) ;
    \draw[->] (S)[bend left=7] to node[above,sloped]{\tiny{${L_{i+1}}\otimes{1}$}} (E) ;
    \draw[->] (E)[bend left=7] to node[below,sloped]{\tiny{${R_{i+1}}\otimes{U_{i}} + {L_i}\otimes{R_{i+1} R_{i}}$}} (S) ;
  \end{tikzpicture}
\end{equation}

\subsection{The $DA$ bimodules}
\label{subsec:DAcross}

In~\cite[Section~\ref{BK:sec:CrossingDA}]{BorderedKnots}, we
associated a bimodule to a crossing.  Specifically, continuing notation
from the previous section (for $i$, $n$, $k$), we constructed a
bimodule $\lsup{\Blg(2n,k)}\Pos^i_{\Blg(2n,k)}$, corresponding to a
partial knot diagram with a crossing in it.
To assist in computations occurring later in this paper, 
we recall the construction presently.

\begin{figure}[ht]
\input{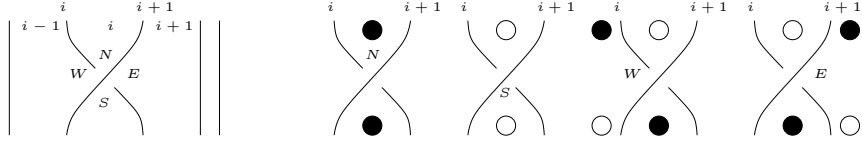}
\caption{\label{fig:PosCrossDA} {\bf{Positive crossing $DA$ bimodule generators.}}  
The four generator types are pictured to the right.}
\end{figure}

As a bimodule over $\IdempRing(2n,k)$-$\IdempRing(2n,k)$, $\Pos^i$ is
the submodule of $\IdempRing(2n,k) \otimes_{\Field}\IdempRing(2n,k)$,
generated by of $\Idemp{\x}\otimes \Idemp{\y}$ where either $\x=\y$ or
there is some $\w\subset \{1,\dots,i-1,i+1,\dots,2n\}$ with
$\x=\w\cup\{i\}$ and $\y=\w\cup \{i-1\}$ or $\y=\w \cup\{i+1\}$.
Thus, as in the case of the type $DD$ bimodules, there are  four
types of generators, of type $\North$, $\South$, $\West$, $\East$, as
pictured in Figure~\ref{fig:PosCrossDA}; i.e.
\begin{align*}
  \sum_{i\in\x} \Idemp{\x}\cdot \North \cdot \Idemp{\x}= \North, \qquad &
  \sum_{i\not\in\x} \Idemp{\x}\cdot \South \cdot \Idemp{\x} = \South, \\
  \sum_{\begin{small}\begin{array}{c}
      i\in\x \\
      i-1\not\in\x
    \end{array}
  \end{small}} \Idemp{\x}\cdot \West \cdot \Idemp{\{i-1\}\cup \x\setminus\{i\}}
  =  \West, \qquad &
  \sum_{\begin{small}
      \begin{array}{c}
        i\in\x \\
        i+1\not\in\x
        \end{array}
        \end{small}}\Idemp{\x}\cdot \East \cdot \Idemp{\{i+1\}\cup \x\setminus\{i\}}=\East.
\end{align*}

In cases where $i=1$, the $\delta^1_1$ and $\delta^1_2$ actions are specified by
the following diagram:
  \begin{equation}
    \label{eq:PositiveCrossingDA}
    \begin{tikzpicture}
    \node at (0,5) (N) {$\North$} ;
    \node at (-5,2) (W) {$\West$} ;
    \node at (5,2) (E) {$\East$} ;
    \node at (0,0) (S) {$\South$} ;
    \draw[->] (N) [loop above] to node[above,sloped]{\tiny{$1\otimes 1 + L_{1} L_{2}\otimes L_{1} L_{2}+ R_{2} R_{1} \otimes R_{2} R_{1}$}} (N);
    \draw[->] (W) [loop left] to node[above,sloped]{\tiny{$1\otimes 1$}} (W);
    \draw[->] (E) [loop right] to node[above,sloped]{\tiny{$1\otimes 1$}} (E);
    \draw[->] (S) [loop below] to node[below,sloped]{\tiny{$1\otimes 1$}} (S);
    \draw[->] (N) [bend right=7] to node[above,sloped]{\tiny {$U_{2}\otimes L_{1} + R_{2} R_{1} \otimes R_{2}$}}  (W)  ;
    \draw[->] (W) [bend right=7] to node[below,sloped,pos=.6]{\tiny {$1\otimes R_{1} + L_{1} L_{2} \otimes L_{2} U_{1}$}}  (N)  ;
    \draw[->] (E)[bend left=7] to node[below,sloped,pos=.6]{\tiny {$1 \otimes L_{2} + R_{2} R_{1} \otimes R_{1} U_{2}$}}  (N)  ;
    \draw[->] (N)[bend left=7] to node[above,sloped]{\tiny {$U_{1}\otimes R_{2} + L_{1} L_{2} \otimes L_{1}$}} (E) ;
    \draw[->,dashed] (W)to node[below,sloped]{\tiny {$L_{1}$}} (S) ;
    \draw[->,dashed] (E) to node[below,sloped]{\tiny{$R_{2}$}} (S) ;
    \draw[->] (W) [bend left=7] to node[above,sloped]{\tiny {$L_{1} L_{2} \otimes U_{1}$}} (E);
    \draw[->] (E) [bend left=7] to node[below,sloped]{\tiny {$R_{2} R_{1}\otimes U_{2}$}} (W);
  \end{tikzpicture}
\end{equation}
(We are using the conventions that a dashed arrow represents an
  operation $\delta^1_1$, so that the labelling algebra element
  appears in the output; and the other arrows indicate $\delta^1_2$
  actions, so that the algebra element on the right tensor factor is
  fed into the $\delta^1_2$ operation and the first tensor factor
  denotes the algebra element appearing in the output.)

These actions are further extended to the algebra with the following conventions.
For any $X\in \{\North,\West,\East,\South\}$ and
any pure algebra element $a\in\Blg(2)=\Blg(2,0)\oplus\Blg(2,1)\oplus\Blg(2,2)\oplus \Blg(2,3)$,
\begin{equation}
  \label{eq:U1U2}
  \delta^1_2(X, U_1 U_2\cdot a)=U_1 U_2 \cdot \delta^1_2(X,a).
\end{equation}
and also:
\begin{itemize}
\item If $b\otimes Y$ appears with non-zero coefficient in $\delta^1_2(\North,a)$,
then $(b\cdot U_2)\otimes Y$ appears with non-zero coefficient in $\delta^1_2(\North,a \cdot U_1)$
and $(b\cdot U_1)\otimes Y$ appears with non-zero coefficient in $\delta^1_2(\North,a \cdot U_2)$.
\item If $b\otimes Y$ appears with non-zero coefficient in 
$\delta^1_2(\West, a)$, then $(U_2\cdot b)\otimes \Yelt$ appears with non-zero coefficient
in $\delta^1_2(\West,U_1\cdot a)$.
\item If $b\otimes \Yelt$ appears with non-zero coefficient in 
$\delta^1_2(\East, a)$, then $(U_1\cdot b)\otimes \Yelt$ appears with non-zero coefficient
in $\delta^1_2(\East,U_2\cdot a)$.
\end{itemize}

Next we specify the actions $\delta^1_3$. To this end,
an algebra element in $\Blg(2)$  is called {\em elementary} if it is of the form
$p \cdot e$, where $p$ is a monomial in $U_{1}$ and $U_{2}$, and 
\[ e\in\{ 1,\qquad L_{1},\qquad R_{1},\qquad L_{2},\qquad R_{2},\qquad L_{1} L_{2},\qquad R_{2} R_{1}\}. \]
So far, we have defined $\delta^1_2$ by specifying the $\delta^1_2$ actions of the form
$\delta^1_2(X,a)$, where
$X\in\{\North,\South,\West,\East\}$ and $a$ is elementary.

We will now specify $\delta^1_3(X,a_1,a_2)$ where
$a_1$ and $a_2$ are elementary. Suppose that
$a_1\otimes a_2\neq 0$ (i.e. there is an idempotent state $\x$ so that 
$a_1\cdot\Idemp{\x}\neq 0$ and $\Idemp{\x}\cdot a_2\neq 0$); and suppose 
moreover that $U_1\cdot U_2$ does not divide either $a_1$ nor $a_2$. In this case,
$\delta^1_3(\South,a_1,a_2)$ is the sum of terms:
\begin{itemize}
  \item $R_{1} U_{1}^t \otimes \East$ if $(a_1,a_2)=(R_{1}, R_{2} U_{2}^t)$ and $t\geq 0$
  \item $L_2 U_1^t U_2^n\otimes \East$ if   $(a_1,a_2)\in$
\[
\begin{array}{llll}
 \{(U_1^{n+1},U_2^t), & (R_1U_1^n,L_1U_2^t),& (L_2U_1^{n+1}, R_2U_2^{t-1})\}
  &{\mbox{when $0\leq n<t$}} \\
 \{(U_2^t,U_1^{n+1}),& (R_1U_2^t, L_1U_1^n), &
(L_2U_2^{t-1}, R_2U_1^{n+1})\}& \mbox{when $1\leq t\leq n$}
\end{array}
\]
  \item $L_{2} U_{2}^n \otimes \West$ if $(a_1,a_2)=(L_{2},L_{1} U_{1}^n)$ and $n\geq 0$
\item $R_1 U_1^t U_2^n\otimes \West$ if 
$(a_1,a_2)\in $ 
\[ \begin{array}{llll}
\{(U_2^{t+1},U_1^n), &
(L_2U_2^t,R_2U_1^n), &
(R_1U_2^{t+1}, L_1U_1^{n-1})\} &\mbox{when $0\leq t<n$}\\
\{(U_1^n,U_2^{t+1}), &
(L_2U_1^n, R_2U_2^t), &
(R_1U_1^{n-1}, L_1U_2^{t+1})\}, &\mbox{when $1\leq n\leq t$}
\end{array}\]
\item $L_2 U_1^t U_2^n\otimes \North$  if
$(a_1,a_2)\in$
\[
\begin{array}{llll}
\{(U_1^{n+1},L_2U_2^t), &
(R_1U_1^n,L_1L_2U_2^t), &
(L_2U_1^{n+1}, U_2^{t})\} &\mbox{when $0\leq n<t$}\\
\{(L_2U_2^{t}, U_1^{n+1}), & (U_2^t,L_2U_1^{n+1})& 
(R_1U_2^t, L_1L_2U_1^n)\} &\mbox{when $1\leq t\leq n$}\\
\{(L_2, U_1^{n+1})\} &  & &\mbox{when $0=t\leq n$}
\end{array}
\]
\item 
$R_1 U_1^t U_2^n\otimes \North$ if 
$(a_1,a_2)$ is in the following list:
\[\begin{array}{llll}
\{(U_2^{t+1},R_1U_1^n), &
(L_2U_2^t,R_2R_1U_1^n), &
(R_1U_2^{t+1}, U_1^{n})\} &\mbox{when $0\leq t<n$ }\\
\{(R_1U_1^{n}, U_2^{t+1}) & 
(U_1^n,R_1U_2^{t+1}), &
(L_2U_1^n, R_2R_1U_2^t)\} &\mbox{when $1\leq n\leq t$}\\
\{(R_1, U_2^{t+1})\}& & & \mbox{when $0= n\leq t$.}
\end{array}\]
\end{itemize}

These operations are generalized to the case of arbitrary $i$,
by first defining a map $t$ from $\Blg(2n,k)$ to expressions in $U_1$, $U_2$, $R_1$, $L_1$, $R_2$, and $L_2$
\begin{equation} t(a)=\left\{
\begin{array}{ll}
R_{2}R_{1}U_{1}^{w_i(a)-\OneHalf} U_{2}^{w_{i+1}(a)-\OneHalf}
&{\text{if $w_i(a)\equiv w_{i+1}(a)\equiv \OneHalf\pmod{\Z}$}} \\
&{\text{and $v^{\x}_{i+1}<v^{\y}_{i+1}$}} \\
L_{1}L_{2}U_1^{w_i(a)-\OneHalf} U_{2}^{w_{i+1}(a)-\OneHalf}
&{\text{if $w_i(a)\equiv w_{i+1}(a)\equiv \OneHalf\pmod{\Z}$}} \\
&{\text{and $v^{\x}_{i+1}>v^{\y}_{i+1}$}} \\
R_{2}U_1^{w_i(a)} U_{2}^{w_{i+1}(a)-\OneHalf}
&{\text{if $w_i(a)\in\Z$ and $w_{i+1}(a)\equiv\OneHalf\pmod{\Z}$,}} \\
& {\text{and $v^{\x}_{i+1}<v^{\y}_{i+1}$}} \\
L_{2}U_1^{w_i(a)} U_{2}^{w_{i+1}(a)-\OneHalf}
&{\text{if $w_i(a)\in\Z$ and $w_{i+1}(a)\equiv\OneHalf\pmod{\Z}$,}} \\
& {\text{and $v^{\x}_{i+1}>v^{\y}_{i+1}$}} \\
R_{1}U_{1}^{w_i(a)-\OneHalf} U_{2}^{w_{i+1}(a)}
&{\text{if $w_i(a)\equiv \OneHalf\pmod{\Z}$ and 
    $w_{i+1}(a)\in\Z$,}} \\
& {\text{and $v^{\x}_{i}<v^{\y}_{i}$}} \\
L_{1}U_1^{w_i(a)-\OneHalf} U_{2}^{w_{i+1}(a)}
&{\text{if $w_i(a)\equiv \OneHalf\pmod{\Z}$ and 
    $w_{i+1}(a)\in\Z$,}} \\
& {\text{and $v^{\x}_{i}>v^{\y}_{i}$}} \\
 U_1^{w_i(a)} U_{2}^{w_{i+1}(a)}
&{\text{if $w_i(a)$ and $w_{i+1}(a)$ are integers.}}
\end{array}
\right.
\label{eq:DefType}
\end{equation}
Similarly, there is a map $t$ from generators of $\Pos^{i}$ to 
the four generators of $\LocPos$, that remembers only the type
($\North$, $\South$, $\West$, $\East$) of the generator of $\Pos^{i}$.

\begin{defn}
  \label{def:DefD}
  For $X\in\Pos^i$, an integer $\ell\geq 1$,
  and a sequence of algebra elements
  $a_1,\dots,a_{\ell-1}$ in $\Blg_0(m,k)$
  with specified weights, so that there exists a sequence of idempotent
  states $\x_0,\dots,\x_{\ell}$ with
  \begin{itemize}
    \item   $X=\Idemp{\x_0}\cdot X\cdot \Idemp{\x_1}$ \
    \item
      $a_t=\Idemp{\x_t}\cdot a_t\cdot \Idemp{\x_{t+1}}$ for
      $t=1,\dots,\ell-1$
      \item $\x_t$ and $\x_{t+1}$ are close enough for $t=0,\dots,\ell-1$
        (in the sense of Definition~\ref{def:CloseEnough}),
        \end{itemize}
  define
  $\delta^1_{\ell}(X,a_1,\dots,a_{\ell-1})\in \Blg(2n,k)\otimes \Pos^i$ 
  as the sum of pairs $b\otimes Y$ where $b\in\Blg(2n,k)$
  and $Y$ is a generator of $\Pos^i$, satisfying the following conditions:
  \begin{itemize}
  \item The weights of $b$ and $Y$ satisfy
    \begin{equation}
      \label{eq:GradedOperations}
      \gr(X)+\tau_i^{\gr}(\gr(a_1)+\dots+\gr(a_{\ell-1}))=\gr(b)+\gr(Y),
      \end{equation}
      where here $\gr(a_i)$ and $\gr(b)$ denote the weight gradings on
      $\Blg_1$ and $\Blg_2$; $\gr(X)$ and $\gr(Y)$ are as in Equation~\eqref{eq:GradeCrossing}, 
      and $\tau_i^{\gr}\colon\Q^{2n}\to\Q^{2n}$ is the linear transformation which acts as
      $\tau_i$ on the basis vectors for $\Q^{2n}$.
  \item
    There are generators $X_0$ and $Y_0$ with the same type
    (i.e. with the same label $\{\North,\South,\West,\East\}$)
    as $X$ and $Y$ respectively, so that 
    $t(b)\otimes Y_0$ appears with non-zero multiplicity in
    $\delta^1_{\ell}(X_0,t(a_1),\dots,t(a_{\ell-1}))$.
    \end{itemize}
\end{defn}

According to~\cite[Proposition~\ref{BK:prop:DAnoS}]{BorderedKnots}, the above operations
give $\Pos^i$ the structure of a type $DA$ bimodule.

Let $\Alg_1=\Alg(n,k,\Partition)$ and
$\Alg_2=\Alg(n,k,\tau_i(\Partition))$.
Our goal is to extend $\lsup{\Blg_2}\Pos^i_{\Blg_1}$ to a type DA bimodule
$\lsup{\Alg_2}\Pos^i_{\Alg_1}$.
For simplicity of notation, suppose that $i=1$. 

First, extend the actions coming from $\Blg_1$, 
so that they commute with the actions of all $C_{\{\beta,\gamma\}}$
with $\{\beta,\gamma\}\in\Partition$, in the sense that
\begin{align}
  \delta^1_2(X,C_{p}\cdot a_1)&=C_{\tau(p)}\cdot \delta^1_2(X,a_1) 
  \label{eq:CEquivariance} \\
  \delta^1_3(X,C_{p}\cdot a_1,a_2)&=
  \delta^1_3(X,a_1,C_{p}\cdot a_2)=
  C_{\tau(p)}\cdot   \delta^1_3(X,a_1,a_2). \nonumber
\end{align}
for all $p\in \Matching$.

When $\{1,2\}\not\in M$,
we add the further terms in $\delta^1_2$
from $\South$ to $\{\North,\West,\East\}$, as follows.
Choose $\alpha$ and $\beta$ so that $\{1,\alpha\},\{2,\beta\}\in\Matching$, and 
write
$C_1=C_{\{1,\alpha\}}$, $C_2=C_{\{2,\beta\}}$,
${\widetilde C}_1=C_{\{1,\beta\}}$, and ${\widetilde C}_2=C_{\{2,\alpha\}}$.
The additional terms are of the form:
\begin{equation}
  \label{ExtendingCs}
  \begin{array}{lll}
  \delta^1_2(\South,C_{2})\rightarrow U_\beta R_1\otimes \West & \\
  \delta^1_2(\South,C_{1})\rightarrow U_\alpha L_2\otimes \East &
  \delta^1_2(\South,C_{1} C_{2})\rightarrow U_{\beta}{\widetilde C}_{2} R_1\otimes \West + U_{\alpha}{\widetilde C}_{1} L_2\otimes \East \\
  \delta^1_{2}(\South,U_1C_{2})\rightarrow U_{\beta}U_1 L_2 \otimes \East &
  \delta^1_{2}(\South,U_1C_{1} C_{2})\rightarrow U_{\beta} U_1 {\widetilde C}_{2} L_2 \otimes \East \\
  \delta^1_{2}(\South,C_{1} U_2)\rightarrow U_{\alpha}R_1  U_2 \otimes \West &
    \delta^1_{2}(\South,C_{1} C_{2} U_2)\rightarrow U_{\alpha}{\widetilde C}_{1} R_1  U_2 \otimes \West \\
  \delta^1_{2}(\South, R_1 C_{2})\rightarrow U_{\beta}R_1\otimes \North &
  \delta^1_{2}(\South, R_1 C_{1} C_{2})\rightarrow U_{\beta} R_1 {\widetilde C}_{2}\otimes \North \\
  \delta^1_{2}(\South, L_2 C_{1})\rightarrow U_{\alpha}L_2\otimes \North
  &
  \delta^1_{2}(\South, C_{1} L_2 C_{2})\rightarrow U_{\alpha}{\widetilde C}_{1} L_2\otimes \North \\
  \delta^1_{2}(\South, R_1 C_{1} U_2)\rightarrow U_{\alpha}R_1 U_2\otimes \North 
  &
  \delta^1_{2}(\South, R_1 C_{1} U_2 C_{2})\rightarrow U_{\alpha}R_1 {\widetilde C}_{1} U_2\otimes \North \\
  \delta^1_{2}(\South, U_1 L_2 C_{2})\rightarrow U_{\beta} L_2 U_1\otimes \North & 
  \delta^1_{2}(\South, U_1 C_{1} L_2 C_{2})\rightarrow U_{\beta}U_1 L_2  {\widetilde C}_{2}\otimes \North 
\end{array}
\end{equation}
These are further extended to commute with multiplication by $U_1 U_2$, and multiplication by algebra elements with $w_1=w_2=0$.

With the above definition, we have, for example:
\[ \delta^1_2(\South,C_2)={\widetilde C}_1 \otimes \South + (U_\beta R_1)\otimes \West.\]

For the case of general $i$, use $i$ and $i+1$ in place of $1$ and $2$ in the subscripts for the algebra elements above.

\begin{prop}
  \label{prop:PosExt}
  The operations defined above give $\Pos^i$ the structure of a type $DA$ bimodule,
  $\lsup{\Alg(n,k,\tau(\Partition))}\Pos^i_{\Alg(n,k,\Partition)}$.
  Moreover, $\Pos^i$ is standard in the sense
  of Definition~\ref{def:Standard}.
\end{prop}
\begin{proof}
  There are two slightly different cases, depending on whether
  or not $i$ and $i+1$ are matched.

  Suppose that they are
  matched. In~\cite[Proposition~\ref{BK:prop:DAnoS}]{BorderedKnots},
  we constructed a type $DA$ bimodule $\lsup{\Blg(2n,k)}\Pos^i_{\Blg(2n,k)}$.
  Multiplication by $dC_{p}$ for all $p\in\Matching$ commute with the
  actions on this bimodule. (This is true also for $dC_{\{i,i+1\}}=U_i
  U_{i+1}$ by the construction from ~\cite{BorderedKnots}.)  It follows that
  the extension of $\Pos^i$ to $\Alg(n,k)$, extending so that all maps are
  $C_{p}$-equivariant, is a $DA$ bimodule, as well.

  The case where $i$ and $i+1$ are not matched requires a little extra
  care, when considering multiplication by $C_{\{\alpha,i\}}$ and
  $C_{\{i+1,\beta\}}$. Recall
  from~\cite[Proposition~\ref{BK:prop:CrossingGeneral}]{BorderedKnots}
  that the $DA$ bimodules for crossings are extended to an algebra
  containing elements $C_i$ with $d C_i=U_i$ and $d C_{i+1}=U_{i+1}$.
  Instead, here, we have here $d C_{\{\alpha,i\}}=U_\alpha U_{i}$ and
  $d C_{\{i+1,\beta\}}=U_\beta U_{i+1}$; and the formulas for the
  extension here have extra factors of $U_\alpha$ or $U_\beta$ in the
  output. Thus, the $DA$ bimodule relations in the present case follow
  from the same analysis as the $DA$ bimodule relations in that
  earlier case.

  Next, we verify that $\Pos^i$ is standard. Property~\ref{prop:LandsInB} is clear from
  the definition, so we turn to Property~\ref{prop:CStandard}. Suppose that
  $\{i,i+1\}\in\Partition$. Then, for all $X\in \Pos^i$
  and $\{\alpha,\beta\}\in\Partition$, 
  \begin{equation}
    \label{eq:CtoC}
    \delta^1_2(X,C_{\alpha,\beta})= C_{\alpha,\beta}\otimes X;
  \end{equation}
  also, if $a_1,\dots,a_m$ is any standard sequence 
  with $m>1$, so that $a_j=C_{\{\alpha,\beta\}}$ for some $j$, then
  \begin{equation}
    \label{eq:HighCsVanish}
    \delta^1_{m+1}(X,a_1,\dots,a_m)=0.
  \end{equation}
  Property~\ref{prop:CStandard} follows readily.
  
  When $\{i,i+1\}\not\in\Partition$, Equation~\eqref{eq:CtoC} is not true when $\{\alpha,\beta\}$
  contains one of $i$ or $i+1$; in those cases,  Equation~\eqref{eq:CEquivariance}
  gives
  \[ \delta^1_2(X,C_{\{\alpha,i\}})=C_{\{\alpha,i+1\}} \otimes X + V
  \qquad\text{and}\qquad
  \delta^1_2(X,C_{\{i+1,\beta\}})=C_{\{i,\beta\}} \otimes X + V'
  \]
  for $V, V'\in \Blg_2\otimes \Pos^i$.
  It still follows
  that 
  \[ \sum_{\{\alpha,\beta\}\in\Partition} \delta^1_{2}(X,C_{\alpha,\beta})
  =  \left(\sum_{\{\alpha,\beta\}\in\tau(\Partition)} C_{\{\alpha,\beta\}}\right) \otimes X + V'',\]
  with $V''\in\Blg_2\otimes \Pos^i$; 
  which, together with Equation~\eqref{eq:HighCsVanish} verifies
  Property~\ref{prop:CStandard}.
\end{proof}

We have the following analogue of ~\cite[Lemma~\ref{BK:lem:CrossingDADD}]{BorderedKnots}:

\begin{prop}
  \label{prop:DualCross}
  Let 
  \[ \Alg_1=\Alg(n,k,\Matching), \qquad
  \Alg_2=\Alg(n,k,\tau_i(\Matching)),\qquad 
  \DuAlg_1=\DuAlg(n,2n+1-k,\Matching).\]
  $\Pos^i$ is dual to $\Pos_i$, in the sense that
  \[ \lsup{\Alg_2}\Pos^i_{\Alg_1}\DT~
    \lsup{\Alg_1,\DuAlg_1}\CanonDD \simeq \lsup{\Alg_2,\DuAlg_1}\Pos_i.\]
\end{prop}

\begin{proof}
  There are two cases, according to whether or not $i$ and $i+1$ are
  matched in $\Matching$.  

  Consider the case where
  $\Matching$ matches $i$ with $\alpha\neq i+1$ and $i+1$ with $\beta\neq i$.

  For notational simplicity, let $i=1$.  A straightforward computation
  shows that $\Pos^1\DT \CanonDD$ is given by
  \begin{equation}
    \label{eq:TensorProdCross}
    \begin{tikzpicture}[scale=2]
    \node at (0,4) (N) {${\mathbf N}$} ;
    \node at (-2,2.5) (W) {${\mathbf W}$} ;
    \node at (2,2.5) (E) {${\mathbf E}$} ;
    \node at (0,0) (S) {${\mathbf S}$} ;
    \draw[->] (S) [bend left=10] to node[below,sloped] {\tiny{$U_\beta R_1 \otimes [E_2,E_\beta]+
        R_1 U_2 \otimes E_2 E_1+L_{2}\otimes R_{2}R_{1}$}}  (W)  ;
    \draw[->] (W) [bend left=10] to node[above,sloped] {\tiny{$L_{1}\otimes 1$}}  (S)  ;
    \draw[->] (E) [bend right=10] to node[above,sloped] {\tiny{$R_{2}\otimes 1$}}  (S)  ;
    \draw[->] (S)[bend right=10] to node[below,sloped] {\tiny{$R_{1} \otimes L_{1} L_{2} + L_2 U_1\otimes E_1 E_2
        + U_{\alpha} L_2 \otimes [E_1,E_\alpha]$}} (E) ;
    \draw[->] (W)[bend right=10]to node[below,sloped] {\tiny{$1\otimes L_{1}$}} (N) ;
    \draw[->] (N)[bend right=10] to node[above,sloped] {\tiny{$U_{2}\otimes R_{1} + R_{2} R_{1} \otimes L_{2}$}} (W) ;
    \draw[->] (E)[bend left=10]to node[below,sloped]{\tiny{$1\otimes R_{2}$}} (N) ;
    \draw[->] (N)[bend left=10] to node[above,sloped]{\tiny{$U_{1}\otimes L_{2} + L_{1} L_{2}\otimes R_{1}$}} (E) ;
    \draw[->] (N) [loop above] to node[above]{\tiny{$U_1\otimes E_2 + U_2\otimes E_1$}} (N);
    \draw[->] (W) [loop left] to node[above,sloped]{\tiny{$U_2\otimes E_{1}$}} (W);
    \draw[->] (E) [loop right] to node[above,sloped]{\tiny{$U_1\otimes E_{2}$}} (E);
    \draw[->] (E) [bend right=5] to node[above,pos=.3] {\tiny{$R_2 R_1 \otimes E_2$}} (W) ;
    \draw[->] (W) [bend right=5] to node[below,pos=.3] {\tiny{$L_1 L_2\otimes E_1$}} (E) ;
    \draw[->] (S) to node[below,sloped,pos=.3] {\tiny{$L_2\otimes E_1 R_2 + R_1\otimes L_1 E_2$}} (N) ;
    \end{tikzpicture}
  \end{equation}
  along with all the usual self-arrows $L_t\otimes R_t$, $R_t\otimes L_t$, $U_t\otimes E_t$
  for $t\neq 1,2$; 
  and $C_{\{m,\ell\}}\otimes [E_{\tau(m)},E_{\tau(\ell)}]$ .

    Consider the 
    map $h^1\colon \Pos^1\DT \CanonDD \to \Pos_1$
    \[ h^1(X) = \left\{\begin{array}{ll}
        \South+ (L_2\otimes E_1)\cdot \East + (R_1\otimes E_2)\cdot 
        \West & {\text{if $X=\South$}} \\
        X &{\text{otherwise.}}
      \end{array}
    \right.\]
    Let $g^1\colon \Pos_1\to \Pos^1\DT \CanonDD$ be given by the same formula.
    It is easy to verify that $h^1$ and $g^1$ are homomorphisms of type $DD$ structures, 
    $h^1\circ g^1 = \Id$, and $g^1\circ h^1=\Id$. 

    When $\{i,i+1\}\in\Matching$, $\Pos^1\DT \CanonDD$ is as shown in Equation~\eqref{eq:TensorProdCross},
    except that in this case, the two terms involving $E_\alpha$ and $E_\beta$ are to be deleted.
    With that modification, the above computations hold.
\end{proof}

\begin{defn}
  \label{def:NegCrossing}
  With $n$, $k$, $\Partition$, and $\tau_i$ as before, take the
  opposite module of  $\lsup{\Alg(n,k,\tau_i(\Partition)}\Pos^i_{\Alg(n,k,\Partition)}$,
  to get a module
  \[ \lsub{\Alg(n,k,\tau_i(\Partition))}({\overline
    \Pos}^i)^{\Alg(n,k,\Partition)} \cong
  \lsup{\Alg(n,k,\Partition)^{\op}}{\overline
    \Pos}^i_{\Alg(n,k,\tau_i(\Partition))^{\op}}; \] and use the
  identifications $\Opposite \colon\Alg(n,k,\Partition')\to
  \Alg(n,k,\Partition')^{\op}$ for both $\Partition'=\Partition$ and
  $\Partition'=\tau_i(\Partition)$ to define
  $\lsup{\Alg(n,k,\tau_i(\Partition))}\Neg^i_{\Alg(n,k,\Partition)}$, the
  {\em bimodule associated to a negative crossing}.
\end{defn}

Since the opposite module of a standard module is clearly standard,
and $\Pos^i$ is standard, it follows that $\Neg^i$ is standard, too.
Also, the analogue of Proposition~\ref{prop:DualCross} holds for 
$\Neg^i$, as well.

\newcommand\NxN{\North\North}
\newcommand\NxW{\North\West}
\newcommand\NxE{\North\East}
\newcommand\ExS{\East\South}
\newcommand\SxS{\South\South}
\newcommand\WxS{\West\South}
\newcommand\ExW{\East\West}
\section{The braid relations}
\label{sec:BraidRelations}

The bimodule $\Neg^i$ inverts $\Pos^i$; and all these bimodules satisfy
the braid relations, according to the following:

\begin{thm}
  \label{thm:BraidRelation}
  Fix integers $n$, $k$ with $0\leq k\leq 2n+1$,
  an integer
  $i$ with $1\leq i\leq 2n-1$,
  and a matching $\Partition$ on $\{1,\dots,2n\}$.
  Let
  $\Alg_1=\Alg(n,k,\Partition)$ and $\Alg_2=\Alg(n,k,\tau_i(\Partition))$.
  Then,
  \begin{equation}
    \lsup{\Alg_1}\Pos^i_{\Alg_2}\DT~^{\Alg_2}\Neg^i_{\Alg_1}\simeq~^{\Alg_1}\Id_{\Alg_1}\!\simeq~
    \lsup{\Alg_1}\Neg^i_{\Alg_2}\DT~^{\Alg_2}\Pos^i_{\Alg_1}.
    \label{eq:InvertPos} 
  \end{equation}
  Given $j\neq i$ with $1\leq j\leq 2n-1$, let
  \[
\Alg_3=\Alg(n,k,\tau_j \tau_i(\Partition))~\qquad\text{and}\qquad
\Alg_4=\Alg(n,k,\tau_j(\Partition)).\]
  If $|i-j|>1$
  \begin{equation}
    \lsup{\Alg_3}\Pos^j_{\Alg_2}\DT~^{\Alg_2}\Pos^i_{\Alg_1}
    \simeq~
    \lsup{\Alg_3}\Pos^i_{\Alg_4}\DT~^{\Alg_4}\Pos^j_{\Alg_1}
    \label{eq:FarBraids}.
    \end{equation}
    while if $j=i+1$,
    let 
    \[\Alg_5=\Alg(n,k,\tau_i\tau_{i+1}\tau_i(\Partition))
    \qquad{\text{and}}\qquad
    \Alg_6=\Alg(n,k,\tau_{i}\tau_{i+1}(\Partition));\]
    then,    
    \begin{equation}
      \lsup{\Alg_5}\Pos^i_{\Alg_3}\DT
      \lsup{\Alg_3}\Pos^{i+1}_{\Alg_2}\DT
      \lsup{\Alg_2}\Pos^i_{\Alg_1}
      \simeq~
      \lsup{\Alg_5}\Pos^{i+1}_{\Alg_6}\DT~^{\Alg_6}\Pos^{i}_{\Alg_4}
      \DT~^{\Alg_4}\Pos^{i+1}_{\Alg _1}.
    \label{eq:NearBraids}
    \end{equation}
\end{thm}

Analogous identities for the bimodules from~\cite{BorderedKnots} were
proved in~\cite[Theorem~\ref{BK:thm:BraidRelation}]{BorderedKnots}, by
direct computation.  A similar direct computation can be used to
verify the above theorem; we prefer instead to show that the
identities from the earlier paper formally imply the identities here.

To this end, recall that in~\cite{BorderedKnots}, we constructed the
Koszul dual algebra of $\Blg(2n,k)=\Blg(2n,k,\emptyset)$, which is an
algebra $\Blg(2n,2n+1-k,\{1,\dots,2n\})$.  This is an algebra that
contains new commuting elements $C_1,\dots,C_{2n}$ with $d C_i = U_i$.
(This algebra is similar to $\DuAlg(n,2n+1-k)$, except that in
$\DuAlg$, some of the new elements $E_i$ do not commute with one
another.)

Let $\Gamma$ be a disjoint union of $2n$ intervals (called ``strands''), equipped with a partition of its boundary
$\partial \Gamma=Y_1 \cup Y_2$, where $Y_1\cong Y_2\cong\{1,\dots,2n\}$. 
Suppose moreover that each interval in $\Gamma$ connects a point in $Y_1$ with a point in $Y_2$,
There is an induced one-to-one correspondence 
$f\colon \{1,\dots,2n\}\to\{1,\dots,2n\}$ between
the incoming strands and the out-going ones.

\begin{defn} 
  \label{def:Special}
  Choose any $0\leq k\leq 2n+1$.
  Let 
  \[ \Blg_1=\Blg(2n,2n+1-k) \qquad \Blg_2'=\Blg(2n,k,\{1,\dots,2n\})
  \qquad \Blg_2=\Blg(2n,k,\emptyset)\subset \Blg_2',\] Fix also
  $\Gamma$ as above, and let $X=\lsup{\Blg_2',\Blg_1}X$ be a type $DD$
  bimodule adapted to $\Gamma$ as in Definition~\ref{def:AdaptedDA}.
  We say $X$ is {\em special} if
  its differential
  \[\delta^1\colon X \to (\Blg_2'\otimes \Blg_1)\otimes X \]
  has  a decomposition $\delta^1=\epsilon + d_0$, where
  \[d_0\colon X \to (\Blg_2\otimes \Blg_1)\otimes X \]  
  and 
  \[ \epsilon(\x)=\left(\sum_{i} C_{f(i)}\otimes U_i\right)\otimes  \x.\]
\end{defn}

Suppose $X$ is special. Consider the map
\[ (\mu^{\Blg_2\otimes\Blg_1}_2\otimes \Id_X)\circ 
(\Id_{\Blg_2\otimes\Blg_1}\otimes d_0)\circ d_0 + 
(\mu^{\Blg_2\otimes\Blg_1}_1\otimes \Id_X)\circ d_0 \colon X \to (\Blg_2\otimes\Blg_1)\otimes X.\]
(Note in fact that $\mu^{\Blg_2\otimes\Blg_1}_1=0$.)
The type $DD$ bimodule structure relation and this hypothesis ensures that this map is multiplication by
$\sum_{i=1}^{2n}U_{f(i)}\otimes U_{i}$.
This allows us to make the following definition:

\begin{defn}
 Choose $\Blg_1$ and $\Blg_2'$ as in Definition~\ref{def:Special} and 
 a special type $DD$ bimodule
 $\lsup{\Blg_2',\Blg_1}X$.
 Let 
 \[ \DuAlg_1=\Alg(n,2n+1-k,\Matching)\qquad\text{and}\qquad 
 \Alg_2=\Alg(n,k,f\circ \Matching).\]
  There is an associated type $DD$ bimodule
  $F(X,f)=\lsup{\Alg_2,\DuAlg_1}F(X,f)$, which is defined by the differential
  \[ \delta^1(\x)=d_0(\x) + \left(\sum_{i} U_{f(i)}\otimes E_{i} + \sum_{i} C_{\{f(i),f(j)\}}\otimes 
    [E_{i},E_{j}]\right)\otimes \x.\]
\end{defn}

It follows immediately from the above remarks that $\lsup{\Alg_2,\DuAlg_1}F(X,f)$
is indeed a type $DD$ bimodule.

Examples of special bimodules are abundant.
Let 
\[\begin{array}{ll}
  \Blg=\Blg(2n,k,\emptyset)  &
  \Blg'=\Blg(2n,2n+1-k,\{1,\dots,2n\})  \\
  \Alg=\Alg(n,k,\Matching) 
  &
  \DuAlg=\DuAlg(n,2n+1-k,\Matching)
\end{array}\]
In~\cite[Section~\ref{BK:subsec:CanonDD}]{BorderedKnots},
we defined a canonical type $DD$ bimodule
$\lsup{\Blg',\Blg}\CanonDD$. 
We also defined bimodules associated to crossings, including
$\lsup{\Blg',\Blg}\Pos_i$ and
$\lsup{\Blg',\Blg}\Neg_i$. All of these are special. 

\begin{lemma}
  \label{lem:Fexamples}
  Choose any matching $\Matching$ on $\{1,\dots,2n\}$, and $i=1,\dots,2n-1$.
  The type $DD$ bimodules 
  $\lsup{\Blg',\Blg}\CanonDD$,
  $\lsup{\Blg',\Blg}\Pos_i$, and
  $\lsup{\Blg',\Blg}\Neg_i$ are special, and
  \[\begin{array}{lll}
    F(\lsup{\Blg',\Blg}\CanonDD,\Id)
    =     \lsup{\Alg,\DuAlg}\CanonDD &
    F(\lsup{\Blg',\Blg}\Pos_i,\tau_i)
    =     \lsup{\Alg,\DuAlg}\Pos_i &
    F(\lsup{\Blg',\Blg}\Neg_i,\tau_i)
    =     \lsup{\Alg,\DuAlg}\Neg_i;
    \end{array}\]
    where here $\Id$ denotes the identity map from $\{1,\dots,2n\}$ to itself.
\end{lemma}

\begin{proof}
  All of these statements are immediate from the definitions of the various bimodules
  (see~\cite[Sections~\ref{BK:subsec:CanonDD} and~\ref{BK:sec:CrossingDD}]{BorderedKnots}
  for the canonical $DD$ bimodule and the crossing bimodules respectively).
\end{proof}
 
\begin{lemma}
  \label{lem:PosPreservesF}
  Choose any $0\leq k \leq 2n+1$, and a braid $\Gamma$ with $2n$ strands inducing the correspondence $f$.
  Choose further any $i=1,\dots,2n-1$, and a matching $\Matching$ on $\{1,\dots,2n\}$. Let
  \[ 
  \begin{array}{ll}\Blg_1=\Blg(2n,2n+1-k),& \Blg_2'=\Blg_3'=\Blg(2n,k,\{1,\dots,2n\}) \\
    \DuAlg_1=\DuAlg(n,2n+1-k,\Matching)),&
    \Alg_2=\Alg(n,k,f(\Matching)) \\
    \Alg_3=\Alg(n,k,\tau_i\circ f(\Matching)).
  \end{array}\]
  Suppose that $\lsup{\Blg_2',\Blg_1}X$ be an special type $DD$ 
  adapted to  $\Gamma$. 
  Then, for the bimodule $\lsup{\Blg_3'}\Pos^i_{\Blg_2'}$ from~\cite{BorderedKnots}, we have that
  $\lsup{\Blg_3'}\Pos^i_{\Blg_2'}\DT \lsup{\Blg_2',\Blg_1}X$ is also special; and  there
  is a homotopy equivalence
  \[\lsup{\Alg_3}\Pos^i_{\Alg_2}\otimes \lsup{\Alg_2,\Alg_1'}F(X,f)
  \simeq 
  F(\lsup{\Blg_3'}\Pos^i_{\Blg_2'} \DT \lsup{\Blg_2',\Blg_1}X,\tau_i\circ f).\]
\end{lemma}

\begin{proof}
  For notational simplicity we assume $i=1$; the general case follows
  with minor notational modifications. There are two cases, according
  to whether or not $1$ and $2$ are matched in $f(\Matching)$. Assume first
  that $1$ and $2$ are not matched.

  The fact that $\lsup{\Blg_3'}\Pos^1_{\Blg_2'}\DT
  \lsup{\Blg_2',\Blg_1}X$ is special (with correspondence $\tau_1\circ f$)
  is a straightforward consequence
  of the bimodule $Q=\lsup{\Blg_3'}\Pos^1_{\Blg_2'}$, which has the property that
  \[\delta_2(X,C_j)-C_{\tau_1(j)}\otimes X \in \Blg_3\otimes Q\subset \Blg_3'\otimes Q.\]

  We will abbreviate $\Pos'=\lsup{\Blg_3'}\Pos^1_{\Blg_2'}$,
  $\Pos=\lsup{\Alg_3}\Pos^1_{\Alg_2}$.
  Consider the linear map
  \[ h^1\colon 
  F(\Pos' \DT \lsup{\Blg_2',\Blg_1}X,\tau_1\circ f)
  \to 
  \Pos\otimes \lsup{\Alg_2,\Alg_1'}F(X,f).\]
  with $h^1(T\otimes X)=T\otimes X$ for $T\in \{\North,\West,\East\}$ and
  \[ h^1(\South\otimes X)=
      \South\otimes X
      +  (R_1\otimes E_{g(2)})\otimes \West\otimes X
      + (L_2\otimes E_{g(1)})\otimes \East\otimes X,
\]
  where $g=f^{-1}$.

  We claim that $h^1$ is a type  $DD$ bimodule homomorphism; i.e.
  \[ (\mu_2\otimes \Id)\circ (\Id\otimes \delta^1)\circ h^1 
  + (\mu_2\otimes \Id)\circ (\Id\otimes h^1)\circ \delta^1 
  + (\mu_1 \otimes \Id_X) \circ h^1=0,\]
  where $\mu_1$ and $\mu_2$ are computed in the algebra $\Alg_3\otimes\Alg_1'$.
  There are various components of this map, decomposed according to how many 
  $E_{g(1)}$ and $E_{g(2)}$ appear in the $\Alg_1'$ factor.
  If there are none, then we have cancellation of terms
  \[ 0=
  \mathcenter{
    \begin{tikzpicture}[scale=.8]
    \node at (0,.5) (Xin) {$X$}; 
    \node at (0,-1) (d1) {$\delta^1$}; 
    \node at (2,-5) (Bout) {};
    \node at (-1,.5) (Pin) {$\Pos'$};
    \node at (-1,-5) (Pout) {};
    \node at (-1,-3) (muP) {$\delta^1_2$}; 
    \node at (-2,-5) (Aout) {}; 
    \node at (0,-4) (h1R) {$h^1$};
    \node at (-1,-4) (h1L) {$h^1$};
    \node at (0,-5) (Xout) {};
    \draw[algarrow] (d1) to node[above,sloped]{\tiny{$U_{g(1)}$}}(Bout) ;
    \draw[modarrow] (Pin) to node[right]{\tiny{$\South$}}(muP) ;
    \draw[modarrow] (muP) to node[right]{\tiny{$\East$}}(h1L) ;
    \draw[modarrow] (h1L) to node[right]{\tiny{$\East$}}(Pout) ;
    \draw[modarrow] (h1R) to node[right]{\tiny{$x$}}(Xout) ;
    \draw[modarrow] (Xin) to node[right]{\tiny{$x$}}(d1); 
    \draw[algarrow] (d1) to node[above,sloped] {\tiny{$C_1$}}(muP); 
    \draw[modarrow] (d1) to node[right]{\tiny{$x$}} (h1R) ;
    \draw[algarrow] (muP) to node[above,sloped]{\tiny{$L_2$}} (Aout);
  \end{tikzpicture}}
  +
  \mathcenter{
    \begin{tikzpicture}[scale=.8]
    \node at (0,.5) (Xin) {$X$}; 
    \node at (0,-1) (h1R) {$h^1$}; 
    \node at (-1,-1) (h1L) {$h^1$}; 
    \node at (1,-2.75) (Bmult) {$\mu_1$};
    \node at (2,-4) (Bout) {};
    \node at (-1,.5) (Pin) {$\Pos'$};
    \node at (-1,-4) (Pout) {};
    \node at (-2,-4) (Aout) {}; 
    \node at (0,-4) (Xout) {};
    \draw[algarrow] (h1R) to node[above,sloped]{\tiny{$E_{g(1)}$}}(Bmult) ;
    \draw[algarrow] (Bmult) to node[above,sloped]{\tiny{$U_{g(1)}$}} (Bout) ;
    \draw[modarrow] (Pin) to node[left]{\tiny{$\South$}}(h1L) ;
    \draw[modarrow] (h1L) to node[left]{\tiny{$\East$}}(Pout) ;
    \draw[modarrow] (Xin) to node[right]{\tiny{$x$}}(h1R); 
    \draw[modarrow] (h1R) to node[right]{\tiny{$x$}} (Xout) ;
    \draw[algarrow] (h1L) to node[above,sloped]{\tiny{$L_2$}} (Aout);
  \end{tikzpicture}}.
\]  
And a similar cancellation for the term involving $U_{g(2)}$.

For terms of the form $a\otimes b\cdot E_{g(1)}$ with $b\in\Blg_1$, we have the following cancellations.
First, for $T\in\{\North,\West\}$, we have:
\[
0=  \!\!\!\! \!\!\!\!
  \mathcenter{
    \begin{tikzpicture}[scale=.8]
   \node at (0,2) (Xin) {$X$}; 
   \node at (0,.5) (h1R) {$h^1$}; 
    \node at (0,-1) (d1) {$\delta^1$}; 
    \node at (2,-5) (Bout) {};
    \node at (-1,2) (Pin) {$\Pos'$};
    \node at (-1,.5) (h1L) {$h^1$};
    \node at (-1,-5) (Pout) {};
    \node at (-1,-3) (muP) {$\delta^1_2$}; 
    \node at (-2,-5) (Aout) {}; 
    \node at (0,-5) (Xout) {};
    \draw[algarrow] (d1) to node[above,sloped]{\tiny{$E_{g(1)}$}}(Bout) ;
    \draw[modarrow] (h1L) to node[right]{\tiny{$T$}}(muP) ;
    \draw[modarrow] (muP) to node[right]{\tiny{$T$}}(Pout) ;
    \draw[modarrow] (Xin) to node[right]{\tiny{$x$}} (h1R);
    \draw[modarrow] (Pin) to node[left]{\tiny{$T$}} (h1L);
    \draw[modarrow] (h1R) to node[right]{\tiny{$x$}}(d1); 
    \draw[algarrow] (d1) to node[above,sloped] {\tiny{$U_1$}}(muP); 
    \draw[modarrow] (d1) to node[right]{\tiny{$x$}} (Xout) ;
    \draw[algarrow] (muP) to node[above,sloped]{\tiny{$U_2$}} (Aout);
  \end{tikzpicture}}
  + 
  {  \mathcenter{
    \begin{tikzpicture}[scale=.8]
    \node at (0,.5) (Xin) {$X$}; 
    \node at (-1,-3)(d1L) {$\delta^1$};
    \node at (0,-3)(d1R) {$\delta^1$};
    \node at (2,-6) (Bout) {};
    \node at (-1,.5) (Pin) {$\Pos'$};
    \node at (-1,-6) (Pout) {};
    \node at (-3,-6) (Aout) {}; 
    \node at (0,-6) (Xout) {};
    \node at (-1,-5) (h1L) {$h^1$}; 
    \node at (0,-5) (h1R) {$h^1$}; 
    \draw[algarrow] (d1R) to node[above,sloped]{\tiny{$E_{g(1)}$}}(Bout) ;
    \draw[algarrow](d1L) to node[below,sloped]{\tiny{$U_2$}}(Aout);
    \draw[modarrow] (Pin) to node[right]{\tiny{$T$}}(d1L) ;
    \draw[modarrow] (d1L) to node[right]{\tiny{$T$}}(h1L) ;
    \draw[modarrow] (Xin) to node[right]{\tiny{$x$}}(d1R) ;
    \draw[modarrow] (d1R) to node[right]{\tiny{$x$}}(h1R) ;
    \draw[modarrow] (h1L) to node[right]{\tiny{$T$}}(Pout) ;
    \draw[modarrow] (h1R) to node[right]{\tiny{$x$}}(Xout) ;
  \end{tikzpicture}}}.
\]
Further cancellations involving $\delta^1_\ell$ with $\ell\leq 2$ in $\Pos'$ are:
\begin{align*}
0&=
  \mathcenter{
    \begin{tikzpicture}[scale=.8]
    \node at (0,.5) (Xin) {$X$}; 
    \node at (-1,-3)(d1L) {$\delta^1$};
    \node at (2,-5) (Bout) {};
    \node at (-1,.5) (Pin) {$\Pos'$};
    \node at (-1,-5) (Pout) {};
    \node at (-1,-1) (z1L) {$h^1$}; 
    \node at (0,-1) (z1R) {$h^1$}; 
    \node at (-3,-5) (Aout) {}; 
    \node at (-2.5,-4) (Amult) {$\mu_2$};
    \node at (0,-5) (Xout) {};
    \draw[algarrow](d1L) to node[below,sloped]{\tiny{$R_2$}}(Amult);
    \draw[algarrow] (z1R) to node[above,sloped]{\tiny{$E_{g(1)}$}}(Bout) ;
    \draw[modarrow] (Pin) to node[right]{\tiny{$\South$}}(z1L) ;
    \draw[modarrow] (z1L) to node[right]{\tiny{$\East$}}(d1L) ;
    \draw[modarrow] (d1L) to node[right]{\tiny{$\South$}}(Pout) ;
    \draw[modarrow] (z1R) to node[right]{\tiny{$x$}}(Xout) ;
    \draw[modarrow] (Xin) to node[right]{\tiny{$x$}}(z1R); 
    \draw[algarrow] (z1L) to node[above,sloped]{\tiny{$L_2$}} (Amult);
    \draw[algarrow] (Amult) to node[above,sloped]{\tiny{$U_2$}}(Aout);
  \end{tikzpicture}}
  + 
  {  \mathcenter{
    \begin{tikzpicture}[scale=.8]
    \node at (0,0) (Xin) {$X$}; 
    \node at (-1,-1)(d1L) {$\delta^1$};
    \node at (0,-1)(d1R) {$\delta^1$};
    \node at (2,-5) (Bout) {};
    \node at (-1,0) (Pin) {$\Pos'$};
    \node at (-1,-5) (Pout) {};
    \node at (-1,-3) (h1L) {$h^1$}; 
    \node at (0,-3) (h1R) {$h^1$}; 
    \node at (-3,-5) (Aout) {}; 
    \node at (0,-5) (Xout) {};
    \draw[algarrow] (d1R) to node[above,sloped]{\tiny{$E_{g(1)}$}}(Bout) ;
    \draw[algarrow](d1L) to node[below,sloped]{\tiny{$U_2$}}(Aout);
    \draw[modarrow] (Pin) to node[right]{\tiny{$\South$}}(d1L) ;
    \draw[modarrow] (d1L) to node[right]{\tiny{$\South$}}(h1L) ;
    \draw[modarrow] (Xin) to node[right]{\tiny{$x$}}(d1R) ;
    \draw[modarrow] (d1R) to node[right]{\tiny{$x$}}(h1R) ;
    \draw[modarrow] (h1L) to node[right]{\tiny{$\South$}}(Pout) ;
    \draw[modarrow] (h1R) to node[right]{\tiny{$x$}}(Xout) ;
  \end{tikzpicture}}}
 \\
0&=
  \mathcenter{
    \begin{tikzpicture}[scale=.8]
   \node at (0,2) (Xin) {$X$}; 
   \node at (0,.5) (h1R) {$h^1$}; 
    \node at (0,-1) (d1) {$\delta^1$}; 
    \node at (2,-5) (Bout) {};
    \node at (-1,2) (Pin) {$\Pos'$};
    \node at (-1,.5) (h1L) {$h^1$};
    \node at (-1,-5) (Pout) {};
    \node at (-1,-3) (muP) {$\delta^1_2$}; 
    \node at (-2,-5) (Aout) {}; 
    \node at (0,-5) (Xout) {};
    \draw[algarrow] (d1) to node[above,sloped]{\tiny{$E_{g(1)}$}}(Bout) ;
    \draw[modarrow] (h1L) to node[right]{\tiny{$\West$}}(muP) ;
    \draw[modarrow] (muP) to node[right]{\tiny{$\East$}}(Pout) ;
    \draw[modarrow] (Xin) to node[right]{\tiny{$x$}} (h1R);
    \draw[modarrow] (Pin) to node[left]{\tiny{$\West$}} (h1L);
    \draw[modarrow] (h1R) to node[right]{\tiny{$x$}}(d1); 
    \draw[algarrow] (d1) to node[above,sloped] {\tiny{$U_1$}}(muP); 
    \draw[modarrow] (d1) to node[right]{\tiny{$x$}} (Xout) ;
    \draw[algarrow] (muP) to node[above,sloped]{\tiny{$L_1 L_2$}} (Aout);
  \end{tikzpicture}}
  +
  \mathcenter{
    \begin{tikzpicture}[scale=.8]
    \node at (0,.5) (Xin) {$X$}; 
    \node at (-1,-1)(d1L) {$\delta^1$};
    \node at (2,-5) (Bout) {};
    \node at (-1,.5) (Pin) {$\Pos'$};
    \node at (-1,-5) (Pout) {};
    \node at (-1,-3) (z1L) {$h^1$}; 
    \node at (0,-3) (z1R) {$h^1$}; 
    \node at (-3,-5) (Aout) {}; 
    \node at (-2.5,-4) (Amult) {$\mu_2$};
    \node at (0,-5) (Xout) {};
    \draw[algarrow](d1L) to node[above,sloped]{\tiny{$L_1$}}(Amult);
    \draw[algarrow] (z1R) to node[above,sloped]{\tiny{$E_{g(1)}$}}(Bout) ;
    \draw[modarrow] (Pin) to node[right]{\tiny{$\West$}}(d1L) ;
    \draw[modarrow] (d1L) to node[right]{\tiny{$\South$}}(z1L) ;
    \draw[modarrow] (z1L) to node[right]{\tiny{$\East$}}(Pout) ;
    \draw[modarrow] (z1R) to node[right]{\tiny{$x$}}(Xout) ;
    \draw[modarrow] (Xin) to node[right]{\tiny{$x$}}(z1R); 
    \draw[algarrow] (z1L) to node[below,sloped]{\tiny{$L_2$}} (Amult);
    \draw[algarrow] (Amult) to node[above,sloped]{\tiny{$L_1 L_2$}}(Aout);
  \end{tikzpicture}} 
\end{align*}
Cancellations using $\delta^1_3$ in $\Pos'$ involve some term
$(a\otimes b)\otimes X$ appearing in $\delta^1(X)$, 
where $a\in\Blg_2'$ will have its type (as in Equation~\eqref{eq:DefType}) specified 
below (it will be $U_2^t$ or $R_1 U_2^{t}$ or $L_2 U_2^{t}$); and $b\in\Blg_1$.
Terms that land in $\East\otimes X$ cancel as below:
\begin{align*}
  0&=
  \mathcenter{
    \begin{tikzpicture}[scale=.8]
    \node at (.5,.5) (Xin) {$X$}; 
    \node at (.5,-1.5) (d1) {$\delta^1$}; 
    \node at (2.5,-5) (Bout) {};
    \node at (2,-3.5) (Bmult) {$\mu_2$};
    \node at (.5,-2.5) (d2) {$\delta^1$};
    \node at (-1,.5) (Pin) {$\Pos'$};
    \node at (-1,-5) (Pout) {};
    \node at (.5,-.5) (h1R) {$h^1$};
    \node at (-1,-.5) (h1L) {$h^1$};
    \node at (-1,-3.5) (muP) {$\delta^1_3$}; 
    \node at (-2,-5) (Aout) {}; 
    \node at (.5,-5) (Xout) {};
    \draw[algarrow] (d2) to node[below,sloped]{\tiny{$b$}}(Bmult) ;
    \draw[algarrow] (d2) to node[below,sloped]{\tiny{$U_2^t$}} (muP);
    \draw[algarrow] (d1) to node[above,sloped]{\tiny{$E_{g(1)}$}}(Bmult) ;
    \draw[modarrow] (Pin) to node[right]{\tiny{$\South$}}(h1L) ;
    \draw[modarrow] (Xin) to node[right]{\tiny{$x$}}(h1R) ;
    \draw[modarrow] (h1L) to node[right]{\tiny{$\South$}}(muP) ;
    \draw[modarrow] (muP) to node[right]{\tiny{$\East$}}(Pout) ;
    \draw[modarrow] (h1R) to node[right]{\tiny{$x$}}(d1); 
    \draw[algarrow] (d1) to node[above,sloped] {\tiny{$U_1$}}(muP); 
    \draw[modarrow,pos=.7] (d1) to node[right]{\tiny{$x$}} (d2) ;
    \draw[modarrow,pos=.7] (d2) to node[right]{\tiny{$y$}} (Xout) ;
    \draw[algarrow] (muP) to node[above,sloped]{\tiny{$L_2 U_1^t$}} (Aout);
    \draw[algarrow] (Bmult) to node[above,sloped]{\tiny{$E_{g(1)}\cdot b$}} (Bout);
  \end{tikzpicture}}
  +
  \mathcenter{
    \begin{tikzpicture}[scale=.8]
    \node at (.5,.5) (Xin) {$X$}; 
    \node at (.5,-.7) (h1R) {$h^1$}; 
    \node at (-1,-.7)(h1L) {$h^1$};
    \node at (2,-5) (Bout) {};
    \node at (1.5,-3) (Bmult) {$\mu_2$};
    \node at (.5,-2) (d2) {$\delta^1$};
    \node at (-1,.5) (Pin) {$\Pos'$};
    \node at (-1,-5) (Pout) {};
    \node at (-1,-3) (muP) {$\delta^1_2$}; 
    \node at (-3,-5) (Aout) {}; 
    \node at (-2,-4) (Amult) {$\mu_2$};
    \node at (.5,-5) (Xout) {};
    \draw[algarrow](h1L) to node[above,sloped]{\tiny{$L_2$}}(Amult);
    \draw[algarrow] (d2) to node[below,sloped]{\tiny{$b$}}(Bmult) ;
    \draw[algarrow] (d2) to node[below,sloped]{\tiny{$U_2^t$}} (muP);
    \draw[algarrow] (h1R) to node[above,sloped]{\tiny{$E_{g(1)}$}}(Bmult) ;
    \draw[modarrow] (Pin) to node[right]{\tiny{$\South$}}(h1L) ;
    \draw[modarrow] (h1L) to node[right]{\tiny{$\East$}}(muP) ;
    \draw[modarrow] (muP) to node[right]{\tiny{$\East$}}(Pout) ;
    \draw[modarrow] (h1R) to node[left]{\tiny{$x$}}(d2) ;
    \draw[modarrow] (Xin) to node[right]{\tiny{$x$}}(h1R); 
    \draw[modarrow,pos=.7] (d2) to node[right]{\tiny{$y$}} (Xout);
    \draw[algarrow] (muP) to node[below,sloped]{\tiny{$U_1^t$}} (Amult);
    \draw[algarrow] (Bmult) to node[above,sloped]{\tiny{$E_{g(1)}\cdot b$}} (Bout);
    \draw[algarrow] (Amult) to (Aout);
  \end{tikzpicture}}
\end{align*}
For terms that land in $\West$, we have similar cancellations.
For terms that land in $\North$, for $t>0$ we have cancellations
\begin{align*}
  0&=
  \mathcenter{
    \begin{tikzpicture}[scale=.8]
    \node at (.5,.5) (Xin) {$X$}; 
    \node at (.5,-1.5) (d1) {$\delta^1$}; 
    \node at (2.5,-5) (Bout) {};
    \node at (2,-3.5) (Bmult) {$\mu_2$};
    \node at (.5,-2.5) (d2) {$\delta^1$};
    \node at (-1,.5) (Pin) {$\Pos'$};
    \node at (-1,-5) (Pout) {};
    \node at (.5,-.5) (h1R) {$h^1$};
    \node at (-1,-.5) (h1L) {$h^1$};
    \node at (-1,-3.5) (muP) {$\delta^1_3$}; 
    \node at (-2,-5) (Aout) {}; 
    \node at (.5,-5) (Xout) {};
    \draw[algarrow] (d2) to node[below,sloped]{\tiny{$b$}}(Bmult) ;
    \draw[algarrow] (d2) to node[below,sloped]{\tiny{$L_2 U_2^t$}} (muP);
    \draw[algarrow] (d1) to node[above,sloped]{\tiny{$E_{g(1)}$}}(Bmult) ;
    \draw[modarrow] (Pin) to node[right]{\tiny{$\South$}}(h1L) ;
    \draw[modarrow] (Xin) to node[right]{\tiny{$x$}}(h1R) ;
    \draw[modarrow] (h1L) to node[right]{\tiny{$\South$}}(muP) ;
    \draw[modarrow] (muP) to node[right]{\tiny{$\North$}}(Pout) ;
    \draw[modarrow] (h1R) to node[right]{\tiny{$x$}}(d1); 
    \draw[algarrow] (d1) to node[above,sloped] {\tiny{$U_1$}}(muP); 
    \draw[modarrow,pos=.7] (d1) to node[right]{\tiny{$x$}} (d2) ;
    \draw[modarrow,pos=.7] (d2) to node[right]{\tiny{$y$}} (Xout) ;
    \draw[algarrow] (muP) to node[above,sloped]{\tiny{$L_2 U_1^t$}} (Aout);
    \draw[algarrow] (Bmult) to node[above,sloped]{\tiny{$E_{g(1)}\cdot b$}} (Bout);
  \end{tikzpicture}}
  +
  \mathcenter{
    \begin{tikzpicture}[scale=.8]
    \node at (.5,.5) (Xin) {$X$}; 
    \node at (.5,-.7) (h1R) {$h^1$}; 
    \node at (-1,-.7)(h1L) {$h^1$};
    \node at (2,-5) (Bout) {};
    \node at (1.5,-3) (Bmult) {$\mu_2$};
    \node at (.5,-2) (d2) {$\delta^1$};
    \node at (-1,.5) (Pin) {$\Pos'$};
    \node at (-1,-5) (Pout) {};
    \node at (-1,-3) (muP) {$\delta^1_2$}; 
    \node at (-3,-5) (Aout) {}; 
    \node at (-2,-4) (Amult) {$\mu_2$};
    \node at (.5,-5) (Xout) {};
    \draw[algarrow](h1L) to node[above,sloped]{\tiny{$L_2$}}(Amult);
    \draw[algarrow] (d2) to node[below,sloped]{\tiny{$b$}}(Bmult) ;
    \draw[algarrow] (d2) to node[below,sloped]{\tiny{$L_2 U_2^t$}} (muP);
    \draw[algarrow] (h1R) to node[above,sloped]{\tiny{$E_{g(1)}$}}(Bmult) ;
    \draw[modarrow] (Pin) to node[right]{\tiny{$\South$}}(h1L) ;
    \draw[modarrow] (h1L) to node[right]{\tiny{$\East$}}(muP) ;
    \draw[modarrow] (muP) to node[right]{\tiny{$\North$}}(Pout) ;
    \draw[modarrow] (h1R) to node[left]{\tiny{$x$}}(d2) ;
    \draw[modarrow] (Xin) to node[right]{\tiny{$x$}}(h1R); 
    \draw[modarrow,pos=.7] (d2) to node[right]{\tiny{$y$}} (Xout);
    \draw[algarrow] (muP) to node[below,sloped]{\tiny{$U_1^t$}} (Amult);
    \draw[algarrow] (Bmult) to node[above,sloped]{\tiny{$E_{g(1)}\cdot b$}} (Bout);
    \draw[algarrow] (Amult) to (Aout);
  \end{tikzpicture}}
\end{align*}
The analogous cancellation when $t=0$ has the form:
\begin{align*}
  0&=
  \mathcenter{
    \begin{tikzpicture}[scale=.8]
    \node at (.5,.5) (Xin) {$X$}; 
    \node at (.5,-1.5) (d1) {$\delta^1$}; 
    \node at (2.5,-5) (Bout) {};
    \node at (2,-3.5) (Bmult) {$\mu_2$};
    \node at (.5,-2.5) (d2) {$\delta^1$};
    \node at (-1,.5) (Pin) {$\Pos'$};
    \node at (-1,-5) (Pout) {};
    \node at (.5,-.5) (h1R) {$h^1$};
    \node at (-1,-.5) (h1L) {$h^1$};
    \node at (-1,-3.5) (muP) {$\delta^1_3$}; 
    \node at (-2,-5) (Aout) {}; 
    \node at (.5,-5) (Xout) {};
    \draw[algarrow] (d2) to node[below,sloped]{\tiny{$E_{g(1)}$}}(Bmult) ;
    \draw[algarrow] (d2) to node[below,sloped]{\tiny{$U_1$}} (muP);
    \draw[algarrow] (d1) to node[above,sloped]{\tiny{$b$}}(Bmult) ;
    \draw[modarrow] (Pin) to node[right]{\tiny{$\South$}}(h1L) ;
    \draw[modarrow] (Xin) to node[right]{\tiny{$x$}}(h1R) ;
    \draw[modarrow] (h1L) to node[right]{\tiny{$\South$}}(muP) ;
    \draw[modarrow] (muP) to node[right]{\tiny{$\North$}}(Pout) ;
    \draw[modarrow] (h1R) to node[right]{\tiny{$x$}}(d1); 
    \draw[algarrow] (d1) to node[above,sloped] {\tiny{$L_2$}}(muP); 
    \draw[modarrow,pos=.7] (d1) to node[right]{\tiny{$y$}} (d2) ;
    \draw[modarrow,pos=.7] (d2) to node[right]{\tiny{$y$}} (Xout) ;
    \draw[algarrow] (muP) to node[above,sloped]{\tiny{$L_2$}} (Aout);
    \draw[algarrow] (Bmult) to node[above,sloped]{\tiny{$E_{g(1)}\cdot b$}} (Bout);
  \end{tikzpicture}}
  +
  \mathcenter{
    \begin{tikzpicture}[scale=.8]
    \node at (.5,.5) (Xin) {$X$}; 
    \node at (.5,-.7) (h1R) {$h^1$}; 
    \node at (-1,-.7)(h1L) {$h^1$};
    \node at (2,-5) (Bout) {};
    \node at (1.5,-3) (Bmult) {$\mu_2$};
    \node at (.5,-2) (d2) {$\delta^1$};
    \node at (-1,.5) (Pin) {$\Pos'$};
    \node at (-1,-5) (Pout) {};
    \node at (-1,-3) (muP) {$\delta^1_2$}; 
    \node at (-3,-5) (Aout) {}; 
    \node at (-2,-4) (Amult) {$\mu_2$};
    \node at (.5,-5) (Xout) {};
    \draw[algarrow](h1L) to node[above,sloped]{\tiny{$L_2$}}(Amult);
    \draw[algarrow] (d2) to node[below,sloped]{\tiny{$b$}}(Bmult) ;
    \draw[algarrow] (d2) to node[below,sloped]{\tiny{$L_2$}} (muP);
    \draw[algarrow] (h1R) to node[above,sloped]{\tiny{$E_{g(1)}$}}(Bmult) ;
    \draw[modarrow] (Pin) to node[right]{\tiny{$\South$}}(h1L) ;
    \draw[modarrow] (h1L) to node[right]{\tiny{$\East$}}(muP) ;
    \draw[modarrow] (muP) to node[right]{\tiny{$\North$}}(Pout) ;
    \draw[modarrow] (h1R) to node[left]{\tiny{$x$}}(d2) ;
    \draw[modarrow] (Xin) to node[right]{\tiny{$x$}}(h1R); 
    \draw[modarrow,pos=.7] (d2) to node[right]{\tiny{$y$}} (Xout);
    \draw[algarrow] (muP) to node[below,sloped]{\tiny{$1$}} (Amult);
    \draw[algarrow] (Bmult) to node[above,sloped]{\tiny{$E_{g(1)}\cdot b$}} (Bout);
    \draw[algarrow] (Amult) to (Aout);
  \end{tikzpicture}}
\end{align*}

There are similar cancellations when the output is $(R_1 U_1^t\otimes
E_{g(1)}\cdot b)\otimes \North\otimes y$. There are also analogous cancellations for
terms of the form $a\otimes b \cdot E_{g(2)}$.

Cancellation of terms involving pairs $E_i$ and $E_j$ (with
$\{i,j\}\cap\{g(1),g(2)\}=\emptyset$ are straightforward.  
For terms involving $E_{g(1)}$ and $E_{g(\alpha)}$ landing in $\East$, we have cancellation:
\begin{align}0&=
  \mathcenter{
    \begin{tikzpicture}[scale=.7]
      \node at (0.5,.5) (Xin) {$X$}; 
    \node at (0.5,-1.5) (d1) {$\delta^1$}; 
    \node at (2.5,-5) (Bout) {};
    \node at (-1,.5) (Pin) {$\Pos'$};
    \node at (-1,-5) (Pout) {};
    \node at (-1,-3) (muP) {$\delta^1_2$}; 
    \node at (-2,-5) (Aout) {}; 
    \node at (0.5,-.5) (z1R) {$h^1$};
    \node at (-1,-.5) (z1L) {$h^1$};
    \node at (0.5,-5) (Xout) {};
    \draw[algarrow] (d1) to node[above,sloped]{\tiny{$[E_{g(1)},E_{g(\alpha)}]$}}(Bout) ;
    \draw[modarrow] (Pin) to node[right]{\tiny{$\South$}}(z1L) ;
    \draw[modarrow] (z1L) to node[left]{\tiny{$\South$}}(muP) ;
    \draw[modarrow] (Xin) to node[right]{\tiny{$x$}}(z1R) ;
    \draw[modarrow] (muP) to node[right]{\tiny{$\East$}}(Pout) ;
    \draw[modarrow] (z1R) to node[right]{\tiny{$x$}}(d1); 
    \draw[algarrow] (d1) to node[above,sloped] {\tiny{$C_{1,\alpha}$}}(muP); 
    \draw[modarrow] (d1) to node[right]{\tiny{$x$}} (Xout) ;
    \draw[algarrow] (muP) to node[above,sloped]{\tiny{$L_2 U_{\alpha}$}} (Aout);
  \end{tikzpicture}}
  \!\!\!\! +  \!\!\!\!
  \mathcenter{
    \begin{tikzpicture}[scale=.7]
    \node at (0,.5) (Xin) {$X$}; 
    \node at (-1,-1)(d1L) {$\delta^1$};
    \node at (0,-1)(d1R) {$\delta^1$};
    \node at (1.5,-3.1)(Bmult) {$\mu_2$};
    \node at (2,-5) (Bout) {};
    \node at (-1,.5) (Pin) {$\Pos'$};
    \node at (-1,-5) (Pout) {};
    \node at (-1,-2.5) (z1L) {$h^1$}; 
    \node at (0,-2.5) (z1R) {$h^1$}; 
    \node at (-3,-5) (Aout) {}; 
    \node at (-2.5,-3.1) (Amult) {$\mu_2$};
    \node at (0,-5) (Xout) {};
    \draw[algarrow](d1L) to node[above,sloped]{\tiny{$U_{\alpha}$}}(Amult);
    \draw[algarrow] (z1R) to node[below,sloped]{\tiny{$E_{g(1)}$}}(Bmult) ;
    \draw[algarrow] (d1R) to node[above,sloped]{\tiny{$E_{g(\alpha)}$}}(Bmult) ;
    \draw[algarrow] (Bmult) to node[above,sloped]{\tiny{$E_{g(\alpha)} E_{g(1)}$}}(Bout) ;
    \draw[modarrow] (Pin) to node[right]{\tiny{$\South$}}(d1L) ;
    \draw[modarrow] (d1L) to node[right]{\tiny{$\South$}}(z1L) ;
    \draw[modarrow] (z1L) to node[right]{\tiny{$\East$}}(Pout) ;
    \draw[modarrow] (z1R) to node[right]{\tiny{$x$}}(Xout) ;
    \draw[modarrow] (Xin) to node[right]{\tiny{$x$}}(d1R); 
    \draw[modarrow] (d1R) to node[right]{\tiny{$x$}}(z1R); 
    \draw[algarrow] (z1L) to node[below,sloped]{\tiny{$L_2$}} (Amult);
    \draw[algarrow] (Amult) to node[above,sloped]{\tiny{$U_{\alpha} L_2$}}(Aout);
  \end{tikzpicture}} 
  \!\!\!\!  +  \!\!\!\!
  \mathcenter{
    \begin{tikzpicture}[scale=.7]
    \node at (.5,.5) (Xin) {$X$}; 
    \node at (-1,-1)(d1L) {$h^1$};
    \node at (.5,-1)(d1R) {$h^1$};
    \node at (2,-3.1)(Bmult) {$\mu_2$};
    \node at (2.5,-5) (Bout) {};
    \node at (-1,.5) (Pin) {$\Pos'$};
    \node at (-1,-5) (Pout) {};
    \node at (-1,-3) (z1L) {$\delta^1_2$}; 
    \node at (.5,-2) (z1R) {$\delta^1$}; 
    \node at (-3,-5) (Aout) {}; 
    \node at (-2.5,-4) (Amult) {$\mu_2$};
    \node at (.5,-5) (Xout) {};
    \draw[algarrow](d1L) to node[above,sloped]{\tiny{$L_2$}}(Amult);
    \draw[algarrow](z1R) to node[below,sloped]{\tiny{$U_{\alpha}$}}(z1L);
    \draw[algarrow] (z1R) to node[below,sloped]{\tiny{$E_{g(\alpha)}$}}(Bmult) ;
    \draw[algarrow] (d1R) to node[above,sloped]{\tiny{$E_{g(1)}$}}(Bmult) ;
    \draw[algarrow] (Bmult) to node[above,sloped]{\tiny{$E_{g(1)} E_{g(\alpha)}$}}(Bout) ;
    \draw[modarrow] (Pin) to node[right]{\tiny{$\South$}}(d1L) ;
    \draw[modarrow] (d1L) to node[right]{\tiny{$\East$}}(z1L) ;
    \draw[modarrow] (z1L) to node[right]{\tiny{$\East$}}(Pout) ;
    \draw[modarrow] (z1R) to node[right]{\tiny{$x$}}(Xout) ;
    \draw[modarrow] (Xin) to node[right]{\tiny{$x$}}(d1R); 
    \draw[modarrow] (d1R) to node[left]{\tiny{$x$}}(z1R); 
    \draw[algarrow] (z1L) to node[below,sloped]{\tiny{$U_{\alpha}$}} (Amult);
    \draw[algarrow] (Amult) to node[above,sloped]{\tiny{$U_{\alpha} L_2$}}(Aout);
  \end{tikzpicture}} 
\label{eq:TermsWithC}
\end{align}
For terms landing in $\South$, the cancellation is easier to see.
There is an analogous cancellation involving $E_{g(2)}$ and $E_{g(\beta)}$.

Finally, 
for terms
involving $E_{g(1)}$ and $E_{g(2)}$, we have the following:
\begin{align}
  0&=
  \mathcenter{
    \begin{tikzpicture}[scale=.8]
    \node at (.5,1) (Xin) {$X$}; 
    \node at (.5,-1) (d1) {$\delta^1$}; 
    \node at (2.5,-5) (Bout) {};
    \node at (1.5,-3) (Bmult) {$\mu_2$};
    \node at (0.5,-2) (d2) {$\delta^1$};
    \node at (-1,1) (Pin) {$\Pos'$};
    \node at (-1,-5) (Pout) {};
    \node at (-1,-3) (muP) {$\delta^1_3$}; 
    \node at (-2,-5) (Aout) {}; 
    \node at (0.5,0) (h1R) {$h^1$};
    \node at (-1,0) (h1L) {$h^1$};
    \node at (0.5,-5) (Xout) {};
    \draw[modarrow] (Pin) to node[right]{\tiny{$\South$}} (h1L);
    \draw[modarrow] (Xin) to node[right]{\tiny{$x$}} (h1R);
    \draw[modarrow] (h1R) to node[right]{\tiny{$x$}} (d1);
    \draw[algarrow] (d2) to node[below,sloped]{\tiny{$E_{g(1)}$}}(Bmult) ;
    \draw[algarrow] (d2) to node[below,sloped]{\tiny{$U_1$}} (muP);
    \draw[algarrow] (d1) to node[above,sloped]{\tiny{$E_{g(2)}$}}(Bmult) ;
    \draw[modarrow] (h1L) to node[right]{\tiny{$\South$}}(muP) ;
    \draw[modarrow] (muP) to node[right]{\tiny{$\West$}}(Pout) ;
    \draw[algarrow] (d1) to node[above,sloped] {\tiny{$U_2$}}(muP); 
    \draw[modarrow] (d1) to node[left]{\tiny{$x$}} (d2) ;
    \draw[modarrow] (d2) to node[left]{\tiny{$x$}} (Xout) ;
    \draw[algarrow] (muP) to node[above,sloped]{\tiny{$R_1 U_2$}} (Aout);
    \draw[algarrow] (Bmult) to node[above,sloped]{\tiny{$E_{g(2)}\cdot E_{g(1)}$}} (Bout);
  \end{tikzpicture}}
  +
  \mathcenter{
    \begin{tikzpicture}[scale=.8]
    \node at (0.5,.5) (Xin) {$X$}; 
    \node at (0.5,-1) (h1R) {$h^1$}; 
    \node at (-1,-1)(h1L) {$h^1$};
    \node at (2.5,-5) (Bout) {};
    \node at (2,-3) (Bmult) {$\mu_2$};
    \node at (0.5,-2) (d2) {$\delta^1$};
    \node at (-1,.5) (Pin) {$\Pos'$};
    \node at (-1,-5) (Pout) {};
    \node at (-1,-3) (muP) {$\delta^1_2$}; 
    \node at (-3,-5) (Aout) {}; 
    \node at (-2,-4) (Amult) {$\mu_2$};
    \node at (0.5,-5) (Xout) {};
    \draw[algarrow](h1L) to node[above,sloped]{\tiny{$R_1$}}(Amult);
    \draw[algarrow] (d2) to node[below,sloped]{\tiny{$E_{g(1)}$}}(Bmult) ;
    \draw[algarrow] (d2) to node[below,sloped]{\tiny{$U_1$}} (muP);
    \draw[algarrow] (d1) to node[above,sloped]{\tiny{$E_{g(2)}$}}(Bmult) ;
    \draw[modarrow] (Pin) to node[right]{\tiny{$\South$}}(h1L) ;
    \draw[modarrow] (h1L) to node[right]{\tiny{$\West$}}(muP) ;
    \draw[modarrow] (muP) to node[right]{\tiny{$\West$}}(Pout) ;
    \draw[modarrow] (h1R) to node[left]{\tiny{$x$}}(d2) ;
    \draw[modarrow] (Xin) to node[right]{\tiny{$x$}}(h1R); 
    \draw[modarrow,pos=.7] (d2) to node[left]{\tiny{$x$}} (Xout);
    \draw[algarrow] (muP) to node[below,sloped]{\tiny{$U_2$}} (Amult);
    \draw[algarrow] (Bmult) to node[above,sloped]{\tiny{$E_{g(2)}\cdot E_{g(1)}$}} (Bout);
    \draw[algarrow] (Amult) to node[above,sloped]{\tiny{$R_1 U_2$}} (Aout);
  \end{tikzpicture}}  \nonumber \\
&+
  \mathcenter{
    \begin{tikzpicture}[scale=.7]
    \node at (0,.5) (Xin) {$X$}; 
    \node at (-1,-1)(d1L) {$\delta^1$};
    \node at (0,-1)(d1R) {$\delta^1$};
    \node at (1.5,-3.1)(Bmult) {$\mu_2$};
    \node at (2,-5) (Bout) {};
    \node at (-1,.5) (Pin) {$\Pos'$};
    \node at (-1,-5) (Pout) {};
    \node at (-1,-2.5) (z1L) {$h^1$}; 
    \node at (0,-2.5) (z1R) {$h^1$}; 
    \node at (-3,-5) (Aout) {}; 
    \node at (-2.5,-3.1) (Amult) {$\mu_2$};
    \node at (0,-5) (Xout) {};
    \draw[algarrow](d1L) to node[above,sloped]{\tiny{$U_{2}$}}(Amult);
    \draw[algarrow] (z1R) to node[below,sloped]{\tiny{$E_{g(2)}$}}(Bmult) ;
    \draw[algarrow] (d1R) to node[above,sloped]{\tiny{$E_{g(1)}$}}(Bmult) ;
    \draw[algarrow] (Bmult) to node[above,sloped]{\tiny{$E_{g(1)} E_{g(2)}$}}(Bout) ;
    \draw[modarrow] (Pin) to node[right]{\tiny{$\South$}}(d1L) ;
    \draw[modarrow] (d1L) to node[right]{\tiny{$\South$}}(z1L) ;
    \draw[modarrow] (z1L) to node[right]{\tiny{$\West$}}(Pout) ;
    \draw[modarrow] (z1R) to node[right]{\tiny{$x$}}(Xout) ;
    \draw[modarrow] (Xin) to node[right]{\tiny{$x$}}(d1R); 
    \draw[modarrow] (d1R) to node[right]{\tiny{$x$}}(z1R); 
    \draw[algarrow] (z1L) to node[below,sloped]{\tiny{$R_1$}} (Amult);
    \draw[algarrow] (Amult) to node[above,sloped]{\tiny{$R_1 U_2$}}(Aout);
  \end{tikzpicture}} 
  \!\!\!\!  +  \!\!\!\!
  \mathcenter{
    \begin{tikzpicture}[scale=.8]
    \node at (0.5,.5) (Xin) {$X$}; 
    \node at (0.5,-1) (h1R) {$h^1$}; 
    \node at (-1,-1)(h1L) {$h^1$};
    \node at (2.5,-5) (Bout) {};
    \node at (2,-3) (Bmult) {$\mu_2$};
    \node at (0.5,-2) (d2) {$\delta^1$};
    \node at (-1,.5) (Pin) {$\Pos'$};
    \node at (-1,-5) (Pout) {};
    \node at (-1,-3) (muP) {$\delta^1_2$}; 
    \node at (-3,-5) (Aout) {}; 
    \node at (-2,-4) (Amult) {$\mu_2$};
    \node at (0.5,-5) (Xout) {};
    \draw[algarrow](h1L) to node[above,sloped]{\tiny{$L_2$}}(Amult);
    \draw[algarrow] (d2) to node[below,sloped]{\tiny{$E_{g(2)}$}}(Bmult) ;
    \draw[algarrow] (d2) to node[below,sloped]{\tiny{$U_1$}} (muP);
    \draw[algarrow] (d1) to node[above,sloped]{\tiny{$E_{g(1)}$}}(Bmult) ;
    \draw[modarrow] (Pin) to node[right]{\tiny{$\South$}}(h1L) ;
    \draw[modarrow] (h1L) to node[right]{\tiny{$\East$}}(muP) ;
    \draw[modarrow] (muP) to node[right]{\tiny{$\West$}}(Pout) ;
    \draw[modarrow] (h1R) to node[left]{\tiny{$x$}}(d2) ;
    \draw[modarrow] (Xin) to node[right]{\tiny{$x$}}(h1R); 
    \draw[modarrow,pos=.7] (d2) to node[left]{\tiny{$x$}} (Xout);
    \draw[algarrow] (muP) to node[below,sloped]{\tiny{$R_2 R_1$}} (Amult);
    \draw[algarrow] (Bmult) to node[above,sloped]{\tiny{$E_{g(1)}\cdot E_{g(2)}$}} (Bout);
    \draw[algarrow] (Amult) to node[above,sloped]{\tiny{$R_1 U_2$}} (Aout);
  \end{tikzpicture}}  
\label{eq:FinalCancel}
\end{align}
There is a similar cancellation for terms involving $E_{g(1)}\cdot E_{g(2)}$
with output in $\East$.

This completes the verification when $1$ and $2$ are not matched.
When they are matched, the cancellations involving products of
$E_{g(1)}$ and $E_{g(2)}$ work a little differently.  Instead of the
cancellation from Equation~\eqref{eq:TermsWithC}, we have the
cancellations from Equation~\eqref{eq:FinalCancel}, noting that in
this case there are two canceling terms outputting $E_{g(1)}\cdot
E_{g(2)}$ and two outputting $E_{g(2)}\cdot E_{g(1)}$. The other
cancellations are as above.

\end{proof}

\begin{lemma}
  \label{lem:Fnatural}
  Let $\lsup{\Blg_2',\Blg_1}Y$ and 
  $\lsup{\Blg_2',\Blg_1}Z$
  be two type $DD$ bimodules adapted to the same braid,
  and let $\phi^1\in\Mor\lsup{\Blg_2',\Blg_1}(Y,Z)$ be a homomorphism with 
  \[ \phi^1\colon Y \to (\Blg_2\otimes\Blg_1)\otimes Z
  \subset (\Blg_2'\otimes\Blg_1)\otimes Z,\]
  then $\phi^1$ induces a morphism
  \[ F(\phi^1)\colon F(Y,f)\to F(Z,f).\]
\end{lemma}

\begin{proof}
  The formula for $F(\phi^1)$ is the same as the formula for $\phi^1$. 
  The structure equation is obviously still satisfied.
\end{proof}

With these pieces in place, the proof of Theorem~\ref{thm:BraidRelation} reduces quickly to its analogue,
from~\cite[Theorem~\ref{BK:thm:BraidRelation}]{BorderedKnots}:

\begin{proof}[Proof of Theorem~\ref{thm:BraidRelation}]
  The theorem is now an easy consequence of the braid relations for the modules from~\cite{BorderedKnots}.

  Since $(\Pos^i\DT\Neg^i)\DT\CanonDD \simeq 
  \Pos^i \DT(\Neg^i\DT \CanonDD) \simeq
  \Pos^i \DT\Neg_i$
  (by associativity of $\DT$ and Proposition~\ref{prop:DualCross}),
  the verification that
  $\Pos^i\DT \Neg^i\simeq \Id_{\Alg_1}$, will follow from the identity
  \begin{equation}
    \label{eq:EasyNegPosInv}
    \Pos^i\DT\Neg_i\simeq  \CanonDD,
  \end{equation}
  
  In more detail, we wish to show that
  \[\lsup{\Alg_1}\Pos^i_{\Alg_2}\DT \lsup{\Alg_2,\DuAlg_1}\Neg_i\simeq \lsup{\Alg_1,\DuAlg_2}\CanonDD.\]
  In~\cite[Section~\ref{BK:sec:Braid}]{BorderedKnots}, we verified that
  \begin{equation}
    \lsup{\Blg'}\Pos^i_{\Blg'}\DT\lsup{\Blg',\Blg}\Neg_i\simeq  \lsup{\Blg',\Blg}\CanonDD,
  \end{equation}
  (i.e. Equation~\eqref{eq:EasyNegPosInv} for the previous algebras)
  using a homotopy equivalence $\phi^1$ satisfying the hypothesis of 
  Lemma~\ref{lem:Fnatural}. (In fact, in this case, 
  \[ \phi^1\colon     \lsup{\Blg'}\Pos^i_{\Blg'}\DT\lsup{\Blg',\Blg}\Neg_i\to
  \lsup{\Blg',\Blg}\CanonDD \]  is the identity map.)
  Combining Lemma~\ref{lem:Fexamples},~\ref{lem:PosPreservesF}, with the the above result,
  we see that
  \begin{align*}
    \lsup{\Alg_1}\Pos^i_{\Alg_2}\DT \lsup{\Alg_2,\DuAlg_1}\Neg_i
  &= \lsup{\Alg_1}\Pos^i_{\Alg_2}\DT F(\lsup{\Blg_2',\Blg_1}\Neg_i)\\
    &\simeq 
  F(\lsup{\Blg_1'}\Pos_{\Blg_2'}\DT \lsup{\Blg_2',\Blg_1}\Neg_i) \\
  &\simeq 
  F(\lsup{\Blg_1',\Blg_1}\CanonDD) \\
  &=
  \lsup{\Alg_1,\DuAlg_1}\CanonDD.
  \end{align*}
  
  Equations~\eqref{eq:FarBraids} and~\eqref{eq:NearBraids} follow from the same logic.
\end{proof}

\section{Bimodule associated to a maximum}
\label{sec:Max}

Fix $\Partition$ a matching on $\{1,\dots,2n\}$, $k$ an integer with
$0\leq k\leq 2n+1$, and $c$ an integer with $1\leq c\leq 2n+1$.

Let $\phi_c\colon \{1,\dots,2n\}\to \{1,\dots,2n+2\}$ be the map
\begin{equation}
\label{eq:DefInsert}
\phi_c(j)=\left\{\begin{array}{ll}
j &{\text{if $j< c$}} \\
j+2 &{\text{if $j\geq c$.}}
\end{array}\right.
\end{equation}
Let 
\begin{equation}
  \label{eq:MaxAlgebras}
  \Alg_1=\Alg(n,k,\Partition)
\qquad{\text{and}}\qquad
\Alg_2=\Alg(n+1,k+1,\phi_c(\Partition)\cup\{c,c+1\})
\end{equation}
We will define the bimodule associated to the partial knot diagram
containing a single local maximum connecting the strands $c$ and $c+1$
(in the output), denoted $\lsup{\Alg_2}\Max_{\Alg_1}$.
Before doing this, we describe its dual type $DD$ bimodule.

\subsection{The type $DD$ bimodule}
\label{subsec:MaxDD}

Let 
\begin{equation}
  \label{eq:DuAlg1}
\DuAlg_1=\DuAlg(n,2n+1-k,\Matching)
\end{equation}
We define  a type $DD$ bimodule $\lsup{\Alg_2,\DuAlg_1}\Max_c$, as follows.

To describe the underlying vector space for $\Max_c$, we proceed as follows.
We call an idempotent state $\y$ for $\Alg_2$ an {\em allowed idempotent state} if
\begin{equation}
  \label{eq:AllowedMaxState}
c\in\y \qquad{\text{and}}\qquad
|\y\cap \{c-1,c+1\}|\leq 1.
\end{equation}
Consider the map $\psi'$ from
allowed idempotent states $\y$ for $\Alg_2$ to idempotent states  for $\DuAlg_1$,
where $\x=\psi'(\y)\subset \{0,\dots,2n\}$ is characterized by the property that
\begin{equation}
  \label{eq:SpecifyPsiPrime}
  |\y\cap \{c-1,c,c+1\}| + 
  |\x\cap \{c-1\}| =2~\qquad{\text{and}}~\qquad \phi_c(\x)\cap \y=\emptyset.
\end{equation}

As a vector space, $\Max_c$ is spanned by vectors that are in
one-to-one correspondence with allowed idempotent states for $\Alg_2$. 
The module structure over the ring of idempotents $\IdempRing(\Alg_2)\otimes\IdempRing(\DuAlg_1)$,
is specified as follows.
If ${\mathbf P}={\mathbf P}_{\y}$ is the generator associated to 
the idempotent state $\y$, then 
\[ (\Idemp{\y}\otimes\Idemp{\psi'(\y)})\cdot {\mathbf P}_{\y}=
{\mathbf P}_{\y}.\]

To specify the differential, 
consider the element 
$A\in \Alg_2\otimes\DuAlg_1$ 
\begin{align*}
A&=(L_{c} L_{c+1}\otimes 1) + 
(R_{c+1} R_{c}\otimes 1) 
 + \sum_{i=1}^{2n} L_{\phi(i)}\otimes R_i + R_{\phi(i)}\otimes L_i \\
& +  C_{\{c,c+1\}}\otimes 1
+ \sum_{i=1}^{2n} U_{\phi(i)}\otimes E_i 
+ \sum_{\{i,j\}\in\Matching} C_{\{\phi(i),\phi(j)\}} \otimes [E_i,E_j]
  \end{align*}
where we have dropped the subscript $c$ from $\phi_c=\phi$.
Let 
\[ \delta^1({\mathbf P}_{\y}) = (\Idemp{\y}\otimes \Idemp{\psi'(\y)})\cdot A \otimes
\sum_{\z} {\mathbf P}_{\z},\]
where the latter sum is taken over all allowed idempotent states $\z$ for $\Alg_2$.

\begin{lemma}
  The space $\lsup{\Alg_2,\DuAlg_1}\Max_c$ defined above, and equipped with the map
  \[ \delta^1\colon \Max_c \to (\Alg_2\otimes \DuAlg_1)\otimes \Max_c,\]
  specified above, is a type $DD$ bimodule over $\Alg_2$ and $\DuAlg_1$.
\end{lemma}

\begin{proof}
  The proof is a straightforward adaptation of Lemma~\ref{lem:CanonicalIsDD}.
\end{proof}

It is helpful to understand $\Max_c$ a little more explicitly.
To this end, we classify the allowed idempotents for $\Alg_2$ into three types,  labeled $\XX$, $\YY$, and $\ZZ$:
\begin{itemize}
  \item $\y$ is of type $\XX$ if $\y\cap \{c-1,c,c+1\}=\{c-1,c\}$,
  \item $\y$ is of type $\YY$ if $\y\cap \{c-1,c,c+1\}=\{c,c+1\}$,
  \item $\y$ is of type $\ZZ$ if $\y\cap \{c-1,c,c+1\}=\{c\}$. 
\end{itemize}
There is a corresponding classification of the generators ${\mathbf P}_{\y}$ into $\XX$, $\YY$, and $\ZZ$, according to the type of $\y$;
see Figure~\ref{fig:DDcap}.

\begin{figure}[ht]
\input{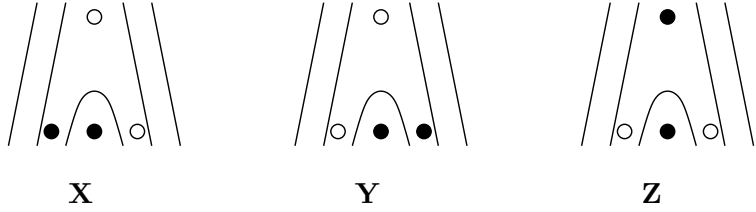}
\caption{\label{fig:DDcap} {\bf{$DD$ bimodule of a maximum.}}  
Three generator types are illustrated.}
\end{figure}

With respect to this decomposition, 
terms in  the differential are of the following types:
\begin{enumerate}[label=($\Omega$-\arabic*),ref=($\Omega$-\arabic*)]
\item 
  \label{type:MOutsideLR}
  $R_{\phi(j)}\otimes L_j$
  and $L_{\phi(j)}\otimes R_{j}$ for all 
  $j\in \{1,\dots,2n\}\setminus \{c-1,c\}$ (connecting
  generators of the same type). 
\item
  \label{type:MUC}
  $U_{\phi(i)}\otimes E_i$ for $i=1,\dots,2n$ (connecting
  generators of the same type). 
\item 
  \label{type:MUC2}
  $C_{\{\phi(i),\phi(j)\}}\otimes [E_i,E_j]$
  for all $\{i,j\}\in\Partition$  (connecting
  generators of the same type). 
\item
  $C_{\{c,c+1\}}\otimes 1$  (connecting
  generators of the same type). 
\item 
  \label{type:MInsideCup}
  Terms in the diagram below connect  generators
  of different types.
\begin{equation}
  \label{eq:CritDiag}
  \begin{tikzpicture}[scale=1.5]
    \node at (-1.5,0) (X) {$\XX$} ;
    \node at (1.5,0) (Y) {$\YY$} ;
    \node at (0,-2) (Z) {$\ZZ$} ;
    \draw[->] (X) [bend right=7] to node[below,sloped] {\tiny{$R_{c+1} R_{c}\otimes 1$}}  (Y)  ;
    \draw[->] (Y) [bend right=7] to node[above,sloped] {\tiny{$L_{c} L_{c+1}\otimes 1$}}  (X)  ;
    \draw[->] (X) [bend right=7] to node[below,sloped] {\tiny{$L_{c-1}\otimes
R_{c-1}$}}  (Z)  ;
    \draw[->] (Z) [bend right=7] to node[above,sloped] {\tiny{$R_{c-1}\otimes L_{c-1}$}}  (X)  ;
    \draw[->] (Z) [bend right=7] to node[below,sloped] {\tiny{$L_{c+2}\otimes R_{c}$}}  (Y)  ;
    \draw[->] (Y) [bend right=7] to node[above,sloped] {\tiny{$R_{c+2}\otimes L_{c}$}}  (Z)  ;
  \end{tikzpicture}
\end{equation}
\end{enumerate}

With the understanding that
if $c=1$, then the terms containing $L_{c-1}$ or $R_{c-1}$ are missing; 
similarly, if $c=2n+1$, the terms containing $R_{c+2}$ and $L_{c+2}$ are missing.

\subsection{The $DA$ bimodule of a maximum}

We will describe now the type $DA$ bimodule
$\lsup{\Alg_2}\Max^c_{\Alg_1}$ promised in the beginning of the
section (where $\Alg_1$ and $\Alg_2$ are specified in
Equation~\eqref{eq:MaxAlgebras}). This is the bimodule associated to a
region in the knot diagram where there are no crossings and a single
local maximum, which connects the $c^{th}$ and the $(c+1)^{st}$
outgoing strand; see Figure~\ref{fig:Maximum}.
\begin{figure}[ht]
\input{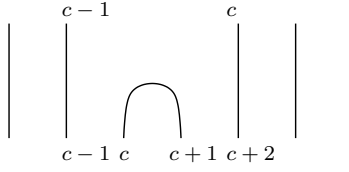}
\caption{\label{fig:Maximum} {\bf{Picture of maximum.}}}
\end{figure}

Recall the definition of allowed idempotent states for $\Alg_2$, 
specified in Equation~\eqref{eq:AllowedMaxState}.
There is a map $\psi$ from allowed idempotent states for
$\Alg_2$ to idempotent states for $\Alg_1$,
given by 
\[ \psi(\x)=\left\{\begin{array}{ll}
\phi^{-1}(\y) & {\text{if $c+1\not\in \y$}} \\
\phi^{-1}(\y)\cup\{c-1\} & {\text{if $c+1\in \y$}}
\end{array}
\right.\]
Observe that $\psi(\y)=\{0,\dots,2n\}\setminus \psi'(\y)$,
where $\psi'$ is the map specified in  Equation~\eqref{eq:SpecifyPsi}.

A basis for the underlying vector space of
$\lsup{\Alg_2}\Max^c_{\Alg_1}$ is specified by the allowed idempotent
states for $\Alg_2$.  The bimodule structure, over the rings of
idempotents $\IdempRing(\Alg_1)$ and $\IdempRing(\Alg_2)$, is
specified as follows.  If ${\mathbf Q}_{\y}$ is the
generator associated to the allowed idempotent state $\y$, then 
\[ \Idemp{\y}\cdot {\mathbf Q}_{\y}\cdot \Idemp{\psi(\y)}={\mathbf Q}_{\y}.\]

The map
\[ \delta^1_1\colon \lsup{\Alg_2}\Max^c_{\Alg_1}\to \Alg_2\otimes_{\IdempRing(\Alg_2)} {\lsup{\Alg_2}\Max^c_{\Alg_1}}\]
is given specified by 
\[ \delta^1({\mathbf Q}_{\y})= \Idemp{\y}\cdot \left(R_{c+1} R_{c} + L_{c} L_{c+1} + 
C_{\{c,c+1\}}\right)\otimes \sum_{\z} {\mathbf Q}_{\z}.\]
where the sum is taken over all allowed idempotents $\z$ for $\Alg_2$.

Split the bimodule
\[ \lsup{\Alg_2}\Max^c_{\Alg_1}\cong \XX\oplus\YY\oplus \ZZ \]
according to the types of the corresponding idempotents as defined in 
Section~\ref{subsec:MaxDD};
i.e.
\begin{align*}
  \XX &= \bigoplus_{\{\y\big|\y\cap\{c-1,c,c+1\}=\{c-1,c\}\}} {\mathbf Q}_\y \\
  \YY &= \bigoplus_{\{\y\big|\y\cap\{c-1,c,c+1\}=\{c,c+1\}\}} {\mathbf Q}_\y \\
  \ZZ &= \bigoplus_{\{\y\big|\y\cap\{c-1,c,c+1\}=\{c\}\}} {\mathbf Q}_\y
\end{align*}

With respect to this splitting, $\delta^1_1$ can be expressed as
\begin{align}
\delta^1_1(\XX)=& C_{\{c,c+1\}}\otimes \XX + R_{c+1}R_{c}\otimes \YY \nonumber \\
\delta^1_1(\YY)=& C_{\{c,c+1\}}\otimes \YY + L_{c}L_{c+1}\otimes \XX \label{eq:DeltaOneOnePos} \\
\delta^1_1(\ZZ)= &  C_{\{c,c+1\}}\otimes \ZZ; \nonumber 
\end{align}

To define $\delta^1_2$, it is helpful to recall the following lemma, which is a
straightforward adaptation
of~\cite[Lemma~\ref{BK:lem:ConstructDeltaTwo}]{BorderedKnots}.

\begin{lemma}
  \label{lem:ConstructDeltaTwo}
  If $\x$ is an allowed idempotent state (for $\Alg_2$) and 
  $\y$ is an idempotent state for $\Alg_1$ so that $\psi(\x)$ and $\y$
  are close enough, then there is an allowed idempotent state $\z$ with
  $\psi(\z)=\y$ so that there is a map
  \[ \Phi_{\x}\colon 
  \Idemp{\psi(\x)}\cdot \Alg_1\cdot \Idemp{\y}\to
  \Idemp{\x}\cdot \Alg_2\cdot \Idemp{\z}\]
  with the following properties:
  \begin{itemize}
    \item $\Phi_\x$ maps the portion of $\Idemp{\psi(\x)}\cdot \Blg_1\cdot \Idemp{\y}$ with weights
      $(v_1,\dots,v_{2n})$ surjectively onto the portion of 
      $\Idemp{\x}\cdot \Blg_2\cdot\Idemp{\z}$ with 
      $w_{\phi_c(i)}=v_i$ and $w_{c}=w_{c+1}=0$, 
    \item $\Phi_{\x}$ further satisfies the relations
      \[
      \Phi_{\x}(U_i\cdot a)=U_{\phi_c(i)} \cdot \Phi_{\x}(a) 
      \qquad{\text{and}}\qquad
      \Phi_{\x}(C_{p}\cdot a)=C_{\phi_c(p)}\cdot \Phi_{\x}(a)\]
      for any $i\in 1,\dots,2n$, $p\in\Matching$, and $a\in \Idemp{\psi(\x)}\cdot \Alg_1\cdot \Idemp{\y}$.
    \end{itemize}
    Moreover, the state $\z$ is uniquely characterized by the existence of 
    such a map $\Phi_\x$.
\end{lemma}

If $\x$ is an allowed idempotent state for $\Alg_2$
and $a=\Idemp{\psi(\x)}\cdot a\cdot \Idemp{\y}\in\Alg_1$ is a non-zero algebra element,
let
\[ \delta^1_2({\mathbf Q}_{\x},a)=\Phi_{\x}(a)\otimes 
{\mathbf Q}_\z;\] 
where  $\z$ is as in Lemma~\ref{lem:ConstructDeltaTwo}.

\begin{thm}
  \label{thm:MaxDA}
  The above specified actions $\delta^1_1$ and $\delta^1_2$ (and
  $\delta^1_\ell=0$ for all $\ell>2$) give
  $\lsup{\Alg_2}\Max^c_{\Alg_1}$ the structure of a $DA$ bimodule.
  Moreover, $\lsup{\Alg_2}\Max^c_{\Alg_1}$ is standard, in the sense of Definition~\ref{def:Standard}.
\end{thm}

\begin{proof}
  This is straightforward; compare~\cite[Theorem~\ref{BK:thm:MaxDA}]{BorderedKnots}.
\end{proof}

The type $DA$ bimodule associated to a local maximum is dual to the
type $DD$ bimodule of a local maximum, in the following sense.

\begin{prop}
  \label{prop:MaxDual}
  Fix $\Alg_1$, $\Alg_2$, and $\DuAlg_1$ as in Equations~\eqref{eq:MaxAlgebras} and~\eqref{eq:DuAlg1},
  and $c$ with $1\leq c \leq 2n+1$. There is an identification
  \[ \lsup{\Alg_2}\Max^c_{\Alg_1}\DT \lsup{\Alg_1,\DuAlg_1}\CanonDD
  \simeq \lsup{\Alg_2,\DuAlg_1}\Max_c.\]
\end{prop}
\begin{proof}
  This is a straightforward computation.
\end{proof}

\subsection{The trident relation}
\label{subsec:MaxTrident}

As in~\cite{BorderedKnots}, an important ingredient in the invariance
proof is the following relation between bimodules for critical points and crossings.

\begin{prop}
  \label{prop:BasicTrident}
  Fix integers $0\leq k\leq 2n+1$, and a matching
  $\Matching_1$ on $\{1,\dots,2n\}$. Let
  \[\begin{array}{ll}
    \Matching_2=\phi_{c+1}(\Matching_1)\cup\{c+1,c+2\} & \Matching_3=\tau_c(\Matching_2) \\
    \Matching_4=\tau_{c+1}(\Matching_3) \\
    \DuAlg_1=\Alg(n,2n+1-k,\Matching_1), &\Alg_2=\Alg(n+1,k+1,\Matching_2) \\
    \Alg_3=\Alg(n+1,k+1,\Matching_3) & \Alg_4=\Alg(n+1,k+1,\Matching_4)
    \end{array}\]
  There is homotopy equivalence of graded bimodules:
  \begin{equation}
    \label{eq:BasicTrident}
    \lsup{\Alg_3}\Neg^{c}_{\Alg_2}\DT~^{\Alg_2,\DuAlg_1}\Max_{c+1}
    \simeq
    \lsup{\Alg_3}\Pos^{c+1}_{\Alg_4}\DT~^{\Alg_4,\DuAlg_1}\Max_{c}
  \end{equation}
\end{prop}

\begin{proof}
  For notational simplicity, suppose that $c=1$.
  A straightforward computation shows that
  $\lsup{\Alg_3}\Pos^{2}_{\Alg_4}\DT~^{\Alg_4,\DuAlg_1}\Max_{1}$ is the bimodule
  generated by $\North$, $\West$, $\South$, $\East$, and with arrows:
  \begin{equation}
    \label{eq:TensorProdTrident}
    \begin{tikzpicture}[scale=1.5]
    \node at (0,4) (N) {${\mathbf N}$} ;
    \node at (-2,2.5) (W) {${\mathbf W}$} ;
    \node at (2,2.5) (E) {${\mathbf E}$} ;
    \node at (0,0) (S) {${\mathbf S}$} ;
    \draw[->] (S) [bend left=10] to node[below,sloped] {\tiny{$U_\alpha R_2 \otimes [E_\alpha,E_1]
        +L_1 L_3 \otimes R_1$}}  (W)  ;
    \draw[->] (W) [bend left=10] to node[above,sloped] {\tiny{$L_{2}\otimes 1$}}  (S)  ;
    \draw[->] (E) [bend right=10] to node[above,sloped] {\tiny{$R_{3}\otimes 1$}}  (S)  ;
    \draw[->] (S)[bend right=10] to node[below,sloped] {\tiny{$U_1 L_3 \otimes 1+R_2 R_1\otimes L_1$}} (E) ;
    \draw[->] (W)[bend right=10]to node[below,sloped] {\tiny{$R_1\otimes 1$}} (N) ;
    \draw[->] (N)[bend right=10] to node[above,sloped] {\tiny{$L_{1} U_{3}\otimes 1 + R_{3} R_{2} \otimes L_1$}} (W) ;
    \draw[->] (E)[bend left=10]to node[below,sloped]{\tiny{$1\otimes R_1$}} (N) ;
    \draw[->] (N)[bend left=10] to node[above,sloped]{\tiny{$U_{2}\otimes L_1 + L_{1} L_{2} L_{3}\otimes 1$}} (E) ;
    \draw[->] (N) [loop above] to node[above]{\tiny{$U_2\otimes E_1$}} (N);
    \draw[->] (E) [loop right] to node[above,sloped]{\tiny{$U_2\otimes E_1$}} (E);
    \draw[->] (E)to node[above,pos=.3] {\tiny{$R_3 R_2 \otimes E_1$}} (W) ;
    \draw[->] (S) to node[below,sloped,pos=.3] {\tiny{$R_2 R_1\otimes E_1$}} (N) ;
    \end{tikzpicture}
  \end{equation}
  along with further self-arrows of the following form:
  $C_{\{1,3\}}\otimes 1$;
  $L_{j+2}\otimes R_j$, $R_{j+2}\otimes L_j$, and $U_{j+2}\otimes E_j$ for
  $j=2,\dots,2n$;  $C_{\{m+2,\ell+2\}}\otimes [E_{m},E_{\ell}]$
  for those $\{m,\ell\}\in\Matching_1$ with $1\not\in\{m,\ell\}$;
  and $C_{\{2,\alpha+2\}}\otimes [E_{1},E_{\alpha}]$ with $\{1,\alpha\}\in\Matching_1$.

  After making the substitution $\South'=\South+ (R_2 \otimes E_1)\otimes \West$,
  we arrive at the more symmetric version:
  \begin{equation}
    \label{eq:SymmetricTrident}
    \begin{tikzpicture}[scale=1.5]
    \node at (0,4) (N) {${\mathbf N}$} ;
    \node at (-2,2) (W) {${\mathbf W}$} ;
    \node at (2,2) (E) {${\mathbf E}$} ;
    \node at (0,0) (S) {${\mathbf S}'$} ;
    \draw[->] (S) [bend left=10] to node[below,sloped] {\tiny{$R_2\otimes U_1
        +L_1 L_3 \otimes R_1$}}  (W)  ;
    \draw[->] (W) [bend left=10] to node[above,sloped] {\tiny{$L_{2}\otimes 1$}}  (S)  ;
    \draw[->] (E) [bend right=10] to node[above,sloped] {\tiny{$R_{3}\otimes 1$}}  (S)  ;
    \draw[->] (S)[bend right=10] to node[below,sloped] {\tiny{$U_1 L_3 \otimes 1+R_2 R_1\otimes L_1$}} (E) ;
    \draw[->] (W)[bend right=10]to node[below,sloped] {\tiny{$R_1\otimes 1$}} (N) ;
    \draw[->] (N)[bend right=10] to node[above,sloped] {\tiny{$L_{1} U_{3} \otimes 1 + R_{3} R_{2} \otimes L_1$}} (W) ;
    \draw[->] (E)[bend left=10]to node[below,sloped]{\tiny{$1\otimes R_1$}} (N) ;
    \draw[->] (N)[bend left=10] to node[above,sloped]{\tiny{$U_{2}\otimes L_1 + L_{1} L_{2} L_{3}\otimes 1$}} (E) ;
    \end{tikzpicture}
  \end{equation}
  along with the same self-arrows as before, and the further
  self-arrow of the form $U_2\otimes E_1$ (now on all four generator types
  rather than on only two, as before).

  A symmetric computation reduces
  $\lsup{\Alg_3}\Neg^{c}_{\Alg_2}\DT~^{\Alg_2,\DuAlg_1}\Max_{c+1}$ to
  the same bimodule.
\end{proof}

\newcommand\nothing{}
\section{Working in the dual algebra}
\label{sec:DuAlg}

For some purposes (especially in Section~\ref{sec:Min}), it will be
convenient to work in the algebra $\DuAlg$ and with bimodules defined
over this algebra.  Many of our previous constructions have a
straightforward adaption to this case.

\subsection{Bimodules associated to positive crossings}
\label{subsec:DuAlgCross}
We construct the type $DA$ bimodule of a positive crossing
$\lsup{\DuAlg(n,k,\tau(\Matching))}\Pos^i_{\DuAlg(n,k,\Matching)}$.

The construction starts from $\lsup{\Blg(2n,k)}\Pos^i_{\Blg(2n,k)}$
from~\cite{BorderedKnots} and recalled in Section~\ref{subsec:DAcross}. 

In Section~\ref{subsec:DAcross}, we explained how to extend the
bimodules from $\Blg(2n,k)$ to $\Alg(n,k)$. The extension to
$\DuAlg(n,k)$ is done similarly, as follows. We spell out the
case where $i=1$; the general case follows with straightforward
notational changes. There are two subcases, according to whether or not
$1$ and $2$ are matched.

First, we extend all the actions coming from
$\Blg_1=\Blg(2n,k)$ so that they commute
with the action of the $E_j$ with $j=1,\dots,2n$, so that the following identities hold:
\begin{align}
  \delta^1_2(X,E_j\cdot a_1)&=E_{\tau(j)}\cdot 
    \delta^1_2(X,a_1) 
  \label{eq:EEquivariance} \\
  \delta^1_3(X,E_j\cdot a_1,a_2)&=
  E_{\tau(j)}\cdot   \delta^1_3(X,a_1,a_2) \nonumber \\
  \delta^1_3(X,a_1\cdot E_j,a_2)&=
  \delta^1_3(X,a_1,E_j\cdot a_2) \nonumber \\
  \delta^1_3(X,a_1,a_2\cdot E_j)&=
  \delta^1_3(X,a_1,a_2)   \cdot E_{\tau(j)},
  \label{eq:LastAction}
\end{align}
where $a_1,a_2\in \DuAlg_1$ are arbitrary algebra elements.
Note that these extension rules are slightly more complicated than the
corresponding rules for the bimodules over $\Alg$ because the new
variables $E_j$ here are not commutative. In particular, the last
identity should be interpreted as follows: multiplication by
$E_{\tau(j)}$ takes place on the right of the algebra output of
$\delta^1_3(X,a_1,a_2)$.

Note that any algebra element in $\DuAlg$ can be written as the
product of an element of $\Blg_1$ with a word in the various
$E_j$. Thus, the first three relations can be regarded as definitions
of the actions, inductively defined in terms of the length of the
words in the $E_j$.  Taking the first three relations as the
definition, Equation~\eqref{eq:LastAction} follows easily.

There are two subcases, according to whether $\{1,2\}\in \Matching$.
Consider first the case where $\{1,2\}\not\in\Matching$.
In this case, we must add more terms in the definition, as follows.
Fix $\alpha$ and $\beta$ so that
$\{1,\alpha\},\{2,\beta\}\in\Matching$. We add further terms in
$\delta^1_2$ from $\South$ to $\{\North,\West,\East\}$, as follows:
\begin{equation}
  \label{eq:ExtendingEs}
  \begin{array}{lll}
  \nothing(\South,E_{2})\rightarrow 
  R_1\otimes \West & \\
  \nothing(\South,E_{1})\rightarrow 
  L_2\otimes \East &   (\South,E_{1} E_{2})\to 
  E_{2} R_1\otimes \West +
  E_{1} L_2\otimes \East \\
  (\South,U_1E_{2})
  \rightarrow 
  U_1 L_2 \otimes \East &   (\South,U_{1}E_{1} E_{2})\to U_1 E_{2} L_2 \otimes \East\\
  \nothing(\South,E_{1} U_2)\rightarrow R_1  U_2 \otimes \West &   (\South,E_{1} E_{2} U_2)\to E_{1} R_1  U_2 \otimes \West \\
  \nothing(\South, R_1 E_{2})\rightarrow R_1\otimes \North &   (\South, R_1 E_{1} E_{2})\to R_1 E_{2}\otimes \North \\
  \nothing(\South, L_2 E_{1})\rightarrow L_2\otimes \North
  &   (\South, E_{1} L_2 E_{2})\to E_{1} L_2\otimes \North \\
  \nothing(\South, R_1 E_{1} U_2)\rightarrow R_1 U_2\otimes \North 
  &   (\South, R_1 E_{1} U_2 E_{2})\to R_1 E_{1} U_2\otimes \North \\
  \nothing(\South, U_1 L_2 E_{2})\rightarrow  L_2 U_1\otimes \North & 
  (\South, U_1 E_{1} L_2 E_{2})\to U_1 L_2  E_{2}\otimes \North
\end{array}
\end{equation}
These further terms are extended as follows. Suppose we have a term $(X,a)\rightarrow b\otimes Y$
as above. Then
for any words $c_1$ and $c_2$ in the $E_k$ with $k\not\in\{1,2\}$ so that $c_1 a c_2\neq 0$, we have a further
term $(X,c_1 a c_2)\to c_1 b c_2\otimes Y$. 

For example,
\begin{align*}
  \delta^1_{2}(\South,E_1 E_\alpha)&=E_2 E_{\alpha}\otimes \South + L_2 E_\alpha\otimes \East \\
  \delta^1_{2}(\South,E_\alpha E_1)&=E_\alpha E_2\otimes \South + L_2 E_\alpha\otimes \East.
\end{align*}

New terms are added so that the following relations hold:
\begin{align*}
  \delta^1_2(X,\llbracket E_1,E_{\alpha}\rrbracket\cdot a)&=\llbracket E_2,E_{\alpha}\rrbracket\cdot \delta^1_2(X,a) \\
  \delta^1_2(X,\llbracket E_2,E_{\beta}\rrbracket\cdot a)&=\llbracket E_1,E_{\beta}\rrbracket\cdot \delta^1_2(X,a) \\
\end{align*}
for all $a\in\DuAlg$.
These new terms in $\delta^1_2$  are further extended to commute with multiplication by $U_1 U_2$; and then they 
are extended to commute with multiplication by algebra
elements with $w_1=w_2=0$.

For example, since $E_1 \cdot  E_\alpha \cdot E_1 = \llbracket E_1,E_\alpha\rrbracket\cdot E_1$, it follows that
\[ \delta^1_2(\South, E_1\cdot E_{\alpha}\cdot E_1)=
\llbracket E_2,E_\alpha\rrbracket\cdot L_2 \otimes \East + 
E_2\cdot E_{\alpha}\cdot E_2 \otimes \South.\]

Note also that
\begin{align*}
  \delta^1_3(X,\llbracket E_1,E_{\alpha}\rrbracket\cdot a_1,a_2)
  &= 
  \delta^1_3(X,a_1,\llbracket E_1,E_{\alpha}\rrbracket\cdot a_2)
  = 
  \llbracket E_2,E_{\alpha}\rrbracket\cdot \delta^1_3(X, a_1,a_2) \\
  \delta^1_3(X,\llbracket E_2,E_{\beta}\rrbracket\cdot a_1,a_2)
  &= 
  \delta^1_3(X,a_1,\llbracket E_2,E_{\beta}\rrbracket\cdot a_2)
  = 
  \llbracket E_1,E_{\beta}\rrbracket\cdot \delta^1_3(X, a_1,a_2).
\end{align*}

When $\{1,2\}\in \Matching$, we add terms as in
Equation~\eqref{eq:ExtendingEs}, further adding terms in the second
column, as well, with the order of $E_1$ and $E_2$ exchanged. For
example, we add the term
\[  (\South,E_{2} E_{1})\to E_{2} R_1\otimes \West +  E_{1} L_2\otimes \East.\]
These are extended over monomials in $E_k\not\in\{1,2\}$ as before.
Finally, we extend all actions so that 
\[ \delta^1_2(X,\llbracket E_1,E_2\rrbracket\cdot a)=\llbracket E_1,E_2\rrbracket\cdot \delta^1_2(X,a)\]
for all $a\in\DuAlg$.

The resulting bimodule has a straightforward generalizations when $i\neq 1$.

\begin{prop}
  \label{prop:PosDuAlg}
  The above operations give 
  $\lsup{\DuAlg(n,k,\tau(\Matching))}\Pos^i_{\DuAlg(n,k,\Matching)}$
  the structure of a type $DA$ bimodule over the specified algebras.
\end{prop}

\begin{proof}
  The proof is similar to the proof of Proposition~\ref{prop:PosExt}.
\end{proof}

We have the following analogue of Proposition~\ref{prop:DualCross}.
For the statement below, the order of two algebras appearing in the type
$DD$ bimodules $\CanonDD$ and $\Pos_i$ are opposite to what they were
in Sections~\ref{sec:DefCanonDD} and~\ref{subsec:DDcross}.  Moreover,
the tensor product in the statement can be formed because
$\lsup{\DuAlg_2}\Pos^i_{\DuAlg_1}$ is suitably bounded; indeed, by the
definition of the bimodule given above, it is clear that
$\delta^1_k=0$ for all $k>3$.

\begin{lemma}
  \label{lem:DuAlgDualCross}
  Let 
  \[ \DuAlg_1=\DuAlg(n,k,\Matching), \qquad
  \DuAlg_2=\DuAlg(n,k,\tau_i(\Matching)),\qquad 
  \Alg_1=\DuAlg(n,2n+1-k,\Matching).\]
  $\Pos^i$ is dual to $\Pos_i$, in the sense that
  \[ \lsup{\DuAlg_2}\Pos^i_{\DuAlg_1}\DT~
    \lsup{\DuAlg_1,\Alg_1}\CanonDD \simeq 
    ~\lsup{\DuAlg_2,\Alg_1}\Pos_i.\]
\end{lemma}

\begin{proof}
  This is a straightforward computation in the spirit of
  Proposition~\ref{prop:DualCross}. Indeed, the computation is
  slightly easier in that there is no need to introduce a homotopy
  equivalence $h^1$ in fact, the computation verifies that
  $\lsup{\DuAlg_2}\Pos^i_{\DuAlg_1}\DT~ \lsup{\DuAlg_1,\Alg_1}\CanonDD
  \cong ~\lsup{\DuAlg_2,\Alg_1}\Pos_i$. (This is analogous to 
  the case of 
  Lemma~\cite[Lemma~\ref{BK:lem:CrossingDADD}]{BorderedKnots}
  where the input algebra has both $C_i$ and $C_{i+1}$ in it.)
\end{proof}

We also have the following:
\begin{prop}
  \label{prop:DAcrosses}
  Let 
  \[ 
  \begin{array}{ll}
    \Alg_1=\Alg(n,k,\Matching), &
    \Alg_2=\Alg(n,k,\tau_i(\Matching)) \\
    \DuAlg_1=\Alg(n,2n+1-k,\Matching), &
    \DuAlg_2=\Alg(n,2n+1-k,\tau_i(\Matching)).
    \end{array}\]
    The $DA$ bimodules of a crossing from Section~\ref{subsec:DAcross} are related to those here by the homotopy equivalence:
    \[ \lsup{\Alg_2}\Pos^i_{\Alg_1}\DT \lsup{\Alg_1,\DuAlg_1}\CanonDD
    \simeq 
    \lsup{\DuAlg_1}\Pos^i_{\DuAlg_2}\DT \lsup{\DuAlg_2,\Alg_2}\CanonDD\]
\end{prop}
\begin{proof}
  This is an immediate consequence of the invertability of $\CanonDD$ (Theorem~\ref{thm:InvertibleDD}),
  Proposition~\ref{prop:DualCross}, and Lemma~\ref{lem:DuAlgDualCross}.
\end{proof}

\subsection{Bimodules associated to negative crossings}
\label{subsec:DuAlgNegCross}

The bimodules associated to negative crossings can be derived from
the positive case. 

With $n$, $k$, $\Partition$, and $\tau_i$ as before, take the opposite
module of
$\lsup{\DuAlg(n,k,\tau_i(\Partition)}\Pos^i_{\DuAlg(n,k,\Partition)}$, to
get a module
  \[ \lsub{\DuAlg(n,k,\tau_i(\Partition))}({\overline
    \Pos}^i)^{\DuAlg(n,k,\Partition)} \cong
  \lsup{\DuAlg(n,k,\Partition)^{\op}}{\overline
    \Pos}^i_{\DuAlg(n,k,\tau_i(\Partition))^{\op}}; \] and use natural
  identifications $\Opposite \colon\DuAlg(n,k,\Partition')\to
  \DuAlg(n,k,\Partition')^{\op}$ for both $\Partition'=\Partition$ and
  $\Partition'=\tau_i(\Partition)$ to define
  $\lsup{\DuAlg(n,k,\tau_i(\Partition))}\Neg^i_{\DuAlg(n,k,\Partition)}$, the
  {\em bimodule over $\DuAlg$associated to a negative crossing}.

\begin{lemma}
  \label{lem:PNinvDual}
  For
  \[\DuAlg_1=\DuAlg(n,k,\Matching) \qquad\DuAlg_2=\DuAlg(n,k,\tau_i(\Matching)),\]
  we have that
  \[ \lsup{\DuAlg_1}\Id_{\DuAlg_1}\simeq~\lsup{\DuAlg_1}\Pos^i_{\DuAlg_2}\DT \lsup{\DuAlg_2}\Neg^i_{\DuAlg_1}
  \simeq~\lsup{\DuAlg_1}\Neg^i_{\DuAlg_2}\DT \lsup{\DuAlg_2}\Pos^i_{\DuAlg_1}.\]
\end{lemma}

\begin{proof}
  From Lemma~\ref{lem:DuAlgDualCross} and Proposition~\ref{prop:DualCross}, it follows
  that
  \[ \lsup{\DuAlg_2}\Pos^i_{\DuAlg_1}\DT \lsup{\DuAlg_1,\Alg_1}\CanonDD
  \simeq~
   \lsup{\Alg_1}\Pos^i_{\Alg_2}\DT \lsup{\Alg_2,\DuAlg_2}\CanonDD.\]
   Taking opposites we also find that
  \[ \lsup{\DuAlg_1}\Neg^i_{\DuAlg_2}\DT \lsup{\DuAlg_2,\Alg_2}\CanonDD
  \simeq~
   \lsup{\Alg_2}\Neg^i_{\Alg_1}\DT \lsup{\Alg_1,\DuAlg_1}\CanonDD.\]
   Combining with Equation~\eqref{eq:InvertPos}, we conclude that
   \begin{align*}
     \lsup{\DuAlg_1}\Neg^i_{\DuAlg_2}\DT
   \lsup{\DuAlg_2}\Pos^i_{\DuAlg_1}\DT \lsup{\DuAlg_1,\Alg_1}\CanonDD
   &~\simeq 
   \lsup{\Alg_1}\Pos^i_{\Alg_2}\DT~
   \lsup{\Alg_2}\Neg^i_{\Alg_1}\DT~\lsup{\DuAlg_1,\Alg_1}\CanonDD \\
   &~\simeq \lsup{\DuAlg_1,\Alg_1}\CanonDD
   \end{align*}
   The desired relation
   \[\lsup{\DuAlg_1}\Id_{\DuAlg_1}\simeq 
   \lsup{\DuAlg_1}\Neg^i_{\DuAlg_2}\DT~\lsup{\DuAlg_2}\Pos^i_{\DuAlg_1}\]
   now follows from  Theorem~\ref{thm:InvertibleDD}.
   The other equivalence (reversing the order of $\Pos^i$ and $\Neg^i$) follows similarly.
\end{proof}

\section{Bimodules associated to a minimum}
\label{sec:Min}

Consider a partial knot diagram that contains exactly one local minimum in it.
Our aim here is to associate a type $DA$ bimodule to this object. Before doing that, 
we start with the more easily defined type $DD$ bimodule.

\subsection{The type $DD$ bimodule of a minimum}
\label{subsec:DDmin}

Fix $1\leq c \leq 2n+1$.
Let 
\[ \phi_c\colon \{1,\dots,2n\}\to\{1,\dots,2n+2\}\] be the function
defined in Equation~\eqref{eq:DefInsert}.  Fix an integer $k$
with $0\leq k\leq 2n+1$. Fix a matching $\Matching_1$ on
$\{1,\dots,2n+2\}$ that does not match $c$ and $c+1$, so we have
$\alpha$, $\beta \in\{1,\dots,2n\}$ so that $\{c,\phi_c(\alpha)\},
\{c+1,\phi_c(\beta)\}\in\Matching_1$.  There is an induced matching
$\Matching_2$ on $\{1,\dots,2n\}$ consisting of $\{\alpha,\beta\}$,
and all pairs $\{i,j\}$ with $\{i,j\}\cap \{\alpha,\beta\}=\emptyset$
so that $\phi_c(i)$ and $\phi_c(j)$ are matched in $\Matching_1$.  Let
\[
  \Alg_1'=\Alg'(n+1,2n+2-k,\Matching_1)
\qquad \text{and}\qquad 
\Alg_2=\Alg(n,k,\Matching_2).
\]

Let $\DDmin_c=\lsup{\Alg_2,\Alg_1'}\DDmin_c$ be the bimodule defined as follows.

We call an idempotent state $\y$ for $\Alg_1'$ an {\em allowed idempotent state for $\Alg_1'$} if
\begin{equation}
  \label{eq:AllowedIdempotents}
  |\y\cap\{c-1,c,c+1\}|\leq 2\qquad\text{and}\qquad c\in\y.
\end{equation}
There is a map $\psi'$ from
allowed idempotent states $\y$ for $\Alg_1'$ to idempotent states  for $\Alg_2$,
where $\x=\psi'(\y)\subset \{0,\dots,2n\}$ is characterized by
\begin{equation}
  \label{eq:SpecifyPsi}
  |\y\cap \{c-1,c,c+1\}| + 
  |\x\cap \{c-1\}| =2~\qquad{\text{and}}~\qquad \phi_c(\x)\cap \y=\emptyset.
\end{equation}

As a vector space, $\DDmin_c$ is spanned by vectors ${\mathbf P}_\y$ that are in
one-to-one correspondence with allowed idempotent states $\y$ for $\Alg_1'$. 
The bimodule structure, over the rings of idempotents $\IdempRing(\Alg_2)\otimes\IdempRing(\Alg_1')$,
is specified by
\[ \left(\Idemp{\psi'(\y)}\otimes\Idemp{\y}\right)\cdot {\mathbf P}_{\y}={\mathbf P}_{\y}.\]

To specify the differential, 
consider the element 
$A\in \Alg_2\otimes \Alg_1'$ 
\begin{align}
A&=(1\otimes L_{c} L_{c+1}) + 
(1\otimes R_{c+1} R_{c}) 
 + \sum_{j=1}^{2n} R_{j} \otimes L_{\phi(j)} + L_{j} \otimes R_{\phi(j)} + U_{j}\otimes E_{\phi(j)}
 \label{eq:defDDmin}\\
&
  + 1\otimes E_{c} U_{c+1}
  + U_{\alpha}\otimes [E_{\phi(\alpha)},E_{c}]E_{c+1} 
  + C_{\{\alpha,\beta\}}\otimes [E_{\phi(\alpha)},E_{c}][E_{c+1},E_{\phi(\beta)}], \nonumber \\
& + \sum_{\{i,j\} \in \Matching_2\setminus \{\alpha,\beta\}} C_{\{i,j\}}\otimes [E_{\phi(i)},E_{\phi(j)}]
\nonumber 
  \end{align}
where we have dropped the subscript $c$ from $\phi_c=\phi$.
Let 
\[ \delta^1({\mathbf P}_{\y}) = (\Idemp{\psi'(\y)}\otimes \Idemp{\y})\cdot A \otimes
\sum_{\z} {\mathbf P}_{\z},\]
where the latter sum is taken over all allowed idempotent states $\z$ for $\Alg_1'$.

\begin{lemma}
  The space $\lsup{\Alg_2,\Alg_1'}\DDmin_c$ defined above,  equipped with the map
  \[ \delta^1\colon \DDmin_c \to (\Alg_2\otimes \Alg_1')\otimes \DDmin_c,\]
  specified above, is a type $DD$ bimodule over $\Alg_2$ and $\Alg_1'$.
\end{lemma}

\begin{proof}
  The proof is a straightforward adaptation of
  Lemma~\ref{lem:CanonicalIsDD}.
\end{proof}

There is a symmetric version of this bimodule, exchanging the roles of
$c$ and $c+1$ in $\delta^1$.  The map $\phi^1(\x)=\x+ (1\otimes E_{c} E_{c+1})\x$ is
easily seen to give an isomorphism between these two bimodules.

\subsection{The type $DA$ bimodule of a minimum when $c=1$}
\label{subsec:DAminCOne}

We will construct a type $DA$ bimodule over the algebra
that is dual to the bimodule of a minimum described
above. Specifically, we continue with the notation at the beginning of
Section~\ref{subsec:DDmin}, further letting
$\Alg_1=\Alg(n+1,k+1,\Matching_1)$.  We start by describing the case
where $c=1$, returning to the general case in
Section~\ref{subsec:GenMin}.

A {\em preferred idempotent state} for $\Alg_1=\Alg(n+1,k+1,\Matching_1)$ is
an idempotent state $\x$ with 
\[ \x\cap\{0,1,2\}\in \{\{0\}, \{2\}, \{0,2\}\}.\]

We define a map $\psi$ from preferred idempotent states of $\Alg_1$ to idempotent states of $\Alg_2$, as follows.
Given preferred idempotent state $\x$ for $\Alg_1$, order the components
$\x=\{x_1,\dots,x_{k+1}\}$ so that
$x_1<\dots<x_{k+1}$. Define
\[
\psi(\x) = \left\{
\begin{array}{ll}
\{0,x_3-2,\dots,x_{k+1}-2\} &{\text{if $|\x\cap\{0,1,2\}|=2$}} \\
\{x_2-2,\dots,x_{k+1}-2\} &{\text{if $|\x\cap\{0,1,2\}|=1$}} \\
\end{array}\right.
\]

Generators of the $DA$ bimodule
$\GenMin^1=\lsup{\Alg_2}\GenMin^1_{\Alg_1}$ correspond to preferred
idempotent states; and its bimodule structure over the idempotent
algebras is specified by the property that if $\x$ is such a preferred
idempotent state, then the corresponding generator $\MinGen_{\x}$
satisfies
\[ \Idemp{\psi(\x)}\cdot \MinGen_\x \cdot \Idemp{\x}=\MinGen_\x.\]

Operations are described in terms of the following graph
$\Gamma_0$ with four vertices, $\XX L_1$, $\YY R_2$,
$\XX$, and $\YY$; and the following families of edges (each indexed by integers $m\geq 0$):
\begin{equation}
  \label{eq:MinActions}
  \begin{tikzpicture}
    \node at (-3,0) (XL1) {$\XX L_1$} ;
    \node at (3,0) (YR2) {$\YY R_2$} ;
    \node at (0,2) (Y) {$\YY$} ;
    \node at (0,-2) (X) {$\XX$} ;
    \draw[->] (XL1) [bend left=15] to node[above,sloped] {\tiny{$U_1^m R_1
\otimes C_2^{\otimes m}$}} (Y) ;
    \draw[->] (Y) [bend left=30] to node[above,sloped] {\tiny{$U_2^{m+1}\otimes C_1^{\otimes m}$}} (X) ;
    \draw[->] (X) [bend left=15] to node[below,sloped] {\tiny{$U_1^{m} L_1\otimes C_2^{\otimes m}$}} (XL1) ;
    \draw[->] (XL1) [loop left] to node[above,sloped] {\tiny{$U_1^{m}\otimes C_2^{\otimes m}$}} (XL1);
    \draw[->] (YR2) [bend left=15] to node[below,sloped] {\tiny{$U_2^m L_2\otimes C_1^{\otimes m}$}} (X) ;
    \draw[->] (X) [bend left=30] to node[above,sloped] {\tiny{$U_1^{m+1}\otimes C_2^{\otimes m}$}} (Y) ;
    \draw[->] (Y) [bend left=15] to node[above,sloped] {\tiny{$U_2^{m} R_2\otimes C_1^{\otimes m}$}} (YR2) ;
    \draw[->] (YR2) [loop right] to node[above,sloped] {\tiny{$U_2^{m}\otimes C_1^{\otimes m}$}} (YR2);
    \end{tikzpicture}
\end{equation}
The labels on the edges consist of some element in
$\Blg_1=\Blg(2n+2,k+1)$ tensored with some tensor powers of
$C_1=C_{\{\alpha+2,1\}}$ and $C_2=C_{\{2,\beta+2\}}$. We call the element of
$\Blg_1$ the {\em $\Blg$-label}.

\begin{defn}
  Recall that a standard sequence $(a_1,\dots,a_m)$ for $\Alg_1$
  has a subsequnce $a_{k_1},\dots,a_{k_\ell}$
  in $\Blg_1$, whose complement consists elements of the form $C_{p}$.
  A standard sequence 
  is a {\em preferred sequence}
  if it satisfies the following properties:
  \begin{enumerate}[label=({$\GenMin$}-\arabic*),ref=({$\GenMin$}-\arabic*)]
  \item there are idempotent states $\x_1,\dots,\x_{m+1}$ so that
    $\Idemp{\x_{i}}\cdot a_i\cdot \Idemp{\x_{i+1}}=a_i$ for $i=1,\dots,m$.
  \item The idempotent states $\x_1$ and $\x_{m+1}$ are  preferred idempotent states.
  \item There is an oriented path $e_1,\dots,e_{\ell}$ in $\Gamma$
    beginning at $\XX L_1$ if $0\in\x_1$ and at $\YY R_2$ if
    $0\not\in\x_1$,  with the following properties:
    \begin{itemize}
      \item The $\Blg$-label $b_i$ on 
        the edge $e_i$ has $w_1(a_{k_i})=w_1(b_i)$ and
        $w_2(a_{k_i})=w_2(b_i)$.
      \item  The sequence $a_{k_i+1},\dots,a_{k_{i+1}-1}$
        (or $a_{k_\ell+1},\dots,a_{m}$ if $i=\ell$)
        consists of the sequence of $C_1$ or $C_2$ elements labeling the
        edge.
      \item The terminal vertex is at $\XX L_1$ or $\YY R_2$.
      \item None of the internal vertices are at $\XX L_1$ or $\YY R_2$.
        \end{itemize}
  \end{enumerate}
\end{defn}

\begin{defn}
  \label{def:AssociatedElement}
For a preferred sequence, there is at most one pure non-zero algebra element $b\in\Blg(2n,k)\subset \Alg_2$,
characterized by the following properties:
\begin{enumerate}[label=(PS-\arabic*),ref=(PS-\arabic*)]
  \item $b=\Idemp{\psi(\x_1)}\cdot b$
  \item 
      \label{ps:GradeMinimum}
    For all $i$ with $1\leq i\leq 2n$,
    $ w_{i}(b)=\sum_{j=1}^{m} w_{i+2}(a_j).$
\end{enumerate}
\end{defn}

\begin{thm}
  \label{thm:MinDA}
  There is a type $DA$ bimodule $\GenMin^1$ uniquely characterized up to homotopy by the following properties:
  If $\x$ is a preferred idempotent state, then 
  \begin{align*}
    \delta^1_{2}(\MinGen_{\x},C_{\{i+2,j+2\}})&=C_{\{i,j\}}\otimes\MinGen_{\x}  
    \qquad \text{for all $\{i,j\}\in\Matching_2\setminus\{\{\alpha,\beta\}\}$}\\
    \delta^1_{3}(\MinGen_{\x},C_{\{\alpha+2,1\}},C_{\{2,\beta+2\}})&=
    C_{\{\alpha,\beta\}}\otimes \MinGen_{\x} 
    \qquad \text{if $0\in\x$} \\
    \delta^1_3(\MinGen_\x,C_{\{2,\beta+2\}},C_{\{\alpha+2,1\}})&=
    C_{\{\alpha,\beta\}}\otimes \MinGen_{\x} 
    \qquad \text{if $0\not\in\x$}
  \end{align*}
  Furthermore, for any preferred sequence $a_1,\dots,a_m$, with
  $a_1=\Idemp{\x}\cdot a_1$ and $a_m= a_m\cdot \Idemp{\y}$, the value of 
  $\delta^1_{m+1}(\MinGen_\x,a_1,\dots,a_m)$ is $b\otimes \MinGen_{\y}$, where
  $b$ is the associated algebra element, defined as above;
  and for all other standard sequences consisting of pure algebra elements, the output is $0$.
\end{thm}

We will prove the above theorem in Subsection~\ref{subsec:AltConstr}.

To illustrate the theorem, observe that there is a closed path from
$\XX L_1$ to itself with three edges, labeled $R_1$, $U_2$, and $L_1$.
According to the theorem, there is a corresponding action on the bimodule
\[ \delta^1_{4}(\x,R_1,U_2,L_1)=\x.\]
Another closed path from $\XX L_1$ to itself with three edges gives:
\[ \delta^1_{7}(\x,U_1^2 R_1, C_2, C_2, U_2^2,C_1,L_1)= U_{\alpha} U_{\beta}^2\otimes \x.\]

The relationship between the two bimodules associated to a minimum is
given in the following:

\begin{lemma}
  \label{lem:MinDual}
  $\GenMin^1$ is dual to $\DDmin_1$, in the following sense.  Given
  $\Alg_1$, $\Alg_1'$, and $\Alg_2$, as above, the following identity
  holds:
  \[ \lsup{\Alg_2} \GenMin^1_{\Alg_1} \DT~\lsup{\Alg_1,\DuAlg_1}\CanonDD 
  \simeq ~\lsup{\Alg_2,\DuAlg_1}\DDmin_1. \]
\end{lemma}

\begin{proof}
  We compute $\lsup{\Alg_2} \GenMin_{\Alg_1} \DT~\lsup{\Alg_1,\DuAlg_1}\CanonDD$
  from  $\CanonDD$ and the actions
  from Equation~\eqref{eq:MinActions}. We obtain a bimodule whose generators
  correspond to allowed idempotent states $\y$ as in Equation~\eqref{eq:AllowedIdempotents}.
  The differential does not have exactly the form as it did for $\DDmin_1$;
  rather, it is described as follows:
  \[
    \begin{tikzpicture}
      \node at (-3,0) (XX) {$\XX$} ;
      \node at (3,0) (YY) {$\YY$} ;
      \draw[->] (XX) [bend left=5] to node[above,sloped] {\tiny{$1\otimes L_1 L_2 + 1\otimes L_1 L_2 E_1 E_2$}} (YY) ; 
      \draw[->] (YY) [bend left=5] to node[below,sloped] {\tiny{$1\otimes R_2 R_1 + 1\otimes R_2 R_1 E_1 E_2$}} (XX) ; 
      \draw[->] (XX) [loop above] to node[above] {\tiny{$1\otimes U_1 E_2+ U_{\beta}\otimes
          [E_2,E_{\beta+2}]E_1+C_{\alpha,\beta} \otimes [E_1,E_{\alpha+2}][E_2,E_{\beta+2}]$}} (XX);
      \draw[->] (YY) [loop below] to node[below] {\tiny{$1\otimes U_2 E_1+ U_{\alpha}\otimes
          [E_1,E_{\alpha+2}]E_2+C_{\alpha,\beta} \otimes [E_1,E_{\alpha+2}][E_2,E_{\beta+2}]$}} (YY);
    \end{tikzpicture}
    \]
        in addition to the usual terms \[\Big(\sum_{i=1}^{2n}
    L_i\otimes R_{i+2}+R_i\otimes L_{i+2}  
    + U_i\otimes E_{i+2}\Big)+ \sum_{\{i,j\}\in\Matching_2-\{\alpha,\beta\}}C_{\{i,j\}}\otimes [E_{i+2},E_{j+2}]\]
    that connect $\XX$ to itself and $\YY$ to itself.
    (Note that $E_1$ and $E_2$ commute with each other.)
    The map sending $\XX$ to $\XX+(1\otimes E_1 E_2)\XX$ 
    and $\YY$ to $\YY$
    identifies this bimodule with $\DDmin_1$.
\end{proof}

To verify Theorem~\ref{thm:MinDA}, we give an alternative construction of the type $DA$ bimodule
in the next section.

\subsection{An alternative construction}
\label{subsec:AltConstr}

Continuing notation on $\Alg_1$, $\Alg_2$, $\Matching_1$, $\Matching_2$
from the previous section, we give an alternative construction of the
bimodule promised in Theorem~\ref{thm:MinDA}, which makes the proof transparent.

Consider the idempotent in $\Alg_1$
\[{\mathbf I}=\sum_{\{\x\big|x_1=1\}}\Idemp{\x}.\] Let $\Blg$ be the
subalgebra of $\Alg(n+1,k+1,\Matching_1)$ 
with $C_{\{\alpha+2,1\}}$ and
$C_{\{2,\beta+2\}}$ (and their multiples) removed from it.

Consider the right $\Blg$-module
\[ M = \frac{{\mathbf I}\cdot \Blg}{L_1 L_2 \Blg}  \oplus 
\frac{{\mathbf I}\cdot \Blg}{L_1 L_2 \Blg}.\]
$M$ can also be viewed as   a left
module over the subalgebra of ${\mathbf I}\cdot \Blg\cdot {\mathbf I}$ 
consisting of elements $w_1=w_2=0$,
which in turn can be
identified with the subalgebra $\Blg_2$ of $\Alg(n,\Partition_2)$ with
$C_{\{\alpha,\beta\}}$ removed. Denote this identification
\[ \phi\colon (\Blg_2\subset \Alg(n,\Matching_2))\to \Idemp{}\cdot \Blg\cdot \Idemp{}.\]
Let $\XX$ and $\YY$ be the generators of the two summands of $M$. 
Let
\[ m_{1|1|0}(b_2,\XX\cdot b_1)=\XX \cdot \phi(b_2)\cdot b_1 \qquad 
m_{1|1|0}(b_2,\YY\cdot b_1)=\YY \cdot \phi(b_2)\cdot b_1,\]
where
$b_2\in\Blg_2$ and $b_1\in\frac{ {\Idemp{}}\cdot \Blg}{L_1 L_2 \cdot \Blg}$.
Equip $M$ with the differential
\[ m_{0|1|0}(\XX)= \YY\cdot U_2 \qquad m_{0|1|0}(\YY)= \XX\cdot U_1.\] 
Think of the right $\Blg$-action as inducing further operations 
\[ m_{0|1|1}(\XX\cdot b_1,b_1')=\XX\cdot (b_1\cdot b_1') \qquad
m_{0|1|1}(\YY\cdot b_1,b_1')=\YY\cdot (b_1\cdot b_1').\] All the
operations described above give $M$ the structure of a $\Blg_2-\Blg$
bimodule, written $\lsub{\Blg_2}M_{\Blg}$.

\begin{lemma}
  \label{lem:BigDAGens}
  There is a type $DA$ bimodule $\lsup{\Blg_2}{\BigMin}_{\Blg}$,
  generated as a right $\Blg$-module by two generators $\XX$ and $\YY$ so that
  $\lsub{\Blg_2}M_{\Blg}=_{\Blg_2}{(\Blg_2)}_{\Blg_2}\otimes
  \lsup{\Blg_2}{\BigMin}_{\Blg}$.
  As a type $D$ structure, $\lsup{\Blg_2}\Theta_{\Blg}$ is generated
  by elements of the form
  $\XX\cdot a$ or $\YY\cdot a$, where $a$ is chosen from:
  \begin{equation}
    \label{eq:DefGamma}
  \Gamma=
    \begin{array}{llllll}
    \{ R_2 U_2^t,& U_2^n, &
    & L_1 U_1^t,& U_1^t\}_{t\geq 0,n>0}.
    \end{array}
    \end{equation}
\end{lemma}

\begin{proof}
  As a left module over ${\mathbf I}\cdot \Blg\cdot
  {\mathbf I}\cong \Blg_2$, $M$ is generated by the above generating
  set.
  Moreover,  each of the $\Blg_2$-modules $\Blg_2\cdot a$, where is $a$ is chosen from the above generating set,
  is isomorphic to $\Blg_2\cdot \Idemp{\x}$, where $a=\Idemp{\x}\cdot a$.
  Both statements are proved in~\cite[Lemma~\ref{BK:lem:BimoduleSplitting}]{BorderedKnots}.
  The existence of $\lsup{\Blg_2}\Theta_{\Blg}$ with the stated generating set is 
  a formal consequence of these facts. The fact that it has the given two generators $\XX$ and $\YY$
  as a right module is clear from the construction of $M$.
\end{proof}

Let abbreviate $C_{\{\alpha,\beta\}}\in \Alg_2$ by $C$;
and $C_{\{\alpha+2,1\}}$ and $C_{\{2,\beta+2\}}$ by $C_1$ and $C_2$ respectively.
To promote the right action on $\lsup{\Blg_2}{\BigMin}_\Blg$ to an action by
$\Alg_1$, we introduce the following actions:
\[\begin{array}{rl}
  \delta^1_2(\XX,C_1)=U_{\alpha}\otimes \YY &
  \delta^1_2(\YY,C_2)=U_{\beta}\otimes \XX\\
  \delta^1_2(\XX,C_1 \cdot C_2)= C \otimes \YY \cdot U_2 &
  \delta^1_2(\YY,C_1 \cdot C_2)= C \otimes \XX \cdot  U_1\\
  \delta^1_3(\XX,C_1,C_2)= C\otimes \XX &
  \delta^1_3(\YY,C_2,C_1)= C\otimes \YY.
  \end{array}
\]
These actions are illustrated in the following picture:
\begin{equation}
  \label{eq:ModuleVersionMin}
    \mathcenter{\begin{tikzpicture}[scale=1.5]
    \node at (0,0) (X) {$\XX$} ;
    \node at (6,0) (Y) {$\YY$} ;
    \draw[->] (X) [bend left=15] to node[above,sloped] {\tiny{$U_2 + U_{\alpha} \otimes 
C_{1} + C \cdot U_2\otimes C_{1}\cdot C_{2}$}}  (Y)  ;
    \draw[->] (Y) [bend left=15]to node[below,sloped]  {\tiny{$U_1+ U_{\beta} \otimes 
C_{2} + C \cdot U_1\otimes C_{1}\cdot C_{2}$}}  (X);
    \draw[->] (X) [loop] to node[above,sloped]{\tiny{$C\otimes
(C_{1},C_{2})$}} (X);
    \draw[->] (Y) [loop] to node[above,sloped]{\tiny{$C\otimes
(C_{2},C_{1})$}} (Y);
\end{tikzpicture}} 
\end{equation}
(Here the arrow labels with $U_2$ and $U_1$
alone represent $\delta^1_1$, actions where the outgoing algebra element
is $1$.)

We denote the result by $\lsup{\Alg_2}{\BigMin}_{\Alg_1}$.

\begin{lemma}
  \label{lem:BigDAGens2}
  The operations described above make $\lsup{\Alg_2}{\BigMin}_{\Alg_1}$ into a type $DA$ bimodule,
  with the generating set described in Lemma~\ref{lem:BigDAGens}.
\end{lemma}

\begin{proof}
  This is straightforward.
\end{proof}

Consider the map $h^1\colon \BigMin \to \BigMin$
determined on the generating set $\XX\cdot a$ and $\YY\cdot a$ (for $a\in \Gamma$) by 
\begin{align*}
h^1(\XX a)&=\left\{\begin{array}{ll}
      \YY a' &{\text{if there is a  $a'\in \Gamma$ with $a=U_1 a'$}} \\
      0 &{\text{otherwise}}
      \end{array}\right. \\
    h^1(\YY a)&=\left\{\begin{array}{ll}
      \XX a' &{\text{if there is a pure $a'\in \Gamma$ with $a=U_2 a'$}} \\
      0 &{\text{otherwise}}
      \end{array}\right. 
\end{align*}

Consider $Q\subset \BigMin$, the two-dimensional vector space spanned by
$\XX  L_1$ and $\YY  R_2$. Observe that
$\delta^1_1$ on $Q$ is identically $0$; so we have an inclusion of type $D$ structures
$i^1\colon Q \to \BigMin\subset \Alg_2\otimes \BigMin$, and a projection
$\pi^1\colon \BigMin \to Q\subset \Alg_2\otimes Q$.

\begin{lemma}
  \label{lem:SubComplex}
  For the above operators, we have the identities:
\[  \begin{array}{lllll}
    (\pi^1\circ i^1) =\Id_{Q}, & i^1\circ \pi^1=\Id_{\BigMin}+ dh^1, &
    h^1\circ h^1=0, & h^1\circ i^1=0, & \pi^1\circ h^1 = 0.
    \end{array}\]
\end{lemma}

\begin{proof}
  This is a straightforward verification.
\end{proof}

\begin{proof}[Proof of Theorem~\ref{thm:MinDA}.]
  Apply homological perturbation theory and Lemma~\ref{lem:SubComplex}
  to give $\XX L_1\oplus \YY R_2$ the structure of a type $DA$
  bimodule.  It is straightforward to verify that the induced structure is
  the one stated in Theorem~\ref{thm:MinDA}.
  To this end, it is helpful to draw the following diagram, where we have written $C_1$, $C_2$, and $C$
  for $C_{\{1,\alpha+2\}}$, $C_{\{2,\beta+2\}}$ and $C_{\{\alpha,\beta\}}$ respectively:
\[
    \mathcenter{\begin{tikzpicture}[scale=1.8]
    \node at (0,.25) (X10) {$\XX L_1$} ;
    \node at (0,-.25) (Y01) {$\YY R_2$} ;
    \node at (1.5,1.2) (Y00) {$\YY$} ;
    \node at (1.5,-1.2) (X00) {$\XX$} ;
    \node at (-1.5,-1.2) (Y02) {$\YY U_2$} ;
    \node at (-3,-2.4) (Y03) {$\YY R_2 U_2$};
    \node at (-4.5,-3.6) (Y04) {$\YY U_2^2$} ;
    \node at (0,2.4)  (Y10) {$\YY L_1$};
    \node at (-1.5,1.2) (X20) {$\XX U_1$};
    \node at (-1.5,3.6)  (Y20) {$\YY U_1$};
    \node at (-3,2.4) (X30) {$\XX U_1 L_1$};
    \node at (-4.5,3.6) (X40) {$\XX U_1^2$};
    \node at (0,-2.4) (X01) {$\XX R_2$};
    \node at (-1.5,-3.6) (X02) {$\XX U_2$};
    \draw[->] (X00) [bend left=15,pos=.45] to node[below,sloped] {\tiny{$U_{\alpha}\otimes C_1+C U_2\otimes C_1 C_2$}} (Y00);
    \draw[->] (Y00) [bend left=15] to node[above,sloped] {\tiny{$U_{\beta}\otimes C_2+C U_1\otimes C_1 C_2
$}} (X00);
    \draw[->] (X01) [bend left=15,pos=.45] to node[below,sloped] {\tiny{$U_{\alpha}\otimes C_1$}} (Y01);
    \draw[->] (X02) [bend left=15,pos=.45] to node[below,sloped] {\tiny{$U_{\alpha}\otimes C_1$}} (Y02);
    \draw[->] (Y10) [bend left=15,pos=.3] to node[above,sloped] {\tiny{$U_{\beta}\otimes C_2$}} (X10);
    \draw[->] (Y20) [bend left=15,pos=.3] to node[above,sloped] {\tiny{$U_{\beta}\otimes C_2$}} (X20);
    \draw[->] (X20) [bend left=15,pos=.3] to (Y00);
    \draw[->] (X30) [bend left=15,pos=.3] to (Y10);
    \draw[->] (X40) [bend left=15,pos=.3] to (Y20);
    \draw[->] (Y02) [bend left=15,pos=.3] to (X00);
    \draw[->] (Y03) [bend left=15,pos=.3] to (X01);
    \draw[->] (Y04) [bend left=15,pos=.3] to (X02);
    \draw[->] (X10) [loop left] to node[above,sloped] {\tiny{$C\otimes (C_1,C_2)$}} (X10);
    \draw[->] (Y01) [loop right] to node[above,sloped] {\tiny{$C\otimes (C_2,C_1)$}} (Y01);
  \end{tikzpicture}} \] The horizontal arrows here indicate the map
$h^1$, and other arrows indicate the actions on
$\lsup{\Alg_2}\Theta_{\Alg_1}$.  To keep the picture clean, we have
suppressed some of the further arrows that follow from the shown ones
by right translation in $\Alg_1$. For example, the label on the arrow
$\YY\to \XX$ by $U_\beta\otimes C_2$ also implies an algebra action
$\YY$ to $\XX L_1$ which would be labeled $U_{\beta} \otimes
C_2 L_1$, whose arrow we have suppressed. Similarly, 
we have suppressed the from $\YY$ to $\YY R_2$ labeled $1\otimes R_2$.  

By homological perturbation theory, the $\Ainfty$ operations induced
on the $\XX L_1\oplus \YY R_2$ are specified by closed paths starting
and ending at $\XX L_1$ and $\YY R_2$, composed of arrows that
alternate between algebra operations and the homotopy operator.
\end{proof}

In the statement of Theorem~\ref{thm:MinDA}, we described only the
$\Ainfty$ operations formed by standard sequences; but there are
others, such as $\delta^1_3(\XX L_1,R_1,C_2L_1)=U_{\beta} \otimes \XX
L_1)$.  The above homological perturbation theory in fact gives all
$\Ainfty$ operations, though the statement is a little complicated,
and unnecessary for our present purposes; we leave it to the
interested reader to work out.

\subsection{Another trident relation}
\label{subsec:MinTrident}

We will give a trident relation analogous to the one from Section~\ref{subsec:MaxTrident};
compare Proposition~\ref{prop:BasicTrident}.
Whereas this relation involves the type $DD$ bimodule of the minimum from Section~\ref{subsec:DDmin}
and the crossing bimodules defined over the dual matched algebras from Section~\ref{sec:DuAlg}, it plays an
important role in the construction of the general $DA$ bimodule of a minimum,
Section~\ref{subsec:GenMin}.

\begin{lemma}
  \label{lem:DuBasicTrident}
  Fix integers $0\leq k\leq 2n+1$, and a matching
  $\Matching_1$ on $\{1,\dots,2n\}$. Let
  \[\begin{array}{ll}
    \Matching_2=\phi_{c+1}(\Matching_1)\cup\{c+1,c+2\} & \Matching_3=\tau_c(\Matching_2) \\
    \Matching_4=\tau_{c+1}(\Matching_3) \\
    \Alg_1=\Alg(n,k,\Matching_1), &\DuAlg_2=\DuAlg(n+1,2n+2-k,\Matching_2) \\
    \DuAlg_3=\DuAlg(n+1,2n+2-k,\Matching_3) & \DuAlg_4=\DuAlg(n+1,2n+2-k,\Matching_4)
    \end{array}\]
  There is homotopy equivalence of graded bimodules:
  \begin{equation}
    \label{eq:BasicTridentMin}
    \lsup{\DuAlg_3}\Neg^{c}_{\DuAlg_2}\DT~^{\DuAlg_2,\Alg_1}\Min_{c+1}
    \simeq
    \lsup{\DuAlg_3}\Pos^{c+1}_{\DuAlg_4}\DT~^{\DuAlg_4,\Alg_1}\Min_{c}
  \end{equation}
\end{lemma}

\begin{proof}
  For notational simplicity, suppose that $c=1$. 
  To ease computation,  we use the homotopy
  equivalent model for $\lsup{\DuAlg_4,\Alg_1}\Min_1$ containing the
  term $U_1 E_2\otimes 1$ rather than $E_1 U_2\otimes 1$ suggested by Equation~\eqref{eq:defDDmin}.
  (Note also that we have reversed the two tensor factors from Equation~\eqref{eq:defDDmin}.)
  
  A straightforward computation shows that
  $\lsup{\DuAlg_3}\Pos^{2}_{\DuAlg_4}\DT~^{\DuAlg_4,\Alg_1}\Min_1$ is
  the bimodule whose arrows are as given in
  Equation~\ref{eq:SymmetricTrident}, along with self-arrows of the
  $U_1 E_3\otimes 1$, $E_2\otimes U_1$, $L_{j+2}\otimes R_{j}$ for
  $j=1,\dots,2n$ $R_{j+2}\otimes L_{j}$ for $j=1\dots,2n$,
  $E_{j+2}\otimes U_j$ for $j=1,\dots,2n$,
  $\llbracket E_{m+2},E_{\ell+2}\rrbracket\otimes C_{\{m,\ell\}}$ for all
  $\{m,\ell\}\in\Matching_1$; and additional self-arrows of the form
  $\llbracket E_{\beta+2},E_{3}\rrbracket\cdot E_1 \otimes U_{\beta}$
  and $\llbracket E_1,E_{\alpha+2}\rrbracket\cdot\llbracket
  E_3,E_{\beta+2}\rrbracket\otimes C_{\alpha,\beta}$.

  A symmetric computation reduces
  $\lsup{\DuAlg_3}\Neg^{c}_{\DuAlg_2}\DT~^{\DuAlg_2,\Alg_1}\Min_{c+1}$ to
  the same bimodule.
\end{proof}

\subsection{The general case}
\label{subsec:GenMin}

We have so far defined $\GenMin^c$ for $c=1$. We can define $\GenMin^c$ in general by the following inductive procedure.
We begin by specifying the algebras. 
Fix integers $n$ and $k$
with $0\leq k\leq 2n+1$. As before, fix a matching $\Matching_1$ on
$\{1,\dots,2n+2\}$ that does not match $c$ and $c+1$, so we have
$\alpha$, $\beta \in\{1,\dots,2n\}$ so that $\{c,\phi_c(\alpha)\},
\{c+1,\phi_c(\beta)\}\in\Matching_1$.  There is an induced matching
$\Matching_2$ on $\{1,\dots,2n\}$ consisting of $\{\alpha,\beta\}$,
and all pairs $\{i,j\}$ with $\{i,j\}\cap \{\alpha,\beta\}=\emptyset$
so that $\phi_c(i)$ and $\phi_c(j)$ are matched in $\Matching_1$.  Let
\[
\begin{array}{lll}
\Alg_2=\Alg(n,k,\Matching_2) &
\Alg_3=\Alg(n+1,k+1,\tau_{c-1}(\Matching_1)&
\Alg_4=\Alg(n+1,k+1,\tau_c\circ\tau_{c-1}(\Matching_1),
\end{array}
\]
as indicated in Figure~\ref{fig:GenMinAlg}.
\begin{figure}[ht]
\input{GenMinAlg.pstex_t}
\caption{\label{fig:GenMinAlg} Algebras used in the construction
  of $\GenMin^{c}$ from $\GenMin^{c-1}$.}
\end{figure}

With the algebras in place, the bimodule is defined inductively by the formula:
\[
\lsup{\Alg_2}\GenMin^{c}_{\Alg_1}=\lsup{\Alg_2}\GenMin^{c-1}_{\Alg_4}\DT
\lsup{\Alg_4}\Pos^{c}_{\Alg_3}\DT~^{\Alg_3}\Pos^{c-1}_{\Alg_1};
\]

The $DA$ bimodule of a minimum defined as above corresponds to the $DD$ bimodule
of a minimum defined in Section~\ref{subsec:DDmin}, according to the following result.
Note that both the statement and proof follow  analogously to~\cite[Proposition~\ref{BK:prop:MinDual}]{BorderedKnots}).

\begin{thm}
  \label{thm:MinDual}
  Fix integers $n$ and $k$ with $0\leq k\leq 2n+1$, and $c\in
  1,\dots,2n-1$, and let $\Alg_1=\Alg(n+1,k+1,\Matching_1)$,
  $\Alg_2=\Alg(n,k,\Matching_2)$, where $\Matching_1$ does not match
  $c$ and $c+1$, and $\Matching_2$ is the induced matching, as in the
  beginning of Section~\ref{subsec:DDmin}.  The above defined type
  $DA$ bimodule $\GenMin^c$ is standard, and it is is dual to
  $\Min_c$, in the sense that
  \begin{equation}
    \label{eq:MinDuality}
    \lsup{\Alg_2} \GenMin^c_{\Alg_1} \DT~\lsup{\Alg_1,\DuAlg_1}\CanonDD 
    \simeq ~\lsup{\Alg_2,\DuAlg_1}\Min_c. 
  \end{equation}
\end{thm}

\begin{proof}
  The verification of Equation~\eqref{eq:MinDuality} proceeds by
  induction on $c$, and the basic case $c=1$ is
  Lemma~\ref{lem:MinDual}.  For the inductive step, with algebras chosen
  as in Figure~\ref{fig:GenMinAlg} (choosing $\DuAlg_i$ to be the algebra
  connected to $\Alg_i$ via the canonical $DD$ bimodule), we compute:
  \begin{align*}
    ^{\Alg_2}\GenMin^c_{\Alg_1}\DT~^{\Alg_1,\DuAlg_1}\CanonDD &\simeq
    ~^{\Alg_2}\GenMin^{c-1}_{\Alg_4} \DT 
    ~^{\Alg_4}\Pos^c_{\Alg_3}\DT
    \Big({}^{\Alg_3}\Pos^{c-1}_{\Alg_1}\DT
    ~^{\Alg_1,\DuAlg_1}\CanonDD\Big) \\
    &\simeq 
    ~^{\Alg_2}\GenMin^{c-1}_{\Alg_4} \DT 
    ~^{\Alg_4}\Pos^c_{\Alg_3}\DT
    \Big({}^{\DuAlg_1}\Pos^{c-1}_{\DuAlg_3}\DT
    ~^{\DuAlg_3,\Alg_3}\CanonDD\Big) \\
    &\simeq 
    ~^{\Alg_2}\GenMin^{c-1}_{\Alg_4} \DT 
    ~^{\DuAlg_1}\Pos^{c-1}_{\DuAlg_3}\DT
    ~\Big({}^{\Alg_4}\Pos^c_{\Alg_3}\DT
    ~^{\Alg_3,\DuAlg_3}\CanonDD\Big) \\
    &\simeq 
    ~^{\Alg_2}\GenMin^{c-1}_{\Alg_4} \DT 
    ~^{\DuAlg_1}\Pos^{c-1}_{\DuAlg_3}\DT
    ~\Big({}^{\DuAlg_3}\Pos^c_{\DuAlg_4}\DT
    ~^{\DuAlg_4,\Alg_4}\CanonDD\Big) \\
    &\simeq 
    ~^{\DuAlg_1}\Pos^{c-1}_{\DuAlg_3}\DT
   ~ ^{\DuAlg_3}\Pos^c_{\DuAlg_4}\DT~\Big({}^{\Alg_2}\GenMin^{c-1}_{\Alg_4} \DT
    ~^{\Alg_4,\DuAlg_4}\CanonDD\Big) \\
    &\simeq 
    ~^{\DuAlg_1}\Pos^{c-1}_{\DuAlg_3}\DT
   ~ ^{\DuAlg_3}\Pos^c_{\DuAlg_4}\DT~^{\DuAlg_4,\Alg_2}\Min_{c-1} \\
    &\simeq 
    ~^{\DuAlg_1}\Pos^{c-1}_{\DuAlg_3}\DT
   ~ ^{\DuAlg_3}\Neg^{c-1}_{\DuAlg_1}\DT~^{\DuAlg_1,\Alg_2}\Min_{c} \\
   &\simeq ~^{\DuAlg_1,\Alg_2}\Min_{c},
  \end{align*}
  using associativity of $\DT$ (bearing in mind that the bimodules
  associated to a crossing are always bounded), Proposition~\ref{prop:DAcrosses},
  the trident relation (Lemma~\ref{lem:DuBasicTrident}), the inductive hypothesis,
  and the fact that $\Pos$ and $\Neg$ are inverses (Lemma~\ref{lem:PNinvDual}); see Figure~\ref{fig:TridentEquations}
\begin{figure}[ht]
\input{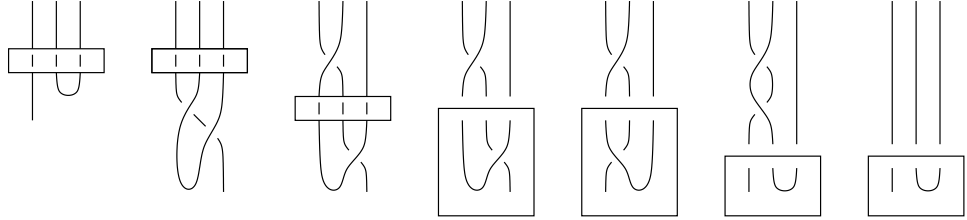}
\caption{\label{fig:TridentEquations} 
  {\bf Verifying the inductive step in Equation~\eqref{eq:MinDuality}.}
  Boxed components correspond to type $DD$ bimodules, and unboxed ones correspond to type $DA$ bimodules;
  algebras above the box are of the form $\DuAlg$ while those below are of the form $\Alg$.
  Since the order of tensor products is not indicated in these pictures, we have skipped those
  steps that correspond to associating in different order from the picture.}
\end{figure}

It is clear from its description in Theorem~\ref{thm:MinDA} that
$\Min^1$ is a standard type $DA$ bimodule.  Since the bimodules for
crossings are also standard, it follows from the inductive definition
of $\Min^c$ and the fact that this property is preserved by tensor
products (Lemma~\ref{lem:StandardTimesStandard}) that $\Min^c$ is also
standard.
\end{proof}

\newcommand\OurRingBig{\mathcal{S}}
\section{The knot invariant}
\label{sec:Construction}
We give now the construction of the knot invariant, and verify its
invariance properties. The basic logic follows as
in~\cite[Section~\ref{BK:sec:ConstructionAndInvariance}]{BorderedKnots}.

\subsection{A symmetry}
Before proceeding to the invariance proof, we establish 
symmetries in the bimodules constructed earlier, which will be useful in the invariance proof.

\begin{lemma}
\label{lem:Rotate}
Fix integers $c$, $k$, and $n$,
with $1\leq c \leq 2n+1$ and $0\leq k\leq 2n+1$; and fix
a matching $\Matching_1$ on $\{1,\dots,2n\}$, and let
$\Matching_2=\phi_c(\Matching_1)\cup\{c,c+1\}$ (as in Definition~\ref{eq:DefInsert}).
Let 
\[
\begin{array}{ll}
\Alg_1=\Alg(n,k,\Matching_1), & \Alg_2=\Alg(n,k,\rho'_n(\Matching_1)) \\
\Alg_3=\Alg(n+1,k+1,\Matching_2), &
\Alg_4=\Alg(n+1,k+1,\rho'_{n+1}(\Matching_2)),
\end{array}
\]
where $\rho'$ (and $\VRot$ below) is as in Section~\ref{sec:AlgSymm}.
The following identities hold:
\begin{align}
  \lsup{\Alg_4}[\VRot]_{\Alg_3}\DT \lsup{\Alg_3}\Max^c_{\Alg_1} 
  &\simeq \lsup{\Alg_4}\Max^{2n+2-c}_{\Alg_2}\DT  \lsup{\Alg_2}[\VRot]_{\Alg_1} 
  \label{eq:ReflectMax}\\
  \lsup{\Alg_2}[\VRot]_{\Alg_1}\DT \lsup{\Alg_1}\Min^c_{\Alg_3} &\simeq 
  \lsup{\Alg_2}\Min^{2n+2-c}_{\Alg_4}\DT \lsup{\Alg_4}[\VRot]_{\Alg_3} 
  \label{eq:ReflectMin}
\end{align}
Also, for any $i=1,\dots,2n-1$, 
let 
$\Alg_5=\Alg(n,k,\rho'_n(\tau_i(\Matching_1)))$
and $\Alg_6=\Alg(n,k,\tau_i(\Matching_1))$.
We have identities
\begin{align}
  \lsup{\Alg_5}[\VRot]_{\Alg_6}\DT \lsup{\Alg_6}\Pos^i_{\Alg_1} &\simeq 
  \lsup{\Alg_5}\Pos^{2n-i}_{\Alg_2}\DT \lsup{\Alg_2}[\VRot]_{\Alg_1} \label{eq:ReflectPos}\\
  \lsup{\Alg_5}[\VRot]_{\Alg_6}\DT \lsup{\Alg_6}\Neg^i_{\Alg_1} &\simeq 
  \lsup{\Alg_5}\Neg^{2n-i}_{\Alg_2}\DT \lsup{\Alg_2}[\VRot]_{\Alg_1} \label{eq:ReflectNeg} 
\end{align}
\end{lemma}

\begin{proof}
  It is straightforward to see that
  \[\lsup{\Alg}[\VRot]_{\Alg}\DT\lsup{\Alg,\Alg'}\CanonDD\cong
  \lsup{\DuAlg}[\VRot]_{\DuAlg}\DT \lsup{\Alg,\Alg'}\CanonDD.\]
  
  We now verify the identities from the lemma by converting $DD$ bimodules, and appealing
  to the fact that $\CanonDD$ is invertible (Theorem~\ref{thm:InvertibleDD}).

  For example, if we let
  \[ 
  \DuAlg_1=\DuAlg(n,2n+1-k,\Matching_1)\qquad
  \DuAlg_5=\DuAlg(n,2n+1-k,\rho'_n(\tau_i(\Matching_1)))),\]
  then it is evident from the description of the type $DD$ bimodule of a positive crossing
  that it has the symmetry
  \[ \left(\lsup{\Alg_5}[\VRot]_{\Alg_6}\otimes
  \lsup{\DuAlg_2}[\VRot]_{\DuAlg_1}\right)\DT \lsup{\Alg_6,\DuAlg_1}\Pos_c =
  \lsup{\Alg_5,\DuAlg_2}\Pos_{2n+2-c}. \] Equation~\eqref{eq:ReflectPos}
  now follows from Proposition~\ref{prop:DualCross} and
  Theorem~\ref{thm:InvertibleDD}. Equation~\eqref{eq:ReflectNeg} follows similarly.

  Similarly, there is an obvious symmetry of
  the type $DD$ bimodule of a maximum
  \[ \left(\lsup{\Alg_4}[\VRot]_{\Alg_3} \otimes
    \lsup{\DuAlg_2}[\VRot]_{\DuAlg_1}\right)\DT
  \lsup{\Alg_3,\DuAlg_1}\Max_c
  =
  \lsup{\Alg_4,\DuAlg_2}\Max_{2n+2-c}.\]
  Equation~\eqref{eq:ReflectMax} follows from this symmetry, Proposition~\ref{prop:MaxDual}, and
  Theorem~\ref{thm:InvertibleDD}.
  
  The symmetry of type $DD$ bimodules that gives Equation~\eqref{eq:ReflectMin},
  \[ \left(\lsup{\DuAlg_4}[\VRot]_{\DuAlg_3}
    \otimes
    \lsup{\Alg_2}[\VRot]_{\Alg_1}\right)
  \DT\lsup{\DuAlg_3,\Alg_1}\Min_c 
  \simeq
  \lsup{\DuAlg_4,\Alg_2}\Min_{2n+2-c}\]
  is supplied by the map $x\to (1+ (E_{2n+2-c} E_{2n+3-c}\otimes 1))\otimes x$.
\end{proof}

\subsection{Invariants associated to partial knot diagrams}
\label{subsec:ConstructInvariant}

Let ${\mathcal D}$ be a planar knot diagram, thought of as lying in the $(x,y)$ plane. 
Let ${\mathcal D}_t$ denote the 
$y=t$ slice of the diagram. Similarly, given $t_1<t_2$, let ${\mathcal D}_{[t_1,t_2]}$ denote
the portion of ${\mathcal D}$ with $t_1\leq y\leq t_2$.

Consider a knot diagram in the plane with a distinguished point.
Recall that the diagram is said to be in {\em bridge position} if for
the projection to the $y$ axis, all critical points are
non-degenerate, and all minima, maxima, and crossings project to
distinct points on the $y$ axis, and the global minimum is the marked
point.

Let ${\mathcal D}$ be a planar knot diagram in bridge position.
Fix some generic $t\in\R$. Then ${\mathcal D}_t$ consists of $2n$ points,
which we think of as the points $\{1,\dots,2n\}$.
There is a naturally associated matching $\Matching$ on these points by the rule that $i$ and $j$ are matched
if they are connected by an arc in ${\mathcal D}\cap (y\geq t)$.
Thus, there is a natural algebra associated to the $t$-slice,
$\Alg(n,n,\Matching)$, which we denote $\Alg({\mathcal D}_t)$.

Choose a knot planar knot diagram ${\mathcal D}$ in bridge position, 
and 
slice it up  into pieces $t_1<\dots<t_k$ so that the following conditions hold:
\begin{itemize}
\item for $i=1,\dots,k-1$, the interval $[t_i,t_{i+1}]$ contains the
  projection onto the $y$ axis of exactly one crossing or critical
  point
\item for $i=1,\dots,k$, $t_i$ is not the projection of any crossing
  or critical point 
\item there are no crossings or critical points whose $y$ value is
  greater than $t_k$ (and so $[t_{k-1},t_k]$ contains the global
  maximum)
\item there are no crossings below $t_1$, 
  and the only critical point whose $y$ value is smaller than $t_1$
  is the global minium.
\end{itemize}

Thus, each piece ${\mathcal D}\cap (y\in[t_i,t_{i+1}])$ is either a maximum, a minimum, or a crossing;
and hence in Sections~\ref{sec:Max},~\ref{sec:Min}, and~\ref{sec:Cross}
we explained how to associate to it a type $DA$ bimodule, 
with incoming algebra associated to the
$t_{i+1}$-slice of the diagram and outgoing algebra associated to the $t_i$-slice of the diagram,
for $i\geq 1$. We denote all of these bimodules by 
$\PartInv({\mathcal D}_{[t_i,t_{i+1}]})$.
More generally, for any $1\leq j\leq m$, define the invariant associate to the partial
knot diagram ${\mathcal D}_{[j,m]}$ to 
be the tensor product of the bimodules associated to the various basic pieces, 
$\PartInv({\mathcal D}_{[i,i+1]})$
with $j\leq i< i+1\leq m$. Up to homotopy, this type $DA$ bimodule is independent of the order
in which the tensor product is taken.

We can think of ${\mathcal D}_{[1,k]}$ as a partial knot diagram
with the global minimum removed. Thus, $\PartInv({\mathcal D}_{[1,k]})$ 
is a type $D$ structure over 
$\Alg(1,1,\{1,2\})$, which gets a grading from  the orientation on the knot.

Let $\OurRingBig=\Field[u,v]/uv=0$. 
The closed knot invariant is constructed via a type $DA$ bimodule
$\lsup{\OurRingBig}\TerMin_{\Alg(1,1,\{1,2\})}$,  as follows.  
The bimodule $\TerMin$  has three generators $\XX$,
$\YY$, and $\ZZ$, with
\[ \XX \cdot \Idemp{\{0\}}=\XX, \qquad 
\YY \cdot \Idemp{\{1\}}=\YY, \qquad 
\ZZ \cdot \Idemp{\{2\}}=\ZZ.\]
When $1$ is oriented upwards, $\TerMin$ is the $DA$ bimodule with $\delta^1_k=0$ for $k\neq 2$, and
all $\delta^1_2$ are determined by
\[ 
\begin{array}{lll}
\delta^1_2(\YY,L_1)= u\otimes \XX, & \delta^1_2(\XX,R_1)= u\otimes \YY,  \\
\delta^1_2(\YY,R_2)= v\otimes \ZZ, & \delta^1_2(\ZZ,L_2)= v\otimes \YY, \\
\delta^1_2(\XX,C_{\{1,2\}})=
\delta^1_2(\YY,C_{\{1,2\}})=
\delta^1_2(\ZZ,C_{\{1,2\}})=0.
\end{array}
\]
When $1$ is oriented downwards, we define the actions as above, exchanging the roles of $u$ and $v$.

\begin{prop}
  \label{prop:RestrictIdempotent}
  The idempotent of the type $D$ invariant of knot diagram with the minimum removed is restricted by
  \[\PartInv({\mathcal D}_{[1,k]})=\Idemp{\{1\}}\cdot \PartInv({\mathcal D}_{[1,k]}).\]
\end{prop}

\begin{proof}
  Given any generic slice $t$ of a knot diagram, 
  and consider the subalgebra 
  \[ \AlgLoc({\mathcal D}_t)=\left(\sum_{\{\x\big|
      \{0,2n\}\cap\x=\emptyset\}}\Idemp{\x}\right)\cdot \Alg({\mathcal
    D}_{t})\cdot \left(\sum_{\{\x\big|
      \{0,2n\}\cap\x=\emptyset\}}\Idemp{\x}\right).\] We claim that if
  we restrict the input algebra of $\PartInv({\mathcal
    D}_{[t_1,t_2]})$ to $\AlgLoc({\mathcal D}_{t_2})$, then the output
  algebra is contained in $\AlgLoc({\mathcal D}_{t_1})\subset
  \Alg({\mathcal D}_{t_1})$.  This is a straightforward verification
  for the bimodules of a crossing, a maximum, and a minimum defined in
  the earlier sections; and it clearly, is a property that is
  preserved under tensor product.  Specializing to the case where the
  partial knot diagram is missing only its global minimum, we arrive
  at the statement of the proposition.
\end{proof}

For a knot diagram, consider $\TerMin\DT\PartInv({\mathcal
  D}_{[1,k]}))$.  Restricting the input algebra of $\TerMin$ to the subalgebra
$\AlgLoc=\AlgLoc\subset \Alg(1,1,\{1,2\})=\Alg$,
we find that the output algebra is contained in the subalgebra
$\OurRing\subset \OurRingBig$ generated by $U=u^2$ and $V=v^2$. 
Thus, in view of Proposition~\ref{prop:RestrictIdempotent}, we can consider the chain complex over
$\OurRing=\Field[U,V]/UV=0$ defined by 
\begin{equation}
  \label{eq:DefCwz}
  \Cwz({\mathcal D})=\lsub{\OurRing}\OurRing_{\OurRing}\DT \lsup{\OurRing}
\TerMin_{\AlgLoc}\DT\lsup{\AlgLoc}\PartInv({\mathcal
  D}_{[1,k]}.
\end{equation} Let$\Hwz({\mathcal D})$ denote its homology, which is also  a
module over $\OurRing$.  We can now prove that $\Hwz({\mathcal D})$
is an invariant of the underlying oriented knot $\orK$:

\begin{proof}[Proof of Theorem~\ref{thm:Invariance}]
  To check that $\Hwz$ is a knot invariant, we must check that it is invariant under bridge moves and Reidemeister moves.
  To recall, bridge moves can be classified into the following:
  \begin{enumerate}
  \item Commutations of distant crossings
  \item Trident moves
  \item Critical points commute with distant crossings
  \item Commuting distant critical points
  \item Pair creation and annihilation
  \end{enumerate}
  Invariance under these various moves is verified as it was
  in~\cite{BorderedKnots}, by comparing type $DD$ bimodules for partial knot
  diagrams, and appealing to the invertibility of $\CanonDD$
  (Theorem~\ref{thm:InvertibleDD}). 

  In fact, commutations of distant
  crossings was already verified in the braid relations.  (See
  especially Equation~\eqref{eq:FarBraids}.)

  Trident moves correspond to passing a local minimum or maximum
  through a crossing. For the maximum, this follows from the trident
  relation for its type $DD$ bimodule
  (Proposition~\ref{prop:BasicTrident}) and the fact that the type
  $DA$ bimodule of a maximum is dual to its type $DD$ bimodule
  (Proposition~\ref{prop:MaxDual}).  For the minimum, we follow a
  similar logic; we spell out the details presently. First note that
  \begin{align*}
    \lsup{\Alg_3}\Min^c_{\Alg_2}\DT 
    \lsup{\Alg_2}\Pos^{c+1}_{\Alg_1}
    \DT \lsup{\Alg_1,\DuAlg_1}\CanonDD
    &\simeq 
        \lsup{\Alg_3}\Min^c_{\Alg_2}\DT \lsup{\Alg_2,\DuAlg_1}\Pos_{c+1} \\
      &\simeq 
        \lsup{\Alg_3}\Min^c_{\Alg_2}\DT 
        \left(\lsup{\DuAlg_1}\Pos^{c+1}_{\DuAlg_2}\DT\lsup{\DuAlg_2,\Alg_2}\CanonDD\right) \\
      &\simeq 
      \lsup{\DuAlg_1}\Pos^{c+1}_{\DuAlg_2}\DT \left(\lsup{\Alg_3}\Min^c_{\Alg_2}\DT \lsup{\Alg_2,\DuAlg_2}\CanonDD \right)\\
      &\simeq 
      \lsup{\DuAlg_1}\Pos^{c+1}_{\DuAlg_2}\DT \lsup{\DuAlg_2,\Alg_3}\Min_{c},
  \end{align*}
  Similarly, 
  \[\lsup{\Alg_3}\Min^{c+1}_{\Alg_4}\DT \lsup{\Alg_4}\Neg^c_{\Alg_1}\DT\lsup{\Alg_1,\DuAlg_1}\CanonDD
  \simeq \lsup{\DuAlg_1}\Neg^c_{\DuAlg_4}\DT \lsup{\DuAlg_4,\Alg_3}\Min_{c+1}.\]
  Thus, the desired trident relation
  \[ 
  \lsup{\Alg_3}\Min^c_{\Alg_2}\DT \lsup{\Alg_2}\Pos^{c+1}_{\Alg_1}\DT \lsup{\Alg_1,\DuAlg_1}\CanonDD
  \simeq 
  \lsup{\Alg_3}\Min^{c+1}_{\Alg_4}\DT \lsup{\Alg_4}\Neg^c_{\Alg_1}\DT\lsup{\Alg_1,\DuAlg_1}\CanonDD
  \]
  follows from Lemma~\ref{lem:DuBasicTrident}.
  
  Commutations between distant crossings and critical points are also straightforward; the identity
  \begin{align*}
    \Pos^{\phi_c(i)}\DT \Max^c\simeq \Max^c\DT \Pos^{i}  \\
  \end{align*}
  is easy to verify on the type $DD$ level (i.e. after tensoring with $\CanonDD$); as 
  is the identity
  \[ 
      \Pos^{i}\DT \Min^{c} \simeq \Min^c\DT \Pos^{\phi_c(i)}.
      \]

  Commutations between distant critical points is also mostly straightforward: 
  for  $i<j$, we claim that
  \begin{align*}
    \Max^{i}\DT\Max^{j-1}&\simeq \Max^{j+1}\DT \Max^{i} \\
    \Min^{i}\DT\Min^{j+1}&\simeq \Min^{j-1}\DT \Min^{i} \\
    \Max^{j-1}\DT\Min^{i}&\simeq \Min^{i}\DT \Max^{j+1} \\
    \Min^{j+1}\DT\Max^{i}&\simeq \Max^{i}\DT \Min^{j-1}.
  \end{align*}
  The first two are clear by tensoring with $\CanonDD$.  The third
  identity can be established similarly when $i=1$, using the explicit
  form of $\Min^{1}$ from Section~\ref{subsec:DAminCOne}. For general
  $i$, it follows from the definition of $\Min^{i}$, and commutation
  of local maxima with crossings. The fourth follows from the third using the symmetry
  Lemma~\ref{lem:Rotate}.
  
  Arbitrary pair creation and annihilations can be reduced to the case that
  \begin{equation}
    \label{eq:PairAnn}
    \Min^1\DT \Max^2\simeq \Id,
  \end{equation}
  using braidlike Reidemeister $2$ moves (Equation~\eqref{eq:InvertPos}), and 
  possible reflections (Lemma~\ref{lem:Rotate}). The verification of Equation~\eqref{eq:PairAnn}
  follows from
  \[ \Min^1\DT \Max_2 \simeq \CanonDD,\]
  which is  a straightforward computation. (Compare~\ref{BK:lem:PairCreation}.)

  Correspondence between the generators of $\Cwz(\Diag)$ and Kauffman states follow from local considerations:
  the bimodules are all associated to partial Kauffman states, and the tensor over the idempotent ring corresponds
  to extending partial Kauffman states. This is exactly as in~\cite{BorderedKnots}.
\end{proof}

Before concluding this section, we note the following result, which is both convenient for computations,
and will also be conceptually useful. (See the proof of Proposition~\ref{prop:ConnSum}.)

\begin{prop}
  \label{prop:InvariantIsStandard}
  For any partial knot diagram ${\mathcal D}$,
  the associated type $DA$ bimodule (or type $D$ structure, in cases where the top of the diagram is empty)
  $\PartInv({\mathcal D})$ is a standard type $DA$ bimodule.
\end{prop}
\begin{proof}
  The $DA$ bimodules for crossings, maxima, and minima are all
  standard (Proposition~\ref{prop:PosExt}, Theorem~\ref{thm:MaxDA},
  and Theorem~\ref{thm:MinDA}). Since tensor products of standard type
  $DA$ bimodules are standard (Lemma~\ref{lem:StandardTimesStandard}), the proposition follows. 
\end{proof}

\newcommand\Tor{\mathrm{Tor}}
\section{Comparison with an earlier invariant}
\label{sec:Compare}

Our aim here is to relate the knot invariant constructed here with the
ones from~\cite{BorderedKnots}, to verify the following more detailed
version of Proposition~\ref{intro:Compare}:

\begin{prop}
  \label{prop:Compare}
  The complex $\Cwz(\orK)/V$ is bigraded homotopy equivalent to the
  complex $C^-(-\orK)$
  from~\cite[Section~\ref{BK:sec:InvarianceProof}]{BorderedKnots},
  thought of as a module over $\Field[U]$.
\end{prop}

The proof compares the two constructions one bimodule at a time.

In~\cite{BorderedKnots}, we defined an algebra $\Blg(m,k,\Upwards)$, where $\Upwards\subset \{1,\dots,m\}$.
There is a corresponding knot invariant, defined in $\Blg(2n,n,\Upwards)$, where $|\Upwards|=n$.
Let $\Upwards$ be a section of $\Partition$; i.e. 
each element of $p\in \Partition$, $p\cap \Upwards$ consists of one element.
There is a map
\[ \phi\colon \Alg(n,k,\Partition)\to \Blg(2n,k,\Upwards),\] with
\[ \phi(C_{\{i,j\}})=U_i\cdot C_j,\]
if $j\in\Upwards$.
There is an associated type $DA$ bimodule $\lsup{\Blg}[\phi]_{\Alg}$.

There is also an algebra map
\[ \psi\colon \Blg(2n,k,\Upwards)\to \Alg'(n,k,\Matching), \] that
induces the identity map on the subalgebra $\Blg(2n,k)$ of both
$\Blg(2n,k,\Upwards)$ and $\Alg'(n,k,\Matching)$, and that has
$\psi(C_j)=E_j$ for all $j\in\Upwards$. This gives a bimodule
$\lsup{\Alg'}[\psi]_{\Blg}$.

The bimodules $[\phi]$ and $[\psi]$ can be used to express identities
between the various type $DD$ bimodules from~\cite{BorderedKnots} with
their analogues in the present work. We start with the simplest case,
the canonical type $DD$ bimodules:

\begin{lemma}
  \label{lem:CompareMorphisms}
  Let $\Blg=\Blg(2n,k,\Upwards)$, 
  $\Blg'=\Blg(2n,2n+1-k,\{1,\dots,2n\}\setminus \Upwards)$,
  $\Alg=\Alg(n,k,\Matching)$, and $\Alg'=\Alg'(n,2n+1-k,\Matching)$.
  There is an isomorphism of type $DD$ bimodules:
  \[ \lsup{\Alg'}[\psi]_{\Blg'}\DT\lsup{\Blg',\Blg}\CanonDD
  \cong \lsup{\Blg}[\phi]_{\Alg}\DT\lsup{\Alg,\Alg'}\CanonDD.\]
\end{lemma}
\begin{proof}
  The generators of
  $\lsup{\Blg,\Alg'}X=~\lsup{\Alg'}[\psi]_{\Blg'}\DT\lsup{\Blg',\Blg}\CanonDD$
  and $\lsup{\Blg,\Alg'}Y=~\lsup{\Blg}[\phi]_{\Alg}\DT\lsup{\Alg,\Alg'}\CanonDD$
  are identified, so there are maps in both directions.
  These identifications are not chain maps, though.
  The differential in 
  $X$ contains the usual terms $L_i\otimes R_i$ and $R_i\otimes L_i$, and 
  \[\sum_{j\in\Upwards'} U_j\otimes E_j+\sum_{j\in\Upwards} C_j\otimes U_j;\] 
  and the differential in $Y$ contains terms of the form
  \[\sum_{j=1}^{2n} U_j\otimes E_j + \sum_{j\in\Upwards, \{i,j\}\in\Matching} U_i C_{j} \otimes [E_i,E_j].\]
  
  Consider maps
  \[ 
  \Phi\colon\lsup{\Blg,\Alg'} X
  \to \lsup{\Blg,\Alg'}Y
  \qquad
  {\text{and}}\qquad
  \Psi\colon\lsup{\Blg,\Alg'} Y
  \to  \lsup{\Blg,\Alg'} X
  \]
  given by the formulas
  \begin{align*}
    \Phi(\x)&=\x + (\sum_{j\in\Upwards} C_j \otimes E_j)\otimes \x \\
    \Psi(\x)&=\x + (\sum_{j\in\Upwards} C_j \otimes E_j)\otimes \x 
  \end{align*}
  It is easy to see that both $\Phi$ and $\Psi$ are $DD$-bimodule homomorphisms; and that
  $\Phi\circ \Psi$ and $\Psi\circ \Phi$ are the identity maps.
\end{proof}

Let ${\mathcal D}$ be a partial knot diagram with two boundaries
$\partial_1{\mathcal D}$ and $\partial_2{\mathcal D}$.  Let
$\lsup{\Blg_2}\KC({\mathcal D})_{\Blg_1}$ be the associated type $DA$
bimodule from~\cite{BorderedKnots}; and let
$\lsup{\Alg_2}\PartInv({\mathcal D})_{\Alg_1}$ denote the bimodules
constructed in the present work, as in Section~\ref{sec:Construction}.

\begin{prop}
  \label{prop:ComparePieces}
  For any partial knot diagram, we have that
  \begin{equation}
    \label{eq:CompareBimodules}
    \lsup{\Blg_2}[\phi]_{\Alg_2}\DT \lsup{\Alg_2}\PartInv({\mathcal D})_{\Alg_1}
  \simeq~
   \lsup{\Blg_2}\KC({\mathcal D})_{\Blg_1}\DT
   \lsup{\Blg_1}[\phi]_{\Alg_1}\end{equation}
\end{prop}

\begin{proof}
  By associativity of tensor products,
  it suffices to check Equation~\ref{eq:CompareBimodules} for the 
  elementary bimodules associated to a  crossing, a maximum, and a minimum.

  We start with a positive crossing.  Recall that
  in~\cite{BorderedKnots}, we defined the bimodule
  $\lsup{\Blg_1,\Blg_2'}\Pos_i$ of a positive crossing, where
  $\Blg_1=\Blg(2n,k,\Upwards)$ and
  $\Blg_2=\Blg(2n,2n+1-k,\{1,\dots,2n\}\setminus \tau(\Upwards))$.
  The generators over the idempotents are the same as for the
  $DD$-bimodule of a positive $\Pos_i$ defined over the matched
  algebras from Section~\ref{subsec:DDcross}.  In fact, the terms of
  Types~\ref{type:OutsideLRP} and~\ref{type:InsideP} are the same as
  those for the $DD$ in Section~\ref{subsec:DDcross}; there are no
  terms of Type~\ref{type:UCCP}, and instead of terms of
  Type~\ref{type:UCP}, we have terms for each $j=1,\dots,2n$ of the
  form $C_j\otimes U_{\tau(j)}$ if $j\in\Upwards$ and $U_j\otimes
  C_{\tau(j)}$ if $j\not\in\Upwards$. 

  To verify 
 \begin{equation}
   \label{eq:WhatWeWantPos}
   \lsup{\Blg_2}[\phi]_{\Alg_2}\DT \lsup{\Alg_2}\Pos^i_{\Alg_1}
  \simeq~ \lsup{\Blg_2}\Pos^i_{\Blg_1}\DT
  \lsup{\Blg_1}[\phi]_{\Alg_1},
  \end{equation}
  we tensor both sides with the canonical type $DD$ bimodule
  $\lsup{\Alg_1,\Alg_1'}\CanonDD$, where $\Alg'_1=\Alg'(n,
  2n+1-k,\Matching)$, and appeal to Theorem~\ref{thm:InvertibleDD}.
  By Proposition~\ref{prop:DualCross},
  the left side gives
\[     \lsup{\Blg_2}[\phi]_{\Alg_2}\DT ~\lsup{\Alg_2,\Alg_1'}\Pos_i
\]
Using Lemma~\ref{lem:CompareMorphisms} and
 the fact that $\Pos^i$ is also dual to $\Pos_i$ for the previous algebras (\cite[Lemma~\ref{BK:lem:CrossingDADD}]{BorderedKnots}), we can identify 
the right hand side with
\begin{align*}
    \lsup{\Blg_2}\Pos^i_{\Blg_1}\DT \lsup{\Blg_1}[\phi]_{\Alg_1}\DT \lsup{\Alg_1,\Alg_1'}\CanonDD  
    & \simeq \lsup{\Blg_2}\Pos^i_{\Blg_1}\DT (\lsup{\Alg'_1}[\psi]_{\Blg'_1}\DT \lsup{\Blg'_1,\Blg_1}\CanonDD) \\
    & \simeq\lsup{\Alg'_1}[\psi]_{\Blg'_1}\DT ( \lsup{\Blg_2}\Pos^i_{\Blg_1}\DT \lsup{\Blg'_1,\Blg_1}\CanonDD) \\
    &\simeq \lsup{\Alg'_1}[\psi]_{\Blg'_1}\DT ~\lsup{\Blg_1',\Blg_2}\Pos_i
    \end{align*}

    Thus, Equation~\eqref{eq:WhatWeWantPos} is reduced to the identification
\[  \lsup{\Blg_2}[\phi]_{\Alg_2}\DT ~\lsup{\Alg_2,\Alg_1'}\Pos_i
\simeq \lsup{\Alg'_1}[\psi]_{\Blg'_1}\DT ~\lsup{\Blg_1',\Blg_2}\Pos_i.\]
To verify that identification, note that the generators of $\lsup{\Blg_2}[\phi]_{\Alg_2}\DT ~\lsup{\Alg_2,\Alg_1'}\Pos_i$
  correspond to those of $\Pos_i$, with terms in the differential of Types~\ref{type:OutsideLRP},~\ref{type:UCP},
  and~\ref{type:InsideP} as before; the terms of Type~\ref{type:UCCP} 
  are replaced by terms of the form
  $U_{\tau(\alpha)} C_{\tau(\beta)}\otimes [E_{\alpha},E_{\beta}]$ for all $\{\alpha,\beta\}\in M$, with 
  $\beta\in\Upwards$.
  Similarly, the generators for $\lsup{\Alg'_1}[\psi]_{\Blg'_1}\DT ~\lsup{\Blg_1',\Blg_2}\Pos_i$
  are the same, with differentials of Types~\ref{type:OutsideLRP} and~\ref{type:InsideP} as before;
  the only terms
  $U_{\tau(j)}\otimes E_j$ 
  of Type~\ref{type:UCP} appearing now are those for which $j\not\in\Upwards$.
  Also, there are no terms of Type~\ref{type:UCCP}.
  As in  the proof of Lemma~\ref{lem:CompareMorphisms},
  there is an isomorphism
  $X\to Y$ (and back) obtained by adding to the identity map
  the terms $\sum_{j\in \Upwards} C_{\tau(j)}\otimes E_j$.
  
  The result for $\Neg^i$ follows formally, since $\Neg^i$ is the inverse of $\Pos^i$.

  Consider next Equation~\eqref{eq:CompareBimodules} in the case where ${\mathcal D}$ consists of 
  a single local maximum, i.e.
  where
  \[\KC({\mathcal D})= ^{\Blg_2}\Max^c_{\Blg_1}
  \qquad{\text{and}}\qquad\PartInv({\mathcal D})=\lsup{\Alg_2}\Max^c_{\Alg_1},\] 
  where 
  \[ 
  \Blg_1=\Blg(2n,k,\Upwards)\qquad{\text{and}}\qquad
  \Blg_2=\Blg(2n+2,k+1,\Upwards_2),\]
  for 
  \[\Upwards_2=\Upwards_1\cup\{c\}\qquad{\text{or}}\qquad
  \Upwards_2=\Upwards_1\cup\{c+1\},\] 
  where $\Upwards_1=\phi(\Upwards)$. Let
  $\Blg_1'=\Blg(2n,2n+1-k,\{1,\dots,2n\}\setminus\Upwards)$.  Recall that in~\cite[Section~\ref{BK:sec:Crit}]{BorderedKnots}, we defined 
  a bimodule $\lsup{\Blg_2,\Blg_1'}\Crit_c$ with the property that
  \[ \lsup{\Blg_2}\Max^c_{\Blg_1}\DT \lsup{\Blg_1,\Blg_1'}\CanonDD 
  \simeq \lsup{\Blg_2,\Blg_1'}\Crit^c.\]
  The generators of $\Crit_c$ are the same as those for $\Max_c$; and the terms in the differential
  of $\Crit_c$ are the terms of Type~\ref{type:MOutsideLR} and~\ref{type:MInsideCup};
  instead of terms of Type~\ref{type:MUC}, we have terms for $j=1,\dots,2n$ of type
  $C_{\phi_c(j)}\otimes U_j$ if $j\in\Upwards_2$ or
  $U_{\phi_c(j)}\otimes C_j$ if $j\not\in\Upwards_2$; 
  and instead of the term of Type~\ref{type:MUC2}, we have
  $C_c U_{c+1}\otimes 1$ if $c\in\Upwards_2$
  and $U_c U_{c+1}\otimes 1$ if $c\in\Upwards_2$.
  Equation~\eqref{eq:CompareBimodules} for $\Max^c$ reduces to verifying
  \[ \lsup{\Blg_2}[\phi]_{\Alg_2}\DT\lsup{\Alg_2,\Alg_1'}\Max_c
  \simeq \lsup{\Alg_1'}[\psi]_{\Blg_1'}\DT \lsup{\Blg_1',\Blg_2}\Crit_c.\]
  This is shown by the usual identification of generators, added to
  $C_{\phi_c}(j)\otimes E_j$ for all $j\in\Upwards_1$.

  Similarly, to verify Equation~\eqref{eq:CompareBimodules} for $C({\mathcal D})=\Min^c$,
  we reduce to the identity
  \[ \lsup{\Blg_2}[\phi]_{\Alg_2}\DT\lsup{\Alg_2,\Alg_1'}\Min_c
  \simeq \lsup{\Alg_1'}[\psi]_{\Blg_1'}\DT \lsup{\Blg_1',\Blg_2}\Crit_c\]
  (where now $\Blg_2=\Blg(2n,k,\Upwards_1)$
  and $\Blg_1=\Blg(2n+2,k+1,\Upwards_1\cup \{c\})$
  or $\Blg(2n+2,k+1,\Upwards_1\cup \{c+1\})$),
  which is verified in the same way.
\end{proof}

\begin{proof}[Proof of Proposition~\ref{prop:Compare}]
  Let $K$ be a knot and let ${\mathcal D}$ be the partial knot diagram
  with the global minimum removed.  When the strand $2$ is oriented
  upwards, $\KC({\mathcal D})$ is a type $D$ structure over the
  algebra
  \[ \Blg=\Idemp{\{1\}}\cdot \Blg(2,1,\{2\})\cdot \Idemp{\{1\}} \cong
  \frac{\Field[U_1,U_2,C_2]}{(C_2^2, U_1 U_2)} \] with $d C_2=U_2$. We
  can think of $\Field[U]$ as a bimodule
  $\lsub{\Field[U]}\Field[U]_{\Blg}$, where the action by $U_2$ and
  $C_2$ are $0$, and $U_1$ acts on the right as as multiplication by $U$.  By
  construction,
  \[ \KC(\orK)=\lsub{\Field[U]}\Field[U]_{\Blg}\DT \lsup{\Blg}\KC({\mathcal D}).\]

  Similarly, by Proposition~\ref{prop:RestrictIdempotent},
  $\PartInv({\mathcal D})$ is a type $D$ structure over the algebra
  \[ \AlgLoc=\Idemp{\{1\}}\cdot \Alg(2,1)\cdot \Idemp{\{1\}}\cong
  \frac{\Field[U,V,C]}{(U V, C^2)},\] equipped with trivial
  differential.  The complex $\lsub{\OurRing}\Cwz(-\orK)$ is obtained
  from $\lsup{\AlgLoc}\PartInv({\mathcal D})$ 
  by tensoring with $\lsub{\OurRing}\OurRing_{\OurRing}\DT
  \lsup{\OurRing}
  \TerMin_{\AlgLoc}=\lsup{\OurRing}\OurRing_{\AlgLoc}$, where
  $\OurRing=\frac{\Field[U,V]}{UV}$ and the right action by $C$ is
  defined to vanish, and $U_1$ acts as multiplication by $U$; i.e.
  \[ \lsub{\OurRing}\Cwz(K)=\lsub{\OurRing}\OurRing_{\AlgLoc}\DT \lsup{\AlgLoc}\PartInv({\mathcal D}).\]

  But
  \begin{align*}
    V\backslash\lsub{\OurRing}\Cwz(-K)&=
        \lsub{\Field[U]}{\Field[U]}_{\AlgLoc}
        \DT \lsup{\AlgLoc}\PartInv({\mathcal D}) \\
        &= 
        \lsub{\Field[U]}{\Field[U]}_{\Blg}\DT \lsup{\Blg}[\phi]_{\AlgLoc}
        \DT \lsup{\AlgLoc}\PartInv({\mathcal D}) \\
    &\simeq 
        \lsub{\Field[U]}{\Field[U]}_{\Blg}\DT \KCm({\mathcal D})
        =\lsub{\Field[U]}\KCm(\orK)
        \end{align*}
        where the last homotopy equivalence uses Proposition~\ref{prop:ComparePieces}. 
\end{proof}

\subsection{Connected sums}

The previous result can be used to verify the following K{\"u}nneth property
for $\KHm$ under connected sums. Specifically, we have the following:

\begin{prop}
  \label{prop:ConnSum}
  Let $\orK_1$ and $\orK_2$ be two oriented knots. The invariant of their connected sum is given by
  \[\KCm(\orK_1\#\orK_2)\simeq \KCm(\orK_1)\otimes_{\Field[U]} \KCm(\orK_2);\]
  and hence 
  \[\KHm(\orK_1\#\orK_2)\cong (\KHm(\orK_1)\otimes_{\Field[U]} \KHm(\orK_2))
  \oplus \Tor_{\Field[U]}(\KHm(\orK_1),\KHm(\orK_2));\]
  where the $\Tor$ appearing above is equipped with its natural shift in
  bigrading.
\end{prop}

\begin{proof}
  Consider a connected sum diagram for $K_1$ and $K_2$, where the
  connected sum region is taken to be the global minimum and the next
  minimum above it, as pictured in Figure~\ref{fig:ConnSum}.

\begin{figure}[ht]
\input{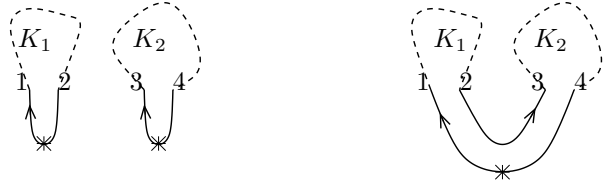}
\caption{\label{fig:ConnSum} {\bf{Connected sums.}}
  Form the disjoint union of the two diagrams on the left;
  and then form the connected sum as shown.}
\end{figure}

Let ${\mathcal D}_1$ and ${\mathcal D}_2$ be disjoint partial knot diagrams
for $K_1$ and $K_2$ with both global minima removed.
We can form the type $D$ structure
$\PartInv({\mathcal D}_1\cup {\mathcal D}_2)$
over $\Alg(2,2,\{\{1,2\},\{3,4\}\})$. By Proposition~\ref{prop:RestrictIdempotent},
the output algebra of this $D$ module is contained in 
$\Idemp{\{1,3\}}\cdot \Alg(2,2,\{\{1,2\},\{3,4\}\})
\cdot \Idemp{\{1,3\}}$. 
There is a natural identification
\begin{align*}
  \left(\Idemp{\{1\}}\cdot \Alg(1,1,\{\{1,2\}\})\cdot \Idemp{\{1\}}\right)
  &\otimes_{\Field}\left(\Idemp{\{1\}}\cdot \Alg(1,1,\{\{1,2\}\})\cdot \Idemp{\{1\}}\right) \\
& \cong 
\Idemp{\{1,3\}}\cdot \Alg(2,2,\{\{1,2\},\{3,4\}\})
\cdot \Idemp{\{1,3\}}.
\end{align*}
Under this identification,
\[ \PartInv({\mathcal D}_1\cup {\mathcal D}_2)=\PartInv({\mathcal
  D}_1)\otimes_{\Field}\PartInv({\mathcal D}_2).\] This can be seen by
drawing ${\mathcal D}_2$ so that all its crossings and critical points
occur above ${\mathcal D}_1$ (at which point there are simply two
strands extending from ${\mathcal D}_2$).

Consider the type $DA$ bimodule $\GenMin^2$ associated to the picture
on the left in Figure~\ref{fig:MinInMid}, with input algebra restricted to 
\[ \Alg=\Idemp{\{1,3\}}\cdot \Alg(2,2,\{\{1,2\},\{3,4\}\})\cdot \Idemp{\{1,3\}}
\cong
\frac{\Field[U_1,U_2,U_3,U_4,C_{\{1,2\}},C_{\{3,4\}}]}
{U_1 U_2=U_3 U_4=C_{\{1,2\}}^2=C_{\{3,4\}}^2=0}.\]
This bimodule is defined to
be the tensor product $\GenMin^1\DT \Pos^2\DT
\Pos^1=\PartInv({\mathcal D}_3)$, as shown on the right on the same
picture.

\begin{figure}[ht]
\input{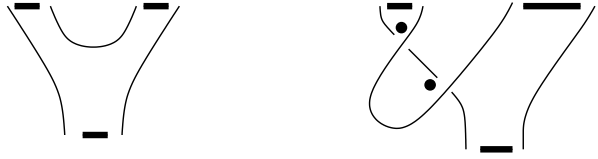}
\caption{\label{fig:MinInMid} {\bf{Bimodule for a minimum.}}
  The bimodule associated to the diagram on the left is a tensor
  product of the bimodules associated to the diagram on the right.}
\end{figure}

To understand the knot invariant for $K_1\# K_2$, we must understand
how $\PartInv({\mathcal D}_3))$ acts on the standard type $D$
structure $\PartInv({\mathcal D}_1\cup{\mathcal D}_2)$.  In fact,
since we consider only $\KCm(K_1\# K_2)\simeq \Cwz(\orK_1\#\orK_2)/V$
(in view of Proposition~\ref{prop:Compare}), it suffices to project
$\PartInv({\mathcal D}_3)$ to the part of the output algebra with
$w_1=0$, and restrict this to above algebra $\Alg$. Split
$\PartInv({\mathcal D}_3)=Q\DT P$, where $P$ is the bimodule
corresponding to the two crossing regions. The actions in $P$ (with
the algebras restricted as above, and $w_3=0$ in the output) are all
specified in the following diagram:

\[
    \begin{tikzpicture}
      \node at (-3,0) (WW) {${\North}\atop{\West}$} ;
      \node at (3,0) (SS) {$\North\atop\South$} ;
      \node at (6,0) (EE) {$\North\atop\East$};
      \draw[->] (WW) [bend left=5] to node[above,sloped] {\tiny{$L_2$}} (SS) ; 
      \draw[->] (SS) [bend left=5] to node[below,sloped] {\tiny{$R_2 U_4 \otimes C_{34}$}} (WW) ; 
      \draw[->] (WW) [loop above] to node[above] {\tiny{$U_4^k\otimes U_4^k+ C_{24}\otimes C_{34}$}} (WW);
      \draw[->] (SS) [loop above] to node[above] 
{\tiny{$U_1^k U_4^{\ell}\otimes U_2^k U_4^{\ell}+ C_{24}\otimes C_{34}$}} (SS);
      \draw[->] (EE) [loop below] to node[below] 
{\tiny{$U_1^k U_2^{\ell}\otimes U_2^k U_3^{\ell}+ C_{24}\otimes C_{34}$}} (SS);
    \end{tikzpicture}
\]

It follows readily from the description of the bimodule associated to
a minimum that the tensor product of this with $Q$ has one generator
(coming from $\North\atop\West$)
and actions 
\[ U_2^k\otimes U_4^k + U_2^{k+\ell} \otimes (U_2^k U_4^{\ell},C_{34}^{\otimes k}) .\] 
Since $\PartInv(\Diag_1\cup\Diag_2)$ is standard, it follows that an output
of $U_4^k$ from $\PartInv(\Diag_1\cup\Diag_2)$ is converted by this bimodule 
to an output $U_2^k$, and 
an output of $U_2^k$ from 
$\PartInv(\Diag_1\cup\Diag_2)$ (followed by a string of $C_{34}$ outputs, which
exist since the bimodule is standard) is converted to an output of $U_2^k$; see for example the
following diagram when $k=1$:
 \begin{equation}
   \begin{tikzpicture}
     \node at (-3,7) (m1in) {$\Min^1$} ;
     \node at (0,7) (m2in) {$\Pos^2$} ;
     \node at (2,7) (m3in) {$\Pos^1$} ;
     \node at (5,7) (m4in) {$\PartInv({\mathcal D}_1)\otimes \PartInv({\mathcal D}_2)$} ;
     \node at (5,6) (m4d1) {$\delta^1$} ;
     \node at (5,5) (m4d2) {$\delta^1$} ;
     \node at (2,5) (m3d1) {$\delta^1_{2}$} ;
     \node at (2,4) (m3d2) {$\delta^1_{2}$} ;
     \node at (0,5) (m2d1) {$\delta^1_{1}$} ;
     \node at (0,4) (m2d2) {$\delta^1_{2}$} ;
     \node at (0,3) (m2d3) {$\delta^1_{2}$} ;
     \node at (-3,3) (m1d) {$\delta^1_{4}$} ;
     \node at (-3,2) (m1out) {} ;
     \node at (0,2) (m2out) {} ;
     \node at (2,2) (m3out) {} ;
     \node at (5,2) (m4out) {} ;
     \node at (-4,2) (algout) {} ;
     \draw[modarrow] (m1in) to (m1d) ;
     \draw[modarrow] (m1d) to (m1out) ;
     \draw[modarrow] (m2in) to node[left] {\tiny{$\West$}} (m2d1) ;
     \draw[modarrow] (m2d1) to node[left] {\tiny{$\South$}}(m2d2) ;
     \draw[modarrow] (m2d2) to node[left] {\tiny{$\South$}} (m2d3) ;
     \draw[modarrow] (m2d3) to node[left] {\tiny{$\West$}} (m2out) ;
     \draw[modarrow] (m3in) to  node[left] {\tiny{$\North$}}(m3d1) ;
     \draw[modarrow] (m3d1) to  node[left] {\tiny{$\North$}}(m3d2) ;
     \draw[modarrow] (m3d2) to node[left]  {\tiny{$\North$}} (m3out) ;
     \draw[modarrow] (m4in) to (m4d1) ;
     \draw[modarrow] (m4d1) to (m4d2) ;
     \draw[modarrow] (m4d2) to (m4out) ;
     \draw[algarrow] (m4d1) to node[above, sloped] {\tiny{$U_2$}} (m3d1);
     \draw[algarrow] (m4d2) to node[above, sloped] {\tiny{$C_{\{3,4\}}$}} (m3d2);
     \draw[algarrow] (m2d1) to node[above,sloped] {\tiny{$L_2$}} (m1d) ;
     \draw[algarrow] (m3d1) to node[above,sloped] {\tiny{$U_1$}} (m2d2);
     \draw[algarrow] (m3d2) to node[above,sloped] {\tiny{$C_{\{3,4\}}$}} (m2d3);
     \draw[algarrow] (m2d2) to node[above,sloped] {\tiny{$U_1$}} (m1d);
     \draw[algarrow] (m2d3) to node[above,sloped] {\tiny{$U_4 R_2$}} (m1d);
     \draw[algarrow] (m1d) to node[above,sloped] {\tiny{$U_2$}} (algout) ;
   \end{tikzpicture}
 \end{equation}

It now follows from the above computation that
\begin{align*}
 (\OurRing/V) \otimes_{\OurRing}  \Cwz({\orK}_1\#{\orK}_2) &=
\lsub{\Field[U]}\Field[U]_{\Alg(1,1,\{1,2\})} \DT \GenMin^2 \DT
(\PartInv({\mathcal D}_1)\otimes\PartInv({\mathcal D}_2)) \\
&\cong
(\Cwz({\mathcal D}_1)/V) \otimes_{\Field[U]} (\Cwz({\mathcal D}_1)/V).
\end{align*} The conclusion about $\KCm$ now follows from
Proposition~\ref{prop:Compare}. The statement on the homological level
follows form the universal coefficient theorem.
\end{proof}

\section{Symmetries}
\label{sec:Symmetries}

We describe some natural symmetries of our knot invariant.

Let $S$ and $\OurRing$ be as in Section~\ref{sec:Intro}.

\begin{prop}
  Let $\orK$ be an oriented knot. Then,
  there is an isomorphism of bigraded $\OurRing$-modules
  $S(\Hwz(\orK))\cong \Hwz(-\orK)$.
\end{prop}

\begin{proof}
  On $\Cwz(\orK)$, the local Alexander gradings
  (cf. Figure~\ref{fig:LocalCrossing}) all change by $-1$ when we
  reverse orientation.  Moreover, the roles of the $U$ and $V$
  variables are exchanged. The complexes otherwise remain the same. Taking homology gives the claimed symmetry.
\end{proof}

Taking mirrors also gives a symmetry. We begin with the following observation:

\begin{prop}
  \label{prop:OppositeBimodules}
  Under the identifications of algebras 
  \[ \Alg(n,k_1,\Matching)\cong
  \Alg(n,k_1,\Matching)^{\op} 
  \qquad{\text{and}}\qquad 
  \DuAlg(n,k_2,\Matching)\cong \DuAlg(n,k_2,\Matching)^{\op}\]
  (cf. Equation~\eqref{eq:OppositeIsomorphism}),
  there is a corresponding identification of type $DD$ bimodules
  \begin{equation}
    \label{eq:OppositeDD}
    \begin{array}{lllll}
    \CanonDD^{\op}\simeq \CanonDD &
    \Max_c^{\op} \simeq \Max_{c} &
    \Min_c^{\op} \simeq \Min_c &
    (\Pos_i)^{\op} \simeq \Neg_i &
    (\Neg_i)^{\op} \simeq \Pos_i
\end{array}
\end{equation}
Similarly, there is an identification of $DA$ bimodules 
  \begin{equation}
    \label{eq:OppositeDA}
    \begin{array}{lllll}
    (\Max^c)^{\op} \simeq (\Max)^{c} &
    (\Min^c)^{\op} \simeq \Min^c &
    (\Pos^i)^{\op} \simeq \Pos^i &
    (\Neg^i)^{\op} \simeq \Pos^i
\end{array}
\end{equation}
\end{prop}

\begin{proof}
  The identification for $\CanonDD$ is clear.

  The identification $\Max_c^{\op}\cong \Max_c$ switches the roles of
  $\XX$ and $\YY$, and fixes $\ZZ$.  For each pair of generators in
  Equation~\eqref{eq:CritDiag}, there are two arrows, and the symmetry
  switches those two arrows; observe that if one arrow is labeled by
  $a\otimes b$, then the other is labeled by $\Opposite(a)\otimes
  \Opposite(b)$.

  There is identification $\Min_c^{\op}\cong \Min_c$ is defined similarly.
  
  The identity $(\Pos_i)^{\op}\simeq \Neg_i$ follows from the
  definition of $\Neg_i$ from Section~\ref{subsec:DDcross}.
  This completes the verification of Equation~\eqref{eq:OppositeDD}.

  Equation~\eqref{eq:OppositeDA} is now a formal consequence of these identities.
\end{proof}

The above proposition has the following immediate consequence:

\begin{prop}
  If $\orK$ is an oriented knot and $m(\orK)$ denotes its mirror,
  obtained by reversing all of the crossings in a diagram for $\orK$,
  then, $\Cwz(m(\orK))\cong\Hom(\Cwz(\orK),\OurRing)$. \qed
\end{prop}

\section{On the module structure of the knot invariant}
\label{sec:CrossingChange}

We turn our attention now to verifying the claims from
Section~\ref{intro:NumKnotInv}.  A key step is a crossing change
morphism established in Section~\ref{subsec:CrossChangeMor}.
Consequences are derived in Section~\ref{subsec:NumericalInvariants}.

\subsection{Crossing change morphisms}
\label{subsec:CrossChangeMor}

\begin{prop}
  \label{prop:CrossingChangeHwz}
  Let $K_+$ and $K_-$ be two knots represented by knot projections ${\mathcal D}_+$ and ${\mathcal D}_-$,
  that differ in a single crossing, which is positive  in ${\mathcal D}_+$ and negative in ${\mathcal D}_-$.
  Then, there are maps
  \[ c_+\colon \Hwz(K_-)\to \Hwz(K_+)\qquad \text{and}\qquad 
   c_-\colon \Hwz(K_+)\to \Hwz(K_-),\]
   where $c_-$ preserves bidegree and $c_+$ is of degree $(-1,-1)$, so that
   $c_+\circ c_-=U$ and $c_-\circ c_+=U$.
\end{prop}

The above proposition hinges on the following local result, stated in
terms of the type $DD$ bimodules $\Pos_i=\lsup{\Alg_2,\Alg_1'}\Pos_i$
and $\Neg_i=\lsup{\Alg_2,\Alg_1'}\Neg_i$ 
in the notation of Section~\ref{subsec:DDcross}.  
For any $j\in\{1,\dots,2n\}$, $U_j\otimes 1$ times the
identity map gives bimodule homomorphisms
\[ U_j\otimes 1 \colon \Pos_i\to \Pos_i\qquad{\text{and}}
\qquad
U_j\otimes 1 \colon \Neg_i\to \Neg_i.\]

\begin{lemma}
  \label{lem:CrossingChange}
  There are bimodule homomorphisms 
  \[ \phi_-\colon \Pos_i\to \Neg_i~\qquad{\text{and}}\qquad\phi_+\colon \Neg_i\to \Pos_i\] so that $\phi_+\circ \phi_-$ and $\phi_-\circ \phi_+$, thought of 
  as endomorphisms of $\Pos_i$ and $\Neg_i$ respectively, are homotopic to 
  $U_{i+1}\otimes 1$ times the corresponding identity maps.
\end{lemma}

\begin{proof}
  We write the formulas when $i=1$; the general case is obtained with
  minor notational changes.  In the diagram below, $\Pos_{1}$ is
  represented by the top row, $\Neg_{1}$ by the bottom row; arrows
  within each row represent the differentials (suppressing all
  self-loops within each generator type; i.e. terms of 
  Type~\ref{type:OutsideLRP},~\ref{type:UCP}, and~\ref{type:UCCP},
  in the notation of Section~\ref{subsec:DDcross}); and
  the maps from the top row to the bottom represent components of a
  map $\phi_-\colon \Pos_1\to\Neg_1$.
  \[
    \begin{tikzpicture}[scale=1.4]
      \node at (0,0) (Wp) {$=\West_+$} ;
      \node at (8,0) (WpR) {$\West_+=$} ;
      \node at (2,0) (Np) {$\North_+$} ;
      \node at (4,0) (Ep) {$\East_+$} ;
      \node at (6,0) (Sp) {$\South_+$} ;
      \node at (0,-2.5) (Wm) {$=\West_-$} ;
      \node at (8,-2.5) (WmR) {$\West_-=$} ;
      \node at (2,-2.5) (Nm) {$\North_-$} ;
      \node at (4,-2.5) (Em) {$\East_-$} ;
      \node at (6,-2.5) (Sm) {$\South_-$} ;

      \draw[->] (Sp) [bend left=7] to node[above,sloped] {\tiny{$R_{1}\!\otimes\! U_{2}+L_{2}\!\otimes\! R_{2}R_{1}$}}  (WpR)  ;
      \draw[->] (WpR) [bend left=7] to node[below,sloped] {\tiny{$L_{1}\otimes 1$}}  (Sp)  ;
      \draw[->] (Ep)[bend right=7] to node[below,sloped] {\tiny{$R_{2}\otimes 1$}}  (Sp)  ;
      \draw[->] (Sp)[bend right=7] to node[above,sloped] {\tiny{$L_{2}\!\otimes\! U_{1} + R_{1} \!\otimes\! L_{1} L_{2}$}} (Ep) ;
      \draw[->] (Wp)[bend right=7] to node[below,sloped] {\tiny{$1\otimes L_{1}$}} (Np) ;
      \draw[->] (Np)[bend right=7] to node[above,sloped] {\tiny{$U_{2}\!\otimes\! R_{1} + R_{2} R_{1} \!\otimes\! L_{2}$}} (Wp) ;
      \draw[->] (Ep)[bend left=7] to node[below,sloped]{\tiny{$1\otimes R_{2}$}} (Np) ;
      \draw[->] (Np)[bend left=7] to node[above,sloped]{\tiny{$U_{1}\!\otimes\! L_{2} + L_{1} L_{2}\!\otimes\! R_{1}$}} (Ep) ;

    \draw[->] (Wm) [bend right=7] to node[below,sloped] {\tiny{${U_{2}}\!\otimes\!{L_{1}}+{L_{1} L_{2}}\!\otimes\!{R_{2}}$}}  (Nm)  ;
    \draw[->] (Nm) [bend right=7] to node[above,sloped] {\tiny{${1}\otimes{R_{1}}$}}  (Wm)  ;
    \draw[->] (Nm)[bend left=7] to node[above,sloped] {\tiny{${1}\otimes{L_{2}}$}}  (Em)  ;
    \draw[->] (Em)[bend left=7] to node[below,sloped] {\tiny{${U_{1}}\!\otimes\!{R_{2}} + {R_{2} R_{1}}\!\otimes\!{L_{1}}$}} (Nm) ;
    \draw[->] (Sm)[bend left=7] to node[above,sloped] {\tiny{${R_{1}}\otimes{1}$}} (WmR) ;
    \draw[->] (WmR)[bend left=7] to node[below,sloped] {\tiny{${L_{1}}\!\otimes\!{U_{2}} + {R_{2}}\!\otimes\!{L_{1} L_{2}}$}} (Sm) ;
    \draw[->] (Sm)[bend right=7] to node[above,sloped]{\tiny{${L_{2}}\otimes{1}$}} (Em) ;
    \draw[->] (Em)[bend right=7] to node[below,sloped]{\tiny{${R_{2}}\!\otimes\!{U_{1}} + {L_{1}}\!\otimes\!{R_{2} R_{1} }$}} (Sm) ;

      \draw[->] (Np) to node[above,sloped]{\tiny{$U_{2}\otimes 1$}} (Nm);
      \draw[->] (Sp) to node[above,sloped]{\tiny{$1\otimes U_{2}$}} (Sm);
      \draw[->] (Np) to node[above,sloped,pos=.7]{\tiny{$R_{2}\otimes L_{2}$}} (Sm);
      \draw[->] (Sp) to node[above,sloped,pos=.7]{\tiny{$L_{2}\otimes R_{2}$}} (Nm);
      \draw[->] (Wp) to node[above,sloped]{\tiny{$1$}} (Wm);
    \end{tikzpicture}
    \]
    (Both terms labelled $\West_+$ are identified in the above diagram
    as are both terms labelled $\West_-$.)
    In the next diagram,  $\Neg_{1}$ is represented by the top row, $\Pos_{1}$ by the bottom row,
  and the maps from the top row to the bottom represent components of a map $\phi_+\colon \Neg_1\to\Pos_1$.
  \[
    \begin{tikzpicture}[scale=1.4]
      \node at (0,-2.5) (Wp) {$=\West_+$} ;
      \node at (8,-2.5) (WpR) {$\West_+=$} ;
      \node at (2,-2.5) (Np) {$\North_+$} ;
      \node at (4,-2.5) (Ep) {$\East_+$} ;
      \node at (6,-2.5) (Sp) {$\South_+$} ;
      \node at (0,0) (Wm) {$=\West_-$} ;
      \node at (8,0) (WmR) {$\West_-=$} ;
      \node at (2,0) (Nm) {$\North_-$} ;
      \node at (4,0) (Em) {$\East_-$} ;
      \node at (6,0) (Sm) {$\South_-$} ;

      \draw[->] (Sp) [bend right=7] to node[below,sloped] {\tiny{$R_{1}\!\otimes\! U_{2}+L_{2}\!\otimes\! R_{2}R_{1}$}}  (WpR)  ;
      \draw[->] (WpR) [bend right=7] to node[above,sloped] {\tiny{$L_{1}\otimes 1$}}  (Sp)  ;
      \draw[->] (Ep)[bend left=7] to node[above,sloped] {\tiny{$R_{2}\otimes 1$}}  (Sp)  ;
      \draw[->] (Sp)[bend left=7] to node[below,sloped] {\tiny{$L_{2}\!\otimes\! U_{1} + R_{1} \!\otimes\! L_{1} L_{2}$}} (Ep) ;
      \draw[->] (Wp)[bend left=7] to node[above,sloped] {\tiny{$1\otimes L_{1}$}} (Np) ;
      \draw[->] (Np)[bend left=7] to node[below,sloped] {\tiny{$U_{2}\!\otimes\! R_{1} + R_{2} R_{1} \!\otimes\! L_{2}$}} (Wp) ;
      \draw[->] (Ep)[bend right=7] to node[above,sloped]{\tiny{$1\otimes R_{2}$}} (Np) ;
      \draw[->] (Np)[bend right=7] to node[below,sloped]{\tiny{$U_{1}\!\otimes\! L_{2} + L_{1} L_{2}\!\otimes\! R_{1}$}} (Ep) ;

    \draw[->] (Wm) [bend left=7] to node[above,sloped] {\tiny{${U_{2}}\!\otimes\!{L_{1}}+{L_{1} L_{2}}\!\otimes\!{R_{2}}$}}  (Nm)  ;
    \draw[->] (Nm) [bend left=7] to node[below,sloped] {\tiny{${1}\otimes{R_{1}}$}}  (Wm)  ;
    \draw[->] (Nm)[bend right=7] to node[below,sloped] {\tiny{${1}\otimes{L_{2}}$}}  (Em)  ;
    \draw[->] (Em)[bend right=7] to node[above,sloped] {\tiny{${U_{1}}\!\otimes\!{R_{2}} + {R_{2} R_{1}}\!\otimes\!{L_{1}}$}} (Nm) ;
    \draw[->] (Sm)[bend right=7] to node[below,sloped] {\tiny{${R_{1}}\otimes{1}$}} (WmR) ;
    \draw[->] (WmR)[bend right=7] to node[above,sloped] {\tiny{${L_{1}}\!\otimes\!{U_{2}} + {R_{2}}\!\otimes\!{L_{1} L_{2}}$}} (Sm) ;
    \draw[->] (Sm)[bend left=7] to node[below,sloped]{\tiny{${L_{2}}\otimes{1}$}} (Em) ;
    \draw[->] (Em)[bend left=7] to node[above,sloped]{\tiny{${R_{2}}\!\otimes\!{U_{1}} + {L_{1}}\!\otimes\!{R_{2} R_{1} }$}} (Sm) ;
      \draw[->] (Nm)[bend right=7] to node[above,sloped,pos=.7]{\tiny{$1$}} (Np);
      \draw[->] (Wm) to node[above,sloped]{\tiny{$U_{2}\otimes 1$}} (Wp);
      \draw[->] (Em) to node[above,sloped]{\tiny{$U_{1}\otimes 1$}} (Ep);
      \draw[->] (Wm) to node[above,sloped,pos=.7]{\tiny{$L_{1} L_{2}\otimes 1$}} (Ep);
      \draw[->] (Em) to node[above,sloped,pos=.7]{\tiny{$R_{2} R_{1}\otimes 1$}} (Wp);
    \end{tikzpicture}
    \]
    Finally, define 
\begin{align*}
  h_+(\North_+)&=h_+(\West_+)=h_+(\East_+)=0 \\
  h_+ (\South_+)&=(L_{2}\otimes 1) \otimes \East_+
\end{align*}
and
\begin{align*}
  h_-(\North_-)&=h_-(\West_-)=h_-(\South_-)=0 \\
  h_-(\East_-)&=(R_{2}\otimes 1)\otimes \South_-.
\end{align*}
    It is straightforward to verify that
    \begin{align*}
      d h_+ &= \phi_+\circ \phi_- + (U_2\otimes 1)\Id_{\Pos_1}   \\
      d h_- &= \phi_-\circ \phi_+ + (U_2\otimes 1)\Id_{\Neg_1},
    \end{align*}
    where the first equation is to be taken as endomorphisms of $\Pos_1$ and the second
    as endomorphisms of $\Neg_1$.
\end{proof}

\begin{remark}
  Observe that there is another map $\phi_-'\colon \Pos_i\to \Neg_i$ with
  the property that $\phi_+\circ \phi_-'$ and $\phi_-'\circ \phi_+$ are chain
  homotopic to $U_{i}\otimes 1$ times the corresponding identity
  maps. This map is obtained by modifying the definition of $\phi_-$,
  switching the roles of $\West_+$ and $\East_+$, the strands $1$ and
  $2$, and $L$ and $R$.
\end{remark}

\begin{proof}[Proof of Proposition~\ref{prop:CrossingChangeHwz}]
  Note that Lemma~\ref{lem:CrossingChange} is stated for type $DD$
  bimodules; but tensoring with the inverse of $\CanonDD$ gives the
  corresponding maps
  \[ \phi^+\colon \Neg^i\to \Pos^i \qquad \phi^-\colon \Pos^i\to
  \Neg^i \] so that $\phi^+\circ\phi^-$ and $\phi^-\circ\phi^+$ are
  homotopic to the bimodule endomorphism 
  of $\Pos^i$ (and $\Neg^i$), $t\colon \Pos^i\to \Pos^i$ 
  (and $t\colon \Neg^i\to \Neg^i$) with
  $t^1_1(\x)=U_{i+1}\otimes \x$ and $t^1_j=0$ for all $j>1$,
  which we call simply ``multiplication by $U_{i+1}$'', and denote $T^{U_{i+1}}$.

  For simplicity, we draw the diagram for ${\mathcal D}_+$ so that the
  distinguished crossing feeds into a local minimum, which is to the lower left of the crossing,
  and to the global
  minimum, which is to the lower right; 
  and so that that both strands through the distinguished
  crossing are oriented upwards (i.e. so that the invariant for
  ${\mathcal D}_+$ uses the bimodule $\Pos^i$ where ${\mathcal D}_-$
  uses $\Neg^i$),  as shown in
  Figure~\ref{fig:OrientOutOfMinimum}.
 \begin{figure}[h]
 \centering
 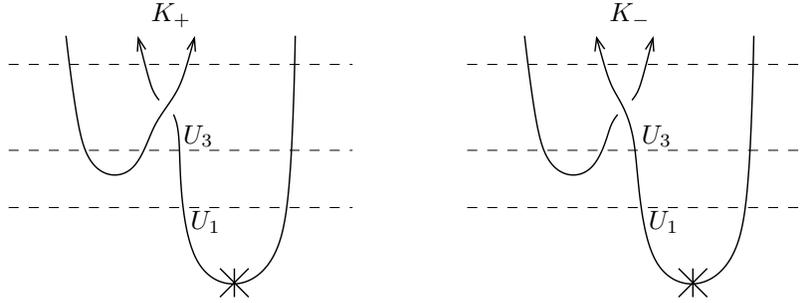
 \caption{{\bf Crossing change.}
 In the proof of Proposition~\ref{prop:CrossingChangeHwz}, we place the global minimum as pictured here.}
 \label{fig:OrientOutOfMinimum}
 \end{figure}

 Thus, 
 \[ \Cwz(\Diag_+)=\TerMin\DT\Min^1\DT\Pos^2\DT \PartInv(\Diag) \qquad \text{and}\qquad
 \Cwz(\Diag_-)=\TerMin\DT\Min^1\DT\Neg^2\DT \PartInv(\Diag).\]
 Moreover, for the given crossing change, the variable $U_{i+1}$
 corresponds to $U_3$ in the output algebra $\Pos^2$ (resp. $\Neg^2$), which is connected to $U_1$ 
 in the output algebra for $\Min^1$, which in turn corresponds to 
 the variable $U$ for $\PartInv({\mathcal D}_+)$ (resp.
 $\PartInv({\mathcal D}_-)$). 

 We claim that the bimodule endomorphisms of $\Min^1\DT \Pos^2$ given by 
 $\Id_{\Min^1}\DT T^{U_3}_{\Pos^2}$ coincides with the endomorphism $T^{U_1}_{\Min^1\DT\Pos^2}$.
 To see this, recall (see for example~\cite[Figure~5]{Bimodules})
 that $\Id_{\Min^1}\DT T^{U_3}_{\Pos^2}$ is defined by
  \[
\mathcenter{
  \begin{tikzpicture}
    \node at (0,0) (tlblank) {$\Min^1$};
    \node at (1.5,0) (trblank) {$\Pos^2$};
    \node at (1.5,-2) (delta1) {$\delta^{\Pos^2}$};
    \node at (1.5,-3) (g) {$t^1$};
    \node at (1.5,-4) (delta2) {$\delta^{\Pos^2}$};
    \node at (0,-5) (deltaM) {$\delta^{\Min^1}$};
    \node at (2.5, -1) (Delta) {$\Delta$};
    \node at (0,-6) (blblank) {};
    \node at (1.5,-6) (brblank) {};
    \node at (2.5,0) (trrblank) {};
    \node at (-1,-6) (bllblank) {};
    \draw[Amodar] (tlblank) to (deltaM);
    \draw[Amodar] (deltaM) to (blblank);
    \draw[Dmodar] (trblank) to (delta1);
    \draw[Dmodar] (delta1) to (g);
    \draw[Dmodar] (g) to (delta2);
    \draw[Dmodar] (delta2) to(brblank);
    \draw[tensorblgarrow, bend right=10] (delta1) to (deltaM);
    \draw[blgarrow] (g) to  node[above,sloped]{\tiny{$U_3$}}  (deltaM);
    \draw[tensorblgarrow] (delta2) to (deltaM);
    \draw[algarrow] (deltaM) to (bllblank);
    \draw[tensorclgarrow] (Delta) to (delta1);
    \draw[tensorclgarrow, bend left=10] (Delta) to (delta2);
    \draw[tensorclgarrow] (trrblank) to (Delta);
  \end{tikzpicture}
}  
  \]        
  By the construction of $\Min^1$, the only non-zero operation
  $\delta^1_\ell(\x,a_1,\dots,a_{\ell-1})$ with some $a_i=U_3$ is 
  $\delta^1_2(\x,U_3)=U_1\otimes \x$; i.e. the only non-zero
  term of the above type is
  \[
\mathcenter{
  \begin{tikzpicture}
    \node at (0,0) (tlblank) {};
    \node at (1.5,0) (trblank) {};
    \node at (1.5,-1) (g) {$t^1$};
    \node at (0,-2) (deltaM) {$\delta^{\Min^1}$};
    \node at (0,-3) (blblank) {};
    \node at (1.5,-3) (brblank) {};
    \node at (2.5,0) (trrblank) {};
    \node at (-1.5,-3) (bllblank) {};
    \draw[Amodar] (tlblank) to (deltaM);
    \draw[Amodar] (deltaM) to (blblank);
    \draw[Dmodar] (trblank) to (g);
    \draw[Dmodar] (g) to (brblank);
    \draw[blgarrow] (g) to  node[above,sloped]{\tiny{$U_3$}}  (deltaM);
    \draw[algarrow] (deltaM) to node[above,sloped]{\tiny{$U_1$}}(bllblank);
  \end{tikzpicture}
},
  \]        
  verifying the claim.

  Feeding the $U_1$ action into $\TerMin$ gives the algebra element $U$. 

  It follows from these considerations that the map
  \[c_-=\Id_{\TerMin}\DT\Id_{\Min^1}\DT \phi_-\DT
  \Id_{\PartInv(\Diag)}\colon \Cwz(\Diag_+)\to\Cwz(\Diag_-),\] composed
  with
  \[c_+=\Id_{\TerMin}\DT\Id_{\Min^1}\DT \phi_+\DT
  \Id_{\PartInv(\Diag)}\colon \Cwz(\Diag_-)\to\Cwz(\Diag_+)\] is chain
  homotopic to multiplication by $U$.
 
 The grading properties of the maps follow immediately from the local descriptions of the maps
 $\phi_+$ and $\phi_-$ from Lemma~\ref{lem:CrossingChange} and the local formulas for the gradings
 from Figure~\ref{fig:LocalCrossing}.
\end{proof}

\subsection{Extracting numerical knot invariants}
\label{subsec:NumericalInvariants}

The chain complexes $\Cwz({\mathcal D})$ lead to natural numerical
knot invariants in the same way that knot Floer homology gives rise to
numerical invariants; see for example~\cite{FourBall}
and~\cite{RasmussenThesis}.  We recall the construction (standard in
knot Floer homology), reformulated slightly to fit with the
perspective of chain complexes over $\OurRing$ stated in the
introduction.

If $W$ is any $\OurRing$-module, and let $W=\bigoplus_{s\in\Z} W_s$ be
its splitting according to its Alexander grading. we can invert $U$ to
form a new module, which we can think of as $W\otimes_{\Field[U]}
\Field[U,U^{-1}]$, equipped with the left $\OurRing$ action. Note that
the action by $V$ on $W\otimes_{\Field[U]}\Field[U,U^{-1}]$ is
trivial.

\begin{lemma}
  \label{lem:NuMakesSense}
  If $W$ is a finitely generated, bigraded $\OurRing$-module,
  with 
  \begin{equation}
    \label{eq:Rank1}
    W\otimes_{\Field[U]} \Field[U,U^{-1}]\cong\Field[U,U^{-1}],
    \end{equation}
  then for all sufficiently large $s$,
  $W_{s}$ is $U$-torsion; 
  and for all sufficiently small $s$,
  there is an element of $W_{s}$ that is not $U$-torsion.
\end{lemma}
\begin{proof}
  Fix a finite graded generating set for $W$, and choose $s_0$ greater
  than $A(\x)$ for any $\x$ in that generating set.  Clearly, for any $s\geq s_0$, $W_s\subset
  V\cdot W_{s-1}$, and therefore $W_s$ contains only $U$-torsion
  elements.

  By classification of finitely generated modules over the PID
  $\Field[U]$, and Equation~\eqref{eq:Rank1}, we can conclude that $W$
  has a single free summand $\Field[U]$, so from our grading
  conventions, it follows that $W_s$ is non-torsion in all
  sufficiently negative grading.
\end{proof}

By Lemma~\ref{lem:NuMakesSense}, we can make the following definition:
\begin{defn}
  Let $W$ be a finitely generated $\OurRing$-module, with the property
  that $W\otimes_{\Field[U]} \Field[U,U^{-1}]\cong \Field[U,U^{-1}]$.
  The {\em $\nu$-invariant of $W$} is the integer
  \[ {\nu}_0(W)=-\max\{s\big| U^d \cdot W_s\neq 0~\forall d\geq 0\}.\]
\end{defn}

\begin{lemma} 
  \label{lem:MonotonicityOfNu}
  If $W^1$ and $W^2$ are two finitely generated
  $\OurRing$-modules with $W^i\otimes_{\Field[U]}\Field[U,U^{-1}]\cong\Field[U,U^{-1}]$,
  and $\phi\colon W^1\to W^2$ is a grading-preserving $\OurRing$-module homomorphism that
  induces an isomorphism after tensoring with $\Field[U,U^{-1}]$, then
  ${\nu}_0(W^1)\geq {\nu}_0(W^2)$.
\end{lemma}
\begin{proof}
  This is an immediate consequence of the fact that a homogeneous, $U$-non-torsion element in $W^1$ has $U$-non-torsion
  image.
\end{proof}

If $C_*$ is a finitely generated chain complex over $\OurRing$,
then since $\OurRing$ is Noetherian, it follows that $H(C_*)$ is also finitely generated as well.
If also $H(C_*)\otimes_{\Field[U]} \Field[U,U^{-1}]\cong\Field[U,U^{-1}]$, we can let
$\nu(C_*)={\nu}_0(H(C_*))$ and 
$\tau(C_*)={\nu}_0(H(C_*/V))$. Moreover, if 
$C^*=\Mor(C,\OurRing)$, we can form
\[ \nu'(C)=\nu(C^*),\qquad \tau'(C)=\tau(C^*).\]
Obviously, if $C^1$ and $C^2$ are quasi-isomorphic as bigraded chain complexes over $\OurRing$, then
${\nu}(C^1)={\nu}(C^2)$, ${\tau}(C^1)={\tau}(C^2)$, 
${\nu}'(C^1)={\nu}'(C^2)$, and
${\tau}'(C^1)={\tau}'(C^2)$.

Observe also that 
${\tau}(C)\leq {\nu}(C)$, 
since the  $\OurRing$-module chain map $C\to C/V$ induces a map on homology satisfying the 
hypotheses of Lemma~\ref{lem:MonotonicityOfNu}, so
that
\[{\nu}(C)={\nu}_0(H(C))
\geq 
{\nu}_0(H(C/V))={\tau}(C).\]

Following~\cite{HomCables}, let
\begin{equation}
  \label{eq:DefEpsilon}
  \epsilon(C)=(\tau(C)-\nu(C))-(\tau'(C)-\nu'(C)).
\end{equation} Using this
function applied to the chain complex computing knot Floer homology,
Hom constructs in~\cite{HomConc} infinitely many linearly independent
smooth concordance homomorphisms to $\Z$ that are non-trivial on
topologically slice knots.

\begin{proof}[Proof of Proposition~\ref{prop:UNonTorsion}]
  According to Proposition~\ref{prop:CrossingChangeHwz},
  $\Hwz(K)\otimes_{\Field[U]} \Field[U,U^{-1}]$ is independent of the choice of knot $K$.
  For $K$ the unknot, it is straightforward to see
  that $\Hwz(K)\cong \OurRing$.
\end{proof}

By Proposition~\ref{prop:UNonTorsion}, we can associate the invariants
$\nu$ and $\tau$ to the complex $\Cwz(\Diag)$, to extract the
corresponding numerical invariants denoted ${\underline \nu}(K)$ and
${\underline\tau}(K)$ in the introduction. We can also define
${\underline \epsilon}(K)$, as in Equation~\eqref{eq:DefEpsilon}
(again, for the chain complex $\Cwz(\Diag)$, to get bordered analogues
of Hom's invariants~\cite{HomCables}.

Proposition~\ref{prop:BoundUnknotting} is an easy consequence of Proposition~\ref{prop:CrossingChangeHwz};
compare~\cite[Chapter~6]{GridBook}:

\begin{proof}[Proof of Proposition~\ref{prop:BoundUnknotting}]
  Since $c_+$ from Proposition~\ref{prop:CrossingChangeHwz} is a bigraded $\OurRing$-module homomorphism
  inducing an isomorphism after we tensor with $\Field[U,U^{-1}]$,
  Lemma~\ref{lem:MonotonicityOfNu} shows that
  ${\underline\nu}(K_+)\geq {\underline\nu}(K_-)$.
  The map $c_-$ is a bigraded homomorphism, too, after introducing a suitable grading shift. This gives the inequality
  ${\underline\nu}(K_-)\geq {\underline\nu}(K_+)-1$.
  Specializing the complexes to $V=0$ gives the analogous bounds for ${\underline\tau}$.
\end{proof}

\newcommand\CwzZ{\Cwz_{\Z}}
\newcommand\HwzZ{\Hwz_{\Z}}
\newcommand\PartInvZ{\PartInv_{\Z}}
\newcommand\Trident{\mathcal T}
\section{Signs}
\label{sec:Signs}

Our aim here is to define a $\Z$ lift of the constructions from this
paper.  First, we define the algebras over $\Z$; then we recall some
generalities on $\Ainfty$ algebras and bimodules with signs.  The
basic bimodules associated to crossings, maxima, and minima are then
defined over $\Z$. Equipped with these definitions, we define
$\CwzZ(\Diag)$ as we defined $\Cwz(\Diag)$ before, only now using the
modules defined over $\Z$; and we can use it as before to define the
invariant of an oriented knot $\orK$; see
Theorem~\ref{thm:InvarianceZ} below.

\subsection{Signs in the algebra}
We define a $\Zmod{2}$ grading on all of the algebras, induced by a
function on homogeneous generators $a$. We call the variables
$C_{\{i,j\}} \in \Alg$ for $\{i,j\}\in\Matching$; and
$C_i\in\Blg(m,k,\Upwards)$ for $i\in\Upwards$, and $E_i\in\Alg'$ for
$i\in \{1,\dots,2n\}$ {\em exterior variables}.  All three algebras
have a mod $2$ grading, called the {\em exterior grading}, which, for
pure algebra elements $a$, counts the number of exterior variable
factors in $a$.  Given $t\in\Zmod{2}$, an algebra element is called
{\em homogeneous} if it can be written as a sum of pure algebra
elements, each of which has exterior grading equal
to $t$.  

The algebras are graded by this exterior grading. This
means that multiplication respects the mod $2$ grading; $d$ reverses
it; and multiplication and differentiation satisfy the Leibniz rule
\[ d(a\cdot b)=(da)\cdot b + (-1)^{|a|}a \cdot (db). \]
Thus, the basic idempotents, and the algebra generators $R_i$
and $L_i$ have even exterior grading; while the exterior grading
of $C_p$ or $E_i$ (in $\Alg$ or $\DuAlg$ respectively) is odd.

In particular, the
algebra $\Blg(2n,k)$ (now defined over $\Z$)
is  in exterior grading $0$. This algebra in turn 
is the quotient of an algebra $\BlgZ(2n,k)$, defined 
exactly as in~\cite[Section~\ref{BK:sec:BlgZ}]{BorderedKnots}, except now
over the base ring $\Z[U_1,\dots,U_{2n}]$. 
Note that the generators $L_i$ and $R_j$ for those algebras satisfy the following commutation 
relations for $|i-j|>1$:
\begin{equation}
  \label{eq:CommutingLsRs}
  \begin{array}{lll}
    L_i \cdot L_j = L_j \cdot L_i,  & L_i \cdot R_j = R_j \cdot L_i,
    & R_i \cdot R_j = R_j \cdot R_i.
    \end{array}
\end{equation}
We will also have some use for alternative generators $L_i'$ and $R_j'$ satisfying the following anti-commutation relations
for all $i,j\in\{1,\dots,2n\}$ with $|i-j|>1$:
\begin{equation}
  \label{eq:AnticommutingLsRs}
  \begin{array}{lll}
 L'_i \cdot L'_j = -L'_j \cdot L'_i,  & L'_i \cdot R'_j = -R'_j \cdot L'_i,
& R'_i \cdot R'_j = -R'_j \cdot R'_i.
\end{array}
\end{equation}
Letting
\[ f(i,\x)=\sum_{\{x\in \x\mid x<i-1\}}x,\]
these generators are related by
\begin{align*}
 \Idemp{\x}\cdot L_i'&= (-1)^{f(i,\x)} \Idemp{\x}\cdot L_i \\
 \Idemp{\x}\cdot R_i'&= (-1)^{f(i,\x)} \Idemp{\x}\cdot R_i.
\end{align*}
Note that $L_i \cdot R_i = L_i'\cdot R_i'$ and $R_i \cdot L_i=R_i'\cdot L_i'$.

With these choices, we also have the 
commutation relations:
\[ U_i\cdot a = a\cdot U_i \]
for all $a\in\Alg$.
Further commutation relations in $\Alg$ are:
\[ C_{\{i,j\}} \cdot a = (-1)^{|a|} a\cdot C_{\{i,j\}}.\]
Similarly, for $\Blg(m,k,\Upwards)$ (which will make a brief appearance 
in this section),
\[ C_i\cdot a = (-1)^{|a|}a \cdot C_i.\]
Finally, in $\DuAlg$,
\[ E_i \cdot E_j = -E_j \cdot E_i\]
if $\{i,j\}\not\in\Matching$;
and
\[ E_i\cdot b = b\cdot E_i\]
if $b\in\Blg(2n,k)$.

The differential in $\Blg(m,k,\Upwards)$ satisfies $dC_i = U_i$, and
differential in $\Alg$ satisfies $d C_{\{i,j\}}= U_i U_j$, and 
the differential in $\Alg'$ satisfies $d E_i = U_i$.

\subsection{Sign conventions for tensor products}
If $X$, $X'$, $Y$, and $Y'$ are $\Zmod{2}$-graded Abelian groups and
$f\colon X\to X'$ and $g\colon Y\to Y'$ are homomorphisms of
Abelian groups with $\Zmod{2}$-grading
$|f|$ and $|g|$ respectively, then
we think of their tensor product
$f\otimes g \colon X\otimes Y \to X'\otimes Y'$
as defined by
\begin{equation}
  \label{eq:TensorMaps}
 (f\otimes g) (x\otimes y)= (-1)^{|g| |x|} f(x) \otimes g(y),
\end{equation}
if
$x\in X$ is homogeneous of degree $|x|$. With this convention, if
$(X,\partial_X)$ and $(Y,\partial_Y)$ are chain complexes,
the endomorphism $\partial_X \otimes \Id_Y + \Id_X \otimes \partial_Y$
gives $X\otimes Y$ the structure of a chain complex.

\subsection{Sign conventions on type $D$ structures}
A graded $D$ module is also equipped with a $\Zmod{2}$ grading, with
the property that the map
\[ \delta^1\colon M_t \to (\Alg\otimes M)_{t-1}.
\]
satisfies the usual the structure relation
\[ (\mu_1\otimes \Id_M) \circ \delta^1 + (\mu_2\otimes \Id_M)\circ
(\Id\otimes \delta^1)\circ \delta^1 = 0. \] Note that certain signs
are introduced by convention from
Equation~\eqref{eq:TensorMaps}. Specifically, if $\{\x_i\}_{i=1}^n$ are generators
for the type $D$ module, and
\[ \delta^1(\x_i)=\sum_{j} a_{i,j}\otimes \x_j,\]
then the structure relations state that
\[ d a_{i,k} + \sum_{j} (-1)^{|a_{i,j}|} a_{i,j}\cdot a_{j,k} = 0.\]
for all pairs $i,k\in\{1,\dots,n\}$.
This map can be iterated to give a map
\[ \delta\colon M \to \Tensor^*\Alg\otimes M; \]
where here $\Tensor^*(\Alg)=\bigoplus_{i=0}^{\infty} \Alg^{\otimes i}$.
Explicitly,
letting
\[ f(a_1,\dots,a_\ell)=\sum_{i=1}^{\ell-1} (\ell-i)\cdot |a_{i}|,\]
we have that
\[ \delta(\x_i) = \sum_{j_1,\dots,j_{\ell}}
(-1)^{f(a_{i,j_1},\dots,a_{j_{\ell-1},j_{\ell}})} a_{i,j_1}\otimes
\dots\otimes a_{j_{\ell-1},j_{\ell}}\otimes \x_{j_{\ell}},\] where the
sum is taken over all sequences $j_1,\dots,j_\ell$ of elements of
$\{1,\dots,n\}$ (including the empty sequence, of length $\ell=0$,
whose contribution is simply $\x_i$).  Note that the sum defining
$\delta$ might have infinitely many non-zero terms; in that case, the
map $\delta$ is to be thought of as a map $\delta\colon M \to
{\overline \Tensor^*(\Alg)} \otimes M$, where
${\overline\Tensor^*(\Alg)}$ is the completion of the tensor
algebra. We will not make this notational distinction in the following
discussion.

Let $\lsup{\Alg}M$ and $\lsup{\Alg}N$ be two type $D$ structures. A morphism
from $M$ to $N$ is a map 
\[ h^1\colon M_t \to (\Alg\otimes N)_{t}.\] The morphism space is equipped with a differential
\[ d h^1 = (\mu_1\otimes \Id_N)\circ h^1 + (\mu_2 \otimes \Id_N) \circ (\Id\otimes \delta^1)\circ h^1 
- (\mu_2 \otimes \Id_N) \circ (\Id\otimes h^1)\circ \delta^1,\]
with sign conventions from Equation~\eqref{eq:TensorMaps}.
A homomorphism is a map $h^1\colon M\to \Alg\otimes N$ as above, with $dh^1=0$.

\subsection{Sign conventions on $\Ainfty$ algebras}
For an $\Ainfty$ algebra, there  is a $\Zmod{2}$ grading so that the maps
satisfy
\[ \mu_\ell \colon (\Alg \otimes \dots\otimes \Alg)_t
\to \Alg_{t+\ell-2}.\]
(We write here $t+\ell-2$; this of course could be written instead as $t+\ell$, since our grading takes values in $\Zmod{2}$. We have chosen
to write instead the expression $t+\ell-2$, which is valid also in the $\Z$-graded case.)
The $\Ainfty$ equation with $n\geq 1$ inputs  takes the form
\begin{equation}
  \label{eq:AinftyWithSigns}
 \sum_{n=r+s+t} (-1)^{r+st} \mu_{r+1+t}(\Id^{\otimes r}\otimes \mu_s \otimes \Id^{\otimes t}) = 0.
\end{equation}
The $\Ainfty$ relation with $n=2$ inputs takes the form
\[ \mu_1 \mu_2 = \mu_2(\mu_1\otimes \Id + \Id\otimes \mu_1);\]
and with $n=3$ inputs it takes the form
\[ \mu_2(\Id \otimes \mu_2 - \mu_2\otimes \Id)
= \mu_1 \mu_3 + \mu_3(\mu_1\otimes \Id \otimes \Id + \Id\otimes \mu_1\otimes \Id + \Id \otimes  \Id\otimes \mu_1).\]

It is easy to check that, with these conventions, the map
\[ \mu' \colon \Tensor^*(\Alg)\to\Tensor^*(\Alg) \]
defined by
\[ \mu'_n(a_1,\dots,a_{n})=
\sum_{n=r+s+t} (-1)^{r+st} \Id^{\otimes r}\otimes \mu_s \otimes \Id^{\otimes t}\]
is a differential; i.e.
\[ \mu'\circ\mu'=0.\]
Of course, the formulas are simpler for a $DG$ algebra, where the only non-zero contributions come from
$s=1$ and $2$.

\subsection{Sign conventions on $\Ainfty$ modules}

The $\Ainfty$ relation (for $\Ainfty$ algebras) has a straightforward generalization to $\Ainfty$ modules.
The maps 
\[m_s\colon M\otimes \overbrace{\Alg\otimes\dots\otimes \Alg}^{s-1}\to M\]
give a map
\[ m\colon M\otimes \Tensor^*(\Alg)\to M.\]
Consider the map
\[ m'\colon M\otimes \Tensor^*(\Alg)\otimes \Tensor^*(\Alg)\to M\otimes \Tensor^*(\Alg) \]
defined by
\begin{equation}
  \label{eq:mPrime}
  m'(\x,(a_1\otimes\dots\otimes a_{s-1})\otimes (a_{s}\otimes\dots a_{s+t-1}))=
(-1)^{st} m_s(\x,a_1\otimes\dots\otimes a_{s-1})\otimes (a_{s}\otimes\dots a_{s+t-1}).
\end{equation}
The $\Ainfty$ relation can be phrased now as the following relation between maps $M\otimes \Tensor^*(\Alg)\to M$:
\[ m\circ m'\circ (\Id_X\otimes \Delta) = m\circ (\Id_M\otimes \mu'), \]
where $\Delta\colon \Tensor^*(\Alg)\to\Tensor^*(\Alg)\otimes\Tensor^*(\Alg)$ is the map induced by
\[
\Delta(a_1\otimes\cdots\otimes
a_i)=\sum_{j=0}^{i}(a_1\otimes\cdots\otimes
a_j)\otimes(a_{j+1}\otimes\cdots\otimes a_i).
\]
The $\Ainfty$ condition is equivalent to the condition that the endomorphism
\begin{equation}
  \label{eq:mComplex}
  \partial_M=m'_M\circ (\Id_M\otimes \Delta)-\Id_M\otimes \mu'  
\end{equation}
is a differential on $M\otimes \Tensor^*(\Alg)$.

Concretely, if $M$ is an $\Ainfty$ module over a DG algebra (i.e. with
$\mu_n=0$ for all $n\geq 3$), the operations
\[ m_{\ell}\colon (M \otimes \Alg^{\otimes {\ell-1}})_t \to M_{t+\ell-2}=M_{t+\ell} \]
satisfy the following $\Ainfty$ relations (for integers $n\geq 1$ and
arbitrary homogeneous elements $x \in M$ and homogeneous $a_1,\dots,a_{n-1}\in \Alg$):
\begin{align*}
  0 =&\sum_{i=0}^{n-1} (-1)^{i}m_{i+1}(m_{n-i}(x,a_1,\dots,a_{n-i-1}),a_{n-i},\dots,a_{n-1}) \\
  &+\sum_{i=1}^{n-2} (-1)^{i} m_{n-1}(x,a_1,\dots,\mu_2(a_i, a_{i+1}),a_{i+2},\dots,a_{n-1}) \\
  &+ (-1)^{n-1}\sum_{i=1}^{n-1} (-1)^{|x|+|a_1|+\dots+|a_{i-1}|}
  m_n(x,a_1,\dots a_{i-1},\mu_1(a_i),a_{i+1},\dots,a_{n-1}).
\end{align*}

Signs for morphisms of $\Ainfty$-modules are constructed similarly. Fix an $\Ainfty$-algebra $\Alg$, and $\Ainfty$ modules $M_{\Alg}$ and $N_{\Alg}$.
A homomorphism
$f \colon M_{\Alg} \to N_{\Alg}$ is a collection of maps
$\{f_i\colon M\otimes \Alg^{\otimes{i-1}}\to N\}_{i=1}^n$ satisfying $\Ainfty$ relations. To phrase those
relations, view $f$ as giving rise to a map
$f\colon M\otimes \Tensor^*(\Alg)\to M$, and 
define 
\[ f'\colon M\otimes\Tensor^*(\Alg)\otimes\Tensor^*(\Alg)\to N\otimes \Tensor^*(\Alg)\] by
\[ f'(\x,(a_1\otimes\dots\otimes a_{s-1})\otimes (a_s\otimes\dots\otimes a_{s+t-1}))
=(-1)^{st} f_s(\x,a_1,\dots,a_{s-1})\otimes (a_s\otimes\dots\otimes a_{s+t-1}).\]
The $\Ainfty$ relation takes the form
\[ (f\circ m'_M)\circ (\Id_X\otimes \Delta)+(m_N\circ f')\circ (\Id_X\otimes \Delta)-f\circ (\Id_X \otimes \mu')=0;\]
i.e.
\[ 
\mathcenter{
\begin{tikzpicture}
  \node at(0,2) (inM) {};
  \node at (2,2) (inAlg) {};
  \node at (1,1) (Delta) {$\Delta$};
  \node at (0,0) (m) {$m'$};
  \node at (0,-1) (phi) {$f$};
  \node at (0,-2) (Mout) {};
  \draw[modarrow] (inM) to (m);
  \draw[tensoralgarrow] (inAlg) to (Delta);
    \draw[modarrow] (m) to (phi);
  \draw[tensoralgarrow] (Delta) to (phi);
  \draw[tensoralgarrow] (Delta) to (m);
  \draw[modarrow] (phi) to (Mout);
\end{tikzpicture}}
+
\mathcenter{
\begin{tikzpicture}
  \node at(0,2) (inM) {};
  \node at (2,2) (inAlg) {};
  \node at (1,1) (Delta) {$\Delta$};
  \node at (0,0) (phi) {$f'$};
  \node at (0,-1) (m) {$m$};
  \node at (0,-2) (Mout) {};
  \draw[modarrow] (inM) to (phi);
  \draw[tensoralgarrow] (inAlg) to (Delta);
    \draw[modarrow] (phi) to (m);
  \draw[tensoralgarrow] (Delta) to (m);
  \draw[tensoralgarrow] (Delta) to (phi);
  \draw[modarrow] (m) to (Mout);
\end{tikzpicture}}
- 
\mathcenter{
\begin{tikzpicture}
  \node at(0,2) (inM) {};
  \node at (2,2) (inAlg) {};
  \node at (1,1) (d) {$\mu'$};
  \node at (0,0) (phi) {$f$};
  \node at (0,-2) (Mout) {};
  \draw[modarrow] (inM) to (phi);
  \draw[tensoralgarrow] (inAlg) to (d);
  \draw[tensoralgarrow] (d) to (phi);
  \draw[modarrow] (phi) to (Mout);
\end{tikzpicture}}=0;\]
or, 
equivalently, the map 
\[ f'\circ (\Id\otimes\Delta)\colon (M\otimes \Tensor^*(\Alg),\partial_M)\to 
(N\otimes \Tensor^*(\Alg),\partial_N)
 \]
is a chain map.

\subsection{The $\Ainfty$ tensor product with signs}
\label{sec:TensorProdSigns}

Given an $\Ainfty$ module $X_{\Alg}$ and a type $D$ structure
$\lsup{\Alg}Y$, we can form their tensor product
$X_{\Alg}\DT\lsup{\Alg}Y$, defining
(as in the unsigned case,  bearing in mind the conventions from
Equation~\eqref{eq:TensorMaps})
\begin{equation}
  \label{eq:DefDBar}
  \partial(\x\otimes \y)=(m\otimes \Id_Y)\circ (\Id_X \otimes \delta).
\end{equation}
As in the unsigned case, a little care must be taken: the {\em a
  priori} infinite sums implicit in the above definition must be
guaranteed to be finite by some boundedness hypothesis.

To make Equation~\eqref{eq:DefDBar} more concrete, suppose that $\lsup{\Alg}Y$ is given with a basis
$\{\y_i\}_{i=1}^n$, so that
\[ \delta^1 (\y_i)=\sum_{j} a_{i,j} \otimes \y_j,\]
and let
\[ e(\x,a_1,\dots,a_{\ell})=\ell|\x|+\sum_{i=1}^{\ell} (\ell-i) |a_i|.\]
Then,
\[ \partial(\x\otimes \y_i)=
\sum_{j_1,\dots,j_{\ell}} (-1)^{e(\x,a_{i,j_1},\dots,a_{j_{\ell-1},j_{\ell}})}
m_{\ell+1}(\x,a_{i,j_1},a_{j_1,j_2},\dots,
a_{j_{\ell-1},j_\ell})\otimes \y_{j_{\ell}};\]
where the sum is taken over all sequences $j_1,\dots,j_\ell$ of elements of $\{1,\dots,n\}$ (including the empty sequence, of 
length $\ell=0$, which contributes $m_1(\x)\otimes \y_i$).

\begin{prop}
  \label{prop:TensorProdSigns}
  If $X_{\Alg}$ is an $\Ainfty$ module and $\lsup{\Alg}Y$ a type $D$
  structure satisfying the needed boundedness conditions
  (e.g. as in Proposition~\ref{prop:AdaptedTensorProd}), then
  the map $\partial$ from Equation~\ref{eq:DefDBar} is a differential.
\end{prop}

\begin{proof}
  The proof that Equation~\eqref{eq:DefDBar} specifies a differential
  follows as in~\cite{BorderedKnots}. For the reader's convenience, we recall the argument here, with a few remarks concerning signs.

  From its construction, $\delta$ satisfies the structure equation
  \[ (\Id_{\Tensor^*(\Alg)}\otimes \delta)\circ \delta + (\Delta\otimes \Id_Y)\circ \delta = 0;\]
  i.e. $\delta$ is a comodule over the coalgebra
  $\Tensor^*(\Alg)$. The structure equation for a type $D$ structure
  is equivalent to the condition
  \[ (\mu'\otimes \Id_Y) \circ \delta = 0;\]
  i.e. $\delta$ is a differential comodule.
  This latter verification makes a straightforward usage of the signs appearing in both $\mu'$ and $\delta$.
  It is similarly easy to verify that:
  \begin{align*} (\Id_X\otimes&\delta^t)\circ (m_s\otimes \Id_Y)(\x\otimes (a_1\otimes\dots\otimes a_{s-1})\otimes \y) \\
  &=(-1)^{st} (m_s\otimes \Id_Y)\circ
  (\Id_{X\otimes \Alg^{\otimes (s-1)}}\otimes\delta^t)(\x\otimes (a_1\otimes\dots\otimes a_{s-1})\otimes \y)
.\end{align*}
 i.e.
 \[ (\Id_X\otimes \delta)\circ (m\otimes \Id_Y)=(m'\otimes \Id_Y)\circ (\Id_{X\otimes \Tensor^*(\Alg)}\circ \delta).\]
 Thus,
  \begin{align*} (m\otimes \Id_Y)&\circ (\Id_X \otimes \delta)\circ
  (m\otimes \Id_Y)\circ (\Id_X \otimes \delta)\\
  &=(m\otimes \Id_Y)\circ
  (m'\otimes \Id_Y)\circ
  (\Id_{X\otimes \Tensor^*(\Alg)}\otimes \delta)\circ
  (\Id_{X}\otimes\delta) \\
  &= (m\otimes \Id_Y)\circ
  (m'\otimes\Id_{Y})\circ
  (\Id_X \otimes \Delta\otimes \Id_Y)\circ (\Id_X\otimes \delta) \\
  &= \Bigg(\Big((m\circ m')\circ (\Id_X\otimes \Delta)
  \Big)\otimes \Id_Y\Bigg) \circ (\Id_X\otimes \delta) 
  \\
  &=((m\circ \mu')\otimes \Id_Y)\circ (\Id_X\otimes \delta) \\
  &=(m\otimes \Id_Y)\circ (\Id_X\otimes \mu'\otimes\Id_Y)\circ (\Id_X\otimes \delta) \\
  &= 0.
  \end{align*}
\end{proof}

\subsection{Bimodules over $\Z$}
A type $DA$ bimodule $\lsup{\Alg}X_{\Blg}$ is a bimodule over the idempotent rings, equipped with a map 
$\delta^1_j\colon X\otimes \Alg^{\otimes j-1}\to \Blg\otimes X$,
which we add up to get a map 
$\delta^1\colon X\otimes \Tensor^*(\Alg)\to \Blg\otimes X$,
satisfying certain $\Ainfty$ relations, which we state now with signs.

Consider the maps
$ X\otimes (\Alg^{\otimes (s-1)})\otimes (\Alg^{\otimes t}) \to \Blg\otimes X \otimes \Alg^{\otimes t}$
defined by
\[ \x\otimes (a_1\otimes\dots\otimes a_{s-1})\otimes (a_s\otimes\dots a_{s+t-1})
\to (-1)^{st}\delta^1_{s}(a_1\otimes\dots\otimes a_{s-1})\otimes (a_s\otimes\dots\otimes a_{s+t-1}).\]
These can be added together to give a map
\[ (\delta^1)'\colon X\otimes \Tensor^*(\Alg)\otimes\Tensor^*(\Alg)\to \Blg\otimes X\otimes \Tensor^*(\Alg).\]
The $\Ainfty$ relations now read:
\[ (\mu_1\otimes \Id_X)\circ \delta^1 - \delta^1\circ (\Id_X\otimes (\mu^\Alg)') +
(\mu_2^{\Blg}\otimes \Id_X)\circ (\Id_{\Blg}\otimes \delta)\circ (\delta^1)' \circ (\Id_X \otimes \Delta)=0,
\]
as maps from $X\otimes \Tensor^*(\Alg)\to\Blg\otimes X$.

Concretely, for any integer $i\geq 1$, 
the $DA$ bimodule relation takes the form
\begin{align*} 0 &= (\mu_1^\Alg\otimes \Id_X)\circ \delta^1_i  \\
&+ (-1)^{i-1} \sum_{j=1}^{i-1} \delta^1_{i}(\Id_{X \otimes \Blg^{j-1}}\otimes \mu_1^{\Blg}\otimes 
\Id_{\Blg^{\otimes (i-j-1)}}) \\
&+ \sum_{j=1}^{i-2} (-1)^{j} \delta^1_{i-1}(\Id_{X \otimes \Blg^{j-1}}\otimes \mu_2^{\Blg}\otimes 
\Id_{\Blg^{\otimes (i-j-2)}}) \\
&+ \sum_{j=1}^i (-1)^{j(i-j)}(\mu_2^{\Alg}\otimes \Id_X)\circ (\Id_{\Blg}\otimes \delta^1_{i-j+1})
\circ (\delta^1_j\otimes \Id_{\Blg^{\otimes (i-j)}}),
\end{align*}
using the conventions of 
compositions from Equation~\eqref{eq:TensorMaps}. 
So, for example, if
$\delta^1_{j}(x,b_1,\dots,b_{j-1})$ contains $a_1\otimes y$, 
and $\delta^1_{i-j+1}(y,b_{j},\dots,b_{i-1})$ contains $a_2\otimes z$,
the corresponding contribution  (to the last term) in the $\Ainfty$ relation with inputs
$(x,b_1,\dots,b_{i-1})$ is
$(-1)^{(|a_1|+j)(i-j+1)} (a_1\cdot a_2) \otimes z$. With this
understood, we can form the $\Ainfty$ tensor products as in
Section~\ref{sec:TensorProdSigns}.

Iterating $\delta^1_j$, we obtain a map
$\lsup{X}\delta \colon X \otimes {\overline \Tensor}^*(\Alg)\to \Tensor^*(\Blg)\otimes X$.

If $\Alg$, $\Blg$, and $\Clg$ are DG algebras, and
$\lsup{\Clg}Y_{\Blg}$, $\lsup{\Blg}X_{\Alg}$ are type $DA$ bimodules,
we define their tensor product
$\lsup{\Clg}Y\DT X_{\Blg}$ to be the DA bimodule with operations given by
$\lsup{Y\DT X}\delta^1 = \lsup{Y}\!\!\delta^1\circ \lsup{X}\!\!\delta$,
when suitable boundedness hypotheses are satisfied.
For example, we have the following signed version of Proposition~\ref{prop:AdaptedTensorProd}:

\begin{prop}
  \label{prop:AdaptedTensorProdZ}
  Let $W_1$ be a disjoint union of finitely many intervals joining
  $Y_1$ to $Y_2$; and let $W_2$ be a disjoint union of finitely many
  intervals joining $Y_2$ to $Y_3$.  Suppose moreover that $W_1\cup
  W_2$ has no closed components, i.e. it is a disjoint union of
  finitely many intervals joining $Y_1$ to $Y_3$.  Given any two
  bimodules $\lsup{\Alg_2}X^1_{\Alg_1}$ and
  $\lsup{\Alg_3}X^2_{\Alg_2}$ adapted to $W_1$ and $W_2$ respectively,
  we can form their tensor product
  $\lsup{\Alg_3}X^2_{\Alg_2}\DT~\lsup{\Alg_2}X^1_{\Alg_1}$ 
  with actions $\lsup{X_2\DT X_1}\delta^1$  as defined above; and moreover, it is a
  bimodule that is adapted to $W_1\cup W_2$.
\end{prop}

\begin{proof}
  The proof from Proposition~\ref{prop:AdaptedTensorProd} shows that
  sums defining $\lsup{X_2\DT X_1}\delta^1$ are in fact finite.  The
  $\Ainfty$ relation holds by a straightforward application of the
  usual rules, applying sign conventions as in the proof of
  Proposition~\ref{prop:TensorProdSigns}.
\end{proof}

\subsection{Homological perturbation lemma over $\Z$}

The homological perturbation lemma plays a central role in the
construction of the bimodule associated to a minimum,
Section~\ref{sec:Min}; where it is used for type $DA$ bimodules.  We
review here how to introduce signs into this basic result, starting
with the case of modules. The proof is standard; see~\cite{Keller} for
a general discussion; see also~\cite[Lemma~9.6]{HFa}.

\begin{lemma}
  \label{lem:HomologicalPerturbation}
  Let $Y_{\Blg}$ be a strictly unital module over $\Blg$
  with grading set
  $\MGradingSet$, and let $Z$ be a chain complex.
  Suppose that there are chain maps $f\colon Z\to Y$ 
  (i.e. as the notation suggests, we are forgetting here about the right $\Blg$-action)
  and $g\colon Y\to Z$
  and a map  $T\colon Y \to~ \!Y$  so that $f$ and $g$ preserve gradings, and $T$ preserves Alexander gradings and shifts $\Delta$ grading by $+1$.
  Suppose moreover that the following identities hold
  \[\begin{array}{lllll}
    g\circ f = \Id_{Z}, & 
    \Id_{Y}+\partial T+T\partial= f\circ g, &
    T\circ T =0, & f\circ T=0, & T\circ g=0.
    \end{array}\]
    Then $Z$ can be turned into
  a strictly unital type $A$ module, denoted $Z_{\Blg}$; and there is an 
  $\Ainfty$ homotopy equivalence 
  $\phi\colon Z_{\Blg}\to  Y_{\Blg}$ with $\phi_{1}=f$. 
\end{lemma}

\begin{proof}
By hypothesis,  $Z$ is already equipped with an action
$m_1$. 
For $j>1$, given ${\underline b}=b_1\otimes \dots \otimes
b_{j-1}$, the operations $m_{j}$ are specified by
the following pictures:
\[   m_+(\x\otimes{\underline{b}})=
    \begin{tikzpicture}[scale=.8,baseline=(x.base)]
        \node at (-1,1) (terminal) {};
        \node at (-1,3) (g) {$g$};
        \node at (-1,4) (m1) {$m_+$};
        \node at (-1,5) (f) {$f$};
        \node at (-1,6) (x) {${\mathbf x}$};
        \node at (0,6) (alg) {${\underline b}$};
        \node at (-2,1) (aout) {};
        \draw[modarrow] (g) to (terminal);
        \draw[othmodarrow] (m1) to (g);
        \draw[modarrow] (x) to (f);
        \draw[othmodarrow] (f) to (m1);
        \draw[tensorblgarrow, bend left=15] (alg) to (m1);
      \end{tikzpicture}
      +
    \begin{tikzpicture}[scale=.8,baseline=(x.base)]
        \node at (-1,0) (terminal) {};
        \node at (-1,2) (g) {$g$};
        \node at (-1,3) (m2) {$m_+$};
        \node at (-1,4) (T1) {$T'$};
        \node at (-1,5) (m1) {$m'_+$};
        \node at (-1,6) (f) {$f$};
        \node at (0,6) (mu) {$\Delta$};
        \node at (-1,7) (x) {${\mathbf x}$};
        \node at (0,7) (alg) {${\underline b}$};
        \node at (-2,0) (aout) {};
        \draw[othmodarrow] (m2) to (g);
        \draw[modarrow] (g) to (terminal);
        \draw[othmodarrow] (T1) to (m2);
        \draw[othmodarrow] (m1) to (T1);
        \draw[modarrow] (x) to (f);
        \draw[othmodarrow] (f) to (m1);
        \draw[tensorblgarrow] (alg) to (mu);
        \draw[tensorblgarrow, bend left=15] (mu) to (m2);
        \draw[tensorblgarrow, bend left=15] (mu) to (m1);
      \end{tikzpicture}
      +
          \begin{tikzpicture}[scale=.8,baseline=(x.base)]
        \node at (-1,-1.5) (terminal) {};
        \node at (-1,0) (g) {$g$};
        \node at (-1,1) (m3) {$m_+$};
        \node at (-1,2) (T2) {$T'$};
        \node at (0,4) (D2) {$\Delta$};
        \node at (-1,3) (m2) {$m'_+$};
        \node at (-1,4) (T1) {$T'$};
        \node at (-1,5) (m1) {$m'_+$};
        \node at (-1,6) (f) {$f$};
        \node at (0,6) (D) {$\Delta$};
        \node at (-1,7) (x) {${\mathbf x}$};
        \node at (0,7) (alg) {${\underline b}$};
        \node at (-2,-1.5) (aout) {};
        \draw[othmodarrow] (m3) to (g);
        \draw[modarrow] (g) to (terminal);
        \draw[othmodarrow] (T2) to (m3);
        \draw[othmodarrow] (m2) to (T2);
        \draw[othmodarrow] (T1) to (m2);
        \draw[othmodarrow] (m1) to (T1);
        \draw[modarrow] (x) to (f);
        \draw[othmodarrow] (f) to (m1);
        \draw[tensorblgarrow] (alg) to (D);
        \draw[tensorblgarrow, bend left=15] (D) to (m1);
        \draw[tensorblgarrow, bend left=15] (D) to (D2);
        \draw[tensorblgarrow, bend left=15] (D2) to (m2);
        \draw[tensorblgarrow, bend left=15] (D2) to (m3);
              \end{tikzpicture}
              + \cdots \] These graphs represent compositions of maps;
              and the primes indicate that maps are weighted with
              signs indicated as in
              Equation~\eqref{eq:AinftyWithSigns}; i.e. $T'\colon
              Y\otimes \Alg^{\otimes t}\to Y\otimes \Alg^{\otimes t}$
              given by $T'(\y\otimes a_1\otimes\dots\otimes
              a_{t})=(-1)^t\cdot T(\y)\otimes (a_1\otimes\dots\otimes
              a_t)$; while $m'$ is as in Equation~\eqref{eq:mPrime};
              and the subscript $+$ indicates that we require there to
              be a non-zero number of algebra inputs into the
              corresponding node. 
              Note that there are no algebra
              elements moving past the $m_+$-labeled nodes, so for those nodes, $m'_+=m_+$.

It follows that $m'\circ (\Id \otimes \Delta)$ is represented by the pictures
\[   m(\x\otimes{\underline{b}})=
    \begin{tikzpicture}[scale=.8,baseline=(x.base)]
        \node at (-1,1) (terminal) {};
        \node at (-1,3) (g) {$g$};
        \node at (-1,4) (m1) {$m'_+$};
        \node at (-1,5) (f) {$f$};
        \node at (0,5) (Delta) {$\Delta$};
        \node at (-1,6) (x) {${\mathbf x}$};
        \node at (0,6) (alg) {${\underline b}$};
        \node at (0,1) (aout) {};
        \draw[modarrow] (g) to (terminal);
        \draw[othmodarrow] (m1) to (g);
        \draw[modarrow] (x) to (f);
        \draw[othmodarrow] (f) to (m1);
        \draw[tensorblgarrow] (alg) to (Delta);
        \draw[tensorblgarrow, bend left=15] (Delta) to (m1);
        \draw[tensoralgarrow, bend left=15] (Delta) to (aout);
      \end{tikzpicture}
      +
    \begin{tikzpicture}[scale=.8,baseline=(x.base)]
        \node at (-1,0) (terminal) {};
        \node at (-1,2) (g) {$g$};
        \node at (-1,3) (m2) {$m_+$};
        \node at (0,4) (D2) {$\Delta$};
        \node at (-1,4) (T1) {$T'$};
        \node at (-1,5) (m1) {$m'_+$};
        \node at (-1,6) (f) {$f$};
        \node at (0,6) (D1) {$\Delta$};
        \node at (-1,7) (x) {${\mathbf x}$};
        \node at (0,7) (alg) {${\underline b}$};
        \node at (0,0) (aout) {};
        \draw[othmodarrow] (m2) to (g);
        \draw[modarrow] (g) to (terminal);
        \draw[othmodarrow] (T1) to (m2);
        \draw[othmodarrow] (m1) to (T1);
        \draw[modarrow] (x) to (f);
        \draw[othmodarrow] (f) to (m1);
        \draw[tensorblgarrow] (alg) to (D1);
        \draw[tensorblgarrow, bend left=15] (D1) to (D2);
        \draw[tensorblgarrow, bend left=15] (D1) to (m1);
        \draw[tensorblgarrow, bend left=15] (D2) to (m2);
        \draw[tensorblgarrow, bend left=15] (D2) to (aout);
      \end{tikzpicture}
      +
          \begin{tikzpicture}[scale=.8,baseline=(x.base)]
        \node at (-1,-1.5) (terminal) {};
        \node at (-1,0) (g) {$g$};
        \node at (-1,1) (m3) {$m_+$};
        \node at (-1,2) (T2) {$T'$};
        \node at (0,4) (D2) {$\Delta$};
        \node at (0,2) (D3) {$\Delta$};
        \node at (-1,3) (m2) {$m'_+$};
        \node at (-1,4) (T1) {$T'$};
        \node at (-1,5) (m1) {$m'_+$};
        \node at (-1,6) (f) {$f$};
        \node at (0,6) (D) {$\Delta$};
        \node at (-1,7) (x) {${\mathbf x}$};
        \node at (0,7) (alg) {${\underline b}$};
        \node at (0,-1.5) (aout) {};
        \draw[othmodarrow] (m3) to (g);
        \draw[modarrow] (g) to (terminal);
        \draw[othmodarrow] (T2) to (m3);
        \draw[othmodarrow] (m2) to (T2);
        \draw[othmodarrow] (T1) to (m2);
        \draw[othmodarrow] (m1) to (T1);
        \draw[modarrow] (x) to (f);
        \draw[othmodarrow] (f) to (m1);
        \draw[tensorblgarrow] (alg) to (D);
        \draw[tensorblgarrow, bend left=15] (D) to (m1);
        \draw[tensorblgarrow, bend left=15] (D) to (D2);
        \draw[tensorblgarrow, bend left=15] (D2) to (m2);
        \draw[tensorblgarrow, bend left=15] (D2) to (D3);
        \draw[tensorblgarrow, bend left=15] (D3) to (m3);
        \draw[tensorblgarrow, bend left=15] (D3) to (aout);
              \end{tikzpicture}
              + \cdots \]

              We must verify that the above specified maps satisfy the
              $\Ainfty$ relation.  To this end, consider 
              $m\circ (\Id\otimes \mu')(\x\otimes {\underline b})$;
              and think of $\mu'$ as labeling some node just below
              the input labeled ${\underline b}$ in some tree above.
              The signs on $T'$, $m'$, and $\mu'$ ensure that if a
              $\mu'$-labeled node is commuted past an $m'$-marked
              node (provided this commutation makes sense; i.e. the
              output of the $\mu'$-marked node does not feed into the
              input of the $m'$-marked one), then the contributions of
              the two trees before and after the contribution add up
              to zero.  Commute each $\mu'$-marked node as far down as
              possible, until it appears immediately above the
              $m'$-marked node into which its output is channeled (an
              even number of commutations), in which case the
              contribution of this tree coincides with the original
              tree in $m\circ (\Id\otimes \mu')$.

              With this observation in place, the proof
              from~\cite[Lemma~9.6]{HFa} applies.  Specifically, for
              trees of the second kind apply the $\Ainfty$ relation to
              the $m'$ labeled node which takes its input from the
              $\mu'$-labeled node.  The $\Ainfty$ relation then
              guarantees that the $m'(\x\otimes \mu'(\underline b))$
              is the count of trees as shown in the definition of
              $m'(\x\otimes{\underline b})$, except that some
              $m'$-labeled node is replaced by two consecutive
              $m'$-labeled nodes, exactly one of which might have no
              algebra inputs.

              Compare with terms in $(m\circ m')\circ (\Id_Y\otimes
              \Delta)$. This latter contribution counts trees as in
              the definition of $m'$, except that some $m$-marked node
              is replaced by a $g$ marked node above an $f$-marked
              node.  Applying the formula $\Id_Y+\partial T +
              T \partial = f\circ g$, we find that the contribution of
              these trees can be thought of as the contributions of
              the following types of trees: trees as in the definition
              of $m'$, only with some $T'$-marked vertex erased (this
              is the contribution of the identity map); or trees as in
              the definition of $m'$, only with one $m'$-marked vertex
              consisting of $m'_1$ (i.e. with no incoming algebra
              element). This is the same as the contributions of
              $m'(\x\otimes \mu'(\underline b))$ from above, verifying
              the $\Ainfty$ relation for the claimed $\Ainfty$ module.

              A morphism $\phi\colon Z \to Y$ is defined by letting
              $\phi_1=f$ and $\phi_k$ with $k>1$ be counts of trees
              similar to the ones used to define $m$, except that the
              final $g$-labeled node is removed. A morphism
              $\psi\colon Y\to Z$ is defined similarly, with
              $\psi_1=g$ and $\psi_k$ with $k>1$ be counts of trees as
              above, only with the initial $f$ node removed.
              Verification of the $\Ainfty$ relations for these
              morphisms proceeds in a similar manner to the above
              proof.
\end{proof}

For example, if $m_k=0$ on $Y$ for $k>2$, then the trees for computing
$m_{j+1}$ on $Z$ are obtained by counting the trees appearing as above
(without the primes on $T$ and $m$); but weighted by $(-1)^{\eta(j)}$,
where
\begin{equation} 
  \label{eq:DefEta}
  \eta(j)=\left\{\begin{array}{ll}
 0 & {\text{if $j\equiv 0~\text{or}~1\pmod{4}$}} \\
 1 & {\text{if $j\equiv 2~\text{or}~3\pmod{4}$,}}
 \end{array}\right.
\end{equation}

In the form we will need it, the lemma has the following form:
\begin{lemma}
  \label{lem:HomologicalPerturbationDA}
  Let $\lsup{\Alg}Y_{\Blg}$ be a strictly unital type $DA$ bimodule
  over algebras $\Alg$ and $\Blg$
  with grading set
  $\MGradingSet$, and let $\lsup{\Alg}Z$ be a type $D$ structure over $\Alg$.
  Suppose that there are chain maps $f\colon \lsup{\Alg}Z\to \lsup{\Alg}Y$ 
  (i.e. again forgetting here about the right $\Blg$-action)
  and $g\colon \lsup{\Alg}Y\to \lsup{\Alg}Z$
  and a type $D$ structure morphism  $T\colon \lsup{\Alg}Y \to~ \lsup{\Alg}Y$  
  so that $f$ and $g$ preserve gradings, and $T$ preserves Alexander gradings and shifts $\Delta$ grading by $+1$.
  Suppose moreover that the following identities hold
  \[\begin{array}{lllll}
    g\circ f = \Id_{Z}, & 
    \Id_{Y}+dT= f\circ g, &
    T\circ T =0 & f\circ T=0 & T\circ g=0.
    \end{array}\]
    Then $Z$ can be turned into
  a strictly unital type $A$ module, denoted $Z_{\Blg}$; and there is an 
  $\Ainfty$ homotopy equivalence 
  $\phi\colon Z_{\Blg}\to  Y_{\Blg}$ with $\phi_{1}=f$. 
\end{lemma}

The proof is as in Lemma~\ref{lem:HomologicalPerturbationDA}, with now a multiplication in $\Alg$ appearing along the leftmost factor;
see~\cite[Lemma~\ref{BK:lem:HomologicalPerturbation2}]{BorderedKnots}.

\subsection{Koszul duality and signs}

The canonical type $DD$ bimodule is defined as in Section~\ref{sec:DefCanonDD}.
As before, its generators correspond to idempotent
states $\x=x_1<\dots<x_{k_1}$ over $\Alg$; or equivalently, the complementary idempotent
$\y=y_1<\dots,<y_{k_2}$. We define the exterior grading of the generator associated to 
the idempontent state $\x$ to be given by  
\begin{equation}
  \label{eq:ModTwoGradeDD}
  \|\y\|=y_1+\dots + y_{k_2}.
\end{equation} 
The algebra element defining the differential is now defined by 
\[
A = \sum_{i=1}^{2n} \left(L_i\otimes R_i' + R_i\otimes L_i'\right) - \sum_{i=1}^{2n}
U_i\otimes E_i + 
\sum_{\{i,j\}\in\Partition} C_{\{i,j\}}\otimes \llbracket E_i,E_j\rrbracket\in
  \Alg\otimes\DuAlg\]
where 
\[ \llbracket E_i,E_j\rrbracket=E_i E_j + E_j E_i.\]
With these choices, $\CanonDD$ is still a type $DD$ bimodule.

Note that the pure algebra elements form a $\Z$-basis for $\Alg$ and $\Alg'$.
Our candidate inverse module 
has the form
\[ \Ynew_{\Alg',\Alg}=\Mor^{\Alg}(\lsub{\Alg'}\Alg'_{\Alg'}\DT
\lsup{\Alg,\Alg'}\CanonDD,\lsup{\Alg}\Id_{\Alg}).\] It is a free
$\Z$-module with spanned by element $({\overline a}|b)$, where
$a\in\Alg'$ is a pure algebra element, $b\in \Alg$, and ${\overline a}\in \Hom(\Alg',\Z)$ takes
$a$ to $1$ and all other pure algebra elements to $0$.

\begin{thm}
  The canonical type $DD$ bimodule is quasi-invertible, in the sense that the tensor product of 
  $\Ynew_{\DuAlg,\Alg}$ with $\lsup{\Alg,\DuAlg}\CanonDD$ over either $\Alg$ or $\DuAlg$
  is homotopy equivalent to the identity type $DA$ bimodules over $\DuAlg$ or $\Alg$ respectively.
\end{thm}

\begin{proof}
To adapt the proof of Theorem~\ref{thm:InvertibleDD}, our goal is to
show that $\Ynew$ is homotopy equivalent to its $\Z$-submodule
generated by idempotents. As in Section~\ref{subsec:Koszul}, this can
be reduced to the case of algebras considered
in~\cite{BorderedKnots}. Namely, we consider algebras
$\Blg_1=\Blg(2n,n,\emptyset)$ and $\Blg_2=\Blg(2n,n,\{1,\dots,2n\})$.
The proof of Lemma~\ref{lem:IdentifySubcomplex} still applies, showing
that $\phi$ is a homotopy equivalence. In adapting the proof, note
that Equation~\eqref{eq:RemainsToCheck} remains true (with $\llbracket
E_i,E_j\rrbracket=E_i\cdot E_j+E_j\cdot E_i$), bearing in mind that
now $\alpha(E_i E_j)=-\alpha(E_j E_i)=C_{i} C_{j}$; so dually
\[ {\overline \alpha}(C_i C_j)=
\overline{E_i \cdot E_j}-{\overline{E_j\cdot E_i}}.\]

Having reduced to the $\Z$ analogue of the invertibility of the
canonical bimodule over $\Blg_1$ and $\Blg_2$, the proof
from~\cite[Theorem~\ref{BK:thm:DDisInvertible}]{BorderedKnots},
showing that $H(\Ynew)$ is a free $\Z$-module generated by the
idempotents in $\Alg$.  Since $\Z$ is free, it follows that $\Ynew$ is
homotopy equivalent to this $\Z$-submodule, as needed.
\end{proof}

\subsection{The $\Ainfty$ tensor product with signs}
Recall that if $\Alg$ and $\Blg$ are $DG$ algebras, $M_{\Alg}$ is a module, and $\lsup{\Alg,\Blg}X$ is 
a type $DD$ bimodule, then we can form their tensor product $(M_{\Alg}\DT\lsup{\Alg,\Blg}X)=\lsup{\Blg}(M\DT X)$.
We explain here how signs work in this construction.

First, $M_{\Alg}$ can be promoted to a type $DA$ bimodule $\lsup{\Blg}M_{\Alg\otimes\Blg}$, 
with the same underlying vector space, and with actions given by
\[ \delta^1_{\ell}(\x,(a_1\otimes b_1),\dots (a_{\ell-1}\otimes b_{\ell-1}))
= (-1)^{s}(b_1\cdots b_{\ell-1}) m_{\ell}(\x,a_1,\dots,a_{\ell-1}),\]
where 
\[ s = |\x|(\sum_{j} |b_j|) + \sum_{i\leq j} |a_i||b_j|.\]

We can think of $M_{\Alg}\DT\lsup{\Alg,\Blg}X$ as given by
$\lsup{\Blg} M_{\Alg\otimes\Blg} \DT \lsup{\Alg\otimes \Blg}X$,
where the signs on the latter tensor product are as specified in Section~\ref{sec:TensorProdSigns}.

For example, if $X_{\Alg}$ is a module and $\lsup{\Alg,\DuAlg}\CanonDD$ is the canonical type 
$DD$ bimodule, and $X$ is a generator with $X\cdot \Idemp{\x}=X$,
then $\partial (X\otimes \gen_\x)$ contains a term of the form $(-1)^{|X|}L_i' \otimes m_2(X,R_i)$
and another one of the form $-E_i\otimes m_2(X,U_i)$.

\subsection{Standard modules over $\Z$}

The notion of standard modules (as in Section~\ref{sec:Standard})
has a natural signed analogue. To state it, we modify Definition~\ref{def:Standard},
replacing Property~\ref{prop:CStandard},
with the 
condition
\begin{equation}
  \label{def:StandardDAZ}
C^2\otimes \x +    \sum_{\ell=0}^{\infty}(-1)^{\ell \cdot |\x|+\frac{\ell^2+\ell}{2}}\delta^1_{1+\ell}(\x,\overbrace{C^1,\dots,C^1}^{\ell})
\in \Blg_2\otimes X.
\end{equation}
With this definition, it is now true that the tensor product of two
standard type $DA$ bimodules is standard, generalizing
Lemma~\ref{lem:StandardTimesStandard} to the signed case.
To see where the sign comes from, note that
\[e=e(\x,\overbrace{C^1,\dots,C^1}^{\ell})=\ell\cdot |\x|+\frac{\ell^2-\ell}{2};\]
and we multiply $(-1)^e$ by $(-1)^{\ell}$, to turn each $C^1$ to $-C^1$.

\subsection{Positive crossing with signs}
\label{subsec:CrossSigns}
The exterior grading on the algebra is compatible with an exterior
grading on all the bimodules considered earlier. 

For example, consider the $DD$ bimodule of a positive crossing from
Subsection~\ref{subsec:DDcross},
$\lsup{\Alg_1,\DuAlg_2}\Pos_i$. Generators of $\Pos_i$ are determined
by an idempotent state $\y$ in $\DuAlg_2$ and a label, which can be
$\North$, $\South$, $\West$, or $\East$.  
Let 
\begin{equation}
  \label{eq:GradeCross}
 |\North|=|\West|=|\East|=0\qquad 
|\South|=-1.
\end{equation}
Let the exterior grading of each generator $X$ with $(1\otimes \Idemp{\y})\cdot X=X$
be given by $\|\y\|+|X|$, 
where the terms are defined in Equations~\eqref{eq:ModTwoGradeDD} and~\eqref{eq:GradeCross} respectively.
To obtain a type $DD$ bimodule, 
the diagram from 
Equation~\eqref{eq:PositiveCrossing} is modified as follows:
\[
    \begin{tikzpicture}[scale=1.8]
    \node at (0,3) (N) {$\North$} ;
    \node at (-2,2) (W) {$\West$} ;
    \node at (2,2) (E) {$\East$} ;
    \node at (0,1) (S) {$\South$} ;
    \draw[->] (S) [bend left=7] to node[below,sloped] {\tiny{$-R_i\otimes U_{i+1}-L_{i+1}\otimes R'_{i+1}R'_i$}}  (W)  ;
    \draw[->] (W) [bend left=7] to node[above,sloped] {\tiny{$-L_{i}\otimes 1$}}  (S)  ;
    \draw[->] (E)[bend right=7] to node[above,sloped] {\tiny{$R_{i+1}\otimes 1$}}  (S)  ;
    \draw[->] (S)[bend right=7] to node[below,sloped] {\tiny{$L_{i+1}\otimes U_i + R_i \otimes L'_{i} L'_{i+1}$}} (E) ;
    \draw[->] (W)[bend right=7] to node[below,sloped] {\tiny{$1\otimes L_i'$}} (N) ;
    \draw[->] (N)[bend right=7] to node[above,sloped] {\tiny{$U_{i+1}\otimes R_i' + R_{i+1} R_i \otimes L_{i+1}'$}} (W) ;
    \draw[->] (E)[bend left=7] to node[below,sloped]{\tiny{$1\otimes R_{i+1}'$}} (N) ;
    \draw[->] (N)[bend left=7] to node[above,sloped]{\tiny{$U_{i}\otimes L_{i+1}' + L_{i} L_{i+1}\otimes R_i'$}} (E) ;
  \end{tikzpicture}
\]
together with outside arrows (connecting generators $X$ of the same type)
$(R_j \otimes L_j'+L_j\otimes R_j')
(-1)^{|X|}$ for all $j\in \{1,\dots,2n\}\setminus \{i,i+1\}$, $-U_{j}\otimes E_{\tau(j)}$
for all $j=1,\dots,2n$, and $C_{\{\alpha,\beta\}}\otimes \llbracket
E_{\tau(\alpha)},E_{\tau(\beta)}\rrbracket$, for all
$\{\alpha,\beta\}\in\Partition$.

For the type $DA$ bimodule, the exterior grading is specified in Equation~\eqref{eq:GradeCross}.
The differential $\delta^1$ is determined by
\[ \delta^1_1(\East)=R_2 \otimes \South, \qquad \delta^1_1(\West)=-L_1
\otimes \South.\] For $b\in\Blg$, the actions $\delta^1_2(\x,b)$ all
give positive multiples (in $\Blg\otimes \Pos^1$).  For
$\delta^1_3(\South,a,b),$ if in $b$ $U_2$ dominates, then the
contribution is positive; while if in $b$ $U_1$ dominates, the sign is
negative.  Recall that in all $\delta^1_3$ actions,
$I(a,I(b,Y))\in\{\East,\West\}$.  When $I(a,I(b,Y))=\East$ then the
sign is $-1$, when $I(a,I(b,Y))=\West$ then the sign is $+1$.

Consider the case where the inputs are not in $\Blg$.
We modify the rules from Equation~\eqref{eq:CEquivariance} as follows:
\begin{align}
  \delta^1_2(X,C_{p}\cdot a_1)&=(-1)^{|X|}C_{\tau(p)}\cdot \delta^1_2(X,a_1) 
  \label{eq:CEquivarianceZ} \\
  \delta^1_3(\South,a_1 \cdot C_{p},a_2)&=
  \delta^1_3(\South,a_1,C_{p}\cdot a_2) \\
    \delta^1_3(\South,C_{p}\cdot a_1, a_2)&=
    C_{\tau(p)}\cdot   \delta^1_3(\South,a_1,a_2). \nonumber
\end{align}
for all $p\in \Matching$.
When $\{1,2\}\not\in\Matching$, 
the actions must be further modified by adding the following terms
(compare Equation~\eqref{ExtendingCs}):
\begin{equation}
  \begin{array}{lll}
  \delta^1_2(\South,C_{2})\rightarrow -U_\beta R_1\otimes \West & \\
  \delta^1_2(\South,C_{1})\rightarrow U_\alpha L_2\otimes \East &
  \delta^1_2(\South,C_{1} C_{2})\rightarrow U_{\beta}{\widetilde C}_{2} R_1\otimes \West + U_{\alpha}{\widetilde C}_{1} L_2\otimes \East \\
  \delta^1_{2}(\South,U_1C_{2})\rightarrow -U_{\beta}U_1 L_2 \otimes \East &
  \delta^1_{2}(\South,U_1C_{1} C_{2})\rightarrow U_{\beta} U_1 {\widetilde C}_{2} L_2 \otimes \East \\
  \delta^1_{2}(\South,C_{1} U_2)\rightarrow U_{\alpha} R_1  U_2 \otimes \West &
    \delta^1_{2}(\South,C_{1} C_{2} U_2)\rightarrow U_{\alpha}{\widetilde C}_{1} R_1  U_2 \otimes \West \\
  \delta^1_{2}(\South, R_1 C_{2})\rightarrow -U_{\beta}R_1\otimes \North &
  \delta^1_{2}(\South, R_1 C_{1} C_{2})\rightarrow U_{\beta} R_1 {\widetilde C}_{2}\otimes \North \\
  \delta^1_{2}(\South, L_2 C_{1})\rightarrow U_{\alpha}L_2\otimes \North
  &
  \delta^1_{2}(\South, C_{1} L_2 C_{2})\rightarrow U_{\alpha}{\widetilde C}_{1} L_2\otimes \North \\
  \delta^1_{2}(\South, R_1 C_{1} U_2)\rightarrow U_{\alpha}R_1 U_2\otimes \North 
  &
  \delta^1_{2}(\South, R_1 C_{1} U_2 C_{2})\rightarrow U_{\alpha}R_1 {\widetilde C}_{1} U_2\otimes \North \\
  \delta^1_{2}(\South, U_1 L_2 C_{2})\rightarrow -U_{\beta} L_2 U_1\otimes \North & 
  \delta^1_{2}(\South, U_1 C_{1} L_2 C_{2})\rightarrow U_{\beta}U_1 L_2  {\widetilde C}_{2}\otimes \North 
\end{array}
\label{eq:ExtendCsSigns}
\end{equation}

We now have the following:

\begin{prop}
  \label{prop:PosExtZ}
  The operations defined above give $\Pos^i$ the structure of a type $DA$ bimodule (over $\Z$),
  $\lsup{\Alg(n,k,\tau(\Partition))}\Pos^i_{\Alg(n,k,\Partition)}$.
  Moreover, $\Pos^i$ is standard (with signs specified in Equation~\eqref{def:StandardDAZ}).
\end{prop}

The proof of the above is obtained by modifying the proof of Proposition~\ref{prop:PosExt}.
We also have the following:

\begin{prop}
  \label{prop:DualCrossZ}
  Let 
  \[ \Alg_1=\Alg(n,k,\Matching), \qquad
  \Alg_2=\Alg(n,k,\tau_i(\Matching)),\qquad 
  \DuAlg_1=\DuAlg(n,2n+1-k,\Matching).\]
  $\Pos^i$ is dual to $\Pos_i$, in the sense that
  \[ \lsup{\Alg_2}\Pos^i_{\Alg_1}\DT~
    \lsup{\Alg_1,\DuAlg_1}\CanonDD \simeq \lsup{\Alg_2,\DuAlg_1}\Pos_i.\]
\end{prop}

\begin{proof}
  This follows as in Proposition~\ref{prop:DualCross}.
  For example when $i=1$ and $1$ and $2$ are not matched, then 
  we find that $\Pos^1\DT \CanonDD$ is given by:
\[  \begin{tikzpicture}[scale=2]
    \node at (0,4) (N) {${\mathbf N}$} ;
    \node at (-2,2.5) (W) {${\mathbf W}$} ;
    \node at (2,2.5) (E) {${\mathbf E}$} ;
    \node at (0,0) (S) {${\mathbf S}$} ;
    \draw[->] (S) [bend left=10] to node[below,sloped] {\tiny{$U_\beta R_1 \otimes 
        \llbracket E_2,E_\beta\rrbracket
        -R_1 U_2 \otimes E_2 E_1-L_{2}\otimes R_{2}'R_{1}'$}}  (W)  ;
    \draw[->] (W) [bend left=10] to node[above,sloped] {\tiny{$-L_{1}\otimes 1$}}  (S)  ;
    \draw[->] (E) [bend right=10] to node[above,sloped] {\tiny{$R_{2}\otimes 1$}}  (S)  ;
    \draw[->] (S)[bend right=10] to node[below,sloped] {\tiny{$R_{1} \otimes L_{1}' L_{2}' + L_2 U_1\otimes E_1 E_2
        - U_{\alpha} L_2 \otimes \llbracket E_1,E_\alpha\rrbracket$}} (E) ;
    \draw[->] (W)[bend right=10]to node[below,sloped] {\tiny{$1\otimes L_{1}'$}} (N) ;
    \draw[->] (N)[bend right=10] to node[above,sloped] {\tiny{$U_{2}\otimes R_{1}' + R_{2} R_{1} \otimes L_{2}'$}} (W) ;
    \draw[->] (E)[bend left=10]to node[below,sloped]{\tiny{$1\otimes R_{2}'$}} (N) ;
    \draw[->] (N)[bend left=10] to node[above,sloped]{\tiny{$U_{1}\otimes L_{2}' + L_{1} L_{2}\otimes R_{1}'$}} (E) ;
    \draw[->] (N) [loop above] to node[above]{\tiny{$-U_1\otimes E_2 - U_2\otimes E_1$}} (N);
    \draw[->] (W) [loop left] to node[above,sloped]{\tiny{$-U_2\otimes E_{1}$}} (W);
    \draw[->] (E) [loop right] to node[above,sloped]{\tiny{$-U_1\otimes E_{2}$}} (E);
    \draw[->] (E) [bend right=5] to node[above,pos=.3] {\tiny{$-R_2 R_1 \otimes E_2$}} (W) ;
    \draw[->] (W) [bend right=5] to node[below,pos=.3] {\tiny{$-L_1 L_2\otimes E_1$}} (E) ;
    \draw[->] (S) to node[below,sloped,pos=.3] {\tiny{$L_2\otimes E_1 R_2' - R_1\otimes L_1'E_2$}} (N) ;
    \end{tikzpicture}
    \]
  along with self-arrows $(-1)^{|X|}L_t\otimes R_t'$, $(-1)^{|X|}R_t\otimes L_t'$, 
  $-U_t\otimes E_t'$
  for $t\neq 1,2$; 
  and $C_{\{m,\ell\}}\otimes \llbracket E_{\tau(m)},E_{\tau(\ell)}\rrbracket$.
    Consider maps $h^1\colon \Pos^1\DT \CanonDD \to \Pos_1$
    and $g^1\colon \Pos_1\to \Pos^1\DT \CanonDD$ given by
    \begin{align*} h^1(X) &= \left\{\begin{array}{ll}
        \South-(L_2\otimes E_1)\cdot \East + (R_1\otimes E_2)\cdot 
        \West & {\text{if $X=\South$}} \\
        X &{\text{otherwise.}}
      \end{array}
    \right.\\
    g^1(X) &= \left\{\begin{array}{ll}
        \South+(L_2\otimes E_1)\cdot \East - (R_1\otimes E_2)\cdot 
        \West & {\text{if $X=\South$}} \\
        X &{\text{otherwise.}}
      \end{array}
    \right.
    \end{align*}
    It is easy to verify that $h^1$ and $g^1$ are homomorphisms of type $DD$ structures, 
    $h^1\circ g^1 = \Id$, and $g^1\circ h^1=\Id$. 
\end{proof}

\subsection{The maximum with signs}

For the type $DD$ bimodule associated to a maximum
$\lsup{\Alg_2,\DuAlg_1}\Max_c$ from Section~\ref{subsec:MaxDD}, the
generators ${\mathbf P}_\x$ correspond to idempotent states $\x$ for $\Alg_2$.
Let
\begin{equation}
  \label{eq:DefExtGradP}
  \gamma(\x)=\#(\x\cap \{0,\dots,c-1\}),
\end{equation}
and define the exterior grading of ${\mathbf P}_\x$ to be the sum
$\gamma(\x)+\|\y\|$,
where the second term is as in  Equation~\eqref{eq:ModTwoGradeDD}
and $\y$ is specified by $(\Idemp{\x}\otimes \Idemp{\y})\cdot {\mathbf P}_\x={\mathbf P}_\x$
(i.e. $\y=\psi'(\x)$ as in  Equation~\eqref{eq:SpecifyPsiPrime}).
Define 
\[ \epsilon = \sum_{\x} (-1)^{\gamma(\x)}\cdot \Idemp{\x}.\]

Signs in the differential are specified by
\begin{align*}
A&=(L_{c} L_{c+1}\otimes 1) +
(R_{c+1} R_{c}\otimes 1)
 + \sum_{i=1}^{2n} \left(L_{\phi(i)}\otimes R_i' + R_{\phi(i)}\otimes L_i'\right)(\epsilon\otimes 1) \\
& \qquad -  C_{\{c,c+1\}}\otimes 1
-\sum_{i=1}^{2n} U_{\phi(i)}\otimes E_i 
+ \sum_{\{i,j\}\in\Matching} C_{\{\phi(i),\phi(j)\}} \otimes \llbracket E_i,E_j\rrbracket.
  \end{align*}

\begin{lemma} With the above definition, $\lsup{\Alg_2,\Alg'}\Max_c$ is a type $DD$ bimodule over $\Z$,
with the specified exterior grading.
\end{lemma}

\begin{proof}
  The proof is mostly straightforward.  
  To see that terms in the third sum anti-commute with $(L_c L_{c+1}\otimes 1)$, it helps
  to note that
  \[ (L_c L_{c+1}) \cdot (L_i \cdot \epsilon) + (L_i \cdot\epsilon)
  \cdot (L_c L_{c+1})=0\]
  for $i\neq c, c+1$.

  Note also that the algebra elements $U_{\phi(k)}\otimes E_i$ and
  $C_{\{\phi(i),\phi(j)}\otimes \llbracket E_i,E_j\rrbracket$ anti-commute with each other.
\end{proof}

For the $DA$ bimodule, the exterior grading of
the generator ${\mathbf Q}_\x$
associated to an allowed idempotent state $\x$ (for $\Alg_2$) is given by 
$\gamma(\x)$ from Equation~\eqref{eq:DefExtGradP}.

The actions by $\delta^1_1$ as specified in
Equation~\eqref{eq:DeltaOneOnePos} are given the following sign
refinements:
\begin{align*}
\delta^1_1(\XX)=& -C_{\{c,c+1\}}\otimes \XX + R_{c+1}R_{c} \otimes \YY \nonumber \\
\delta^1_1(\YY)=& -C_{\{c,c+1\}}\otimes \YY + L_{c}L_{c+1} \otimes \XX \\
\delta^1_1(\ZZ)= &  -C_{\{c,c+1\}}\otimes \ZZ; \nonumber 
\end{align*}
Similarly, 
the maps $\delta^1_2$ are defined now by
\[ \delta^1_2({\mathbf Q}_{\x},a)=(-1)^{\gamma(\x)|a|}\Phi_{\x}(a)\otimes 
{\mathbf Q}_\z,\] 
where  $\z$ and $\Phi$ are as in Lemma~\ref{lem:ConstructDeltaTwo},
and $a$ is a pure algebra element. For example,
\[ \delta^1_2({\mathbf Q}_\y,C_{\{i,j\}})=(-1)^{\gamma(\y)} C_{\{\phi_c(i),\phi_c(j)\}} \otimes {\mathbf Q}_\y, \]
for $c\not\in\{i,j\}$.

These are the modifications needed for the constructions from
Section~\ref{sec:Max} (notably, Theorem~\ref{thm:MaxDA} and
Proposition~\ref{prop:MaxDual}) to hold over $\Z$.

\subsection{The minimum with signs}

We consider now changes needed to adapt Section~\ref{sec:Min} to work over
$\Z$. 

First, recall that in Section~\ref{subsec:DDmin} we defined a type
bimodule $\lsup{\Alg_2,\Alg_1'}\DDmin_c$ associated to a minimum,
whose generators ${\mathbf P}_\y$ correspond to allowed idempotent
states $\y$ for $\DuAlg_1$.  The exterior grading is 
given by $\|\y\|+\gamma(\y)$ from Equations~\eqref{eq:ModTwoGradeDD} and~\eqref{eq:DefExtGradP} respectively.  To specify the differential, we refine
Equation~\eqref{eq:defDDmin} as follows:
\begin{align*}
A&=(1\otimes L'_{c} L'_{c+1}) + 
(1\otimes R'_{c+1} R'_{c}) 
 + \sum_{j=1}^{2n} \left(R_{j} \otimes L'_{\phi(j)} + L_{j} \otimes R'_{\phi(j)}\right) (1\otimes \epsilon)
 \\
&
  - 1\otimes E_{c} U_{c+1}- \sum_{j=1}^{2n} U_{j}\otimes E_{\phi(j)}   + U_{\alpha}\otimes \llbracket E_{\phi(\alpha)},E_{c}\rrbracket E_{c+1}  \nonumber\\
&  - C_{\{\alpha,\beta\}}\otimes \llbracket E_{\phi(\alpha)},E_{c}\rrbracket \llbracket E_{c+1},E_{\phi(\beta)}\rrbracket + \sum_{\{i,j\} \in \Matching_2\setminus \{\alpha,\beta\}} C_{\{i,j\}}\otimes \llbracket E_{\phi(i)},E_{\phi(j)}\rrbracket.
  \end{align*}
With these choices, it is clear that $\lsup{\Alg_2,\DuAlg_1}\DDmin_c$ is a type $DD$ bimodule over $\Z$.

We can alternatively replace the roles of $c$ with $c+1$ and $\alpha$
with $\beta$ to obtain a homotopy equivalent bimodule.

Next, we turn to the signed version of the $DA$ bimodule of a minimum.

Consider first the larger bimodule given in Section~\ref{subsec:AltConstr}, denoted $M$.
That module is generated by two generators $\XX$ and $\YY$. We specify the exterior grading
by defining $|\XX|=1$ and $|\YY|=0$. (This coincides with $\gamma$.)
Signs are put in the model from Equation~\eqref{eq:ModuleVersionMin}, as follows:
\[
    \mathcenter{\begin{tikzpicture}[scale=1.5]
    \node at (0,0) (X) {$\XX$} ;
    \node at (6,0) (Y) {$\YY$} ;
    \draw[->] (X) [bend left=15] to node[above,sloped] {\tiny{$U_2 - U_{\alpha} \otimes 
C_{1} - C \cdot U_2\otimes C_{1}\cdot C_{2}$}}  (Y)  ;
    \draw[->] (Y) [bend left=15]to node[below,sloped]  {\tiny{$U_1+ U_{\beta} \otimes 
C_{2} + C \cdot U_1\otimes C_{1}\cdot C_{2}$}}  (X);
    \draw[->] (X) [loop] to node[above,sloped]{\tiny{$C\otimes
(C_{1},C_{2})$}} (X);
    \draw[->] (Y) [loop] to node[above,sloped]{\tiny{$C\otimes
(C_{2},C_{1})$}} (Y);
\end{tikzpicture}} \] 
With these signs, the bimodule relations hold; indeed,  homological perturbation theory gives 
the following analogue of Lemma~\ref{lem:BigDAGens2}:

\begin{lemma}
  \label{lem:BigDAGens2s}
  The operations described above make $\lsup{\Alg_2}{\BigMin}_{\Alg_1}$ into a type $DA$ bimodule,
  with the generating set described in Lemma~\ref{lem:BigDAGens}
\end{lemma}
 
\begin{proof} 
  We verify the $\Ainfty$ relations with these sign choices. 
  Half of the non-trivial cases are shown below;
  the remaining half follow the same way:
\begin{align*}
  0&= 
        \mathcenter{\begin{tikzpicture}[scale=.8]
  \node at (0,1) (inD) { };
  \node at (2,1) (inB2) {$C_2$};
  \node at (0,-0.5) (d1) {$\delta^1_1$};
  \node at (0,-2) (d2) {$\delta^1_2$};
  \node at (-1.5,-2.3) (mu) {$\mu_2$};
  \node at (0,-3.5) (outD) {};
  \node at (-3,-3) (outA) {};
  \draw[algarrow] (inB2) to (d2);
  \draw[modarrow] (inD) to node[left]{\tiny{$\XX$}}(d1); 
  \draw[algarrow] (d1) [bend right=15] to  (mu);
  \draw[modarrow] (d1) to node[left]{\tiny{$U_2 \YY$}}(d2);
  \draw[algarrow] (d2) to node[above,sloped]{\tiny{$U_\beta$}}(mu);
  \draw[algarrow] (mu) to node[above,sloped] {\tiny{$U_\beta$}}(outA);
  \draw[modarrow] (d2) to node[left]{\tiny{$U_2 \XX$}} (outD);
\end{tikzpicture}}
      +
      (-1)^{|\XX|}
    \mathcenter{\begin{tikzpicture}[scale=.8]
        \node at (-1,2) (inD) { };
        \node at (-1,-1) (d) {$\delta^1_2$};
        \node at (0.5,.5) (mu) {$\mu_1$};
        \node at (-3,-2) (outA) {};
        \node at (1,2) (inB2) {$C_2$};
        \node at (-1,-2.5) (outD) {};
        \draw[modarrow] (inD) to node[left]{\tiny{$\XX$}} (d); 
        \draw[algarrow] (inB2) to (mu);
        \draw[algarrow] (mu) to node[above,sloped]{\tiny{$U_b$}} (d);
        \draw[modarrow] (d) to node[left]{\tiny{$U_2 \XX$}} (outD);
        \draw[algarrow] (d) to node[above,sloped]{\tiny{$U_\beta$}} (outA);
      \end{tikzpicture}} 
\\
  0&=
        (-1)^{|U_\beta|}\mathcenter{\begin{tikzpicture}[scale=.8]
  \node at (0,1) (inD) { };
  \node at (2,1) (inB2) {$C_2$};
  \node at (0,-0.5) (d1) {$\delta^1_2$};
  \node at (0,-2) (d2) {$\delta^1_1$};
  \node at (-1.5,-2.3) (mu) {$\mu_2$};
  \node at (0,-3) (outD) {};
  \node at (-3,-3) (outA) {};
  \draw[algarrow] (inB2) to (d1);
  \draw[modarrow] (inD) to node[left]{\tiny{$\YY$}}(d1); 
  \draw[algarrow] (d1) [bend right=15] to node[above,sloped] {\tiny{$U_\beta$}} (mu);
  \draw[modarrow] (d1) to node[left]{\tiny{$U_2 \XX$}}(d2);
  \draw[algarrow] (d2) to node[above,sloped]{\tiny{$1$}}(mu);
  \draw[algarrow] (mu) to node[above,sloped] {\tiny{$U_\beta$}}(outA);
  \draw[modarrow] (d2) to node[left]{\tiny{$U_2 \YY$}} (outD);
\end{tikzpicture}}
        -
      (-1)^{|\YY|}
    \mathcenter{\begin{tikzpicture}[scale=.8]
        \node at (-1,2) (inD) { };
        \node at (-1,-1) (d) {$\delta^1_2$};
        \node at (0.5,-.5) (mu) {$\mu_1$};
        \node at (-2.5,-2.5) (outA) {};
        \node at (1,2) (inB2) {$C_2$};
        \node at (-1,-2.5) (outD) {};
        \draw[modarrow] (inD) to node[left]{\tiny{$\YY$}} (d); 
        \draw[algarrow] (inB2) to (mu);
        \draw[algarrow] (mu) to node[above,sloped]{\tiny{$U_2 U_b$}} (d);
        \draw[modarrow] (d) to node[left]{\tiny{$U_2  \YY$}} (outD);
        \draw[algarrow] (d)[bend right=15] to node[above,sloped]{\tiny{$U_\beta$}} (outA);
      \end{tikzpicture}}  
\\
  0&=
        \mathcenter{\begin{tikzpicture}
  \node at (0,1) (inD) { };
  \node at (1,1) (inB1) {$C_1$};
  \node at (2,1) (inB2) {$C_2$};
  \node at (0,-0.5) (d1) {$\delta^1_2$};
  \node at (0,-2) (d2) {$\delta^1_2$};
  \node at (-1.5,-2.3) (mu) {$\mu_2$};
  \node at (0,-3) (outD) {};
  \node at (-3,-3) (outA) {};
  \draw[algarrow] (inB1) to (d1);
  \draw[algarrow] (inB2) to (d2);
  \draw[modarrow] (inD) to node[left]{\tiny{$\XX$}}(d1); 
  \draw[algarrow] (d1) to node[above,sloped] {\tiny{$-U_\alpha$}} (mu);
  \draw[modarrow] (d1) to node[left]{\tiny{$\YY$}}(d2);
  \draw[algarrow] (d2) to node[above,sloped]{\tiny{$U_\beta$}}(mu);
  \draw[algarrow] (mu) to node[above,sloped] {\tiny{$-U_\alpha U_\beta$}}(outA);
  \draw[modarrow] (d2) to node[left]{\tiny{$\XX$}} (outD);
\end{tikzpicture}}
+
        \mathcenter{\begin{tikzpicture}
  \node at (0,1) (inD) { };
  \node at (1,1) (inB1) {$C_1$};
  \node at (2,1) (inB2) {$C_2$};
  \node at (0,-0.5) (d1) {$\delta^1_3$};
  \node at (-1,-1.5) (mu) {$\mu_1$};
  \node at (0,-3) (outD) {};
  \node at (-3,-3) (outA) {};
  \draw[algarrow] (inB1) to (d1);
  \draw[algarrow] (inB2) to (d1);
  \draw[modarrow] (inD) to node[left]{\tiny{$\XX$}}(d1); 
  \draw[algarrow] (d1) to node[above,sloped] {\tiny{$C$}} (mu);
  \draw[modarrow] (d1) to node[left]{\tiny{$\XX$}}(outD);
  \draw[algarrow] (mu) to node[above,sloped] {\tiny{$U_\alpha U_\beta$}}(outA);
\end{tikzpicture}} \\
  0&=
        \mathcenter{\begin{tikzpicture}
  \node at (0,1) (inD) { };
  \node at (1,1) (inB1) {$C_1$};
  \node at (2,1) (inB2) {$C_2$};
  \node at (3,1) (inB3) {$C_1$};
  \node at (0,-0.5) (d1) {$\delta^1_3$};
  \node at (0,-2) (d2) {$\delta^1_2$};
  \node at (-1.5,-2.3) (mu) {$\mu_2$};
  \node at (0,-3) (outD) {};
  \node at (-3,-3) (outA) {};
  \draw[algarrow] (inB1) to (d1);
  \draw[algarrow] (inB2) to (d1);
  \draw[algarrow] (inB3) to (d2);
  \draw[modarrow] (inD) to node[left]{\tiny{$\XX$}}(d1); 
  \draw[algarrow] (d1) to node[above,sloped] {\tiny{$C$}} (mu);
  \draw[modarrow] (d1) to node[left]{\tiny{$\XX$}}(d2);
  \draw[algarrow] (d2) to node[above,sloped]{\tiny{$-U_\alpha$}}(mu);
  \draw[algarrow] (mu) to node[above,sloped] {\tiny{$-C U_\alpha$}}(outA);
  \draw[modarrow] (d2) to node[left]{\tiny{$\YY$}} (outD);
\end{tikzpicture}}
-
(-1)^{|U_{\alpha}|}
        \mathcenter{\begin{tikzpicture}
  \node at (0,1) (inD) { };
  \node at (1,1) (inB1) {$C_1$};
  \node at (2,1) (inB2) {$C_2$};
  \node at (3,1) (inB3) {$C_1$};
  \node at (0,-0.5) (d1) {$\delta^1_2$};
  \node at (0,-2) (d2) {$\delta^1_3$};
  \node at (-1.5,-2.3) (mu) {$\mu_2$};
  \node at (0,-3) (outD) {};
  \node at (-3,-3) (outA) {};
  \draw[algarrow] (inB1) to (d1);
  \draw[algarrow] (inB2) to (d2);
  \draw[algarrow] (inB3) to (d2);
  \draw[modarrow] (inD) to node[left]{\tiny{$\XX$}}(d1); 
  \draw[algarrow] (d1) to node[above,sloped] {\tiny{$-U_\alpha$}} (mu);
  \draw[modarrow] (d1) to node[left]{\tiny{$\YY$}}(d2);
  \draw[algarrow] (d2) to node[above,sloped]{\tiny{$C$}}(mu);
  \draw[algarrow] (mu) to node[above,sloped] {\tiny{$-C U_\alpha$}}(outA);
  \draw[modarrow] (d2) to node[left]{\tiny{$\YY$}} (outD);
\end{tikzpicture}} 
\end{align*}
\begin{align*}
0&=
    (-1)\mathcenter{\begin{tikzpicture}
        \node at (-1,2) (inD) { };
        \node at (-1,-1) (d) {$\delta^1_2$};
        \node at (0,0) (mu) {$\mu_2$};
        \node at (-3,-2) (outA) {};
        \node at (1,2) (inB2) {$C_1$};
        \node at (2,2) (inB3) {$C_2$};
        \node at (-1,-2) (outD) {};
        \draw[modarrow] (inD) to node[left]{\tiny{$\XX$}} (d); 
        \draw[algarrow] (inB2) to (mu);
        \draw[algarrow] (inB3) to (mu);
        \draw[algarrow] (mu) to node[above,sloped]{\tiny{$C_1 C_2$}} (d);
        \draw[modarrow] (d) to node[left]{\tiny{$U_2\YY$}} (outD);
        \draw[algarrow] (d) to node[above,sloped]{\tiny{$-C $}} (outA);
      \end{tikzpicture}} 
    + (-1)^{|C|}
        \mathcenter{\begin{tikzpicture}
  \node at (0,1) (inD) { };
  \node at (1,1) (inB1) {$C_1$};
  \node at (2,1) (inB2) {$C_2$};
  \node at (0,-0.5) (d1) {$\delta^1_3$};
  \node at (0,-2) (d2) {$\delta^1_1$};
  \node at (-1.5,-2.3) (mu) {$\mu_2$};
  \node at (0,-3.5) (outD) {};
  \node at (-3,-3) (outA) {};
  \draw[algarrow] (inB1) to (d1);
  \draw[algarrow] (inB2) to (d1);
  \draw[modarrow] (inD) to node[left]{\tiny{$\XX$}}(d1); 
  \draw[algarrow] (d1) [bend right=15] to node[above,sloped] {\tiny{$C$}} (mu);
  \draw[modarrow] (d1) to node[left]{\tiny{$U_2 \XX$}}(d2);
  \draw[algarrow] (d2) to node[above,sloped]{\tiny{$1$}}(mu);
  \draw[algarrow] (mu) to node[above,sloped] {\tiny{$C$}}(outA);
  \draw[modarrow] (d2) to node[left]{\tiny{$U_2 \YY$}} (outD);
\end{tikzpicture}}
\end{align*}
The above relations hold since $|U_i|=0=|\YY|$, $|\XX|=1=|C|$.
Another consistency check is offered by considering the $\Ainfty$ relation with a single algebra input, that is $C_1 C_2$.
\end{proof}

We modify the the homotopy operator 
$h^1\colon \BigMin \to \BigMin$
considered in Section~\ref{subsec:AltConstr}, to include signs.
The operator is characterized by the property that for pure algebra element $a$,
\begin{align*}
h^1(\XX\otimes a)=\left\{\begin{array}{ll}
      -\YY\otimes a' &{\text{if there is an $a'\in \Gamma$ with $a=U_1 a'$}} \\
      0 &{\text{otherwise}}
      \end{array}\right. \\
    h^1(\YY\otimes a)=\left\{\begin{array}{ll}
      -\XX\otimes a' &{\text{if there is an $a'\in \Gamma$ with $a=U_2 a'$}} \\
      0 &{\text{otherwise,}}
      \end{array}\right. 
\end{align*}
where $\Gamma$ is as in Equation~\eqref{eq:DefGamma}. This operator
satisfies the equation
\[ i\circ \pi =  \Id + \partial \circ h^1 + h^1 \circ \partial.\]

For the purpose of the next lemma, we adapt the associated element from Definition~\ref{def:AssociatedElement}:
\begin{defn}
  \label{def:AssociatedElementSign}
For a preferred sequence, where each pure algebra generator appears with coefficient $+1$,
there is at most one pure algebra generator (again with appearing with coefficient $+1$) $b\in\Blg(2n,k)\subset \Alg_2$,
characterized by the following properties:
\begin{enumerate}[label=(PS-\arabic*),ref=(PS-\arabic*)]
  \item $b=\Idemp{\psi(\x_1)}\cdot b$
  \item 
    For all $i$ with $1\leq i\leq 2n$,
    $ w_{i}(b)=\sum_{j=1}^{m} w_{i+2}(a_j).$
\end{enumerate}
\end{defn}

\begin{lemma}
  \label{lem:MinWithSigns}
  The $\Ainfty$ action of standard sequences 
  on $\Min_1$ are given by
  \begin{align*}
    \delta^1_3(\XX L_1,C_1,C_2)&=C\otimes \XX L_1 \\
    \delta^1_3(\YY R_2,C_2,C_1)&=C\otimes \YY R_2 
  \end{align*}
  with further actions  governed  by the diagram
\begin{equation}
  \label{eq:MinActionsZ}
  \begin{tikzpicture}
    \node at (-3,0) (XL1) {$\XX L_1$} ;
    \node at (3,0) (YR2) {$\YY R_2$} ;
    \node at (0,2) (Y) {$\YY$} ;
    \node at (0,-2) (X) {$\XX$} ;
    \draw[->] (XL1) [bend left=15] to node[above,sloped] {\tiny{$U_1^m R_1
\otimes C_2^{\otimes m}$}} (Y) ;
    \draw[->] (Y) [bend left=30] to node[above,sloped] {\tiny{$U_2^{m+1}\otimes (-C_1)^{\otimes m}$}} (X) ;
    \draw[->] (X) [bend left=15] to node[below,sloped] {\tiny{$ U_1^{m} L_1\otimes C_2^{\otimes m}$}} (XL1) ;
    \draw[->] (XL1) [loop left] to node[above,sloped] {\tiny{$U_1^{m}\otimes C_2^{\otimes m}$}} (XL1);
    \draw[->] (YR2) [bend left=15] to node[below,sloped] {\tiny{$U_2^m L_2\otimes (-C_1)^{\otimes m}$}} (X) ;
    \draw[->] (X) [bend left=30] to node[above,sloped] {\tiny{$U_1^{m+1}\otimes C_2^{\otimes m}$}} (Y) ;
    \draw[->] (Y) [bend left=15] to node[above,sloped] {\tiny{$U_2^{m} R_2\otimes (-C_1)^{\otimes m}$}} (YR2) ;
    \draw[->] (YR2) [loop right] to node[above,sloped] {\tiny{$U_2^{m}\otimes (-C_1)^{\otimes m}$}} (YR2);
    \end{tikzpicture}
\end{equation}
  The contribution of a path of $k$ algebra elements
  in this diagram is multiplied by $(-1)^{\eta(k)+k-1}$, where
  $\eta$ is as in Equation~\eqref{eq:DefEta},
 to give the $\Ainfty$ operations on the type $DA$ bimodule;
 i.e. if $a_1,\dots,a_k$ is a preferred 
 sequence of pure algebra elements
 with $a_1=\Idemp{\x}\cdot a_1$ and $a_k=a_k\cdot \Idemp{\y}$
 and $b$ is a pure algebra element 
 as in Definition~\eqref{def:AssociatedElementSign}, then 
 $\delta^1_{k+1}(\MinGen_{\x},a_1,\dots,a_k)$  is given by
 $(-1)^{\#(C_1\in (a_1,\dots,a_k))+\eta(k)+k-1}b\otimes \MinGen_{\y}$.
\end{lemma}

\begin{proof}
  The homological perturbation lemma (applied to type $DA$ bimodules)
  endows $\XX L_1\oplus \YY R_2$ with the structure of a type $DA$
  bimodule.  To find the sign, observe in the descriptions of the new
  $\delta^1_{k+1}$ actions for $k>1$, only the $\delta^1_2$ actions on
  $M$ appear, and the $h^1$ nodes contribute $(-1)^{\sum_{i=1}^{k-1}
    i}=\eta(k)$; moreover each $\delta^1_{k+1}$ action uses the
  homotopy operator $h^1$ a total of $k-1$ times.
\end{proof}

For example,
\[
\begin{array}{l}
  \delta^1_{3}(\YY R_2,U_2,C_1)= -U_{\alpha} \YY R_2 \\
  \delta^1_{4}(\YY R_2,U_2^2,C_1,C_1)= -U_{\alpha}^2 \YY R_2 \\
  \delta^1_{4}(\YY R_2, L_2, U_1, R_2)= - \YY R_2 \\
  \delta^1_{6}(\YY R_2, L_2, U_1, U_2,U_1, R_2)=  \YY R_2 
\end{array}
\]

These are the modifications needed for the constructions from
Section~\ref{sec:Min} (notably, Theorem~\ref{thm:MinDA} and
Lemma~\ref{lem:MinDual}) to hold over $\Z$.

\subsection{Invariance over $\Z$}

We can now define $\PartInvZ({\mathcal D})$ to be the chain complex
over $\Z$ defined by forming the tensor products of the above
bimodules. If ${\mathcal D}$ denotes the diagram for $\orK$ with the
minimum removed, then $\PartInvZ({\mathcal D})$ can be used to define
a complex denoted $\CwzZ(\orK)$ over $\Z[U,V]/U V=0$ following methods from
Section~\ref{subsec:ConstructInvariant}.

By its construction, $\PartInvZ({\mathcal D})$ and hence also
$\CwzZ(\orK)$ inherits the exterior $\Zmod{2}$ grading, which was
constructed in order to impose a sign convention.  The homology of the
complex, thought of as equipped with an absolute $\Zmod{2}$ grading,
cannot be a knot invariant: it evidently changes under certain
Reidemeister 1 moves. However, as a relative $\Zmod{2}$ grading, it
agrees with the Maslov grading. This can be shown by identify the
generators of the complex with Kauffman states, and proving that both
the exterior and the Maslov gradings change parity under clock
transformations. Since any two Kauffman states can be connected
by
a sequence of clock transformations (see~\cite{Kauffman}) it follows
that the two relative $\Zmod{2}$-gradings agree.

Let $\HwzZ(\orK)$ denote the homology of $\CwzZ(\orK)$, thought of as a module over $\Z[U,V]/U V =0$,
equipped with its Alexander and Maslov gradings.
To see that $\HwzZ(\orK)$ is a knot invariant, we follow the logic from
Section~\ref{sec:Construction}, with a few adaptations.

Theorem~\ref{thm:BraidRelation} can be proved by direct computation,
in the spirit of the proof of the analogous result for the earlier
algebras~\cite[Theorem~\ref{BK:thm:BraidRelation}]{BorderedKnots}.
We leave the details to the interested reader. 

The proof of the trident relations with a maximum,
Proposition~\ref{prop:BasicTrident}, works with the following
adjustments.  Consider
$\lsup{\Alg_3}\Pos^{2}_{\Alg_4}\DT~^{\Alg_4,\DuAlg_1}\Max_{1}$, and
let $X$ be a generator of this tensor product. $X$ can be decomposed
as a tensor product of a generator $\Pos^2$, which has type ($\North$,
$\South$, $\West$, $\East$) and a generator of $\Max_1$, which is
determined by its idempotent $\x$ in $\Alg_4$. 
Let $\y$ denote the corresponding idempotent in $\DuAlg_1$.
The type of $X$ agrees with the type of its $\Pos^2$ factor.
Let
\[ \sigma(X)=
 \left(\begin{array}{rl}
-1 & {\text{if $X$ is of type $\South$}} \\
+1 & {\text{otherwise}}
\end{array}\right)
\cdot
\left(\begin{array}{rl}
-1 &{\text{if $0\in\x$}} \\
1 &{\text{otherwise}}
\end{array}
\right)
\]
Clearly, $\sigma(X)(-1)^{\|\y\|}$ is $(-1)$ to the
mod two grading of $X$. 

The module 
$\lsup{\Alg_3}\Pos^{2}_{\Alg_4}\DT~^{\Alg_4,\DuAlg_1}\Max_{1}$
has 
outside arrows (connecting generators $X$ of the same type) of the 
form $-C_{\{1,3\}}\otimes 1$;
$\sigma \cdot (L_{j+2}\otimes R'_j + R_{j+2}\otimes L'_j)$ 
for all $j=1,\dots 2n$;
$-U_{j+2}\otimes E_{j}$ for all $j=1,\dots,2n$; 
and $C_{\{m+2,\ell+2\}}\otimes \llbracket E_{m},E_{\ell}\rrbracket$
for all $\{m,\ell\}\in\Matching_1$ with $1\not\in\{m,\ell\}$,
and $C_{\{2,\alpha+2\}}\otimes \llbracket E_1,E_{\alpha}\rrbracket$ with
$\{1,\alpha\}\in\Matching_1$.
The arrows connecting different types 
appear in the following sign-refined version of
Equation~\eqref{eq:TensorProdTrident}:
  \begin{equation}
    \label{eq:TensorProdTridentZ}
    \begin{tikzpicture}[scale=1.5]
    \node at (0,4) (N) {${\mathbf N}$} ;
    \node at (-2,2.5) (W) {${\mathbf W}$} ;
    \node at (2,2.5) (E) {${\mathbf E}$} ;
    \node at (0,0) (S) {${\mathbf S}$} ;
    \draw[->] (S) [bend left=10] to node[below,sloped] {\tiny{$U_\alpha R_2 \otimes 
        \llbracket E_\alpha,E_1\rrbracket
        +L_1 L_3 \otimes R'_1$}}  (W)  ;
    \draw[->] (W) [bend left=10] to node[above,sloped] {\tiny{$-L_{2}\otimes 1$}}  (S)  ;
    \draw[->] (E) [bend right=10] to node[above,sloped] {\tiny{$R_{3}\otimes 1$}}  (S)  ;
    \draw[->] (S)[bend right=10] to node[below,sloped] {\tiny{$-U_1 L_3 \otimes 1+R_2 R_1\otimes L'_1$}} (E) ;
    \draw[->] (W)[bend right=10]to node[below,sloped] {\tiny{$R_1\otimes 1$}} (N) ;
    \draw[->] (N)[bend right=10] to node[above,sloped] {\tiny{$L_{1} U_{3}\otimes 1 + R_{3} R_{2} \otimes L'_1$}} (W) ;
    \draw[->] (E)[bend left=10]to node[below,sloped]{\tiny{$1\otimes R'_1$}} (N) ;
    \draw[->] (N)[bend left=10] to node[above,sloped]{\tiny{$U_{2}\otimes L'_1 + L_{1} L_{2} L_{3}\otimes 1$}} (E) ;
    \draw[->] (N) [loop above] to node[above]{\tiny{$-U_2\otimes E_1$}} (N);
    \draw[->] (E) [loop right] to node[above,sloped]{\tiny{$-U_2\otimes E_1$}} (E);
    \draw[->] (E)to node[above,pos=.3] {\tiny{$-R_3 R_2 \otimes E_1$}} (W) ;
    \draw[->] (S) to node[below,sloped,pos=.3] {\tiny{$-R_2 R_1\otimes E_1$}} (N) ;
    \end{tikzpicture}
  \end{equation}
Consider 
the more symmetric bimodule $\lsup{\Alg_3,\DuAlg_1}\Trident$
with the same generators, and differentials
\begin{equation}
\label{eq:SymmWithSign}    \begin{tikzpicture}[scale=1.5]
    \node at (0,4) (N) {${\mathbf N}$} ;
    \node at (-2,2) (W) {${\mathbf W}$} ;
    \node at (2,2) (E) {${\mathbf E}$} ;
    \node at (0,0) (S) {${\mathbf S}$} ;
    \draw[->] (S) [bend left=10] to node[below,sloped] {\tiny{$-R_2\otimes U_1
        +L_1 L_3 \otimes R'_1$}}  (W)  ;
    \draw[->] (W) [bend left=10] to node[above,sloped] {\tiny{$-L_{2}\otimes 1$}}  (S)  ;
    \draw[->] (E) [bend right=10] to node[above,sloped] {\tiny{$R_{3}\otimes 1$}}  (S)  ;
    \draw[->] (S)[bend right=10] to node[below,sloped] {\tiny{$-U_1 L_3 \otimes 1+R_2 R_1\otimes L'_1$}} (E) ;
    \draw[->] (W)[bend right=10]to node[below,sloped] {\tiny{$R_1\otimes 1$}} (N) ;
    \draw[->] (N)[bend right=10] to node[above,sloped] {\tiny{$L_{1} U_{3} \otimes 1 + R_{3} R_{2} \otimes L'_1$}} (W) ;
    \draw[->] (E)[bend left=10]to node[below,sloped]{\tiny{$1\otimes R'_1$}} (N) ;
    \draw[->] (N)[bend left=10] to node[above,sloped]{\tiny{$U_{2}\otimes L'_1 + L_{1} L_{2} L_{3}\otimes 1$}} (E) ;
    \end{tikzpicture}
\end{equation}
equipped with the above self-arrows; compare Equation~\eqref{eq:SymmetricTrident}.
The map
\[ h^1\colon \lsup{\Alg_3}\Pos^2_{\Alg_4}\DT \lsup{\Alg_4,\DuAlg_1}\Max_1 
\to \lsup{\Alg_3,\DuAlg_1}\Trident \]
defined by
\[
    h^1(X) = \left\{\begin{array}{ll}
        \South+(R_2\otimes E_1)\cdot \West \\
        X &{\text{otherwise}}
      \end{array}\right.
\]
gives a homotopy equivalence between the two $DD$ bimodules.

The relation from Proposition~\ref{prop:BasicTrident} now works.

For the proof of the trident relation involving minima,
Lemma~\ref{lem:DuBasicTrident}, we worked in the dual algebra
$\DuAlg$. That algebra can be given signs as explained earlier.
Moreover, as needed in the proof, we can construct the $DA$ bimodule
of crossings over $\DuAlg$.  This is constructed by following the
rules from Section~\ref{subsec:DuAlgCross}, and using the signs
outlined in Section~\ref{subsec:CrossSigns}. So, for example, we have
relations as in Equation~\eqref{eq:CEquivarianceZ} and further
extensions as in Equation~\eqref{eq:ExtendCsSigns}, except now with
$E$-variables instead of $C$-variables.  With these straightforward
changes, we find that, as in the proof of
Lemma~\ref{lem:DuBasicTrident},
$\lsup{\DuAlg_3}\Pos^{2}_{\DuAlg_4}\DT~^{\DuAlg_4,\Alg_1}\Min_1$ is
the bimodule 
whose arrows are obtained from Equation~\eqref{eq:SymmWithSign}, switching
the prime markings from the second to the first tensor factor, 
equipped with
further self-arrows from each generator $X$ to itself
of the form $-U_1 E_3\otimes 1$, $-E_2\otimes
U_1$, $\sigma\cdot L'_{j+2}\otimes R_{j}$ for $j=1,\dots,2n$ 
(with $\sigma$ defined
as before),
$\sigma R'_{j+2}\otimes L_{j}$ for $j=1\dots,2n$, $-E_{j+2}\otimes U_j$
for $j=1,\dots,2n$,
$\llbracket E_{m+2},E_{\ell+2}\rrbracket\otimes C_{\{m,\ell\}}$
for all $\{m,\ell\}\in\Matching_1$; and additional self-arrows of the
form $\llbracket E_{\beta+2},E_{3}\rrbracket\cdot E_1 \otimes U_{\beta}$ and
$-\llbracket E_1,E_{\alpha+2}\rrbracket\cdot\llbracket E_3,E_{\beta+2}\rrbracket\otimes C_{\alpha,\beta}$.
By a natural symmetry of this bimodule, the trident relation involving a minimum holds.

\begin{defn}
  Let $\HwzZ(\orK)$ denote the bigraded module over $\Z[U,V]/U V = 0$, which is obtained by taking the homology of $\CwzZ(\Diag)$.
\end{defn}

\begin{thm}
  \label{thm:InvarianceZ}
  The bigraded module $\HwzZ(\orK)$ is an invariant of the oriented knot $\orK$.
 \qed \end{thm}

Obviously, $\HwzZ(\orK)=H(\CwzZ(\orK))$ refines  
$\Hwz(\orK)=H(\Cwz(\orK))$ considered earlier;
since
\[ \CwzZ(\orK)\otimes_{\Z}{\Zmod{2}}\cong \Cwz(\orK);\] so their
homologies are related by a universal coefficient theorem.

\subsection{Crossing change morphisms over $\Z$}

Proposition~\ref{prop:CrossingChangeHwz} has the following obvious generalization:

\begin{prop}
  \label{prop:CrossingChangeHwzZ}
  Let $K_+$ and $K_-$ be two knots represented by knot projections ${\mathcal D}_+$ and ${\mathcal D}_-$,
  that differ in a single crossing, which is positive  in ${\mathcal D}_+$ and negative in ${\mathcal D}_-$.
  Then, there are maps
  \[ c_+\colon \HwzZ(K_-)\to \HwzZ(K_+)\qquad \text{and}\qquad 
   c_-\colon \HwzZ(K_+)\to \HwzZ(K_-),\]
   where $c_-$ preserves bidegree and $c_+$ is of degree $(-1,-1)$, so that
   $c_+\circ c_-=U$ and $c_-\circ c_+=U$.
\end{prop}

\begin{proof}
  We insert signs into the homomorphisms from Lemma~\ref{lem:CrossingChange}.
  The 
  map $\phi_-\colon \lsup{\Alg_1,\DuAlg_2}\Pos_1\to\lsup{\Alg_1,\DuAlg_2}\Neg_1$ is represented by
  \[
    \begin{tikzpicture}[scale=1.4]
      \node at (0,0) (Wp) {$=\West_+$} ;
      \node at (8,0) (WpR) {$\West_+=$} ;
      \node at (2,0) (Np) {$\North_+$} ;
      \node at (4,0) (Ep) {$\East_+$} ;
      \node at (6,0) (Sp) {$\South_+$} ;
      \node at (0,-2.5) (Wm) {$=\West_-$} ;
      \node at (8,-2.5) (WmR) {$\West_-=$} ;
      \node at (2,-2.5) (Nm) {$\North_-$} ;
      \node at (4,-2.5) (Em) {$\East_-$} ;
      \node at (6,-2.5) (Sm) {$\South_-$} ;

      \draw[->] (Sp) [bend left=7] to node[above,sloped] {\tiny{\hskip.3cm$-R_{1}\!\otimes\! U_{2}-L_{2}\!\otimes\! R'_{2}R'_{1}$}}  (WpR)  ;
      \draw[->] (WpR) [bend left=7] to node[below,sloped] {\tiny{\hskip.3cm$-L_{1}\otimes 1$}}  (Sp)  ;
      \draw[->] (Ep)[bend right=7] to node[below,sloped] {\tiny{$R_{2}\otimes 1$}}  (Sp)  ;
      \draw[->] (Sp)[bend right=7] to node[above,sloped] {\tiny{$L_{2}\!\otimes\! U_{1} + R_{1} \!\otimes\! L'_{1} L'_{2}$}} (Ep) ;
      \draw[->] (Wp)[bend right=7] to node[below,sloped] {\tiny{$1\otimes L'_{1}$}} (Np) ;
      \draw[->] (Np)[bend right=7] to node[above,sloped] {\tiny{$U_{2}\!\otimes\! R'_{1} + R_{2} R_{1} \!\otimes\! L'_{2}$}} (Wp) ;
      \draw[->] (Ep)[bend left=7] to node[below,sloped]{\tiny{$1\otimes R'_{2}$}} (Np) ;
      \draw[->] (Np)[bend left=7] to node[above,sloped]{\tiny{$U_{1}\!\otimes\! L'_{2} + L_{1} L_{2}\!\otimes\! R'_{1}$}} (Ep) ;

    \draw[->] (Wm) [bend right=7] to node[below,sloped] {\tiny{${U_{2}}\!\otimes\!{L'_{1}}+{L_{1} L_{2}}\!\otimes\!{R'_{2}}$}}  (Nm)  ;
    \draw[->] (Nm) [bend right=7] to node[above,sloped] {\tiny{${1}\otimes{R'_{1}}$}}  (Wm)  ;
    \draw[->] (Nm)[bend left=7] to node[above,sloped] {\tiny{${1}\otimes{L'_{2}}$}}  (Em)  ;
    \draw[->] (Em)[bend left=7] to node[below,sloped] {\tiny{${U_{1}}\!\otimes\!{R'_{2}} + {R_{2} R_{1}}\!\otimes\!{L'_{1}}$}} (Nm) ;
    \draw[->] (Sm)[bend left=7] to node[above,sloped] {\tiny{$-{R_{1}}\otimes{1}$}} (WmR) ;
    \draw[->] (WmR)[bend left=7] to node[below,sloped] {\tiny{\hskip.3cm$-{L_{1}}\!\otimes\!{U_{2}} - {R_{2}}\!\otimes\!{L'_{1} L'_{2}}$}} (Sm) ;
    \draw[->] (Sm)[bend right=7] to node[above,sloped]{\tiny{${L_{2}}\otimes{1}$}} (Em) ;
    \draw[->] (Em)[bend right=7] to node[below,sloped]{\tiny{${R_{2}}\!\otimes\!{U_{1}} + {L_{1}}\!\otimes\!{R'_{2} R'_{1} }$}} (Sm) ;

      \draw[->] (Np) to node[above,sloped]{\tiny{$U_{2}\otimes 1$}} (Nm);
      \draw[->] (Sp) to node[above,sloped]{\tiny{$1\otimes U_{2}$}} (Sm);
      \draw[->] (Np) to node[above,sloped,pos=.7]{\tiny{$-R_{2}\otimes L'_{2}$}} (Sm);
      \draw[->] (Sp) to node[above,sloped,pos=.7]{\tiny{$-L_{2}\otimes R'_{2}$}} (Nm);
      \draw[->] (Wp) to node[above,sloped]{\tiny{$1$}} (Wm);
    \end{tikzpicture}
    \]
    (where once again we have suppressed self-arrows of 
    Type~\ref{type:OutsideLRP},~\ref{type:UCP}, and~\ref{type:UCCP}, which now occur with signs;
    e.g. including self-arrows at each vertex labeled by $-U_2\otimes E_1-U_1\otimes E_2$)
    and
    $\phi_+\colon \Neg_1\to\Pos_1$ is represented (with self-arrows suppressed) by
  \[
    \begin{tikzpicture}[scale=1.4]
      \node at (0,-2.5) (Wp) {$=\West_+$} ;
      \node at (8,-2.5) (WpR) {$\West_+=$} ;
      \node at (2,-2.5) (Np) {$\North_+$} ;
      \node at (4,-2.5) (Ep) {$\East_+$} ;
      \node at (6,-2.5) (Sp) {$\South_+$} ;
      \node at (0,0) (Wm) {$=\West_-$} ;
      \node at (8,0) (WmR) {$\West_-=$} ;
      \node at (2,0) (Nm) {$\North_-$} ;
      \node at (4,0) (Em) {$\East_-$} ;
      \node at (6,0) (Sm) {$\South_-$} ;

      \draw[->] (Sp) [bend right=7] to node[below,sloped] {\tiny{-$R_{1}\!\otimes\! U_{2}-L_{2}\!\otimes\! R'_{2}R'_{1}$}}  (WpR)  ;
      \draw[->] (WpR) [bend right=7] to node[above,sloped] {\tiny{$-L_{1}\otimes 1$}}  (Sp)  ;
      \draw[->] (Ep)[bend left=7] to node[above,sloped] {\tiny{$R_{2}\otimes 1$}}  (Sp)  ;
      \draw[->] (Sp)[bend left=7] to node[below,sloped] {\tiny{$L_{2}\!\otimes\! U_{1} + R_{1} \!\otimes\! L'_{1} L'_{2}$}} (Ep) ;
      \draw[->] (Wp)[bend left=7] to node[above,sloped] {\tiny{$1\otimes L'_{1}$}} (Np) ;
      \draw[->] (Np)[bend left=7] to node[below,sloped] {\tiny{$U_{2}\!\otimes\! R'_{1} + R_{2} R_{1} \!\otimes\! L'_{2}$}} (Wp) ;
      \draw[->] (Ep)[bend right=7] to node[above,sloped]{\tiny{$1\otimes R'_{2}$}} (Np) ;
      \draw[->] (Np)[bend right=7] to node[below,sloped]{\tiny{$U_{1}\!\otimes\! L'_{2} + L_{1} L_{2}\!\otimes\! R'_{1}$}} (Ep) ;

    \draw[->] (Wm) [bend left=7] to node[above,sloped] {\tiny{${U_{2}}\!\otimes\!{L'_{1}}+{L_{1} L_{2}}\!\otimes\!{R'_{2}}$}}  (Nm)  ;
    \draw[->] (Nm) [bend left=7] to node[below,sloped] {\tiny{${1}\otimes{R'_{1}}$}}  (Wm)  ;
    \draw[->] (Nm)[bend right=7] to node[below,sloped] {\tiny{${1}\otimes{L'_{2}}$}}  (Em)  ;
    \draw[->] (Em)[bend right=7] to node[above,sloped] {\tiny{${U_{1}}\!\otimes\!{R'_{2}} + {R_{2} R_{1}}\!\otimes\!{L'_{1}}$}} (Nm) ;
    \draw[->] (Sm)[bend right=7] to node[below,sloped] {\tiny{$-{R_{1}}\otimes{1}$}} (WmR) ;
    \draw[->] (WmR)[bend right=7] to node[above,sloped] {\tiny{$-{L_{1}}\!\otimes\!{U_{2}} - {R_{2}}\!\otimes\!{L'_{1} L_{2}}$}} (Sm) ;
    \draw[->] (Sm)[bend left=7] to node[below,sloped]{\tiny{${L_{2}}\otimes{1}$}} (Em) ;
    \draw[->] (Em)[bend left=7] to node[above,sloped]{\tiny{${R_{2}}\!\otimes\!{U_{1}} + {L_{1}}\!\otimes\!{R'_{2} R'_{1} }$}} (Sm) ;
      \draw[->] (Nm)[bend right=7] to node[above,sloped,pos=.7]{\tiny{$1$}} (Np);
      \draw[->] (Wm) to node[above,sloped]{\tiny{$U_{2}\otimes 1$}} (Wp);
      \draw[->] (Em) to node[above,sloped]{\tiny{$U_{1}\otimes 1$}} (Ep);
      \draw[->] (Wm) to node[above,sloped,pos=.7]{\tiny{$L_{1} L_{2}\otimes 1$}} (Ep);
      \draw[->] (Em) to node[above,sloped,pos=.7]{\tiny{$R_{2} R_{1}\otimes 1$}} (Wp);
    \end{tikzpicture}
    \]
It is straightforward to check that $\phi_+$ and $\phi_-$ are $DD$ bimodule homomorphisms.

Letting
\begin{align*}
  h_+(\North_+)&=h_+(\West_+)=h_+(\East_+)=0 \\
  h_+ (\South_+)&=(L_{2}\otimes 1) \otimes \East_+
\end{align*}
and
\begin{align*}
  h_-(\North_-)&=h_-(\West_-)=h_-(\South_-)=0 \\
  h_-(\East_-)&=(R_{2}\otimes 1)\otimes \South_-,
\end{align*}
    it is straightforward to verify that
    \begin{align*}
      d h_+ &=  (U_2\otimes 1)\Id_{\Pos_1}  - \phi_+\circ \phi_-  \\
      d h_- &= (U_2\otimes 1)\Id_{\Neg_1} - \phi_-\circ \phi_+.,
    \end{align*}
\end{proof}

For any prime $p$, we can define 
$\Hwz(\orK,s;\Zmod{p})$ to be the portion of $H(\CwzZ(\orK)\otimes \Z/p\Z)$ in Alexander grading $s$.
We can define invariants ${\underline\nu}_{p}(\orK)$ for any prime $p$ by
\[ {\underline\nu}_p(\orK)=-\max\{s\big| U^d\cdot H(\orK,s;\Zmod{p})\neq 0~~ \forall d\geq 0\}.\]
It would be interesting to find examples of knots $\orK$ for which ${\underline\nu}_p(\orK)\neq {\underline\nu}_q(\orK)$
for $p\neq q$.

\newcommand\AverageD{\gamma}
\section{Fast computations}
\label{sec:Fast}

\subsection{Working with standard $D$ modules}

We introduce a shorthand.  Given a sequence of algebra elements
$(\beta_1,\dots,\beta_k)$ in $\Blg$, we say that a standard sequence
$(a_1,\dots,a_\ell)$ augments the sequence, if it is obtained from
$(\beta_1,\dots,\beta_k)$ by inserting copies of $-C$. 

Let $X$ be a standard type $DA$ module, and let
$\AverageD_{k}(\x,\beta_1,\dots,\beta_{k})$ be a suitable signed count of the sum
of $\delta^1_{k+m}$ over all the ways of extending the sequence
$(\beta_1,\dots,\beta_{k})$ by sequences of $-C$.  Specifically,
\begin{equation}
  \AverageD_k(\x,\beta_1,\dots,\beta_k)=\sum_{\ell=k}^{\infty} (-1)^{\epsilon(|\x|,|a_1|,\dots,|a_\ell|)} \delta^1_{\ell+1}(\x,a_1,\dots,a_{\ell}),
\end{equation}
where the sum is taken over all standard sequences $(a_1,\dots,a_\ell)$ augmenting $(\beta_1,\dots,\beta_k)$.

To justify our use of $\AverageD$, 
note that if $X$ is also adapted to a one-manifold with $W$ with
$\partial W = Y_1 \cup Y_2$ with non-empty $Y_2$, then the sum
defining $\AverageD_k(\x,\beta_1,\dots,\beta_k)$ is finite. If the sum were
not finite, then there would be output algebra elements with
arbitrarily large weight, since $X$ respects the Alexander
multigrading, and $\sum_{p} C_p$ has positive weight everywhere.  But
this contradicts the universal weight bound supplied by the
$\Delta$-grading, which is guaranteed since each term in the sum
contributes the same $\Delta$ grading.

In practice, if $X$ and $Y$ are standard type $DA$ bimodules, then
their $\AverageD$-operations of $X\DT Y$ are determined by the
$\AverageD$-operations of $X$ and $Y$. Moreover, to determine the
chain complex $\Cwz$, we need only understand these
$\AverageD$-operations.  (Retaining only the $\gamma$-operations, we
can think of $X$ as a curved module over $\Blg$, and the tensor product
then corresponds to a suitable curved analogue of the tensor product.
This is the point of
view taken in~\cite{HolKnot}; compare also~\cite{LOTtorus}.)

We describe these data explicitly in the case of $\GenMin_1$. Since the
idempotent of $\XX L_1$ does not appear in the output algebras, we
will describe the part of the module with input $\YY R_2$. The
function $\AverageD_{2k+1}(\YY R_2,\beta_1,\dots,\beta_{2k+1})$ is
determined by the graph:
\begin{equation}
  \label{eq:MinGraph}
  \begin{tikzpicture}
    \node at (3,0) (YR2) {$\YY R_2$} ;
    \node at (0,2) (Y) {$\YY$} ;
    \node at (0,-2) (X) {$\XX$} ;
    \draw[->] (Y) [bend left=30] to node[above,sloped] {\tiny{$U_2^{m+1}$}} (X) ;
    \draw[->] (YR2) [bend left=15] to node[below,sloped] {\tiny{$U_2^m L_2$}} (X) ;
    \draw[->] (X) [bend left=30] to node[above,sloped] {\tiny{$U_1^{m+1}$}} (Y) ;
    \draw[->] (Y) [bend left=15] to node[above,sloped] {\tiny{$U_2^{m} R_2$}} (YR2) ;
    \draw[->] (YR2) [loop right] to node[above,sloped] {\tiny{$U_2^{m}$}} (YR2);
    \end{tikzpicture}
\end{equation}
as given in the following lemma
(compare Equation~\eqref{eq:MinActions}):

The $\gamma$-actions are computed by the following proposition (extended as usual to the full 
algebra as in Definition~\ref{def:AssociatedElement}):
\begin{prop}
  Suppose that $\beta_1,\dots,\beta_{2k+1}$ is a sequence of algebra
  elements labeling the edges in the graph of
  Equation~\eqref{eq:MinGraph}, to give a path from $\YY R_2$ to
  itself, so that the intermediate vertices alternate between $\XX$
  and $\YY$, then letting
  \[ v_1= \left(\sum_{i=1}^{2k+1} w_1(\beta_i)\right)-k; \qquad v_2=\left(\sum_{i=1}^{2k+1} w_2(\beta_i)\right)-k.\]
  \[ \gamma_{2k+1}(\x,\beta_1,\dots,\beta_{2k+1})=
  (-1)^{k} U_{\alpha}^{v_2} U_{\beta}^{v_1}\otimes \x.\] 
\end{prop}
\begin{proof}
  Clearly, $v_1$ and $v_2$ are the number of algebra elements
  $C_2$ and $C_1$ to be added to the sequence to make
  it a sequence as in Equation~\eqref{eq:MinActions}.
  Thus, we see that the output algebra element contains $\pm U_{\alpha}^{v_2} U_{\beta}^{v_1}$,
  as in Definition~\ref{def:AssociatedElement}.

  We wish to compute the sign in
  $\gamma_{2k+1}(\x,\beta_1,\dots,\beta_{2k+1})$.
  By Lemma~\eqref{lem:MinWithSigns}, (suppressing the module generator $\x=\YY R_2$ from the notation),
  \begin{align}
    \delta^1_{j+1}&(\beta_1,(-C_1)^{j_{2k+1}},\beta_2,(-C_2)^{j_{2k}},\dots,\beta_{2k+1},(-C_1)^{j_{1}}) \nonumber \\
    &=(-1)^{\# C_1 + \#C_2} \delta^1_{j+1}(\beta_1,C_1^{j_{2k+1}},\beta_2,C_2^{j_{2k}},\dots,\beta_{2k+1},C_1^{j_{1}}) \nonumber \\
    &= (-1)^{\#C_2 + \eta(j)+ j-1} U_\alpha^{v_2} U_\beta^{v_1}
    = (-1)^{\eta(j)+\# C_1} U_\alpha^{v_2} U_\beta^{v_1}, \label{eq:D1}
  \end{align}
  since $j-1\equiv \# C_1 + \#C_2\pmod{2}$.
  Moreover, since $|\YY R_2|=0=|\beta_i|$, and $|C_i|=1$,
  \begin{align*} \epsilon(\YY &R_2,\beta_1,(-C_1)^{j_{2k+1}},\beta_2,(-C_2)^{j_{2k}},\dots,\beta_{2k+1},(-C_1)^{j_{1}}) \\
  =& 
  \left(\sum_{i=1}^{j-1} i\right) - j_1- (j_1+j_2+1)-(j_1+j_2+j_3+2)+\dots -(j_1+\dots+j_{2k+1}+2k). \\
  \end{align*}
  Note that $\eta(m)\equiv \sum_{i=1}^{m-1}i \pmod{2}$.
  It follows that
  \begin{align}
    \epsilon(\YY R_2,\beta_1,&(-C_1)^{j_{2k+1}},\beta_2,(-C_2)^{j_{2k}},\dots,\beta_{2k+1},(-C_1)^{j_{1}}) \nonumber \\
    &\equiv \eta(j) + \eta(2k+1)+\sum_{i=0}^{k} j_{2i+1}=\eta(j)+\eta(2k+1)+\# C_1. \label{eq:D2}
    \end{align}
    Combining Equations~\eqref{eq:D1}, ~\eqref{eq:D2}, and the observation that $\eta(2k+1)\equiv k\pmod{2}$,
    the result follows.
\end{proof}

\subsection{Contracting arrows}

In practice, we inductively construct the complexes, say, over a field
as follows. The top gives a standard type $D$ module,
tensor that with the next type $DA$ bimodule, to obtain a new standard
type $D$ module. Next, cancel differentials until algebra elements
appearing as the coefficients of $\delta^1(\x)$ all have non-zero
weight, and then proceed. 
More formally:

\begin{defn}
  A $\Delta$-graded standard type $D$ structure $X$ over $\Alg$ is said 
  to be {\em small}, if for any two generators $\x,\y\in X$ with 
  $\Delta(\y)=\Delta(\x)-1$, the $\Alg\otimes \y$ coefficient of
  $\delta^1(\x)=0$.
\end{defn}

Since all algebra elements $a\in\Alg$ with $\Delta(a)=0$ are in the idempotent ring,
the above condition is equivalent to the condition that 
for all $\x\in X$ with $\Idemp{\x}\cdot \x=\x$, 
$\Idemp{\x}\y$ does not appear with non-zero coefficient in $\delta^1(\x)$.

The cancellation is done via the following:

\begin{lemma}
  \label{lem:Small}
  Any finitely generated, $\Delta$-graded standard type $D$ module
  over $\Alg\otimes\Field$ (where $\Field$ is a field; consider
  for example $\Field=\Zmod{p}$ or $\Q$), 
  is homotopy equivalent to a small, finitely generated,
  $\Delta$-graded standard type $D$ module over $\Alg\otimes\Field$.
\end{lemma}

\begin{proof}
  Given a type $D$ module $Y$ with a generating set $\y_1,\dots,\y_m$,
  with the property that 
  $\y_i=\Idemp{\x_i}\cdot \y_i$ for some idempotent state $\x_i$.
  We can write
  \[ \delta^1(\y_i) = \sum_{i,j} a_{i,j}\otimes \y_j,\] where
  $\Idemp{\x_i}\cdot a_{i,j}\cdot \Idemp{\x_j}=a_{i,j}$. 
  If $Y$ is not small, we can assume without loss of generality (after
  renumbering generators $\y_i$ and possibly rescaling some by an 
  element of $\Field$, if needed) that
  $\Idemp{\x_1}=\Idemp{\x_2}=a_{1,2}$.  Then, we can find a new type $D$
  structure $Y'$ with two fewer generators, obtained by canceling
  $\y_1$ and $\y_2$. Explicitly, $Y'$ has generators $\y_3',\dots,\y_m'$
  with $a'_{i,j}= a_{i,j}-a_{i,2} \cdot a_{1,j}$. Here, $Y'$ is the
  submodule of $Y$ with $\y_j'=\y_j - a_{j,2}\cdot \y_1$.  Clearly, if
  $Y$ is standard, then so is $Y'$.
\end{proof}

By applying this cancellation after each step, we gain some control
over the complexity of the calculations.  Software for implementing
this algorithm is included with this preprint submission.

\bibliographystyle{plain}
\bibliography{biblio}

\end{document}